\documentclass{article}
\usepackage[utf8]{inputenc}
\usepackage{comment}
\usepackage{amssymb,amsmath,amsthm,graphicx,xcolor,mathtools,enumerate,mathrsfs}
\usepackage{enumitem,thmtools}
\usepackage{thm-restate}
\usepackage{hyperref}
\usepackage{cleveref}
\usepackage[left=4cm, right=4cm, top=4cm, bottom=2.25cm]{geometry}

\usepackage{tikz}
\usetikzlibrary{math,patterns}
\usetikzlibrary{snakes,calc,decorations.pathmorphing,through}
\tikzset{snake it/.style={decorate, decoration=snake}}

\theoremstyle{plain}
\newtheorem{theorem}{Theorem}[section]
\crefname{theorem}{Theorem}{Theorems}

\newtheorem{proposition}[theorem]{Proposition}
\crefname{proposition}{Proposition}{Propositions}

\newtheorem{corollary}[theorem]{Corollary}
\crefname{corollary}{Corollary}{Corollaries}

\newtheorem{lemma}[theorem]{Lemma}
\crefname{lemma}{Lemma}{Lemmas}

\crefname{conjecture}{Conjecture}{Conjectures}

\crefname{problem}{Problem}{Problem}

\newtheorem{claim}[theorem]{Claim}
\crefname{claim}{Claim}{Claims}

\crefname{observation}{Observation}{Observations}

\crefname{setup}{Setup}{Setups}

\crefname{fact}{Fact}{Facts}

\crefname{algorithm}{Algorithm}{Algorithms}

\crefname{remark}{Remark}{Remarks}

\crefname{example}{Example}{Examples}

\theoremstyle{definition}
\newtheorem{definition}[theorem]{Definition}
\crefname{definition}{Definition}{Definitions}

\crefname{construction}{Construction}{Constructions}

\crefname{question}{Question}{Questions}

\numberwithin{equation}{section}

\newcommand{\red}[1]{#1_{\text{red}}}
\newcommand{\blue}[1]{#1_{\text{blue}}}

\def\eps{\varepsilon}

\renewcommand{\int}[1]{\mathop{\mkern 0mu\mathrm{int}}\nolimits(#1)}

\newcommand\ceil[1]{\left\lceil#1\right\rceil}
\newcommand\floor[1]{\left\lfloor#1\right\rfloor}
\allowdisplaybreaks

\definecolor{DarkDesaturatedBlue}{HTML}{3A3556}
\definecolor{VividOrange}{HTML}{F15918}
\definecolor{PureOrange}{HTML}{FFBA00}
\definecolor{LightGrayishPink}{HTML}{EEC5D5}
\definecolor{VerySoftBlue}{HTML}{B5AFDB}

\usepackage{tikz}
\usetikzlibrary{math}
\usetikzlibrary{calc,decorations.pathmorphing,through}
\tikzset{snake it/.style={decorate, decoration=snake}}

\definecolor{DarkDesaturatedBlue}{HTML}{3A3556}
\definecolor{VividOrange}{HTML}{F15918}
\definecolor{PureOrange}{HTML}{FFBA00}
\definecolor{LightGrayishPink}{HTML}{EEC5D5}
\definecolor{VerySoftBlue}{HTML}{B5AFDB}
  \makeatletter
  \newcommand{\labelinthm}[1]{%
     \label{temp#1}
     \protected@write \@auxout {}{\string \newlabel{#1}{{\emph{\ref{temp#1}}}{\thepage}{\emph{\ref{temp#1}}}{temp#1}{}} }%
  }
  \makeatother
\newcounter{propcounter}

\title{Ramsey numbers of trees}
\author{Richard Montgomery\thanks{Mathematics Institute, University of Warwick, Coventry CV4 7AL, UK. richard.montgomery@warwick.ac.uk}
\and Mat\'ias Pavez-Sign\'e\thanks{Departamento de Ingenier\'ia Matem\'atica, Universidad de Chile, and Centro de Modelamiento Matem\'atico, CNRS  IRL2807, mpavez@dim.uchile.cl} \and Jun Yan\thanks{
Mathematics Institute, University of Oxford, Oxford OX2 6GG, UK. jun.yan@maths.ox.ac.uk.
\newline \vspace{-0.2cm}
\newline RM and MPS supported by the European
Research Council (ERC) under the European Union Horizon 2020 research and innovation programme (grant agreement No.\ 947978). MPS additionally supported by ANID-FONDECYT Regular grant No.\ 1241398 and by ANID Basal Grant CMM FB210005.
JY was supported by the Warwick Mathematics Institute CDT, and by funding from the UK EPSRC (Grant number: EP/W523793/1) when most of this work was done.}}

\begin{document}

	\maketitle

 \begin{abstract}
    We show that there exists a constant $c>0$ such that every $n$-vertex tree $T$ with $\Delta(T)\le cn$ has Ramsey number $R(T)=\max\{t_1+2t_2,2t_1\}-1$, where $t_1\ge t_2$ are the sizes of the bipartition classes of $T$. This improves an asymptotic result of Haxell, \L uczak, and Tingley from 2002, and shows that, though Burr's 1974 conjecture on the Ramsey numbers of trees has long been known to be false for certain `double stars', it is true for trees with up to small linear maximum degree.
 \end{abstract}

\section{Introduction}\label{sec:intro}
The Ramsey number of a graph $G$, denoted by $R(G)$, is the smallest positive integer $N$ such that every red/blue edge colouring of the complete $N$-vertex graph $K_N$ contains a monochromatic copy of $G$. The existence of $R(G)$ follows from Ramsey's foundational result in 1930~\cite{RamseyOnAP}, but determining good bounds on Ramsey numbers has since proved extremely challenging.
The most notorious and natural case is where $G$ is the complete $n$-vertex graph $K_n$. The famous upper bound by Erd\H{o}s and Szekeres~\cite{erdos1935combinatorial} in 1935 and lower bound by Erd\H{o}s~\cite{erdos1947some} in 1947 showed that the rate of growth of $R(K_n)$ is exponential in $n$. Since then, these bounds saw only modest improvements until the recent remarkable breakthrough of Campos, Griffiths, Morris, and Sahasrabudhe~\cite{campos2023} finally gave an exponential improvement to the upper bound of Erd\H{o}s and Szekeres (see also~\cite{balister2024upper,gupta2024optimizing}).

Away from complete graphs, the sparser $G$ is, the more ambitious we can reasonably be in bounding $R(G)$. For example, a classical result of Chv\'atal, R\"odl, Szemer\'edi, and Trotter~\cite{chvatal1983ramsey} from 1983 states that the Ramsey number of every $n$-vertex graph with bounded maximum degree is linear in $n$. That is, for every $\Delta$, there is some $c_\Delta$ such that any $n$-vertex graph $G$ with maximum degree at most $\Delta$ satisfies $R(G)\le c_\Delta n$.  Burr and Erd\H os~\cite{Burr1973ONTM} had conjectured in 1975 that, moreover, this should hold with maximum degree replaced by degeneracy, and this was proved by Lee~\cite{lee2017ramsey} in 2017.

There are not many graphs $G$ for which we can muster any hope of determining $R(G)$ exactly. Aside from the smallest of graphs, the main candidates are trees and cycles.
In 1967, Gerencs\'er and Gy\'arf\'as~\cite{gerencser1967ramsey} determined the Ramsey number of the $n$-vertex path $P_{n-1}$, showing that $R(P_{n-1})=\floor{3n/2}-1$. For the $n$-vertex star $K_{1,n-1}$, note that $R(K_{1,n-1})-1$ is the size of the largest graph such that both it and its complement have maximum degree at most $n-2$. Thus, as shown by Harary~\cite{Harary1972} in 1972, $R(K_{1,n-1})=2n-2$ if $n$ is even, and $R(K_{1,n-1})=2n-3$ if $n$ is odd. The Ramsey number of the $n$-vertex cycle $C_n$ is known due to independent work in the early 1970's by Bondy and Erd\H os~\cite{BONDY197346},  Faudree and Schelp~\cite{faudree1974all}, and Rosta~\cite{rosta1973ramsey}, where we have $R(C_3)=R(C_4)=6$, $R(C_n)=2n-1$ for odd $n\geq5$, and $R(C_n)=3n/2-1$ for even $n\geq6$.

For a general tree $T$, two constructions of Burr~\cite{Burr1974} in 1974 (see Figure~\ref{figure:Burr}) show that if $T$ has  bipartition classes of sizes $t_1$ and $t_2$, where $t_1\geq t_2$, then
\begin{equation}\label{intro:R}R(T)\geq \max\{t_1+2t_2,2t_1\}-1.\end{equation}
From the results quoted above, this bound is tight when $T$ is a path or a star of odd size, and Burr conjectured~\cite{Burr1974} that this bound is tight for every tree $T$ with $t_1\geq t_2\geq2$. However, this was disproved in 1979 by Grossman, Harary, and Klawe~\cite{GROSSMAN1979247} for certain trees called \emph{double stars}. For each $t_1\geq t_2\geq2$, let $S_{t_1,t_2}$ be the tree formed by joining the central vertices of the stars $K_{1,t_1-1}$ and $K_{1,t_2-1}$ with an edge, noting that $S_{t_1,t_2}$ has bipartition classes with sizes $t_1$ and $t_2$.
Grossman, Harary, and Klawe~\cite{GROSSMAN1979247} showed that if $t_1\geq3t_2-2$, then $R(S_{t_1,t_2})=2t_1$, and thus the bound at \eqref{intro:R} is off by 1 in this case.
In 1982, Erd\H os, Faudree, Rousseau, and Schelp~\cite{erdHos1982ramsey} attempted to rescue Burr's conjecture by conjecturing that the bound at \eqref{intro:R} is tight when $t_1=2t_2$. However, this was strongly disproved by Norin, Sun, and Zhao~\cite{norin2016asymptotics} in 2016, who showed in particular that $R(S_{2t,t})\ge (4.2-o(1))t$ (see also~\cite{dubo2025ramsey}), and thus the bound at \eqref{intro:R} can be off by a multiplicative factor.

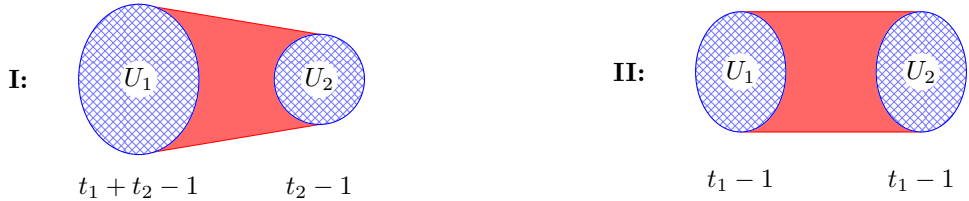
\begin{figure}[h]
\centering
\begin{minipage}{0.49\textwidth}
\centering
\begin{tikzpicture}[scale=.8]
\def\spacer{3};
\def\Ahgt{1.25};
\def\Bhgt{0.75};
\coordinate (U_1) at (0,0);
\coordinate (U_2) at (\spacer,0);

\draw[red,fill=red!60] ($(U_1)+(0,\Ahgt)$) -- ($(U_1)-(0,\Ahgt)$) -- ($(U_2)-(0,\Bhgt)$) -- ($(U_2)+(0,\Bhgt)$) -- cycle;

\draw[white,fill=white] (U_1) circle [y radius=\Ahgt cm,x radius=1cm];
\draw[white,fill=white] (U_2) circle [y radius=\Bhgt cm,x radius=0.75cm];
\draw[blue,pattern=crosshatch, pattern color=blue!50] (U_1) circle [y radius=\Ahgt cm,x radius=1cm];
\draw[blue,pattern=crosshatch, pattern color=blue!50] (U_2) circle [y radius=\Bhgt cm,x radius=0.75cm];
\draw[white,fill=white] (U_1) circle [radius=0.3cm];
\draw[white,fill=white] (U_2) circle [radius=0.3cm];

\draw (U_1) node {$U_1$};
\draw (U_2) node {$U_2$};

\draw ($(U_1)+(-2,0)$) node {\textbf{I:}};
\draw ($(U_1)+(0,-1.8)$) node {$t_1+t_2-1$};
\draw ($(U_2)+(0,-1.8)$) node {$t_2-1$};
\end{tikzpicture}\end{minipage}
\begin{minipage}{0.49\textwidth}
\centering
\begin{tikzpicture}[scale=.8]
\def\spacer{3};
\def\Ahgt{1};
\def\Bhgt{1};
\coordinate (U_1) at (0,0);
\coordinate (U_2) at (\spacer,0);

\draw[red,fill=red!60] ($(U_1)+(0,\Ahgt)$) -- ($(U_1)-(0,\Ahgt)$) -- ($(U_2)-(0,\Bhgt)$) -- ($(U_2)+(0,\Bhgt)$) -- cycle;

\draw[white,fill=white] (U_1) circle [y radius=\Ahgt cm,x radius=0.75cm];
\draw[white,fill=white] (U_2) circle [y radius=\Bhgt cm,x radius=0.75cm];
\draw[blue,pattern=crosshatch, pattern color=blue!50] (U_1) circle [y radius=\Ahgt cm,x radius=0.75cm];
\draw[blue,pattern=crosshatch, pattern color=blue!50] (U_2) circle [y radius=\Bhgt cm,x radius=0.75cm];
\draw[white,fill=white] (U_1) circle [radius=0.3cm];
\draw[white,fill=white] (U_2) circle [radius=0.3cm];
\draw (U_1) node {$U_1$};
\draw (U_2) node {$U_2$};

\draw ($(U_1)+(-1.85,0)$) node {\textbf{II:}};
\draw ($(U_1)+(0,-1.8)$) node {$t_1-1$};
\draw ($(U_2)+(0,-1.8)$) node {$t_1-1$};
\end{tikzpicture}
\end{minipage}
\caption{Burr's extremal constructions for $R(T)$ when $T$ is a tree with bipartition classes of sizes $t_1\geq t_2$. \\
\textbf{I:} Disjoint blue cliques on $U_1$ and $U_2$, with $|U_1|=t_1+t_2-1$, $|U_2|=t_2-1$, and every edge between $U_1$ and $U_2$ coloured red. Any connected blue subgraph has at most $t_1+t_2-1<|T|$ vertices, and any connected red subgraph is bipartite with fewer than $t_2$ vertices in one class.\\
\textbf{II:} Disjoint blue cliques on $U_1$ and $U_2$, with $|U_1|=|U_2|=t_1-1$, and every edge between $U_1$ and $U_2$ coloured red. Any connected blue subgraph has at most $t_1-1<|T|$ vertices, and any connected red subgraph is bipartite with fewer than $t_1$ vertices in each class. \\
Thus, in both \textbf{I} and \textbf{II} there is no monochromatic copy of $T$.}
\label{figure:Burr}
\end{figure}

All the known counterexamples to Burr's conjecture, however, have large maximum degree, and thus the bound at \eqref{intro:R} may still be tight for trees with small maximum degree.
Towards this, Haxell, \L uczak, and Tingley~\cite{haxell2002ramsey} showed in 2002 that the bound at \eqref{intro:R} is approximately tight for trees with up to small linear maximum degree. That is, they showed that, for every $\eps>0$, there exists some $c>0$ such that any $n$-vertex tree $T$ with maximum degree $\Delta(T)\leq cn$ and bipartition classes of sizes $t_1\geq t_2$ satisfies $R(T)\le(1+\eps)\max\{t_1+2t_2,2t_1\}$.

In this paper, we will show that Burr's bound at \eqref{intro:R} is tight for all trees with up to small linear maximum degree, as follows.
\begin{theorem}\label{thm:main} There exists a constant $c>0$ such that the following holds.
Any $n$-vertex tree $T$ with $\Delta(T)\le cn$ and bipartition classes of sizes $t_1\geq t_2$ satisfies $R(T)=\max\{2t_1,t_1+2t_2\}-1.$
\end{theorem}

The existence of such a constant $c$ in Theorem~\ref{thm:main} answers in the positive a question asked explicitly by Stein~\cite{stein2020tree} in 2020. Our value of $c$ is very small due to the use of regularity methods, and is likely very far from optimal. It follows from the double star examples given in~\cite{norin2016asymptotics} that $c$ cannot be improved beyond $7/11+o(1)$. For a tree $T$ with large maximum degree we do not have a good conjecture for the exact value of $R(T)$, though Burr and Erd\H os~\cite{burr1976extremal} conjectured in 1976 that for any $n$-vertex tree $T$, $R(T)\leq2n-2$ when $n$ is even and $R(T)\leq2n-3$ when $n$ is odd, or in other words $R(T)\leq R(K_{1,n-1})$.
In 2011, Zhao~\cite{zhao2011proof} showed that this is true for all large even $n$, as a consequence of his resolution for large $n$ of Loebl's $n/2-n/2-n/2$ conjecture~\cite{erdHos1995discrepancy}. Burr and Erd\H os's conjecture follows directly from the Erd\H os-S\'os conjecture, and thus for large $n$ follows from the proof of the Erd\H os-S\'os conjecture for large trees announced by
 Ajtai, Koml\'os, Simonovits, and Szemer\'edi in the early 1990s (see~\cite{pokrovskiy2024hyperstability} for a discussion of this result).
A wide-ranging discussion of further results on Ramsey numbers can be found in the dynamic survey by Radziszowski~\cite{radziszowski2012small}.

To prove Theorem~\ref{thm:main} we will conduct what is known as a stability analysis. When the red/blue coloured host graph is not close to one of the extremal constructions in Figure~\ref{figure:Burr}, we will develop the work of Haxell, \L uczak, and Tingley~\cite{haxell2002ramsey}, and find a monochromatic copy of our tree $T$ using methods involving Szemer\'edi's regularity lemma. If instead the colouring is close to an extremal construction, then we will analyse the structure more closely to still find a monochromatic copy of $T$, often using randomised embeddings. We call these two parts of the proof the `stability part' and `extremal part', and both of them will be rather involved. For the stability part of the argument, we will be able to start from a certain monochromatic structure found by Haxell, \L uczak, and Tingley~\cite{haxell2002ramsey} in the reduced graph, but will still need to do much more work to cover all the non-extremal cases, with the main problem to overcome being a deficit of vertices in this initial structure. For the extremal part, proving Theorem~\ref{thm:main} when the colouring approximates either extremal construction turns out to be surprisingly delicate. For instance, looking at the first extremal construction in Figure~\ref{figure:Burr}, one might expect that a blue copy of an $n$-vertex tree $T$ with $\Delta(T)\leq cn$ would appear once $U_1$ contains one more vertex, even if a small linear proportion of the edges within $U_1$ are red. However, a famous example of Koml\'os, S\"ark\'ozy and Szemer\'edi~\cite{komlos2001spanning} shows that this is not true. As such, proving Theorem~\ref{thm:main} in the extremal part will require a careful consideration of both the structure of the tree $T$ and the presence of edges of the `wrong' colour in the extremal colouring, to decide to where, and in which colour, the tree should be embedded in different cases.

This paper is organised as follows. In Section~\ref{sec:prelim} we first give a brief overview of our proof of Theorem~\ref{thm:main}, focusing on how it can be divided into the stability part (Sections~\ref{sec:embed regularity} and~\ref{sec:stages}) and the extremal part (Sections~\ref{sec:extremearg1} and~\ref{sec:extremearg2}), then collect all the basic notations and preliminary results. Then, in Section~\ref{sec:staboutline}, we give a detailed outline of the stability part of our proof of Theorem~\ref{thm:main}. In Section~\ref{sec:embed regularity}, we prove a series of technical regularity embedding lemmas, each of which allows us to embed a monochromatic copy of the tree $T$ into a red/blue coloured reduced graph that contains a certain suitable structure. In Section~\ref{sec:stages}, we use these embedding lemmas to move through 4 stages of embedding attempts, and eventually conclude either that we can find a monochromatic copy of $T$ using regularity, or that the reduced graph and thus the original graph are both extremal, in the sense that they approximate one of the extremal constructions. Depending on which of the two extremal constructions our original graph approximates, we show in Section~\ref{sec:extremearg1} and Section~\ref{sec:extremearg2} respectively that a monochromatic copy of $T$ can still be found.

\section{Proof overview and preliminaries}\label{sec:prelim}
In this section, we begin by giving a short overview of our proof of Theorem~\ref{thm:main} in Section~\ref{sec:division}, specifically on how it divides into the stability part and the extremal part, and formalising what it means for a colouring to approximate an extremal construction. Then, we record the basic notations we use in Section~\ref{sec:notation}, and collect a series of preliminary results in Sections~\ref{sec:concentration} to~\ref{sec:regularity}.
\subsection{Division of the proof of Theorem~\ref{thm:main} into stability and extremal parts}\label{sec:division}
We start by recapping the situation in Theorem~\ref{thm:main}. Let $t_1,t_2$ be positive integers so that $n=t_1+t_2$ and $t_1\ge t_2$. Let $T$ be an $n$-vertex tree with $\Delta(T)\le cn$ and bipartition classes $V_1$ and $V_2$, such that $|V_1|=t_1$ and $|V_2|=t_2$. Note that $\Delta(T)\le cn$ implies $t_2\ge c^{-1}$. Let $N=\max\{t_1+2t_2,2t_1\}-1$ and let $G$ be a red/blue coloured complete graph on $N$ vertices. Our aim, then, is to find a monochromatic copy of $T$ in $G$.


Our proof of Theorem~\ref{thm:main} consists of two main parts, the stability part (Sections~\ref{sec:embed regularity} and~\ref{sec:stages}) and the extremal part
(Sections~\ref{sec:extremearg1} and~\ref{sec:extremearg2}). In the stability part, we show that either $G$ contains a monochromatic copy of $T$, or $G$ is close to one of the two extremal constructions in Figure~\ref{figure:Burr}.
Then in the extremal part we show that a monochromatic copy of $T$ still exists even if $G$ approximates an extremal construction. 
These parts are quite separate, and both are quite involved, so we will sketch their proofs later (in Section~\ref{sec:outline:stability} for the stability part, and in Sections~\ref{sec:outline:extcase1} and~\ref{sec:outline:extcase2} for the two different cases of the extremal part). Here, we state the main results for both parts, and put them together to prove Theorem~\ref{thm:main}. We start with the following definition of what it means for $G$ to be close to one of the extremal constructions.

\begin{definition}\label{def:extremality}
Let $0<\mu<1$ and let $G$ be a red/blue coloured complete graph. We say $G$ is \textit{Type I $(\mu,t_1,t_2)$-extremal} if with $n=t_1+t_2$, there are disjoint subsets $U_1,U_2\subset V(G)$ such that 
\begin{itemize}
    \item $|U_1|\geq(1-\mu)n$ and $|U_2|\geq(1-\mu)t_2$,
    \item for every $u\in U_1$, $\red{d}(u,U_1)\leq \mu n$, and
    \item for every $i\in[2]$ and every $u\in U_i$, $\blue{d}(u,U_{3-i})\leq \mu n$,
\end{itemize}
or with red and blue swapped. On the other hand, we say $G$ is \textit{Type II $(\mu,t_1,t_2)$-extremal} if with $n=t_1+t_2$, there are disjoint subsets $U_1,U_2\subset V(G)$ such that 
\begin{itemize}
    \item $|U_1|,|U_2|\geq (1-\mu)t_1$, and
    \item for each $i\in [2]$ and $u\in U_i$, $\red{d}(u,U_i)\leq \mu n$ and $\blue{d}(u,U_{3-i})\leq\mu n$,
\end{itemize}
or with red and blue swapped. If $G$ is either or Type I or Type II $(\mu,t_1,t_2)$-extremal, we say it is \emph{$(\mu,t_1,t_2)$-extremal}.
\end{definition}

Using this, we can now state the main result of the stability part of our proof, as follows.

\begin{theorem}\label{thm:stability} Let $1/n\ll c\ll\mu \ll 1$ and let $t_1,t_2\in \mathbb{N}$ satisfy $t_1+t_2=n$ and $t_1\geq t_2$. Let $G$ be a red/blue coloured complete graph with $\max\{t_1+2t_2,2t_1\}-1$ vertices.
Then, at least one of the following is true. 
\begin{itemize}
    \item $G$ contains a monochromatic copy of every $n$-vertex tree $T$ with $\Delta(T)\le cn$ and bipartition class sizes $t_1$ and $t_2$.
    \item $G$ is Type I $(\mu,t_1,t_2)$-extremal.
    \item $t_1\geq(2-\mu)t_2$ and $G$ is Type II $(\mu,t_1,t_2)$-extremal.
\end{itemize}
\end{theorem}
 
Theorem~\ref{thm:stability} will be proved in Section~\ref{sec:stages}, and reduces the proof of Theorem~\ref{thm:main} to the following two results, which find a monochromatic copy of $T$ even when $G$ is close to one of the extremal constructions. We will prove them in Sections~\ref{sec:extremearg1} and~\ref{sec:extremearg2}, respectively.

\begin{theorem}\label{theorem:extremal:case1}
Let $1/n\ll c\ll\mu \ll 1$ and let $t_1,t_2\in \mathbb{N}$ satisfy $t_1+t_2=n$ and $t_1\geq t_2$. If $G$ is a Type I $(\mu,t_1,t_2)$-extremal graph with $\max\{t_1+2t_2,2t_1\}-1$ vertices, then $G$ contains a monochromatic copy of every $n$-vertex tree $T$ with $\Delta(T)\le cn$ and bipartition class sizes $t_1$ and $t_2$.
\end{theorem}

\begin{theorem}\label{theorem:extremal:case2}
Let $1/n\ll c\ll\mu \ll 1$ and let $t_1,t_2\in \mathbb{N}$ satisfy $t_1+t_2=n$ and $t_1\geq(2-\mu)t_2$. If $G$ is a Type II $(\mu,t_1,t_2)$-extremal graph with $\max\{t_1+2t_2,2t_1\}-1$ vertices, then $G$ contains a monochromatic copy of every $n$-vertex tree $T$ with $\Delta(T)\le cn$ and bipartition class sizes $t_1$ and $t_2$.
\end{theorem}

Given Theorems~\ref{thm:stability}--\ref{theorem:extremal:case2}, Theorem~\ref{thm:main} follows essentially immediately, but we will conclude this overview by formally making this deduction, as follows.

\begin{proof}[Proof of Theorem~\ref{thm:main}] Let $\mu$ satisfy $c\ll \mu \ll 1$. If $G$ is not $(\mu,t_1,t_2)$-extremal, then $G$ contains a monochromatic copy of $T$ by Theorem~\ref{thm:stability}. If $G$ is Type I $(\mu,t_1,t_2)$-extremal, then $G$ contains a monochromatic copy of $T$ by Theorem~\ref{theorem:extremal:case1}, while if $t_1\geq(2-\mu)t_2$ and $G$ is Type II $(\mu,t_1,t_2)$-extremal, then $G$ contains a monochromatic copy of $T$ by Theorem~\ref{theorem:extremal:case2}.
\end{proof}

Finally, we remark that in these proofs, we can often assume that $t_1\leq 2t_2+1$. Indeed, if $t_1\geq 2t_2+2$, then we can take $2/c$ vertices in $V_1$ with degree 1, which are guaranteed to exist by Lemma~\ref{lemma:leaves:V1}, and attach $\lfloor (t_1-2t_2)/2\rfloor$ new leaves to them, with none of them receiving more than $cn$ new leaves. Let $T'$ be the new tree obtained in this way, and note that the bipartition classes of $T'$ have sizes $t_1'=t_1$ and $t_2'=\floor{t_1/2}$, satisfying $t_1'\leq2t_2'+1$ and $\max\{t'_1+2t'_2,2t'_1\}-1=2t_1'-1=2t_1-1=\max\{t_1+2t_2,2t_1\}-1$. Thus, the number of vertices in $G$ remains unchanged, and it is clear that if $G$ contains a monochromatic copy of $T'$, then $G$ also contains a monochromatic copy of $T$.

\subsection{Notation}\label{sec:notation}
For a positive integer $n\in\mathbb N$, we write $[n]=\{1,\dots,n\}$  and $[n]_0=[n]\cup\{0\}$. We will use the standard hierarchy notation, that is, for $a,b\in (0,1]$, we will use $a\ll b$ to mean that there exists a non-decreasing function $f:(0,1]\to (0,1]$ such that if $a\le f(b)$ then the following statement holds. For $a,b\geq1$, we write $a\ll b$ if $1/b\ll1/a$. Hierarchies with more constants are defined in a similar way. For simplicity we will sometimes ignore floor and ceiling signs when doing so does not affect the argument.

Given a graph $G$, we use $V(G)$ and $E(G)$ to denote the set of vertices and edges of $G$, respectively, and write $|G|=|V(G)|$ and $e(G)=|E(G)|$. For not necessarily disjoint subsets $A,B\subset V(G)$, we denote the number of edges in $G$ with one endpoint in $A$ and one in $B$ by $e(A,B)$. For a subset $S\subset V(G)$, we use $G[S]$ to denote the graph with vertex set $S$ and all the edges from $G$ with both endpoints in $S$, and we write $G-S$ for the graph $G[V(G)\setminus S]$. Given two disjoint subsets $S,S'\subset V(G)$, we use $G[S,S']$ to denote the bipartite graph with parts $S$ and $S'$, and all edges of the form $ss'\in E(G)$ with $s\in S$ and $s'\in S'$.

For a vertex $v\in V(G)$, the set of neighbours of $v$ is denoted by $N(v)$, and $d(v)=|N(v)|$ denotes the degree of $v$. The maximum degree and the minimum degree of $G$ are denoted by $\Delta(G)$ and $\delta(G)$, respectively. Given a subset $S\subset V(G)$, its external neighbourhood is $N(S)=(\cup_{s\in S}N(s))\setminus S$. For a vertex $v\in V(G)$ and subsets $S,U\subset V(G)$, we write $N(v,S)=N(v)\cap S$, $d(v,S)=|N(v,S)|$, and $N(U,S)=N(U)\cap S$. When working with more than one graph, we add subscripts to denote which graph we are working with. For example, $d_G(x)$ refers to the degree of $x$ in the graph $G$.

Say $G$ is a red/blue coloured graph if every edge in $E(G)$ is coloured with either red or blue. We let $\red{G}$ and $\blue{G}$ denote the graphs spanned by the red edges and the blue edges, respectively. For brevity, we write $\red{d}(x)$ instead of $d_{\red{G}}(x)$ and $\blue{d}(x)$ instead of $d_{\blue{G}}(x)$, and use similar notations for the red and blue neighbourhoods of a vertex or a set of vertices.

For $\mu\in[0,1]$, a graph $G$ is \textit{$\mu$-almost complete} if $\delta(G)\geq(1-\mu)|G|$, and is \textit{$\mu$-almost empty} if $\Delta(G)\leq\mu|G|$. A bipartite graph $H$ with bipartition classes $A,B$ is \textit{$\mu$-almost complete} if $d(a)\geq(1-\mu)|B|$ for every $a\in A$ and $d(b)\geq(1-\mu)|A|$ for every $b\in B$, and is \textit{$\mu$-almost empty} if $d(a)\leq\mu|B|$ for every $a\in A$ and $d(b)\leq\mu|A|$ for every $b\in B$.


\subsection{Concentration results}\label{sec:concentration}
We will need the following well-known concentration results.
\begin{lemma}[Chernoff's Bound~{\cite[Corollary 2.3, Theorem 2.10]{JLR2000}}]\label{lemma:chernoff}
Let $X$ be either a binomial random variable or a hypergeometric random variable. Then, for all $0<\eps\le 3/2$,
\[\mathbb{P}\left(\big|X-\mathbb E[X]\big|\ge \eps\mathbb E[X]\right)\le 2\exp(-\eps^2\mathbb{E}[X]/3).\]
\end{lemma}
\begin{lemma}[Azuma's Inequality~{\cite[Lemma 4.2]{wormald1999differential}}]\label{lemma:azuma}
Let $X_1,\ldots, X_m$ be a sequence of random variables such that for each $i\in[m]$, there exist constants $a_i\in\mathbb{R}$ and $c_i>0$ with $|X_i-a_i|\leq c_i$.
\begin{itemize}
    \item If $\mathbb{E}[X_i\mid X_1,\ldots,X_{i-1}]\geq a_i$ for every $i\in[m]$, then for every $t>0$, \[\textstyle\mathbb{P}\left(\sum_{i=1}^m(X_i-a_i)\leq-t\right)\leq\exp\left(-\frac{t^2}{2\sum_{i=1}^mc_i^2}\right).\]
    \item If $\mathbb{E}[X_i\mid X_1,\ldots,X_{i-1}]\leq a_i$ for every $i\in[m]$, then for every $t>0$, \[\textstyle \mathbb{P}\left(\sum_{i=1}^m(X_i-a_i)\geq t\right)\leq\exp\left(-\frac{t^2}{2\sum_{i=1}^mc_i^2}\right).\]
\end{itemize}
\end{lemma}
\begin{lemma}[McDiarmid's Inequality~{\cite[Lemma 1.2]{McDiarmid_1989}}]\label{lemma:mcdiarmid}
Let $X_1,\ldots,X_m$ be independent random variables taking values in a set $\Omega$. Let $c_1,\ldots,c_m\geq 0$ and suppose $f:\Omega^m\to\mathbb{R}$ is a function such that for every $i\in[m]$ and every $x_1,\ldots,x_m,x_i'\in\Omega$, we have $\Big|f(x_1,\ldots,x_i,\ldots,x_m)-f(x_1,\ldots,x_i',\ldots,x_m)\Big|\leq c_i$.
Then, for all $t>0$,
\[\mathbb{P}\left(\Big|f(X_1,\ldots,X_m)-\mathbb{E}[f(X_1,\ldots,X_m)]\Big|\geq t\right)\leq2\exp\left(\frac{-2t^2}{\sum_{i=1}^mc_i^2}\right).\]
\end{lemma}

\subsection{Matchings in bipartite graphs}
In many of our later tree embedding arguments, we will first embed all but a small set of vertices in $T$ with degrees 1 or 2. To finish the embedding, the following well-known Hall's matching theorem and its generalisation are useful.
\begin{lemma}[Hall's matching theorem {\cite[Theorem 1]{hall1935representatives}}]\label{lemma:Hall}
Let $G$ be a bipartite graph with bipartition classes $A$ and $B$. If $|N(S)|\ge |S|$ for every $S\subset A$, then $G$ contains a matching covering all vertices in $A$.
\end{lemma}
\begin{lemma}[{\cite[Corollary 11]{bollobás1998modern}}]\label{lemma:hallmatching}
Let $G$ be a bipartite graph with bipartition classes $A$ and $B$, and let $(f_a)_{a\in A}$ be a tuple of non-negative integers indexed by elements of $A$. Suppose that $|N(S)|\geq\sum_{a\in S}f_a$ for every $S\subset A$. Then, there exists a collection of vertex-disjoint stars $(S_a)_{a\in A}$ in $G$, such that for each $a\in A$, $S_a$ is centred at $a$ and has exactly $f_a$ leaves.
\end{lemma}
The conditions in Lemma~\ref{lemma:Hall} and Lemma~\ref{lemma:hallmatching} will both be referred as \textit{Hall's matching condition}.

\subsection{Trees}
We now record several useful results on tree embeddings and tree decompositions.
\begin{lemma}\label{lemma:leaves:V1}If $T$ is an $n$-vertex tree with bipartition classes $V_1$ and $V_2$ such that $|V_1|=t_1$, $|V_2|=t_2$, and $t_1\ge t_2$, then $T$ contains at least $t_1-t_2+1$ leaves in $V_1$.
\end{lemma}
\begin{proof}
Let $L$ be the set of leaves of $T$ in $V_1$. Then
	\[n-1=e(V_1,V_2)=\sum_{v\in V_1}d(v)\ge |L|+2|V_1\setminus L|=2t_1-|L|,\]
	from which it follows that $|L|\geq2t_1-n+1=t_1-t_2+1$.
\end{proof}

A path $P$ in a tree $T$ is a \textit{bare path} if all of its internal vertices have degree 2 in $T$. By the following well-known result, every tree has either many leaves or many bare paths.
\begin{lemma}[{\cite[Lemma 2.1]{Krivelevichtrees}}]\label{lemma:paths-leaf}Let $k,\ell, n\in\mathbb N$ and let $T$ be an $n$-vertex tree with at most $\ell$ leaves. Then $T$ contains a collection of at least $\frac{n}{k+1}-(2\ell-2)$ vertex-disjoint bare paths, each of length $k$.
\end{lemma}

In many of our tree embeddings, we will first divide the tree into two parts that are then embedded with different methods and different aims. For this, we use the following definition.
\begin{definition}For a tree $T$, we say that subgraphs $T_1,T_2$ of $T$ form a \textit{decomposition} of $T$ if they are edge-disjoint subforests of $T$ such that $E(T)=E(T_1)\cup E(T_2)$.
\end{definition}
We will use the following result to decompose a tree into two subtrees so that each subtree in the decomposition contains a large proportion of a set chosen in advance.

\begin{lemma}[{\cite[Proposition 3.19]{montgomery2019spanning}}]\label{lemma:divide:vertices}Let $T$ be a tree and let $Q\subset V(T)$. Then, $T$ has a decomposition into subtrees $T_1$ and $T_2$ with a unique common vertex such that $|Q\cap V(T_1)|\ge |Q|/3$ and $|Q\cap V(T_2)|\ge |Q|/3$.\end{lemma}
The following almost immediate corollary is obtained by taking $Q=V(T)$ in Lemma~\ref{lemma:divide:vertices}.
\begin{corollary}\label{cor:splittree} Every $n$-vertex tree $T$ decomposes into subtrees $T_1$ and $T_2$ with a unique common vertex such that $\ceil{n/3}\leq|T_1|\leq|T_2|\leq\ceil{2n/3}$.
\end{corollary}
\begin{proof}
Apply Lemma~\ref{lemma:divide:vertices} with $Q=V(T)$, we get a decomposition of $T$ into subtrees $T_1$ and $T_2$ with a unique common vertex $v$ such that $\ceil{n/3}\leq|T_1|\leq|T_2|\leq n-\ceil{n/3}+1$. If $n$ is congruent to 1 or 2 modulo 3, then $n-\ceil{n/3}+1=\ceil{2n/3}$, so we are done. If $n=3k$ for some integer $k\geq1$, then the only situation where the result does not follow immediately is when $|T_1|=k$ and $|T_2|=2k+1$. Assume that this holds. 

If $d_T(v,T_2)=1$, let $v'$ be the unique neighbour of $v$ in $T_2$, then $T_1+vv'$ and $T_2-v$ are subtrees decomposing $T$ with a unique common vertex $v'$, and contains $k+1$ and $2k$ vertices, respectively, as required. If $d_T(v,T_2)\geq2$, let $S$ be the smallest component in $T_2-v$, so $1\leq|S|\leq k$. Then $V(T_1)\cup S$ and $V(T_2)\setminus S$ induce two subtrees decomposing $T$ with a unique common vertex $v$, and contain $k+1\leq k+|S|\leq 2k$ and $k+1\leq 2k+1-|S|\leq 2k$ vertices, respectively, finishing the proof.  
\end{proof}

The following results show that we can cut a tree into smaller subtrees using few vertices.
\begin{lemma}\label{lemma:treehalfcomponent}For every $n$-vertex tree $T$, there exists a vertex $v\in T$ so that each component of $T-v$ has size at most $n/2$.
\end{lemma}
\begin{proof}Choose an arbitrary vertex as the root of $T$. Let $v$ be a vertex at a maximal distance from the root subject to the condition that the tree $T'$ induced by $v$ and all of its descendents has size at least $n/2$. By the choice of $v$, each component of $T'-v$ has size less than $n/2$. The only component of $T-v$ that is not a component of $T'-v$ is $T-T'$, which has size at most $n/2$ as $|T'|\geq n/2$.
\end{proof}

\begin{lemma}[{\cite[Proposition 4.1]{besomi2019degree}}]\label{lemma:trees:cutset}
Let $1/n\ll \xi\ll1$ and let $T$ be an $n$-vertex tree. Then, there exists a subset $X\subset V(T)$ with $|X|\le2\xi^{-1}$, such that every component of $T-X$ has size at most $\xi n$.
\end{lemma}

The following two results state that trees can be greedily embedded into graphs with large minimum degrees, and will be used throughout the paper without any further reference.
\begin{lemma}\label{lemma:greedy}
Let $T$ be an $n$-vertex tree containing a vertex $t$. If $G$ is a graph with $\delta(G)\ge n-1$, then, for any vertex $v\in G$, there is a copy of $T$ in $G$ with $t$ copied to $v$.
\end{lemma}

\begin{lemma}\label{lemma:bipartitegreedy}
Let $T$ be a tree with bipartition classes $V_1$ and $V_2$ of sizes $t_1$ and $t_2$, respectively. Suppose that $G$ is a bipartite graph with bipartition classes $U_1$ and $U_2$, such that
\begin{itemize}
    \item every vertex in $U_1$ has at least $t_2$ neighbours in $U_2$, and
    \item every vertex in $U_2$ has at least $t_1$ neighbours in $U_1$.
\end{itemize}
Then, for any $i\in[2]$ and any vertices $t\in V_i$ and $u\in U_i$, there exists a copy of $T$ in $G$ such that $V_1$ is copied to $U_1$, $V_2$ is copied to $U_2$, and $t$ is copied to $u$.
\end{lemma}


\subsection{Szemerédi's regularity lemma}\label{sec:regularity}
Let $G$ be a bipartite graph with bipartition classes $A$ and $B$. For sets $X\subset A$ and $Y\subset B$, the \textit{density} between $X$ and $Y$ is defined as
\[d(X,Y)=\frac{e(X,Y)}{|X||Y|}.\]
We say $G$ is \textit{$\varepsilon$-regular} if for every $X\subset A$ and every $Y\subset B$ with $|X|\ge \varepsilon |A|$ and $|Y|\ge \varepsilon |B|$, we have $|d(X,Y)-d(A,B)|\le\varepsilon$. Furthermore, we say $G$ is \textit{$(\varepsilon,d)$-regular} if $G$ is $\varepsilon$-regular and $d(A,B)\ge d$. 
The following results are standard.
\begin{lemma}\label{lemma:regularity:1}
Let $\eps\leq1/4$, and let $G$ be a bipartite graph with bipartition classes $A$ and $B$ that is $(\varepsilon,d)$-regular. Suppose $X\subset A$ and $Y\subset B$ satisfy $|X|\ge\sqrt\varepsilon |A|$ and $|Y|\ge\sqrt\varepsilon|B|$, then $G[X,Y]$ is $(\sqrt\varepsilon,d-\eps)$-regular.
\end{lemma}

\begin{lemma}\label{lemma:regularity:3}
Let $G$ be a bipartite graph with bipartition classes $A$ and $B$ that is $(\varepsilon,d)$-regular. Suppose $Y\subset B$ satisfies $|Y|\ge\varepsilon|B|$, then there are less than $\varepsilon|A|$ vertices $v\in A$ for which $d(v,Y)<(d-\varepsilon)|Y|$.
\end{lemma}


\begin{lemma}\label{lemma:regularity:2} Let $G$ be a graph containing disjoint subsets $V_0, V_1,\ldots, V_r\subset V(G)$, such that $G[V_0,V_i]$ is $(\varepsilon,d)$-regular for each $i\in [r]$. Let $U_i\subset V_i$ have size $|U_i|\ge \varepsilon |V_i|$ for each $i\in[r]$. Then, there are less than $\sqrt{\varepsilon}|V_0|$ vertices $v\in V_0$ such that $d(v,U_i)<(d-\varepsilon)|U_i|$ for at least $\sqrt\varepsilon r$ indices $i\in [r]$.
\end{lemma}

The following colourful variant of Szemer\'edi's Regularity Lemma is well-known, and is the starting point of the stability part of our proof. 
\begin{theorem}[Coloured Regularity Lemma{~\cite[Theorem~1.18]{komlos1995szemeredi}}]\label{theorem:regularity}
Let $1/k_2\ll 1/k_1\ll \varepsilon$. Every red/blue coloured graph $G$ on $n\ge k_1$ vertices contains disjoint subsets $V_1,\ldots,V_k\subset V(G)$ with $k_1\le k\le k_2$ that satisfy the following.
\begin{enumerate}[label=\upshape (\roman{enumi})]
    \item\label{colreg:1} $|V(G)\setminus(V_1\cup\cdots\cup V_k)|\le \varepsilon n$.
    \item\label{colreg:2} $|V_1|=\cdots=|V_k|$.
    \item\label{colreg:3} For all but at most $\varepsilon k^2$ indices $1\le i<j\le k$, both $\red{G}[V_i,V_j]$ and $\blue{G}[V_i,V_j]$ are $\varepsilon$-regular.
\end{enumerate}\end{theorem}

For technical reasons, we sometimes require the sets $V_i$ to have different sizes, but do not necessarily need them to cover all but $\eps n$ vertices in $G$. As this is a minor point, we do not introduce more notation and instead use the standard term \textit{$\eps$-regular partition} under the following more relaxed definition. 
\begin{definition}\label{def:regpartition}
Let $1/n\ll\eps\ll d\leq1$, and let $G$ be a red/blue coloured graph on $n$ vertices. An \textit{$\eps$-regular partition} in $G$ is a collection of disjoint subsets $V_1,\ldots,V_k\subset V(G)$, such that for all but at most $\varepsilon k^2$ pairs of indices $1\le i<j\le k$, both $\red{G}[V_i,V_j]$ and $\blue{G}[V_i,V_j]$ are $\varepsilon$-regular. Each set $V_i$ is called a \textit{cluster}. 

Given an $\eps$-regular partition $V_1\cup\cdots\cup V_k$ in $G$, its corresponding \textit{$(\eps,d)$-reduced graph} $R$ is a red/blue coloured graph with vertex set $[k]$, such that for each $\ast\in\{\text{red},\text{blue}\}$ and any distinct $i,j\in[k]$, there is an $ij$ edge of colour $\ast$ in $R$ if and only if $G_\ast[V_i,V_j]$ is $(\eps,d)$-regular. 
\end{definition}

Note that if $V_1,\ldots,V_k$ form an $\eps$-regular partition in a red/blue coloured complete graph and $\eps\ll d\leq 1/2$, then for all but at most $\eps k^2$ pairs of indices $1\leq i<j\leq k$, there is either a red edge $ij$ or a blue edge $ij$ (or both) in the corresponding $(\eps,d)$-reduced graph $R$.

Finally, we prove the following refinement result that will be used later.
\begin{lemma}\label{lemma:regularity:refine}
Let $1/k,1/m\ll\eps\ll\eta\ll\alpha\ll d\leq1$. Suppose $G$ is a graph containing disjoint subsets $V_1,\ldots,V_k\subset V(G)$, each of size $m$. Let $R$ be a graph on $[k]$ such that for every $ij\in E(R)$,  $G[V_i,V_j]$ is $(\eps,d)$-regular. Suppose there exists a partition $[k]=I_1\cup I_2$, with $|I_1|=k_1,|I_2|=k_2$, and $k_1,k_2\geq\alpha k$, such that $R[I_1,I_2]$ is $\eta$-almost complete. Then, there exist two collections of disjoint sets $\{U_i:i\in J\}$ and $\{W_i:i\in J\}$ such that 
\begin{itemize}
    \item $|U_i|=|U_j|$ and $|W_i|=|W_j|$ for any $i,j\in J$,
    \item $\sum_{i\in J}|U_i|\geq(1-\alpha)\sum_{i\in I_1}|V_i|$, $\sum_{i\in J}|W_i|\geq(1-\alpha)\sum_{i\in I_2}|V_i|$, and
    \item $G[U_i,W_i]$ is $(\sqrt\eps,d-\eps)$-regular for every $i\in J$.
\end{itemize}
\end{lemma}
\begin{proof}
Let $\eta\ll\gamma\ll\alpha$. For each $i\in I_1$, pick a largest  collection of disjoint subsets of $V_i$ of size $\gamma k_1m/(k_1+k_2)$. For each $i\in I_2$, pick a largest  collection of disjoint subsets of $V_i$ of size $\gamma k_2m/(k_1+k_2)$. Let $\{U_i:i\in J_1\}$ and $\{W_i:i\in J_2\}$ be the collections of refined subsets coming from $\{V_i:i\in I_1\}$ and $\{V_i:i\in I_2\}$, respectively. Note that at most a $\gamma$-proportion of vertices are lost from each $V_i$ in this refinement process. Let $R'$ be a graph with vertex set $J_1\cup J_2$, such that $ij\in E(R')$ if $G[U_i,W_j]$ is $(\sqrt\eps,d-\eps)$-regular. Then, as $R[I_1,I_2]$ is $\eta$-almost complete, $R'[J_1,J_2]$ is $\eta$-almost complete as well by Lemma~\ref{lemma:regularity:1}. Therefore, we can greedily find a matching $M$ of size $(1-\eta)\min\{|J_1|,|J_2|\}$ in $R'[J_1,J_2]$. To finish, observe that $\sum_{i\in J_1\cap V(M)}|U_i|\geq(1-\eta)(1-\gamma)\sum_{i\in I_1}|V_i|\geq(1-\alpha)\sum_{i\in I_1}|V_i|$, and similarly $\sum_{i\in J_2\cap V(M)}|U_i|\geq(1-\alpha)\sum_{i\in I_2}|V_i|$. 
\end{proof}


\section{Outline of the proof of Theorem~\ref{thm:stability}: Stability}\label{sec:staboutline}\label{sec:outline:stability}
\noindent\textbf{Simplifications for the discussion.}
The main technical tool for the stability part of the proof of Theorem~\ref{thm:main}, i.e.\ the proof of Theorem~\ref{thm:stability}, is Szemer\'edi's regularity lemma. The following outline of the proof of Theorem~\ref{thm:stability} assumes a working knowledge of the regularity lemma and simple embeddings using it, and further surpresses two technical details that we will explain momentarily. Readers less familiar with regularity techniques may find it useful to start with Section~\ref{sec:regularity}, and readers finding this outline too scant in detail may find it valuable instead as a blueprint when reading the formal proofs in Section~\ref{sec:embedregularity} and Section~\ref{sec:stages}.

There are two main technicalities that we will suppress in the following outline. Most notably, due to the imbalance in the sizes $t_1$ and $t_2$ of the bipartition classes of the tree $T$, we will sometimes work with regularity partitions whose clusters have different sizes. In several cases, clusters will have two different sizes that are in the ratio $t_1:t_2$. Moreover, in many of our more intricate arguments, the same cluster may change its role throughout the proof, having vertices from either the larger or the smaller side of the bipartition embedded into it. To facilitate this, we sometimes need to refine the regularity partition that we started with, partitioning all clusters into smaller clusters of suitable sizes, before pairing them up again to form regular pairs with the right size ratio (see e.g. Lemma~\ref{lemma:regularity:refine}). In this outline, we will skim over this aspect of the proof. That is, we will only work with a fixed regularity partition here and not worry about the technicalities regarding cluster sizes and refinements. The crucial point we focus on in the outline here is the number of vertices in the original graph that are covered by a certain set of clusters, where for any set $I\subset [k]$, we say that $I$ \emph{covers} the vertices $\cup_{i\in I}V_i$.

The other technicality is one common to many uses of the regularity lemma: we will need to have many constants of decreasing sizes in some hierarchy. To avoid this burden here, we will informally use $m^+$ or $m^-$ to denote a number equal to $m+\alpha n$ or $m-\alpha n$, respectively, for some small and suitable constant $\alpha>0$. In particular, $0^+$ will represent $\alpha n$ for some $\alpha>0$. The constants $\alpha$ involved in different instances of these notations are all different and will be chosen later carefully in the formal proofs. To give a rough idea of the relation of parameters, we can expect these $\alpha$ to satisfy $\eps\ll\alpha\ll 1$, where $\eps$ is the regularity parameter.

\medskip

\noindent\textbf{Set-up in the proof of Theorem~\ref{thm:stability}.}
In Theorem~\ref{thm:stability}, we have an $n$-vertex tree $T$ satisfying $\Delta(T)\leq cn$ that has bipartition classes $U_1$ and $U_2$ with sizes $t_1$ and $t_2$ respectively, where $t_1\geq t_2$. We also have a red/blue coloured complete graph $G$ with $\max\{t_1+2t_2,2t_1\}-1$ vertices, and wish to either find a monochromatic copy of $T$ in $G$ or show that the colouring of $G$ is close to one of the two extremal constructions, where this proximity is controlled with the parameter $\mu$. As mentioned in Section~\ref{sec:division}, by adding leaves to the $t_2$-side of the tree if necessary, we may assume that $t_1\leq 2t_2+1$. In fact, as will be justified in Section~\ref{sec:stabmain}, we can even assume that $t_1\leq 2t_2$ in the proof of Theorem~\ref{thm:stability}, which we will do from now on. In particular then, the graph $G$ has $t_1+2t_2-1$ vertices.

\medskip

\noindent\textbf{Stages, situations and embedding methods.} Let $\eps$ be a suitable regularity parameter satisfying $1/n\ll c\ll\eps\ll\mu\ll1$. We begin by applying a result (Theorem~\ref{thm:hltstart}) of Haxell, \L uczak, and Tingley~\cite{haxell2002ramsey} to find an $\eps$-regular partition in $V(G)$ that contains a certain monochromatic structure in the reduced graph (see the top left of Figure~\ref{fig:stages}). Having found this, we say we have an \textbf{A-situation}. We then work through a sequence of 4 stages. At each stage, we either find a monochromatic copy of the tree, or deduce that $G$ must be close to an extremal construction, or find more useful structure in the reduced graph. If the last of these is true at the end of a stage, we reach another named situation (see Figure~\ref{fig:stages}), and if we reach the end of these 4 stages, then we have an \textbf{E-situation} (see the right of Figure~\ref{fig:stages}), which will imply that $G$ is close to an extremal construction. 

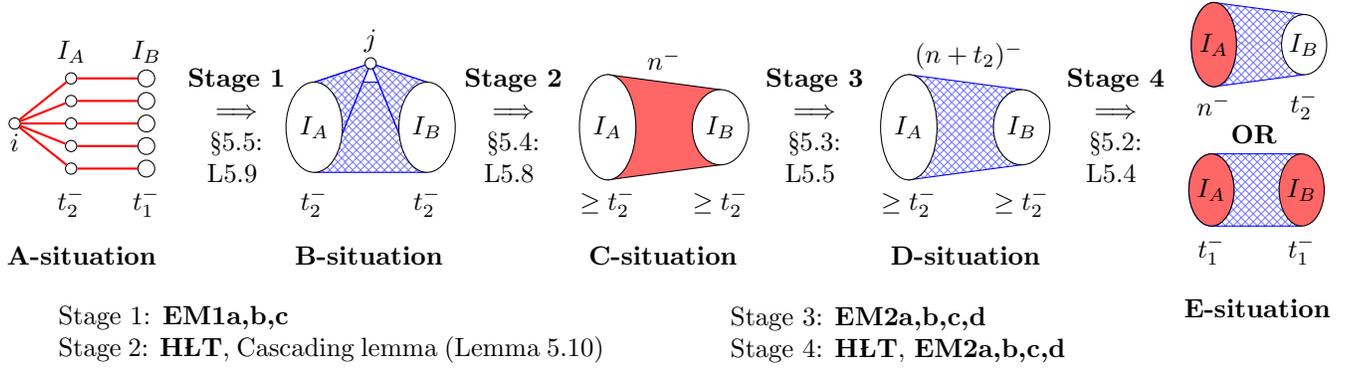
\begin{figure}[p]
\hspace{-0.1\textwidth}\makebox[1.2\textwidth][c]{\begin{tikzpicture}[scale=.5,main1/.style = {circle,draw,fill=none,inner sep=1.5}, main2/.style = {circle,draw,fill=none,inner sep=2.25}]
\def\spacer{2};
\def\spacerr{1.5};
\def\Ahgt{1.2};
\def\Bhgt{1.2};
\coordinate (A) at (0,0);
\coordinate (B) at (\spacer,0);


\foreach \n/\offf in {0/2,1/1,2/0,3/-1,4/-2}
{
\node[main1] (A\n) at ($0.5*\offf*(0,\Ahgt)$) {};
\node[main2] (B\n) at ($0.5*\offf*(0,\Ahgt)+(\spacer,0)$) {};
}
\node[main1] (V) at ($(-\spacerr,0)$) {};

\foreach \n in {0,1,2,3,4}
{
\draw [thick,red] (V) -- (A\n);
\draw [thick,red] (A\n) -- (B\n);
}


\draw ($(A4)-(0,0.9)$) node {$t_2^-$};
\draw ($(B4)-(0,0.9)$) node {$t_1^-$};

\draw ($(V)-(0,0.5)$) node {$i$};

\draw ($(A0)+(0,0.7)$) node {$I_A$};
\draw ($(B0)+(0,0.7)$) node {$I_B$};

\end{tikzpicture}
\begin{minipage}{1.3cm}
\begin{center}
\textbf{Stage 1}

$\implies$

\textsection\ref{sec:stage1}: L\ref{lem:stage1}
\vspace{2.5cm}
\end{center}
\end{minipage}
\begin{tikzpicture}[scale=.5,main1/.style = {circle,draw,fill=white,inner sep=1.5}]
\def\spacer{3};
\def\Ahgt{1.2};
\def\Bhgt{1.2};
\def\ver{0.5};
\def\wid{0.75};

\coordinate (A) at (0,0);
\coordinate (B) at (\spacer,0);

\coordinate (V) at ($(0.5*\spacer,\Ahgt+\ver)$) {};

\draw[blue,pattern=crosshatch, pattern color=blue!50] ($(A)+(0,\Ahgt)$) -- ($(A)-(0,\Ahgt)$) -- ($(B)-(0,\Bhgt)$) -- ($(B)+(0,\Bhgt)$) -- cycle;

\draw[blue,fill=white] (V) -- ($(A)+(0,\Ahgt)$) -- ($(A)+(\wid,0)-(0.3,0.9)$) -- (V);
\draw[blue,fill=white] (V) -- ($(B)-(\wid,0)-(-0.3,0.9)$) -- ($(B)+(0,\Bhgt)$) -- cycle;
\draw[blue,pattern=crosshatch, pattern color=blue!50] (V) -- ($(A)+(0,\Ahgt)$) -- ($(A)+(\wid,0)-(0.3,0.9)$) -- (V);
\draw[blue,pattern=crosshatch, pattern color=blue!50] (V) -- ($(B)-(\wid,0)-(-0.3,0.9)$) -- ($(B)+(0,\Bhgt)$) -- cycle;

\draw[fill=white] (A) circle [y radius=\Ahgt cm,x radius=0.75cm];
\draw[fill=white] (B) circle [y radius=\Bhgt cm,x radius=0.75cm];

\draw (A) node {$I_A$};
\draw (B) node {$I_B$};

\draw ($(A)+(0,-2)$) node {$t_2^-$};
\draw ($(B)+(0,-2)$) node {$t_2^-$};

\node[main1] (V1) at ($(0.5*\spacer,\Ahgt+\ver)$) {};

\draw ($(V)+(0,0.6)$) node {$j$};

\end{tikzpicture}\begin{minipage}{1.3cm}
\begin{center}
\textbf{Stage 2}

$\implies$

\textsection\ref{sec:stage2}: L\ref{lem:stage2}
\vspace{2.5cm}
\end{center}
\end{minipage}
\begin{tikzpicture}[scale=.5]
\def\spacer{3};
\def\Ahgt{1.4};
\def\Bhgt{1.0};
\coordinate (A) at (0,0);
\coordinate (B) at (\spacer,0);

\draw[fill=red!60] ($(A)+(0,\Ahgt)$) -- ($(A)-(0,\Ahgt)$) -- ($(B)-(0,\Bhgt)$) -- ($(B)+(0,\Bhgt)$) -- cycle;

\draw [fill=white] (A) circle [y radius=\Ahgt cm,x radius=0.75cm];
\draw [fill=white] (B) circle [y radius=\Bhgt cm,x radius=0.75cm];

\draw (A) node {$I_A$};
\draw (B) node {$I_B$};

\draw ($(A)+(0,-2)$) node {$\geq t_2^-$};
\draw ($(B)+(0,-2)$) node {$\geq t_2^-$};

\draw ($0.5*(A)+0.5*(B)+(0,\Ahgt)+(0,0.4)$) node {$n^-$};

\end{tikzpicture}\begin{minipage}{1.3cm}
\begin{center}
\textbf{Stage 3}

$\implies$

\textsection\ref{sec:stage3}: L\ref{lem:stage3}
\vspace{2.5cm}
\end{center}
\end{minipage}
\begin{tikzpicture}[scale=.5]
\def\spacer{3};
\def\Ahgt{1.4};
\def\Bhgt{1.0};
\coordinate (A) at (0,0);
\coordinate (B) at (\spacer,0);

\draw[blue,pattern=crosshatch, pattern color=blue!50] ($(A)+(0,\Ahgt)$) -- ($(A)-(0,\Ahgt)$) -- ($(B)-(0,\Bhgt)$) -- ($(B)+(0,\Bhgt)$) -- cycle;

\draw [fill=white] (A) circle [y radius=\Ahgt cm,x radius=0.75cm];
\draw [fill=white] (B) circle [y radius=\Bhgt cm,x radius=0.75cm];

\draw (A) node {$I_A$};
\draw (B) node {$I_B$};

\draw ($(A)+(0,-2)$) node {$\geq t_2^-$};
\draw ($(B)+(0,-2)$) node {$\geq t_2^-$};

\draw ($0.5*(A)+0.5*(B)+(0,\Ahgt)+(0.1,0.5)$) node {$(n+t_2)^-$};

\end{tikzpicture}\begin{minipage}{1.3cm}
\begin{center}
\textbf{Stage 4}

$\implies$

\textsection\ref{sec:stage4}: L\ref{lem:stage4}
\vspace{2.5cm}
\end{center}
\end{minipage}\;\;
\begin{minipage}{1.75cm}

\vspace{-1.75cm}

\begin{tikzpicture}[scale=.4]
\def\spacer{3};
\def\Ahgt{1.4};
\def\Bhgt{1.0};
\coordinate (A) at (0,0);
\coordinate (B) at (\spacer,0);

\draw[blue,pattern=crosshatch, pattern color=blue!50] ($(A)+(0,\Ahgt)$) -- ($(A)-(0,\Ahgt)$) -- ($(B)-(0,\Bhgt)$) -- ($(B)+(0,\Bhgt)$) -- cycle;

\draw [blue,fill=white] (A) circle [y radius=\Ahgt cm,x radius=0.75cm];
\draw [fill=red!60] (A) circle [y radius=\Ahgt cm,x radius=0.75cm];
\draw [fill=white] (B) circle [y radius=\Bhgt cm,x radius=0.75cm];

\draw [red!60,fill=red!60] (A) circle [radius=0.5cm];
\draw (A) node {$I_A$};
\draw (B) node {$I_B$};

\draw ($(A)+(0,-2)$) node {$n^-$};
\draw ($(B)+(0,-2)$) node {$t_2^-$};


\end{tikzpicture}

\vspace{-1cm}

\begin{center}
\textbf{OR}
    \end{center}

\vspace{-0.25cm}

\begin{tikzpicture}[scale=.4]
\def\spacer{3};
\def\Ahgt{1.2};
\def\Bhgt{1.2};
\coordinate (A) at (0,0);
\coordinate (B) at (\spacer,0);

\draw[blue,pattern=crosshatch, pattern color=blue!50] ($(A)+(0,\Ahgt)$) -- ($(A)-(0,\Ahgt)$) -- ($(B)-(0,\Bhgt)$) -- ($(B)+(0,\Bhgt)$) -- cycle;

\draw [blue,fill=white] (A) circle [y radius=\Ahgt cm,x radius=0.75cm];
\draw [fill=red!60] (A) circle [y radius=\Ahgt cm,x radius=0.75cm];
\draw [fill=red!60] (B) circle [y radius=\Bhgt cm,x radius=0.75cm];

\draw [red!60,fill=red!60] (A) circle [radius=0.5cm];
\draw (A) node {$I_A$};
\draw (B) node {$I_B$};

\draw ($(A)+(0,-2)$) node {$t_1^-$};
\draw ($(B)+(0,-2)$) node {$t_1^-$};


\end{tikzpicture}
\end{minipage}
}

\vspace{-1.8cm}

\hspace{-0.8cm}\begin{tikzpicture}
\def\spacerrrr{3.74cm};
\foreach \n/\lab/\exxx in {0/A/0,1/B/0.08,2/C/0.25,3/D/0.53}
{
\draw ($\n*(\spacerrrr,0)+(\exxx,0)$) node {$\text{\bfseries \lab-situation}$};
}
\foreach \n/\lab/\exxx in {4/E/0.67}
{
\draw ($\n*(\spacerrrr,0)+(\exxx,-0.7)$) node {$\text{\bfseries \lab-situation}$};
}
\end{tikzpicture}

\vspace{-0.3cm}

\begin{minipage}{0.45\textwidth}
Stage 1: \textbf{EM1a,b,c}

Stage 2: \textbf{H\L T}, Cascading lemma (Lemma~\ref{lemma:improvemaximummatching})
\end{minipage}\;\;\;\;\;\;\;\;\;\;\;\;\;\;\;\;
\begin{minipage}{0.35\textwidth}
Stage 3: \textbf{EM2a,b,c,d}

Stage 4: \textbf{H\L T}, \textbf{EM2a,b,c,d}
\end{minipage}
\caption{The different situations we find in the reduced graph, and the stages we use to work through them, along with the sections they are carried out in and the corresponding lemmas. Which embedding methods are used at which stages is recorded underneath. The numbers $n^-$ and $(n+t_2)^-$ on top refer to the total number of vertices in $G$ covered by $I_A\cup I_B$, while the numbers beneath refer to the number of vertices of $G$ covered by $I_A$ or $I_B$ as appropriate. The shaded areas indicate that the corresponding edges in the reduced graph are mostly that colour.}\label{fig:stages}
\end{figure}


\begin{figure}[p]
\input{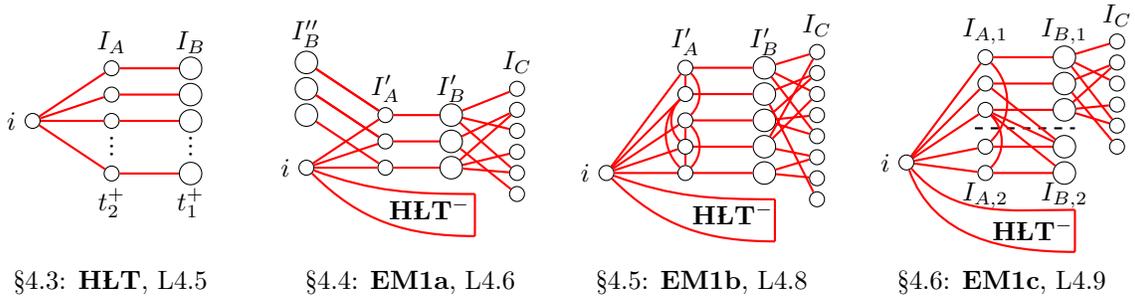}
\caption{The structure in the reduced graph required for embedding methods \textbf{H\L T} and \textbf{EM1a-c}, along with their corresponding sections, names, and lemmas. On the left, $t_2^+$ and $t_1^+$ refer to the number of vertices in $G$ covered by $I_A$ and $I_B$. In each other structure, \textbf{H\L T$^-$} refers to the same structure as \textbf{H\L T} attached to $i$ but covering $t_2^-$ and $t_1^-$ vertices, and comprises the majority of the required structure, while the remaining structure pictured covers $0^+$ vertices in $G$. In \textbf{EM1a-c}, each vertex in $I_B'$ or $I_{B,1}$ has some red neighbours in $I_C$. In \textbf{EM1b} there are some red edges within $I_{A}'$, while in \textbf{EM1c} there are sets $I_{A,2}$ and $I_{B,2}$ matched together in red and every vertex in these sets has some red neighbours in $I_{A,1}$.}\label{fig:embeddings:1}
\end{figure}

\begin{figure}[p]
\input{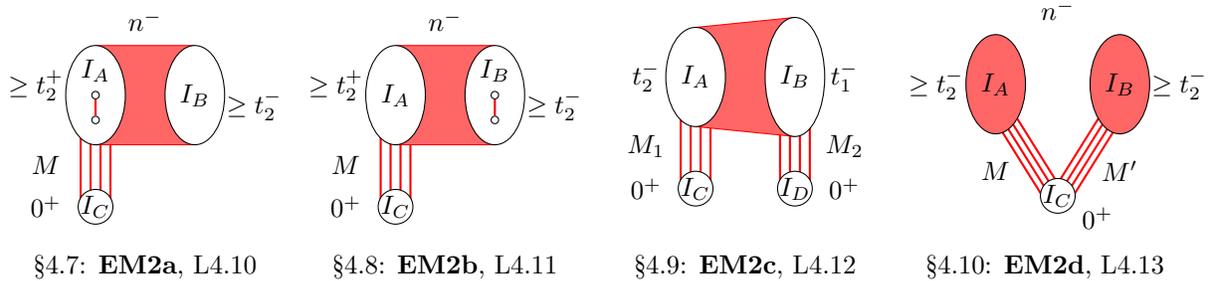}
\caption{The structure in the reduced graph required for \textbf{EM2a-d}, along the corresponding sections, names, and lemmas. In all cases $I_A$ and $I_B$ together cover $n^-$ vertices, and lower bounds or sizes for the number of vertices covered by $I_A$ and by $I_B$ are given in each case. In \textbf{EM2a-c}, almost all of the edges in $R$ between $I_A$ and $I_B$ are red. In \textbf{EM2a} and \textbf{EM2b}, there is a red edge in $R[I_A]$ and $R[I_B]$, respectively. In \textbf{EM2d} almost all edges in $R$ within $I_A$ and within $I_B$ are red.}\label{fig:embeddings:2}
\end{figure}

At each stage, our deductions will often say that if there is a certain structure in the reduced graph $R$, then we can find a monochromatic copy of the tree $T$ in $G$ using a named embedding method. These embedding methods include one due to Haxell, \L uczak, and Tingley~\cite{haxell2002ramsey} that we will refer to as \textbf{H\L T}, as well as a series of new ones denoted by \textbf{EM1a-c} and \textbf{EM2a-d} that we will prove in Section~\ref{sec:embedregularity}. The structure required for each of these methods is depicted in either Figure~\ref{fig:embeddings:1} or Figure~\ref{fig:embeddings:2}, marked with relevant references to the corresponding lemmas and the sections they are proved in. Formal definitions of the required structures can be found in the relevant sections, though informal descriptions of these structures are provided in the captions. Which embedding methods are used for which stages are noted in Figure~\ref{fig:stages}. To give a rough idea of how these embedding lemmas are proved, we will discuss and sketch a proof of simplest embedding method \textbf{H\L T} below, and briefly relate it to the other methods. All of these proofs use the same general framework relying on a technical embedding lemma using regularity proved in Section~\ref{sec:maintechnicalembedding}.

\medskip

\noindent\textbf{Embedding method H\L T.} As shown by Haxell, \L uczak, and Tingley~\cite{haxell2002ramsey}, if the left-most structure in Figure~\ref{fig:embeddings:1} can be found in the reduced graph $R$ in red, say, then we can find a red copy of $T$ in $G$. More precisely, the structure is defined as follows. There is an $\eps$-regular partition $V_1\cup\cdots\cup V_{2k+1}$ in $G$ with a corresponding $(\eps,d)$-reduced graph $R$, an index $i\in [2k+1]$ and a partition $[2k+1]\setminus\{i\}=I_A\cup I_B$, such that $ia$ is an edge in $\red{R}$ for each $a\in I_A$, and there is a perfect matching $M$ in $\red{R}$ between $I_A$ and $I_B$. Furthermore, $I_A$ covers $t_2^+$ vertices of $G$ while $I_B$ covers $t_1^+$ vertices of $G$, and the ratio between the sizes of the clusters indexed by $I_B$ and by $I_A$ is around $t_1:t_2$.

Observe that the structure described here is a bipartite subgraph of the reduced graph, so if we are to embed the tree $T$ into the $\eps$-regular pairs corresponding to edges in this structure, it is necessary that $I_A$ covers $t_2^+$ vertices of $G$ so that there is enough room for vertices in $U_2$ to be embedded among them, and similarly that $I_B$ covers $t_1^+$ vertices of $G$ for the embedding of vertices in $U_1$. To prove their approximate version of Theorem~\ref{thm:main} in~\cite{haxell2002ramsey}, Haxell, \L uczak, and Tingley started with a red/blue coloured complete graph on $(t_1+2t_2)^+$ vertices, and showed that its reduced graph will always contain, in red or blue, the structure required to apply \textbf{H\L T}. However, as our graph $G$ has only $t_1+2t_2-1$ vertices, we cannot find this structure in full. Instead, we use their result to show that we can find a slightly scaled down version of \textbf{H\L T} with the corresponding sets $I_A$ and $I_B$ covering $t_2^-$ and $t_1^-$ vertices respectively (see Theorem~\ref{thm:hltstart}). This is the \textbf{A-situation} depicted in Figure~\ref{fig:stages}, and we denote this structure by \textbf{H\L T$^-$}.

Roughly speaking, in each of our other embedding methods we will be able to embed most of the tree relatively easily, but need to somehow make up for a lack of vertices in both $I_A$ and $I_B$ in the \textbf{H\L T$^-$} structure. This can be seen in the structures required for \textbf{EM1a-c} depicted in Figure~\ref{fig:embeddings:1}, where part of the structure in each case is \textbf{H\L T$^{-}$}, and we need to find some additional structure to complete the embedding.
A major driver of the complexity in our embedding is that the tree $T$ could have a linear maximum degree, so the diameter of $T$ could be as small as 5. This means that structures in the reduced graph with larger diameters are often not very useful for us.

As an example, we now give a sketch of how to embed $T$ into the structure \textbf{H\L T} described above. We start by finding a constant-sized set of vertices $X\subset V(T)$ such that $T-X$ has only small components (see Lemma~\ref{lemma:trees:cutset} and Lemma~\ref{lem:maintreedecomp}). These components could be linear-sized, but must be much smaller compared to the regularity clusters. Let $i_1=i$. Remove an arbitrary edge from $M$, let $i_2$ be the endpoint of this edge in $I_A$, and note that $i_1i_2$ is an edge in $\red{R}$. This small modification is depicted in Figure~\ref{fig:HLTproof}.

Our aim is to embed $T$ into $\red{G}$ so that vertices in $X\cap U_2$ are embedded into $V_{i_2}$, vertices in $(X\cap U_1)\cup N_T(X\cap U_2)$ are embedded into $V_{i_1}$, and for each component $K$ of $T-X$, we can assign to it some edge $ab$ in $M$ with $a\in I_A$ and $b\in I_B$, such that all vertices in $K$ not mentioned so far are embedded into either $V_a$ or $V_b$ depending on if they are in $U_2$ or $U_1$, respectively. 
As $X$ is constant-sized, by choosing the maximum degree parameter $c$ to be small enough, $N_T(X\cap U_2)$ will be a linear-sized set small enough to be embedded along with $X\cap U_1$ into the regularity cluster $V_{i_1}$. We are only embedding $X\cap U_2$ into the cluster $V_{i_2}$, so there is plenty of room there. To make sure that we have enough room for the rest of the embedding, for each component of $T-X$ we will decide which clusters to embed it into by picking an edge $ab$ in $M$ independently and uniformly at random. As each component of $T-X$ is small, with high probability this will distribute the components of $T-X$ across the edges of $M$ without too many vertices assigned to any one edge. This ensures we have enough room in each cluster, so standard regularity techniques now apply to find a red copy of $T$.

\medskip

\noindent\textbf{Overview of the 4 stages in the proof of Theorem~\ref{thm:stability}.}
We start the 4-stage proof of Theorem~\ref{thm:stability} by applying the aforementioned Haxell, \L uczak, and Tingley~\cite{haxell2002ramsey} result (see Theorem~\ref{thm:hltstart}) to our red/blue coloured complete graph $G$ to find a monochromatic \textbf{H\L T$^-$} structure in the reduced graph $R$. This is the \textbf{A-situation} depicted in the left of Figure~\ref{fig:stages}, and we can now proceed to \textbf{Stage 1}. 

\vspace{-0.2cm}

\begin{center}
\rule{5cm}{0.2pt}\;\;\textbf{Stage 1}\;\;\rule{5cm}{0.2pt}
\end{center}

\vspace{-0.3cm}

Given an \textbf{A-situation} in $R$, say in red, $I_A$ and $I_B$ each does not cover enough vertices to let us use the embedding method \textbf{H\L T}. Let $I_C=V(R)\setminus (I_A\cup I_B\cup \{i\})$, so that $I_C$ covers $(n+t_2)^--t_2^--t_1^-=t_2^-$ vertices. If almost all of the edges in $R$ between $I_B$ and $I_C$ are blue, then we would have a \textbf{C-situation} in blue with $I_C$ in place of $I_A$, and can skip ahead to \textbf{Stage 3}.
Suppose, then, that there are at least some red edges between $I_B$ and $I_C$. In particular, if $I_{B,1}$ is the set of vertices in $I_B$ with at least some red neighbours in $I_C$, then $I_{B,1}$ is non-empty. 
Let $I_{A,1}$ be the set of vertices matched with $I_{B,1}$ by the matching $M$.
Then, take $I_{A,3}$ to be the set of vertices in $I_A\setminus I_{A,1}$ with at least some red neighbours in $I_C$, and let $I_{B,3}$ be the set of vertices matched with $I_{A,3}$ by $M$. Finally, let $I_{A,2}=I_A\setminus (I_{A,1}\cup I_{A,3})$ and $I_{B,2}=I_B\setminus (I_{B,1}\cup I_{B,3})$. See the left of Figure~\ref{fig:stage12} for a depiction of these vertex sets.

\begin{figure}
\input{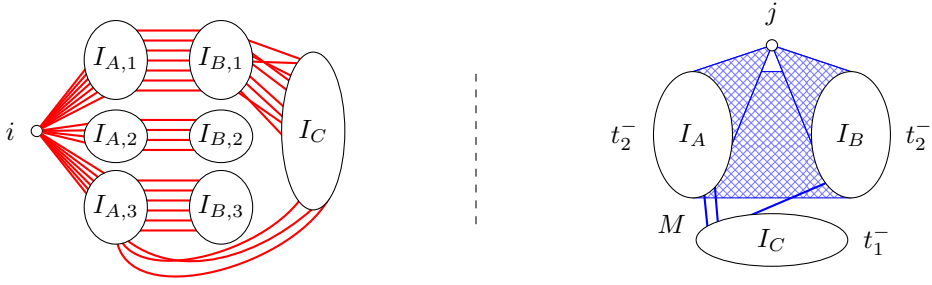}

\vspace{-0.5cm}

\caption{On the left, the main structure in Stage 1. On the right, the main structure in Stage 2.}\label{fig:stage12}
\end{figure}

Suppose there is no copy of $T$ in red, and thus the structure required to use any of \textbf{EM1a-c} does not exist in red. We will be able to show, then, that \textbf{i)} the edges between $I_{A,1}$ and $I_C$ are almost all blue, \textbf{ii)} the edges between $I_{A,1}$ and $I_{B,3}$ are almost all blue, \textbf{iii)} most of the edges in $I_{A,1}$ are blue, and \textbf{iv)} for most of the edges $i_Ai_B\in M[I_{A,2},I_{B,2}]$, one of $i_A$ or $i_B$ will have mostly blue neighbours in $I_{A,1}$. These deductions are commented on below,
but assuming \textbf{i)}--\textbf{iv)} hold, we find a \textbf{B-stituation} as follows.

From \textbf{iv)}, we can find a subset $I_{AB,2}\subset I_{A,2}\cup I_{B,2}$ containing a vertex in almost every edge in $M[I_{A,2},I_{B,2}]$, so that the edges between $I_{AB,2}$ and $I_{A,1}$ are mostly blue. Combined with \textbf{ii)} and \textbf{iii)}, the edges between $I_D:=I_{AB,2}\cup I_{B,3}\cup I_{A,1}$ and $I_{A,1}$ are mostly blue. Furthermore, since $I_D$ contains a vertex in almost every edge in $M$, and every cluster indexed by $I_{B}$ contains more vertices than one indexed by $I_{A}$, we see that $I_D$ covers at least close to the same number of vertices as $I_A$ does, and thus $I_D$ covers at least $t_2^-$ vertices. Moreover, as $I_{A,1}$ is non-empty and \textbf{i)} holds, the edges from $I_C$ to $I_D$ are almost all blue, and we can select some $j\in I_{A,1}$ with almost all blue edges to $I_C\cup I_D$. Thus, $j$, $I_C$ and $I_D\setminus\{j\}$ give the structure required for a \textbf{B-situation} in blue. 

We finish this discussion of \textbf{Stage 1} by commenting briefly on the four deductions mentioned above and the embedding methods required for them, using the same labels as in Figure~\ref{fig:embeddings:1} where possible.

\medskip

\noindent\textbf{i) $R[I_{A,1},I_C]$ is almost all blue.} If this does not hold, then let $I'_{A,1}$ be a small set of vertices in $I_{A,1}$ with some red neighbours in $I_C$, and let $I'_{B,1}$ be the vertices matched with $I'_{A,1}$ by $M$. We can then find a perfect red matching between $I_{A,1}'$ and some $I_{B,1}''\subset I_C$, and set $I'_C=I_C\setminus I''_{B,1}$ to obtain the structure required for \textbf{EM1a}.

\medskip

\noindent\textbf{ii) $R[I_{A,1},I_{B,3}]$ is almost all blue.} If this does not hold, then, similar to \textbf{i)}, let $I'_{A,1}$ be a small set of vertices in $I_{A,1}$ with some red neighbours in $I_{B,3}$, let $I'_{B,1}$ be the vertices matched with $I_{A,1}'$ by $M$, and find a perfect red matching between $I_{A,1}'$ and some $I''_{B,1}\subset I_{B,3}$. Unlike \textbf{i)}, using $I''_{B,1}$ for the structure in \textbf{EM1a} will `orphan' the vertices in $I_{A,3}$ matched to $I''_{B,1}$ by $M$, say those in $I'_{A,3}$. However, from the definition of $I_{A,3}$, we can find a perfect red matching between $I'_{A,3}$ and some $I'_{B,3}\subset I_C$, which can be used to replace $I''_{B,1}$ to complete the structure required for \textbf{EM1a}.


\medskip

\noindent\textbf{iii) $R[I_{A,1}]$ is almost all blue.} If this does not hold, then we have the structure required for \textbf{EM1b} -- some red edges within a set $I_{A,1}$ whose neighbours under $M$ (i.e., the vertices in $I_{B,1}$) all have some red neighbours in $I_C$.

\medskip

\noindent\textbf{iv) For most edges in $M[I_{A,2},I_{B,2}]$, one endpoint has mostly blue edges to $I_{A,1}$.}
If this does not hold, then we have the structure required for \textbf{EM1c} -- a small red submatching of $M[I_{A,2},I_{B,2}]$ whose vertices all have some neighbours within $I_{A,1}$.



\vspace{-0.2cm}

\begin{center}
\rule{5cm}{0.2pt}\;\;\textbf{Stage 2}\;\;\rule{5cm}{0.2pt}
\end{center}

\vspace{-0.3cm}

Suppose, now, we have a blue \textbf{B-situation} in $R$: a vertex $j$ with blue edges to almost every vertex in two disjoint sets $I_A$ and $I_B$, with both $I_A$ and $I_B$ covering $t_2^-$ vertices of $G$, and almost every edge between them being blue. Let $M$ be a maximum blue matching between $I_A\cup I_B$ and $I_C$ (see the right of Figure~\ref{fig:stage12}). If $I_C\cap V(M)$ covers more than $(t_1-t_2)^+$ vertices in $G$, then using this matching and the \textbf{B-situation} structure, we can find the structure required to embed $T$ in blue using \textbf{H\L T}, where we use $j$ as the vertex $i$ in the \textbf{H\L T} structure and use that $I_A\cup I_B\cup (I_C\cap V(M))$ covers $2t_2^-+(t_1-t_2)^+=n^+$ vertices in $G$.

Therefore, we can assume that $I_C\cap V(M)$ covers at most $(t_1-t_2)^+$ vertices, so $I_C\setminus V(M)$ covers at least $(n+t_2)^--2t_2^--(t_1-t_2)^+=t_2^-$ vertices. If almost all edges between a subset $I\subset I_A\cup I_B$ and $I_C\setminus V(M)$ are red, with $I\cup(I_C\setminus V(M))$ covering $n^-$ vertices in $G$, then we have a red \textbf{C-situation}. Unfortunately, the maximality of $M$ only immediately gives that almost all edges between $(I_A\cup I_B)\setminus V(M)$ and $I_C\setminus V(M)$ are red, and they might cover only $(n+t_2)^--2(t_1-t_2)^+=(4t_2-t_1)^-$ vertices together, which could be as small as $2t_2^-$ if $t_1\approx 2t_2$.

To combat this, we exploit the maximality of the matching $M$ in a more sophisticated way using what we call a `cascading argument' (see Lemma~\ref{lemma:improvemaximummatching}). Note that any blue edge between some $c\in I_C\setminus V(M)$ and $(I_A\cup I_B)\cap V(M)$ would allow us to exchange that edge into the matching $M$ to create a matching $M'$ that has the same intersection with $I_A\cup I_B$ as $M$, but whose intersection with $I_C$ includes $c$ and omits some vertex $c'\in I_C\cap V(M)$. The maximality of $M$ then implies the edges between $c'$ and $(I_A\cup I_B)\setminus V(M)$ are mostly red. Iterating such an argument will eventually allow us to find two large subsets $X$ and $Y$ with $(I_A\cup I_B)\setminus V(M)\subset X\subset I_A\cup I_B$ and $I_C\setminus V(M)\subset Y\subset I_C$, such that the edges between $X$ and $Y$ are mostly red, $X$ and $Y$ both cover at least $t_2^-$ vertices in $G$, and $X\cup Y$ cover at least $n^-$ vertices in $G$. This gives a red \textbf{C-situation}.


\vspace{-0.2cm}

\begin{center}
\rule{5cm}{0.2pt}\;\;\textbf{Stage 3}\;\;\rule{5cm}{0.2pt}
\end{center}

\vspace{-0.3cm}

Suppose then we have a red \textbf{C-situation} in $R$, which consists of two disjoint vertex sets $I_A$ and $I_B$, each covering at least $t_2^-$ vertices and together covering $n^-$ vertices, such that $R[I_A,I_B]$ is mostly red. Let $I_C=V(R)\setminus (I_A\cup I_B)$, which covers $(n+t_2)^--n^-=t_2^-$ vertices. In the rest of \textbf{Stage 3} we will use two different sequences of deductions (\textbf{Claim A} and \textbf{Claim B}) several times. Before continuing then, we state roughly what they are and summarise the arguments for them.

\begin{figure}[h]
\input{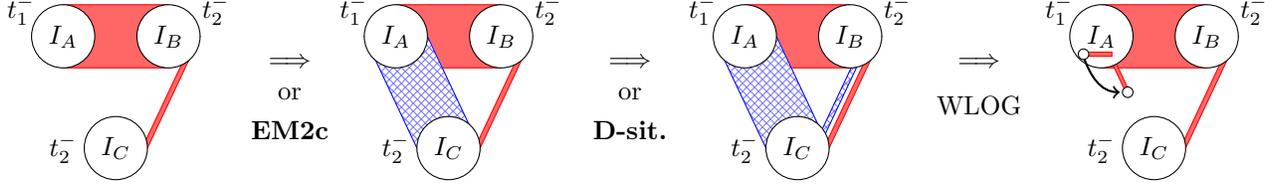}

\vspace{-0.5cm}
\caption{The deductions for \textbf{Claim A}, before finally \textbf{EM2c} is applied to get a red copy of $T$.}\label{fig:claimA}
\end{figure}

\noindent\textbf{Claim A:} \emph{If $I_A$ and $I_B$ cover $t_1^-$ and $t_2^-$ vertices respectively, and there are some red edges between $I_B$ and $I_C$, then we can either find a monochromatic copy of $T$ or reach a \emph{\textbf{D-situation}}.}

\noindent\textbf{Argument for Claim A:} If there are some red edges between $I_A$ and $I_C$ then \textbf{EM2c} applies. Thus, we can assume $R[I_A,I_C]$ is almost all blue. If $R[I_B,I_C]$ is mostly red, then we have a \textbf{D-situation} in red using $I_A\cup I_C$ and $I_B$, so there must be some blue edges between $I_B$ and $I_C$, as well as some red edges as part of the assumption. If there are some red edges in $I_A$, then we can find a small red matching in $I_A$ and move one side of this matching out of $I_A$ to get the structure required for \textbf{EM2c} in red. Finally, if there are some blue edges in $I_A$ then we can similarly apply \textbf{EM2c} in blue.

\begin{figure}[h]
\input{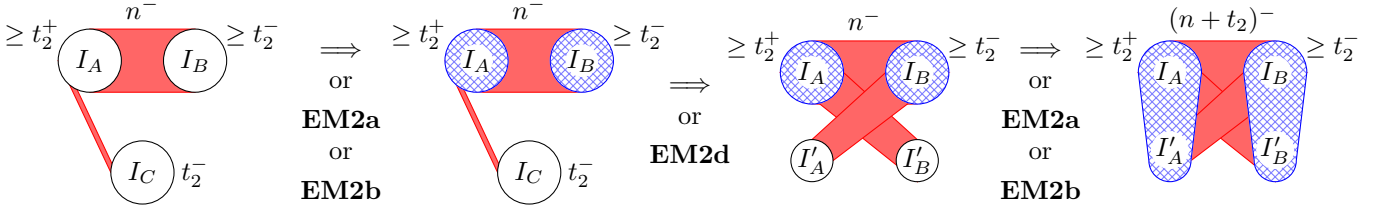}

\vspace{-0.5cm}
\caption{The deductions for \textbf{Claim B}, before finally either \textbf{EM2d} applies or we have a \textbf{D-situation}.}\label{fig:claimB}
\end{figure}

\noindent\textbf{Claim B:} \emph{Suppose $I_A$ and $I_B$ cover at least $t_2^+$ and $t_2^-$ vertices respectively, and $n^-$ vertices in total. If there are some red edges between $I_A$ and $I_C$, then we can either find a monochromatic copy of $T$ or reach a \emph{\textbf{D-situation}}.}

\noindent\textbf{Argument for Claim B:} If there are some red edges in $R[I_A]$ or $R[I_B]$, then \textbf{EM2a} or \textbf{EM2b} applies respectively, so assume that $R[I_A]$ and $R[I_B]$ are both mostly blue. If some vertices in $I_C$ have some blue neighbours in both $I_A$ and $I_B$, then we can use \textbf{EM2d}. Thus, we can partition most of $I_C$ into $I_A'\cup I_B'$, such that $R[I_A,I_B']$ and $R[I_A',I_B]$ are both mostly red.

Like above, if there are some red edges in $R[I_A\cup I_A']$ or $R[I_B\cup I_B']$, then \textbf{EM2a} or \textbf{EM2b} applies respectively, so we can assume that both $R[I_A\cup I_A']$ and $R[I_B\cup I_B']$ are mostly blue. If there is some vertex in $I_A'$ with some blue neighbours in $I_B\cup I_B'$, or if there is some vertex in $I_B'$ with some blue neighbours in $I_A\cup I_A'$, then we can apply \textbf{EM2d}. Thus, we can assume that $R[I_A\cup I_A',I_B\cup I_B']$ is mostly red, which gives a \textbf{D-situation}.

\medskip

\noindent\textbf{Stage 3 using Claim A and B.} Using these two claims, we can now carry out \textbf{Stage 3}.
Note first that if $R[I_A\cup I_B,I_C]$ is mostly blue, then they form a \textbf{D-situation}, so assume that there are some red edges either between $I_A$ and $I_C$ or between $I_B$ and $I_C$.
If both $I_A$ and $I_B$ cover at least $t_2^+$ vertices, then we can use \textbf{Claim B} to find a monochromatic copy of $T$ or reach a \textbf{D-situation}.

Thus, we can assume without loss of generality that $I_B$ covers at most $t_2^+$ vertices, so $I_A$ covers at least $t_1^-$ vertices. If there are some red edges between $I_B$ and $I_C$, then we can use \textbf{Claim A} to find a monochromatic copy of $T$ or reach a \textbf{D-situation}. Otherwise, there must be some red edges between $I_A$ and $I_C$. If $t_1\leq t_2^+$, then we can apply \textbf{Claim A} with $I_A$ and $I_B$ swapped as $t_1\approx t_2$, while if $t_1\geq t_2^+$, then we can apply \textbf{Claim B}.

\vspace{-0.2cm}

\begin{center}
\rule{5cm}{0.2pt}\;\;\textbf{Stage 4}\;\;\rule{5cm}{0.2pt}
\end{center}

\vspace{-0.3cm}

Suppose finally that we have a blue \textbf{D-situation} in $R$, which consists of two disjoint vertex sets $I_A$ and $I_B$, each covering at least $t_2^-$ vertices and together covering $(n+t_2)^-$ vertices, such that $R[I_A,I_B]$ is mostly blue. First, suppose in addition that both $I_A$ and $I_B$ cover at least $t_2^+$ vertices. Then, if there is a blue edge in either $R[I_A]$ or $R[I_B]$, we can take some vertices out of $I_B$ to form $I_C$, which allows us to apply \textbf{EM2a} or \textbf{EM2b}, respectively, to find a blue copy of $T$. Thus, we can assume that $R[I_A]$ and $R[I_B]$ are both mostly red. If the larger of $I_A$ and $I_B$, which we can assume is $I_A$, covers at least $t_1^+$ vertices, then we can take some vertices out of $I_A$ to form $I_D$ and some vertices out of $I_B$ to form $I_C$, so that we can apply \textbf{EM2c} to get a blue copy of $T$. If $I_A$ covers at most $t_1^+$ vertices, then we have $(n+t_2)^-\leq 2t_1^+$, and so $t_1\approx2t_2$ and both $I_A$ and $I_B$ must cover at least $t_1^-$ vertices in $G$. This gives an \textbf{E-situation}, and will imply that $G$ is close to a Type II extremal construction.

Now suppose that the smaller of $I_A$ and $I_B$, which we can assume is $I_B$, covers at most $t_2^+$ vertices. Then, $I_A$ covers $(n+t_2)^--t_2^+=n^-$ vertices in $G$. If there are some blue edges in $R[I_A]$, then we can use them to take some vertices out of $I_A$ to form $I_C$, and then apply \textbf{EM2a}. Thus, we can assume that $R[I_A]$ is mostly red. If $I_A$ covers at least $n^+$ vertices, then we can easily find the structure required to apply \textbf{EM2c} in $R[I_A]$. Therefore, we can assume that $I_A$ covers at most $n^+$ vertices, and so we have an \textbf{E-situation} that will imply that $G$ is close to a Type I extremal construction.


\section{Embedding methods for the proof of Theorem~\ref{thm:stability}: Stability}\label{sec:embed regularity}\label{sec:embedregularity}
In this section, we prove a series of embedding lemmas using regularity, each of which says that if a certain structure exists in the reduced graph $R$, then we can embed $T$ into $G$. In Section~\ref{sec:newtreedecomp}, we prove a tree decomposition lemma phrased in terms of graph homomorphisms, which cuts the tree $T$ into small pieces by removing very few vertices. Then, in Section~\ref{sec:maintechnicalembedding}, we prove our main technical lemma, Lemma~\ref{lem:maintechnicalembedding}. Roughly speaking, it says that given a suitable structure in the reduced graph and an appropriate assignment of each small piece of the tree $T$ to a part of the structure, we can find a copy of $T$ by embedding each piece between a randomly chosen regular pair within the part it is assigned to. This is then applied to prove embedding methods \textbf{H\L T} in Section~\ref{sec:HLT}, \textbf{EM1a-c} in Sections~\ref{sec:EM1a}--\ref{sec:EM1c}, and \textbf{EM2a-d} in Sections~\ref{sec:EM2a}--\ref{sec:EM3}. In each application, the structure in the reduced graph provided by the assumption is transformed into a substructure of the one used in Lemma~\ref{lem:maintechnicalembedding}, then we find a proper assignment of each piece of the tree $T$ to a part of this structure, so that on average no cluster has too many vertices assigned to it. 

As mentioned at the end of Section~\ref{sec:division}, by adding leaves to the $t_2$-side of the tree if necessary, we can assume that $t_1\leq 2t_2+1$. In fact, as we will show later in Section~\ref{sec:stages} when we prove Theorem~\ref{thm:stability}, it can even be assumed that $t_1\leq 2t_2$. As such, all of the embedding methods we prove below in this section will have the assumption that $t_2\leq t_1\leq 2t_2$.  

\subsection{Tree decomposition}\label{sec:newtreedecomp}
In this subsection, we prove the following lemma phrased in terms of graph homomorphisms that cuts the tree $T$ into small pieces by removing very few vertices. Recall that for graphs $H_1$ and $H_2$, a function $\phi:H_1\to H_2$ is a \textit{graph homomorphism} if for any edge $uv$ in $H_1$, $\phi(u)\phi(v)$ is also an edge in $H_2$.


\begin{lemma}\label{lem:maintreedecomp}
Let $1/n\ll c\ll \xi$. Let $S$ be the following graph:

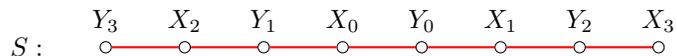
\vspace{-0.2cm}

\begin{center}
\begin{tikzpicture}[scale=.7,main1/.style = {circle,draw,fill=none,inner sep=1.5}, main2/.style = {circle,draw,fill=none,inner sep=2.25}]
\def\spacer{2};
\def\spacerr{1.5};
\def\Ahgt{1.2};
\def\Bhgt{1.2};

\foreach \n in {1,...,8}
{
\node[main1] (A\n) at (\n*\spacerr,0) {};
}
\foreach \n/\m in {1/2,2/3,3/4,4/5,5/6,6/7,7/8}
{
\draw [thick,red] (A\n) -- (A\m);
}
\foreach \where/\what in {1/Y_3,2/X_2,3/Y_1,4/X_0,5/Y_0,6/X_1,7/Y_2,8/X_3}
{
\draw ($(A\where)+(0,0.5)$) node {$\what$};
}
\draw ($(A1)-(1.5,0)$) node {$S:$};
\end{tikzpicture}

\end{center}

\vspace{-0.3cm}

\noindent Let $T$ be an $n$-vertex tree. Then, there is a homomorphism $\phi:T\to S$ such that each component of $T-\phi^{-1}(X_0\cup Y_0)$ has size at most $\xi n$ and $|\phi^{-1}(X_0\cup Y_0\cup X_1\cup Y_1)|\leq \xi n$.
\end{lemma}
\begin{proof}
Let the bipartition classes of $T$ be $V_1$ and $V_2$. By Lemma~\ref{lemma:trees:cutset}, there exists $Z\subset V(T)$ such that each component of $T-Z$ has size at most $\xi n$ and $|Z|\leq 2\xi^{-1}$. It follows that $T-Z$ contains at most $|Z|\cdot\Delta(T)\leq2cn\xi^{-1}\leq\xi n/10$ components. Arbitrarily pick $t_1\in Z$, view $T$ as being rooted at $t_1$, and extend $t_1$ to an ordering $t_1, t_2,\ldots, t_n$ of the vertices of $T$, such that for each $2\leq i\leq n$, $t_i$ has a unique neighbour, namely its parent in $T$, to its left in this ordering, and vertices in the same component of $T-Z$ appear consecutively. 

Let $A\subset V(T-Z)$ be the set consisting of each vertex in a component of $T-Z$ that appears first in the ordering. Let $B$ be the set of parents of vertices in $Z$, and let $C$ be the set of parents and children of vertices in $B$. Note that the sets $A,B,C$ could overlap. For each $\ast\in\{A,B,C,Z\}$ and $j\in[2]$, let $\ast_j=\ast\cap V_j$.

We now define a homomorphism $\phi:T\to S$, so that the following conditions are maintained. 
\stepcounter{propcounter}
\begin{enumerate}[label =\textbf{\Alph{propcounter}\arabic{enumi}}]
    \item\label{hom1} $\phi(t)=X_0$ for each $t\in Z_1$ and $\phi(t)=Y_0$ for each $t\in Z_2$.
    \item\label{hom2} $\phi(t)\in\{X_0,X_1\}$ for each $t\in B_1$ and $\phi(t)\in\{Y_0,Y_1\}$ for each $t\in B_2$.
    \item\label{hom3} $\phi^{-1}(\{X_0,X_1,Y_0,Y_1\})\subset A\cup B\cup C\cup Z$. 
\end{enumerate}
To initialise, let $\phi(t_1)=X_0$ if $t_1\in Z_1$ and $\phi(t_1)=Y_0$ if $t_1\in Z_2$. Suppose we have just finished defining $\phi$ for a vertex in $Z$ or a component in $T-Z$. 

Suppose first that the next vertex in the ordering is a vertex $t_i\in Z$. If $t_i\in Z_1$, then we can define $\phi(t_i)=X_0$ as the image of its parent in $T$ under $\phi$ is adjacent to $X_0$ in $S$ by~\ref{hom2}. Similarly, if $t_i\in Z_2$, then we can safely define $\phi(t_1)=Y_0$.

Now suppose the next vertex in the ordering is a vertex $t_i\in A$ within a component $K$ of $T-Z$. Assume that $t_i\in A_1$, then $\phi$ sends its parent to $Y_0$ by~\ref{hom1}, so we can define $\phi(t_i)=X_1$. Note that~\ref{hom2} is also maintained if $t_i$ happens to be in $B_1$. Define $\phi$ on the remaining vertices $t\in K$ as follows.
\begin{center}
\begin{tabular}{|c|c|c|c|c|c|c|c|c|}
    \hline$t\in$ & $A_1$ & $B_1$ & $B_2$ & $C_1$ & $V_1\setminus(A_1\cup B_1\cup C_1)$ & $V_2\setminus B_2$ \\\hline
    $\phi(t)=$ & $X_1$ & $X_1$ & $Y_0$ & $X_1$ & $X_3$ & $Y_2$ \\\hline
\end{tabular}
\end{center}
One can check that this defines a valid homomorphism using the definitions of $A,B,C$. For example, if $t\in B_2\cap K$, then $\phi(t)$ is defined to be $Y_0$. The parent and children of $t$ are in $C_1$ from definition, and are sent by $\phi$ to $X_1$, which is valid.  Moreover, both~\ref{hom2} and~\ref{hom3} are maintained. The case when $t_i\in A_2$ is symmetric so is omitted. 

Therefore, we can define a homomorphism $\phi:T\to S$ satisfying~\ref{hom1}--\ref{hom3}. Every component of $T-\phi^{-1}(\{X_0,Y_0\})$ has size at most $\xi n$ as it is contained in a component of $T-Z$. Finally, by~\ref{hom3}, $|\phi^{-1}(\{X_0,X_1,Y_0,Y_1\})|\leq|A|+|B|+|C|+|Z|\leq\xi n/10+2\xi^{-1}+2\xi^{-1}\cdot\Delta(T)+2\xi^{-1}\leq\xi n$, as required. 
\end{proof}


\subsection{Main technical embedding lemma}\label{sec:maintechnicalembedding}
In this subsection, we prove our main technical embedding lemma. Roughly speaking, it says that given a reduced graph structure $R$ and a tree $T$ cut into many small pieces, if each piece can be assigned appropriately into a part of $R$ (represented below as a homomorphism $\varphi:T\to R'$), so that on average no cluster has too many vertices assigned to it (see~\ref{prop:maintech2}), then we can find a copy of $T$ in $G$. 
\begin{figure}
\begin{center}
\begin{tikzpicture}[scale=.7,main1/.style = {circle,draw,fill=none,inner sep=1.5}, main2/.style = {circle,draw,fill=none,inner sep=2.25}]
\def\spacer{2};
\def\spacerr{1.5};
\def\Ahgt{1.2};
\def\Bhgt{1.2};

\coordinate (B) at (\spacer,0);

\node[main1] (V1) at ($(-\spacerr,1.35)$) {};
\node[main1] (V2) at ($(-\spacerr,-0.1)$) {};
\node[main1] (V3) at ($(-\spacerr,-1.45)$) {};

\coordinate (A0) at (0,2.75);
\coordinate (A1) at (0,1.35);
\coordinate (A2) at (0,-0.05);
\coordinate (A3) at (0,-1.45);
\foreach \n in {0,1,2,3}
{
\coordinate (B\n) at ($(A\n)+(\spacer,0)$);
\coordinate (C\n) at ($(B\n)+(\spacer,0)$);
}

\foreach \n in {1,...,4}
{
\coordinate (C0\n) at ($(A0)+(0,0.5)-\n*(0,0.19)$);
\coordinate (D0\n) at ($(B0)+(0,0.5)-\n*(0,0.19)$);
\coordinate (E0\n) at ($(C0)+(0,0.5)-\n*(0,0.19)$);
\draw [thick,red] (C0\n) -- (D0\n);
\draw [thick,red] (E0\n) -- (D0\n);
}

\foreach \n in {1,...,7}
{
\coordinate (C1\n) at ($(A1)+(0,0.75)-\n*(0,0.19)$);
\coordinate (D1\n) at ($(B1)+(0,0.75)-\n*(0,0.19)$);
\coordinate (E1\n) at ($(C1)+(0,0.75)-\n*(0,0.19)$);
\draw [thick,red] (C1\n) -- (D1\n);
\draw [thick,red] (E1\n) -- (D1\n);
}

\foreach \n in {1,...,4}
{
\coordinate (C2\n) at ($(A2)+(0,0.5)-\n*(0,0.19)$);
\coordinate (D2\n) at ($(B2)+(0,0.5)-\n*(0,0.19)$);
\coordinate (E2\n) at ($(C2)+(0,0.5)-\n*(0,0.19)$);
\draw [thick,red] (C2\n) -- (D2\n);
\draw [thick,red] (E2\n) -- (D2\n);
}

\foreach \n in {1,...,6}
{
\coordinate (C3\n) at ($(A3)+(0,0.7)-\n*(0,0.19)$);
\coordinate (D3\n) at ($(B3)+(0,0.7)-\n*(0,0.19)$);
\coordinate (E3\n) at ($(C3)+(0,0.7)-\n*(0,0.19)$);
\draw [thick,red] (C3\n) -- (D3\n);
\draw [thick,red] (E3\n) -- (D3\n);

}

\draw [thick,red] (V1) -- (V2) -- (V3);

\foreach \n in {1,...,4}
{
\draw [thick,red] (V1) -- (C0\n);
}
\foreach \n in {1,...,7}
{
\draw [thick,red] (V1) -- (C1\n);
}
\foreach \n in {1,...,4}
{
\draw [thick,red] (V2) -- (C2\n);
}
\foreach \n in {1,...,6}
{
\draw [thick,red] (V3) -- (C3\n);
}


\draw ($(V1)-(0.5,0)$) node {$i_1$};
\draw ($(V2)-(0.5,0)$) node {$i_2$};
\draw ($(V3)-(0.5,0)$) node {$i_3$};

\draw ($(V2)-(1.5,0)$) node {$R:$};

\draw [fill=white] (A0) circle [x radius=0.6cm,y radius=0.5cm];
\draw [fill=white] (A1) circle [x radius=0.6cm,y radius=0.75cm];
\draw [fill=white] (A2) circle [x radius=0.6cm,y radius=0.5cm];
\draw [fill=white] (A3) circle [x radius=0.6cm,y radius=0.7cm];
\draw [fill=white] (B0) circle [x radius=0.6cm,y radius=0.5cm];
\draw [fill=white] (B1) circle [x radius=0.6cm,y radius=0.75cm];
\draw [fill=white] (B2) circle [x radius=0.6cm,y radius=0.5cm];
\draw [fill=white] (B3) circle [x radius=0.6cm,y radius=0.7cm];
\draw [fill=white] (C0) circle [x radius=0.6cm,y radius=0.5cm];
\draw [fill=white] (C1) circle [x radius=0.6cm,y radius=0.75cm];
\draw [fill=white] (C2) circle [x radius=0.6cm,y radius=0.5cm];
\draw [fill=white] (C3) circle [x radius=0.6cm,y radius=0.7cm];

\draw (A0) node { $I_{0,1}$};
\draw (B0) node { $I_{0,2}$};
\draw (C0) node { $I_{0,3}$};
\draw (A1) node { $I_{1,1}$};
\draw (B1) node { $I_{1,2}$};
\draw (C1) node { $I_{1,3}$};
\draw (A2) node { $I_{2,1}$};
\draw (B2) node { $I_{2,2}$};
\draw (C2) node { $I_{2,3}$};
\draw (A3) node { $I_{3,1}$};
\draw (B3) node { $I_{3,2}$};
\draw (C3) node { $I_{3,3}$};

\end{tikzpicture}\hspace{1.5cm}
\begin{tikzpicture}[scale=.7,main1/.style = {circle,draw,fill=none,inner sep=1.5}, main2/.style = {circle,draw,fill=none,inner sep=2.25}]
\def\spacer{1.25};
\def\spacerr{1.25};
\def\Ahgt{1.2};
\def\Bhgt{1.2};


\coordinate (B) at (\spacer,0);

\node[main1] (V1) at ($(-\spacerr,1.35)$) {};
\node[main1] (V2) at ($(-\spacerr,-0.1)$) {};
\node[main1] (V3) at ($(-\spacerr,-1.45)$) {};

\node[main1] (A1) at (0,1.35) {};
\node[main1] (A2) at (0,-0.1) {};
\node[main1] (A3) at (0,-1.45) {};
\foreach \n in {1,2,3}
{
\node[main1] (B\n) at ($(A\n)+(\spacer,0)$) {};
\node[main1] (C\n) at ($(B\n)+(\spacer,0)$) {};
}

\node[main1] (A0) at ($(A1)+(0,1.35)$) {};
\node[main1] (B0) at ($(B1)+(0,1.35)$) {};
\node[main1] (C0) at ($(B0)+(\spacer,0)$) {};
\draw [thick,red] (V1) -- (A0) -- (B0) -- (C0);

\foreach \n in {1,2,3}
{
\draw [thick,red] (V\n) -- (A\n) -- (B\n) -- (C\n);
}

\draw [thick,red] (V1) -- (V2) -- (V3);

\draw ($(V1)-(0.5,0)$) node {$i_1$};
\draw ($(V2)-(0.5,0)$) node {$i_2$};
\draw ($(V3)-(0.5,0)$) node {$i_3$};
\foreach \where/\what in {A1/I_{1,1},B1/I_{1,2},C1/I_{1,3},A2/I_{2,1},B2/I_{2,2},C2/I_{2,3},A3/I_{3,1},B3/I_{3,2},C3/I_{3,3},A0/I_{0,1},B0/I_{0,2},C0/I_{0,3}}
{
\draw ($(\where)+(-0.1,0.5)$) node {$\what$};
}

\draw ($(-\spacerr,0.625)-(1.8,0)$) node {$R':$};
\end{tikzpicture}

\end{center}
\caption{Auxiliary graphs $R$ and $R'$ used in the statement of Lemma~\ref{lem:maintechnicalembedding}}\label{fig:mainembedstatement}
\end{figure}
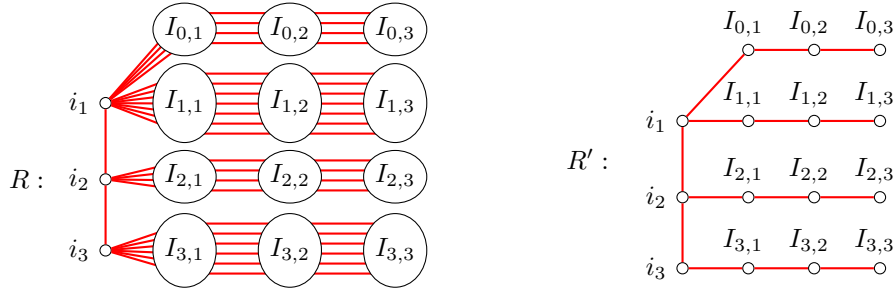

\begin{lemma}\label{lem:maintechnicalembedding} Let $0<1/n\ll c\ll \xi \ll 1/k\ll \eps\ll \alpha\ll d\leq1$. Let $i_1,i_2,i_3\in [k]$ be distinct and let $I_{0,1}, I_{0,2}, I_{0,3}$, $I_{1,1}, I_{1,2}, I_{1,3}$, $I_{2,1}, I_{2,2}$, $I_{2,3}, I_{3,1}, I_{3,2},I_{3,3}$ partition $[k]\setminus \{i_1,i_2,i_3\}$ such that $|I_{0,1}|=|I_{0,2}|=|I_{0,3}|$, $|I_{1,1}|=|I_{1,2}|=|I_{1,3}|$, $|I_{2,1}|=|I_{2,2}|=|I_{2,3}|$, and $|I_{3,1}|=|I_{3,2}|=|I_{3,3}|$.

As depicted in Figure~\ref{fig:mainembedstatement}, let $R$ be a graph with vertex set $[k]$ and edge set consisting of the edges in
\[
\{i_1i_2,i_2i_3\}\cup\{i_1i:i\in I_{0,1}\}\cup \{i_ai:a\in [3],i\in I_{a,1}\},
\]
along with a perfect matching between each of the eight pairs of vertex sets $(I_{a,1},I_{a,2})$ and $(I_{a,2},I_{a,3})$ for every $a\in[3]_0$. Let $R'$ be the graph with vertex set $\{i_1,i_2,i_3\}\cup\{I_{a,b}:a\in[3]_0,b\in[3]\}$ and edge set
\[
\{i_1i_2,i_2i_3,i_1I_{0,1},i_1I_{1,1},i_2I_{2,1},i_3I_{3,1}\}\cup\{I_{a,1}I_{a,2}:a\in[3]_0\}\cup\{I_{a,2}I_{a,3}:a\in[3]_0\}.
\]

Let $G$ be a graph on at most $2n$ vertices with a vertex partition $V_1\cup V_2\cup\cdots\cup V_k$, such that for each $ij\in E(R)$, $G[V_i,V_j]$ is $(\eps,d)$-regular, and for each $I\in V(R')\setminus\{i_1,i_2,i_3\}$, the sets $V_i$ with $i\in I$ all have the same size. Assume also that $n/10k\leq|V_i|\leq\eps n$ for each $i\in \{i_1,i_2,i_3\}$.

Let $T$ be an $n$-vertex tree with $\Delta(T)\leq cn$. Suppose $\varphi:T\to R'$ is a homomorphism such that the following hold.
\stepcounter{propcounter}
\begin{enumerate}[label = \emph{\textbf{\Alph{propcounter}\arabic{enumi}}}]
\item $\varphi^{-1}(\{i_1,i_2,i_3\})\neq \emptyset$, and every component of $T-\varphi^{-1}(\{i_1,i_2,i_3\})$ has size at most $\xi n$.\labelinthm{prop:maintech1}
\item For each $I\in V(R')\setminus\{i_1,i_2,i_3\}$ with $\varphi^{-1}(I)\neq \emptyset$, $|\varphi^{-1}(I)|+\alpha n \leq \sum_{i\in I}|V_i|$.\labelinthm{prop:maintech2}
\item For each $i\in \{i_1,i_2,i_3\}$, $|\varphi^{-1}(i)|\leq\xi n$.\labelinthm{prop:maintech3}
\end{enumerate}
Then, $G$  contains a copy of $T$.
\end{lemma}
\begin{proof} 
Since $\varphi$ is a homomorphism, the image of every component of $T-\varphi^{-1}(\{i_1,i_2,i_3\})$ under $\varphi$ is entirely contained in $I_a:=I_{a,1}\cup I_{a,2}\cup I_{a,3}$ for some $a\in[3]_0$. For each $a\in[3]_0$ then, let $r_a$ be the number of components of $T-\varphi^{-1}(\{i_1,i_2,i_3\})$ whose images are contained in $I_{a,1}\cup I_{a,2}\cup I_{a,3}$, and label these components as $T_{a,1},\ldots,T_{a,r_a}$.
Let $x_1\in \varphi^{-1}(\{i_1,i_2,i_3\})$, view $T$ as being rooted at $x_1$, and extend it to an ordering $x_1,\ldots, x_{n}$ of $V(T)$ so that $T[x_1,\ldots, x_j]$ is a tree for each $j\in [n]$, and vertices in the same component of $T-\varphi^{-1}(\{i_1,i_2,i_3\})$ appear consecutively. Moreover, we can ensure that for every $x\in\varphi^{-1}(\{i_1,i_2,i_3\})$, the components of $T-\varphi^{-1}(\{i_1,i_2,i_3\})$ that directly descends from $x$ appear right after $x$ in this ordering.
For each $a\in [3]_0$ and $\ell\in [r_a]$, let $p_{a,\ell}$ be the smallest index such that $x_{p_{a,\ell}}$ is a
vertex in $T_{a,\ell}$. By relabelling if necessary, assume that $p_{a,1}<p_{a,2}<\cdots<p_{a,r_a}$ for each $a\in [3]_0$. 

We now provide a random algorithm that, with positive probability, produces an assignment function $\sigma:V(T)\to V(R)$ consistent with $\varphi$ that guides an embedding $\psi:T\to G$. To initialise, let $\sigma(x)=i$ for every $x\in \varphi^{-1}(i)$ and $i\in \{i_1,i_2,i_3\}$, and let $\psi$ be the empty function. Now, for each $s\in[n]$ in turn, if $s=p_{a,\ell}$ for some $a\in [3]_0$ and $\ell\in [r_a]$, then we extend the definition of $\sigma$ to include all vertices in $T_{a,\ell}$ in a random manner defined below, while we do nothing to $\sigma$ otherwise. Then, if possible, we extend $\psi$ by embedding $x_s$ into $G$ so that the following properties hold, otherwise we stop this process. For notational convenience, let $i_0=i_1$, and let $k_a=|I_{a,1}|$ for every $a\in[3]_0$.
\stepcounter{propcounter}
\begin{enumerate}[label = {\textbf{\Alph{propcounter}\arabic{enumi}}}]
\item $\psi(x_j)\in V_{\sigma(x_j)}$ for each $j\in[s]$.\label{prop:forreg1}
\item For every $j\in[s]$ and $j'>j$ satisfying $j'\notin \{p_{a,\ell}:a\in [3]_0,\ell\in [r_a]\}$ and $x_jx_{j'}\in E(T)$,
we have $d_G(\psi(x_j),V_{\sigma(x_{j'})}\setminus \psi(\{x_1,\ldots,x_{j-1}\}))\geq d|V_{\sigma(x_{j'})}|/4$ if $\sigma(x_{j'})\in\{i_1,i_2,i_3\}$, and $d_G(\psi(x_j),V_{\sigma(x_{j'})}\setminus \psi(\{x_1,\ldots,x_{j-1}\}))\geq d\alpha n/4k_a$ if $\sigma(x_{j'})\in I_a$ for some $a\in[3]_0$.\label{prop:forreg2}
\item For every $a\in[3]_0$ and $j\in[s]$, if $\varphi(x_j)=i_a$, then for all but at most $\alpha k_a/100$ values of $i\in I_{a,1}$, there exists a set $W_{i,j}\subset N(\psi(x_j),V_{i}\setminus \psi(\{x_1,\ldots,x_{j-1}\}))$ with size $d\alpha n/8k_a$, such that if $j<j'\leq s$ and $x_{j'}$ is a vertex in a component that directly descends from $x_j$, then $\psi(x_{j'})$ avoids these sets $W_{i,j}$ unless $x_{j'}\in N_T(x_j)$.\label{prop:forreg3} 
\end{enumerate}
First, we show how $\sigma$ is randomly extended when $s=p_{a,\ell}$ for some $a\in [3]_0$ and $\ell\in [r_a]$. Note that $s>1$ as $x_1\in \varphi^{-1}(\{i_1,i_2,i_3\})$. Let $1\leq s'<s$ be the unique index satisfying $x_{s'}x_s\in E(T)$, and observe that as $T_{a,\ell}$ is a component of $T-\varphi^{-1}(\{i_1,i_2,i_3\})$, we have $\varphi(x_{s})=I_{a,1}$, and thus $\sigma(x_{s'})=i_a$. By \ref{prop:forreg3}, there exists $I_{a,1,\ell}\subset I_{a,1}$ with $|I_{a,1,\ell}|\geq(1-\alpha/100)k_a$, such that for every $i\in I_{a,1,\ell}$, there exists $W_{i,s'}\subset N(\psi(x_{s'}),V_{i}\setminus \psi(\{x_1,\ldots,x_{s'-1}\}))$ with size $d\alpha n/8k_a$.
Pick $i_{a,1,\ell}\in I_{a,1,\ell}$ uniformly at random. Let $i_{a,2,\ell}\in I_{a,2}$ and $i_{a,3,\ell}\in I_{a,3}$ be such that $i_{a,1,\ell}i_{a,2,\ell},i_{a,2,\ell}i_{a,3,\ell}\in E(R)$.
For each $b\in [3]$ and each $x\in V(T_{a,\ell})\cap \varphi^{-1}(I_{a,b})$, set $\sigma(x)=i_{a,b,\ell}$. This extends the definition of $\sigma$ to include vertices in $V(T_{a,b})$.

Now, for each $a\in [3]_0$, $b\in [3]$, $i\in I_{a,b}$ and $\ell\in [r_a]$, if $i_{a,b,\ell}=i$, then let
\[Z_{a,b,i,\ell}=|V(T_{a,\ell})\cap \varphi^{-1}(I_{a,b})|,\]
otherwise, including if the process above stops early, let $Z_{a,b,i,\ell}=0$.

\begin{claim}\label{clm:Zajib} For each $a\in [3]_0$, $b\in [3]$, and $i\in I_{a,b}$, with probability at least $1-1/20k$, we have
\[
\sum_{\ell\in [r_a]}Z_{a,b,i,\ell}\leq \frac{|\varphi^{-1}(I_{a,b})|}{k_a}+\frac{\alpha n}{2k_a}.
\]
\end{claim}
\begin{proof}[Proof of Claim~\ref{clm:Zajib}] Fix $a\in [3]_0$, $b\in [3]$, and $i\in I_{a,b}$. For each $\ell\in [r_a]$, we have 
\[\mathbb{E}(Z_{a,b,i,\ell}\mid Z_{a,b,i,1},\ldots,Z_{a,b,i,\ell-1})\leq\frac{|V(T_{a,\ell})\cap \varphi^{-1}(I_{a,b})|}{|I_{a,1,\ell}|}\leq\frac{|V(T_{a,\ell})\cap \varphi^{-1}(I_{a,b})|}{(1-\alpha/100)k_a},\] 
and
\[\frac{1}{(1-\alpha/100)k_a}\sum_{\ell\in [r_a]}|V(T_{a,\ell})\cap \varphi^{-1}(I_{a,b})|
\leq\frac{|\varphi^{-1}(I_{a,b})|}{k_a}+\frac{\alpha n}{4k_a}.\]
Furthermore, as $|T_{a,\ell}|\leq \xi n$ for each $\ell\in [r_a]$, we have $Z_{a,b,i,\ell}\leq |T_{a,\ell}|\leq\xi n$. Thus, using $\sum_{\ell\in [r_a]}|T_{a,\ell}|\leq n$, we have
\[\sum_{\ell\in [r_a]}|T_{a,\ell}|^2\leq\frac n{\xi n}\cdot(\xi n)^2= \xi n^2.\]
Therefore, applying Lemma~\ref{lemma:azuma} with $t=\alpha n/4k_a$, we have that
\begin{align*}
\mathbb{P}\left(\sum_{\ell\in [r_a]}Z_{a,b,i,\ell}>\frac{|\varphi^{-1}(I_{a,b})|}{k_a}+\frac{\alpha n}{2k_a}\right)&
\leq\exp\left(-\frac{(\alpha n/4k_a)^2 }{2\xi n^2}\right)\leq\exp\left(-\frac{\alpha^2}{32\xi k^2}\right)\leq \frac{1}{20k},
\end{align*}
as required, where we have used that $k_a\leq k$ and $\xi\ll 1/k,\alpha\ll1$.
\renewcommand{\qedsymbol}{$\boxdot$}
\end{proof}
\renewcommand{\qedsymbol}{$\square$}


\begin{claim}\label{clm:extendingpossible}
Suppose for some $0\leq s<n$, $\psi(x_1),\ldots,\psi(x_{s})$ satisfy \emph{\ref{prop:forreg1}}--\emph{\ref{prop:forreg3}}, and for every $a\in [3]_0$, $b\in [3]$ and $i\in I_{a,b}$, we have
\begin{equation}\label{eqn:nottoomuch}
\sum_{\ell\in [r_a]:p_{a,\ell}\leq s}Z_{a,b,i,\ell}\leq \frac{|\varphi^{-1}(I_{a,b})|}{k_a}+\frac{\alpha n}{2k_a}.
\end{equation}
Then, we can extend the embedding $\psi$ to include $x_{s+1}$ so that \emph{\ref{prop:forreg1}}--\emph{\ref{prop:forreg3}} still hold with $s+1$ in place of $s$.
\end{claim}
\begin{proof}[Proof of Claim~\ref{clm:extendingpossible}]
Let $s'\in[s]$ be such that $x_{s'}x_{s+1}\in E(T)$.

If $s+1=p_{a,\ell}$ for some $a\in [3]_0$ and $\ell\in [r_a]$, then from above we have $\sigma(x_{s+1})=i$ for some $i\in I_{a,1}$, and by~\ref{prop:forreg3} there exists $W_{i,s'}\subset N(\psi(x_{s'}),V_{i}\setminus \psi(\{x_1,\ldots,x_{s'-1}\}))$ with size $d\alpha n/8k_a$. Using~\ref{prop:forreg3}, $\Delta(T)\leq cn$, and that components directly descending from $x_{s'}$ appear right after $x_{s'}$ in the ordering, we see that at most $cn$ vertices in $W_{i,s'}$ have been used, thus the set $Y_{s+1}:=W_{i,s'}\setminus\psi(\{x_1,\ldots,x_{s}\})$ has size at least $d\alpha n/10k_a$. 

If $s+1\not\in\{p_{a,\ell}:a\in [3]_0, \ell\in [r_a]\}$, then $\sigma(x_{s'})\not\in\{i_1,i_2,i_3\}$. If $\sigma(x_{s+1})=i_a$ for some $a\in[3]$, then by~\ref{prop:forreg2}, we have $d_G(\psi(x_{s'}),V_{i_a}\setminus\psi(\{x_1,\ldots,x_{s'-1}\}))\geq d|V_{i_a}|/4$. By~\ref{prop:maintech3}, $Y_{s+1}:=N_G(\psi(x_{s'}),V_{i_a}\setminus\psi(\{x_1,\ldots,x_{s}\}))$ has size at least $d|V_{i_a}|/4-\xi n\geq d|V_{i_a}|/5$. If instead $\sigma(x_{s+1})=i\in I_a$ for some $a\in[3]_0$, then $x_{s'}$ and $x_s$ are in the same component of $T-\varphi^{-1}(\{i_1,i_2,i_3\})$. Say this component is directly descended from $x_{s''}\in\varphi^{-1}(\{i_1,i_2,i_3\})$. By~\ref{prop:forreg2}, we have $d_G(\psi(x_{s'}),V_i\setminus\psi(\{x_1,\ldots,x_{s'-1}\}))\geq d\alpha n/4k_a$. Since vertices in the same component appear consecutively in the ordering, by~\ref{prop:maintech1}, at most $\xi n$ vertices are embedded between $x_{s'}$ and $x_s$. Thus, the set $Y_{s+1}:=N_G(\psi(x_{s'}),V_i\setminus(\psi(\{x_1,\ldots,x_{s}\})\cup W_{i,s''}))$ has size at least $d\alpha n/4k_a-\xi n-d\alpha n/8k_a\geq d\alpha n/10k_a$.

We now embed $x_{s+1}$ to a suitable vertex in $Y_{s+1}$ by splitting into the following two cases, depending on whether $\sigma(x_{s+1})\in\{i_1, i_2, i_3\}$.

\medskip

\noindent \textbf{Case I.} $\sigma(x_{s+1})=i_a$ for some $a\in [3]$. For every $i\in N_R(i_a)\setminus \{i_1,i_2,i_3\}\subset I_{a,1}$, by~\ref{prop:forreg1}, \eqref{eqn:nottoomuch} and~\ref{prop:maintech2}, we have
\[
|V_{i}\cap \psi(\{x_1,\ldots,x_s\})|\leq \frac{|\varphi^{-1}(I_{a,1})|}{k_a}+\frac{\alpha n}{2k_a}
\leq \frac{\sum_{i'\in I_{a,1}}|V_{i'}|}{k_a}-\frac{\alpha n}{2k_a}.
\]
Since $|V_{i}|=\frac{1}{k_a}\sum_{i'\in I_{a,1}}|V_{i'}|$, we have $|V_{i}\setminus \psi(\{x_1,\ldots,x_s\})|\geq \alpha n/2k_a\gg\eps|V_i|$. Since $i_ai\in E(R)$, $G[V_{i_a},V_i]$ is $(\eps,d)$-regular, so by Lemma~\ref{lemma:regularity:2}, for all but at most $\sqrt\eps|V_{i_a}|$ vertices $y\in Y_{s+1}$, 
\[d(y,V_{i}\setminus \psi(\{x_1,\ldots,x_s\}))\geq d|V_{i}\setminus \psi(\{x_1,\ldots,x_s\})|/2\geq d\alpha n/4k_a>d\alpha n/8k_a\]
for all but at most $\sqrt\eps k_a\leq\alpha|I_{a,1}|/100$ indices $i\in I_{a,1}$. Similarly, when $a=1$, for all but at most $\sqrt\eps|V_{i_1}|$ vertices $y\in Y_{s+1}$, $d(y,V_{i}\setminus \psi(\{x_1,\ldots,x_s\}))\geq d\alpha n/8k_0$ for all but at most $\alpha|I_{0,1}|/100$ indices $i\in I_{0,1}$.

Furthermore, for each $i\in N_R(i_a)\cap \{i_1,i_2,i_3\}$, by using~\ref{prop:maintech3} and Lemma~\ref{lemma:regularity:3} instead of~\ref{prop:maintech2} and Lemma~\ref{lemma:regularity:2}, we have that for all but at most $\eps|V_{i_a}|$ vertices $y\in Y_{s+1}$, $d(y,V_{i}\setminus \psi(\{x_1,\ldots,x_s\}))\geq d|V_i|/4$.
As $|Y_{s+1}|\geq d|V_{i_a}|/5\geq10\sqrt\eps|V_{i_a}|$, we can pick $\psi(x_{s+1})\in Y_{s+1}$ such that all of the above hold, so~\ref{prop:forreg1}--\ref{prop:forreg3} hold with $s+1$ in place of $s$, as required.

\medskip

\noindent \textbf{Case II.} $\sigma(x_{s+1})\in I_a$ for some $a\in[3]_0$. Similar to \textbf{Case I}, we can deduce from Lemma~\ref{lemma:regularity:3}, \ref{prop:forreg1}, \eqref{eqn:nottoomuch}, and either~\ref{prop:maintech2} or~\ref{prop:maintech3} that for every $i\in N_R(\sigma(x_{s+1}),I_a)$, all but at most $\eps |V_{\sigma(x_{s+1})}|$ vertices $y\in Y_{s+1}$ satisfy $d(y,V_{i}\setminus \psi(\{x_1,\ldots,x_s\}))\geq d\alpha n/4k_a$, and for every $i\in N_R(\sigma(x_{s+1}),\{i_1,i_2,i_3\})$, all but at most $\eps |V_{\sigma(x_{s+1})}|$ vertices $y\in Y_{s+1}$ satisfy $d(y,V_{i}\setminus \psi(\{x_1,\ldots,x_s\}))\geq d|V_i|/4$. 
Then, using that $d_R(\sigma(x_{s+1}))\leq 2$, and $|Y_{s+1}|\geq d\alpha n/10k_a\geq20\eps n/k_a\geq10\eps|V_{\sigma(x_{s+1})}|$, we can pick $\psi(x_{s+1})\in Y_{s+1}$ such that all of the above hold, so~\ref{prop:forreg1}--\ref{prop:forreg3} hold with $s+1$ in place of $s$, as required.
\renewcommand{\qedsymbol}{$\boxdot$}
\end{proof}
\renewcommand{\qedsymbol}{$\square$}

Finally, note that by a union bound over all $a\in[3]_0, b\in[3]$, and $i\in I_{a,b}$, Claim~\ref{clm:Zajib} and Claim~\ref{clm:extendingpossible} combine to show that the process above embeds $T$ into $G$ with strictly positive probability, and thus $G$ contains a copy of $T$.
\end{proof}


\subsection{Embedding method H\L T}\label{sec:HLT}

The following result appeared in the work of Haxell, \L uczak, and Tingley~\cite{haxell2002ramsey}. For completion and to illustrate our method, we include a proof using our framework.

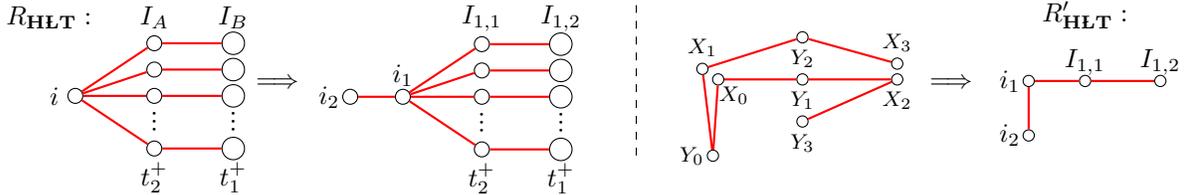
\begin{figure}[h]
\begin{center}
\begin{minipage}{0.20\textwidth}\centering
\begin{tikzpicture}[scale=0.7,main1/.style = {circle,draw,fill=none,inner sep=2}, main2/.style = {circle,draw,fill=none,inner sep=3}]
\node[main1] (A0) at (0.5,0) {};
\node[main1] (A1) at (2,1) {};
\node[main1] (A2) at (2,0.5) {};
\node[main1] (A3) at (2,0) {};
\node[draw=none,fill=none] at (2,-0.65) {$\cdot$};
\node[draw=none,fill=none] at (2,-0.5) {$\cdot$};
\node[draw=none,fill=none] at (2,-0.35) {$\cdot$};

\node[draw=none,fill=none] at (3.5,-0.65) {$\cdot$};
\node[draw=none,fill=none] at (3.5,-0.5) {$\cdot$};
\node[draw=none,fill=none] at (3.5,-0.35) {$\cdot$};
\node[main1] (A4) at (2,-1) {};
\node[main2] (B1) at (3.5,1) {};
\node[main2] (B2) at (3.5,0.5) {};
\node[main2] (B3) at (3.5,0) {};
\node[main2] (B4) at (3.5,-1) {};

\draw ($(A0)+(-0.5,1.5)$) node {$R_{\text{{\bfseries H\L T}}}:$};

\draw[red,thick] (A0) -- (A1);
\draw[red,thick] (A0) -- (A2);
\draw[red,thick] (A0) -- (A3);
\draw[red,thick] (A0) -- (A4);

\foreach \n in {1,2,3,4}
\draw[red,thick] (A\n) -- (B\n);

\node[draw=none,fill=none] at (2,-1.55) {$t_2^+$};
\node[draw=none,fill=none] at (3.5,-1.55) {$t_1^+$};

\draw ($(A0)-(0.4,0)$) node {$i$};

\draw ($(A1)+(0,0.5)$) node {$I_A$};
\draw ($(B1)+(0,0.5)$) node {$I_B$};
\end{tikzpicture}
\end{minipage}
\begin{minipage}{0.05cm}
$\implies$
\end{minipage}\;\;\;\;\;
\begin{minipage}{0.25\textwidth}\centering
\begin{tikzpicture}[scale=0.7,main1/.style = {circle,draw,fill=none,inner sep=2}, main2/.style = {circle,draw,fill=none,inner sep=3}]
\node[main1] (J0) at (-0.5,0) {};
\node[main1] (A0) at (0.5,0) {};
\node[main1] (A1) at (2,1) {};
\node[main1] (A2) at (2,0.5) {};
\node[main1] (A3) at (2,0) {};
\node[draw=none,fill=none] at (2,-0.65) {$\cdot$};
\node[draw=none,fill=none] at (2,-0.5) {$\cdot$};
\node[draw=none,fill=none] at (2,-0.35) {$\cdot$};
\node[main1] (A4) at (2,-1) {};
\node[main2] (B1) at (3.5,1) {};
\node[main2] (B2) at (3.5,0.5) {};
\node[main2] (B3) at (3.5,0) {};
\node[draw=none,fill=none] at (3.5,-0.65) {$\cdot$};
\node[draw=none,fill=none] at (3.5,-0.5) {$\cdot$};
\node[draw=none,fill=none] at (3.5,-0.35) {$\cdot$};
\node[main2] (B4) at (3.5,-1) {};

\draw[red,thick] (A0) -- (J0);
\draw[red,thick] (A0) -- (A1);
\draw[red,thick] (A0) -- (A2);
\draw[red,thick] (A0) -- (A3);
\draw[red,thick] (A0) -- (A4);

\foreach \n in {1,2,3,4}
\draw[red,thick] (A\n) -- (B\n);

\node[draw=none,fill=none] at (2,-1.55) {$t_2^+$};
\node[draw=none,fill=none] at (3.5,-1.55) {$t_1^+$};

\draw ($(A0)+(0,0.4)$) node {$i_1$};

\draw ($(J0)-(0.4,0)$) node {$i_2$};

\draw ($(A1)+(0,0.5)$) node {$I_{1,1}$};
\draw ($(B1)+(0,0.5)$) node {$I_{1,2}$};
\end{tikzpicture}
\end{minipage}\;\;
\begin{minipage}{0.3cm}
\begin{center}\vspace{-0.5cm}
\begin{tikzpicture}
\draw [dashed] (0,0) -- (0,2);
\end{tikzpicture}
\end{center}
\end{minipage}\;\;
\begin{minipage}{0.20\textwidth}\centering
\begin{tikzpicture}[scale=.7,main1/.style = {circle,draw,fill=none,inner sep=1.5}, main2/.style = {circle,draw,fill=none,inner sep=2.25}]
\def\spacer{2};
\def\spacerr{1.5};
\def\Ahgt{1.2};
\def\Bhgt{1.2};

\coordinate (V1) at ($(-\spacerr,1.35)$);
\coordinate (V2) at ($(-\spacerr,-0.1)$);
\coordinate (A1) at ($(0,1.35)+(0.2,0)$);
\coordinate (B1) at ($(A1)+(\spacer,0)-(0.2,0)$);

\node[main1] (X1) at ($(V1)+(-0.2,0.2)$) {};
\node[main1] (X0) at ($(V1)+(0.2,-0.2)+(-0.1,0.2)$) {};
\node[main1] (Y0) at ($(V2)$) {};
\node[main1] (Y1) at ($(A1)+(0,0)$) {};
\node[main1] (Y2) at ($(A1)+(0,0.8)$) {};
\node[main1] (Y3) at ($(A1)+(0,-0.8)$) {};
\node[main1] (X2) at ($(B1)+(0,0)$) {};
\node[main1] (X3) at ($(B1)+(0,0.3)$) {};

\draw [thick,red] (Y3) -- (X2) -- (Y1) -- (X0) -- (Y0) -- (X1) -- (Y2) -- (X3);

\foreach \where/\what in {X1/X_1,X3/X_3}
{
\draw ($(\where)+(0,0.4)$) node {\footnotesize $\what$};
}
\foreach \where/\what in {X2/X_2,Y3/Y_3,Y2/Y_2,Y1/Y_1}
{
\draw ($(\where)-(0,0.4)$) node {\footnotesize $\what$};
}
\draw ($(Y0)-(0.4,0)$) node {\footnotesize $Y_0$};
\draw ($(X0)+(0.3,-0.3)$) node {\footnotesize $X_0$};

\end{tikzpicture}
\end{minipage}
\begin{minipage}{0.3cm}
$\implies$
\end{minipage}\;\;\;\;\;
\begin{minipage}{0.15\textwidth}\centering
\begin{tikzpicture}[scale=.5,main1/.style = {circle,draw,fill=none,inner sep=1.5}, main2/.style = {circle,draw,fill=none,inner sep=2.25}]
\def\spacer{2};
\def\spacerr{1.5};
\def\Ahgt{1.2};
\def\Bhgt{1.2};

\node[main1] (V1) at ($(-\spacerr,1.35)$) {};
\node[main1] (V2) at ($(-\spacerr,-0.1)$) {};
\node[circle,draw=white,fill=none,inner sep=1.5] (V3) at ($(-\spacerr,-1.8)$) {};

\node[main1] (A1) at (0,1.35) {};
\foreach \n in {1}
{
\node[main1] (B\n) at ($(A\n)+(\spacer,0)$) {};
}

\foreach \n in {1}
{
\draw [thick,red] (V\n) -- (A\n) -- (B\n);
}
\draw [thick,red] (V1) -- (V2);

\draw ($(V1)-(0.5,0)$) node {$i_1$};
\draw ($(V2)-(0.5,0)$) node {$i_2$};
\foreach \where/\what in {A1/I_{1,1},B1/I_{1,2}}
{
\draw ($(\where)+(0,0.5)$) node {$\what$};
}

\draw ($(A1)+(0,1.75)$) node {$R'_{\text{{\bfseries H\L T}}}:$};
\end{tikzpicture}
\end{minipage}
\end{center}

\vspace{-0.2cm}

\caption{On the left, the slight refinement of the initial reduced graph $R_{\text{\bfseries H\L T}}$ used in the proof of Lemma~\ref{lemma:hlt}. On the right, a depiction of the rule of the embedding used at \eqref{eq:hltembedding}, condensing $S$ from Lemma~\ref{lem:maintreedecomp} into a subgraph $R'_{\text{{\bfseries H\L T}}}$ of $R'$ from Figure~\ref{fig:mainembedstatement}.}\label{fig:HLTproof}
\end{figure}

\begin{lemma}[\textbf{H\L T}]\label{lemma:hlt}
Let $1/n\ll c\ll 1/k\ll\varepsilon\ll\alpha\ll d\leq 1$. Let $T$ be an $n$-vertex tree with $\Delta(T)\leq cn$ and bipartition classes of sizes $t_1$ and $t_2$ satisfying $t_2\le t_1\le 2t_2$. Let $G$ be a graph on at most $2n$ vertices with a partition $V(G)=V_1\cup\cdots\cup V_{2k+1}$. Let $R_{\text{\emph{\bfseries H\L T}}}$ be a graph with vertex set $[2k+1]$, such that if $ij\in E(R_{\text{\emph{\bfseries H\L T}}})$ then $G[V_i,V_j]$ is $(\eps,d)$-regular. Let $i\in [2k+1]$ and suppose there is a partition $[2k+1]\setminus \{i\}=I_A\cup I_B$, with $|I_A|=|I_B|=k$, such that the following hold for some $m_A,m_B$ (see Figure~\ref{fig:HLTproof}). 
\stepcounter{propcounter}
\begin{enumerate}[label = \emph{\textbf{\Alph{propcounter}\arabic{enumi}}}]
    \item $|V_a|=m_A$ for each $a\in I_A$, $|V_b|=m_B$ for each $b\in I_B$, and $|V_i|\geq n/10k$.\labelinthm{prop:hlt:1}
    \item $km_A\geq t_2+\alpha n$ and $km_B\geq t_1+\alpha n$.\labelinthm{prop:hlt:2}
    \item In $R_{\text{\emph{\bfseries H\L T}}}$, $i$ is adjacent to each vertex in $I_A$, and there is a perfect matching $M$ between $I_A$ and $I_B$.\labelinthm{prop:hlt:3}
\end{enumerate}
Then, there is a copy of $T$ in $G$.
\end{lemma}
\begin{proof} Let $\xi$ satisfy $c\ll\xi\ll 1/k$. Let $S$ be the graph defined in Lemma~\ref{lem:maintreedecomp}. Using this lemma, we can take a homomorphism $\phi:T\to S$ such that each component of $T-\phi^{-1}(X_0\cup Y_0)$ has size at most $\xi n$ and $|\phi^{-1}(X_0\cup Y_0\cup X_1\cup Y_1)|\leq \xi n$. Assume, by relabelling if necessary, that $|\phi^{-1}(X_0\cup X_1\cup X_2\cup X_3)|=t_1$ and $|\phi^{-1}(Y_0\cup Y_1\cup Y_2\cup Y_3)|=t_2$.

Let $i_1=i$. Pick some $i_2\in I_A$ and suppose it is matched with $i_3\in I_B$ by $M$. Let $I_{1,1}=I_A\setminus \{i_2\}$ and $I_{1,2}=I_B\setminus \{i_3\}$ (see Figure~\ref{fig:HLTproof}). 
Let $R'_{\text{\bfseries H\L T}}$ be the graph on the right in Figure~\ref{fig:HLTproof}, and note that it is a subgraph of $R'$ in Figure~\ref{fig:mainembedstatement}.
Then, for each $v\in V(T)$, as depicted in Figure~\ref{fig:HLTproof}, let
\begin{equation*}\label{eq:hltembedding}
\varphi(v)=\left\{\begin{array}{ll}
i_1&\text{\;\; if }\phi(v)\in \{X_0,X_1\}\\
i_2 &\text{\;\; if }\phi(v)=Y_0\\
I_{1,1} &\text{\;\; if }\phi(v)\in \{Y_1,Y_2,Y_3\}\\
I_{1,2}&\text{\;\; if }\phi(v)\in \{X_2,X_3\}
\end{array},\right.
\end{equation*}
so that $\varphi$ is a homomorphism from $T$ to $R'_{\text{\bfseries H\L T}}$, and thus to $R'$, with $|\varphi^{-1}(\{i_1,i_2\})|\leq \xi n$, $|\varphi^{-1}(I_{1,1})|\leq t_2$, and $|\varphi^{-1}(I_{1,2})|\leq t_1$.

We will now check the conditions required for an application of Lemma~\ref{lem:maintechnicalembedding}. First,~\ref{prop:maintech1} holds as every component of $T-\varphi^{-1}(\{i_1,i_2,i_3\})$ is contained in a component of $T-\phi^{-1}(X_0\cup Y_0)$, which has size at most $\xi n$.
Next, using \ref{prop:hlt:1}, \ref{prop:hlt:2}, and $1/n\ll 1/k\ll\alpha$, we have
\[|\varphi^{-1}(I_{1,1})|+\alpha n/2\leq t_2+\alpha n/2\leq(k-1)m_A=\sum_{j\in I_{1,1}}|V_j|,\]
and similarly $|\varphi^{-1}(I_{1,2})|+\alpha n/2\leq \sum_{j\in I_{1,2}}|V_j|$, so \ref{prop:maintech2} holds as $\varphi^{-1}(I)=\emptyset$ for each $I\in V(R')\setminus\{i_1,i_2,I_{1,1},I_{1,2}\}$.
Finally, for each $j\in \{i_1,i_2,i_3\}$, $|\varphi^{-1}(j)|\leq |\phi^{-1}(X_0\cup Y_0\cup X_1\cup Y_1)|\leq \xi n$, so~\ref{prop:maintech3} holds. Therefore, we can apply Lemma~\ref{lem:maintechnicalembedding} to find a copy of $T$ in $G$, as required.
\end{proof}

\subsection{EM1a Embedding Method}\label{sec:EM1a}
\begin{figure}[h]
\input{figures/EM1a/pfEM1a}

\vspace{-0.3cm}

\caption{On the left, the initial reduced graph $R_{\text{\bfseries EM1a}}$ transformed into the subgraph used to embed the tree in Lemma~\ref{lemma:em1a}. On the right, the auxiliary graph $R_{\text{\bfseries EM1a}}'$ used when applying Lemma~\ref{lem:maintechnicalembedding}.}\label{fig:EM1aproof}
\end{figure}


\begin{lemma}[\textbf{EM1a}]\label{lemma:em1a}
Let $1/n\ll 1/m\ll c\ll1/k\ll\eps\ll\alpha\ll d\leq1$. Let $T$ be an $n$-vertex tree with $\Delta(T)\le cn$ and bipartition classes of sizes $t_1$ and $t_2$ satisfying $t_2\leq t_1\leq 2t_2$. Let $G$ be a graph with with a partition $V(G)=V_0\cup V_1\cup\cdots\cup V_k$. Let $R_{\text{\emph{\bfseries EM1a}}}$ be a graph with vertex set $[k]_0$, such that if $ij\in E(R_{\text{\emph{\bfseries EM1a}}})$ then $G[V_i,V_j]$ is $(\eps,d)$-regular. Suppose there is a partition $[k]=I_A\cup I_A'\cup I_B\cup I_B'\cup I_B''\cup I_C$, such that the following properties hold (see Figure~\ref{fig:EM1aproof}).
 \stepcounter{propcounter}
 \begin{enumerate}[label = \emph{\textbf{\Alph{propcounter}\arabic{enumi}}}]
    \item\labelinthm{lemma:reg:em1:1} $|V_i|=m$ for all $i\in\{0\}\cup I_A\cup I_A'\cup I_C$ and $|V_i|=t_1m/t_2$ for all $i\in I_B\cup I_B'\cup I_B''$.
    \item\labelinthm{lemma:reg:em1:2} $|\cup_{i\in I_A\cup I_A'}V_i|\geq(1-\alpha^2) t_2, |\cup_{i\in I_B\cup I_B'}V_i|\geq(1-\alpha^2)t_1$, $|\cup_{i\in I_A'}V_i|=\alpha t_2$, $|\cup_{i\in I_B'}V_i|=|\cup_{i\in I_B''}V_i|=\alpha t_1$, and $|\cup_{i\in I_C}V_i|\geq\frac23t_2$.
    \item\labelinthm{lemma:reg:em1:3} In $R_{\text{\emph{\bfseries EM1a}}}$, 0 is adjacent to every vertex in $I_A\cup I_A'$.
    \item\labelinthm{lemma:reg:em1:4} In $R_{\text{\emph{\bfseries EM1a}}}$, there is a perfect matching between $I_A$ and $I_B$, a perfect matching between $I_A'$ and $I_B'$, and a perfect matching between $I_A'$ and $I_B''$.
    \item\labelinthm{lemma:reg:em1:5} In $R_{\text{\emph{\bfseries EM1a}}}$, every vertex in $I_B'$ is adjacent to at least $10\alpha|I_C|$ vertices in $I_C$.
\end{enumerate}
Then, $G$ contains a copy of $T$.
\end{lemma}
\begin{proof}
We begin with a claim that will also be used later in the proofs of Lemma~\ref{lemma:em1b} and Lemma~\ref{lemma:em1c}.
\begin{claim}\label{claim:backandforth}
For any $Z\subset I_B'$ with $\alpha^{-2}\ll|Z|\leq5\alpha|I_C|$, there exists $z\in Z$ and $Z'\subset Z\setminus\{z\}$ with $|Z'|=5\alpha|Z|$, and a matching $M$ in $R_{\mathbf{EM1a}}[Z',I_C]$ covering $Z'$, with $z$ adjacent to every $i\in I_C\cap V(M)$.
\end{claim}
\begin{proof}[Proof of Claim~\ref{claim:backandforth}]
By~\ref{lemma:reg:em1:5}, $e(R_{\text{\bfseries EM1a}}[Z,I_C])\geq10\alpha |Z||I_C|$, so there exists a set $I_C'\subset I_C$ with size $5\alpha|I_C|$ such that every $i\in I_C'$ has at least $5\alpha|Z|+1$ neighbours in $Z$, as otherwise \[e(R_{\text{\bfseries EM1a}}[Z,I_C])<5\alpha|I_C|\cdot|Z|+(1-5\alpha)|I_C|\cdot(5\alpha |Z|+1)<10\alpha|Z||I_C|,\] a contradiction. Then, since $e(R_{\text{\bfseries EM1a}}[Z,I_C'])\geq5\alpha|Z||I_C'|$, by averaging, there exists $z\in Z$ with at least $5\alpha|I_C'|$ neighbours in $I_C'$. Let $I_C''$ be a set of $5\alpha|I_C'|$ neighbours of $z$ in $I_C'$. Greedily, and using $5\alpha|Z|\leq|I_C''|$, we can find a matching $M$ between $I_C''$ and $Z\setminus\{z\}$ with size $5\alpha|Z|$, which proves the claim.
\renewcommand{\qedsymbol}{$\boxdot$}
\end{proof}
\renewcommand{\qedsymbol}{$\square$}
Let $i_1=0$. Since $|I_B'|=\alpha t_2/m<10\alpha t_2/3m\leq5\alpha|I_C|$, we can apply Claim~\ref{claim:backandforth} to find $Z_B\subset I_B'$ with $|Z_B|=5\alpha|I_B'|=5\alpha^2t_2/m$, $i_3\in I_B'\setminus Z_B$, and a perfect matching $M$ between $Z_B$ and some $Z_C\subset I_C$, such that $i_3$ is adjacent to all $i\in Z_C$. Suppose that in the matchings given by~\ref{lemma:reg:em1:4},  $i_3$ is matched with $i_2\in I_A'$, and $i_2$ is matched with $i_2'\in I_B''$.

Let $V_i'=V_i$ for all $i\in I_B\cup I_B''\cup Z_B\cup Z_C\cup\{i_1,i_2,i_3\}$. For every $i\in(I_A\cup I_A')\setminus\{i_2\}$, let $V_i'\subset V_i$ have size $(1-\alpha^2/2)|V_i|$. Let $I_{1,1}=(I_A\cup I_A')\setminus\{i_2\}$ and $I_{1,2}=(I_B\cup I_B'')\setminus\{i_2'\}$. Partition $Z_B$ as evenly as possible into two sets $I_{0,2}$ and $I_{3,2}$, and say they are matched by $M$ with subsets $I_{0,3}$ and $I_{3,1}$ of $Z_C$, respectively. Note that $i_3$ is adjacent to every vertex in $I_{3,1}$. Finally, take a new index set $I_{0,1}$ with size $|I_{0,2}|$, say $I_{0,2}$ is matched with $I_{0,1}'\subset I_A'$ in the matching given by~\ref{lemma:reg:em1:4}, and relabel the collection $\{V_i\setminus V_i':i\in I_{0,1}'\}$ as $\{V_j':j\in I_{0,1}\}$. Note that by Lemma~\ref{lemma:regularity:1}, if $ij$ is an edge in the graph depicted in the middle of Figure~\ref{fig:EM1aproof}, then $G[V_i',V_j']$ is $(\sqrt\eps,d-\eps)$-regular.

Let $\xi$ satisfy $c\ll\xi\ll 1/k$. Let $S$ be the graph defined in Lemma~\ref{lem:maintreedecomp}. Using that lemma, take a homomorphism $\phi:T\to S$ such that each component of $T-\phi^{-1}(X_0\cup Y_0)$ has size at most $\xi n$ and $|\phi^{-1}(X_0\cup Y_0\cup X_1\cup Y_1)|\leq \xi n$. Without loss of generality, say $|\phi^{-1}(X_0\cup X_1\cup X_2\cup X_3)|=t_1$ and $|\phi^{-1}(Y_0\cup Y_1\cup Y_2\cup Y_3)|=t_2$. Let the components of $T-\phi^{-1}(X_0\cup Y_0)$ be $\{K_j:j\in J\}$, and note that each of these component has neighbours in exactly one of $\phi^{-1}(X_0)$ and $\phi^{-1}(Y_0)$. Thus, we can partition $J$ as $J_X\cup J_Y$, such that $N_T(K_j)\subset\phi^{-1}(X_0)$ for each $j\in J_X$, and $N_T(K_j)\subset\phi^{-1}(Y_0)$ for each $j\in J_Y$.

Let $J_X'\subset J_X$ and $J_Y'\subset J_Y$ both be random sets with each element being included independently with probability $2\alpha^2$. Then, by Lemma~\ref{lemma:mcdiarmid}, with positive probability we have both of the following, so fix such a choice of $J_X', J_Y'$.   
\begin{equation}\label{em1a:eq:1}
\sum_{j\in J_X'\cup J_Y'}|K_j\cap\phi^{-1}(Y_1\cup Y_2\cup Y_3)|=2\alpha^2 t_2\pm\alpha^2 n/100.
\end{equation}
\begin{equation}\label{em1a:eq:2}
\sum_{j\in J_X'\cup J_Y'}|K_j\cap\phi^{-1}(X_1\cup X_2\cup X_3)|=2\alpha^2 t_1\pm\alpha^2 n/100.
\end{equation}

Let $R_{\text{\bfseries EM1a}}'$ be the graph on the right of Figure~\ref{fig:EM1aproof}. Define a homomorphism $\varphi:T\to R_{\text{\bfseries EM1a}}'$ as follows. Let $\varphi(v)=i_1$ for every $v\in\phi^{-1}(X_0)$, and let $\varphi(v)=i_2$ for every $v\in\phi^{-1}(Y_0)$. For every $K_j$ with $j\in J_X'$, define $\varphi$ on $K_j$ by composing $\phi$ with the function sending $Y_1,X_2,Y_3$ to $I_{0,1},I_{0,2},I_{0,3}$, respectively; while if $j\in J_X\setminus J_X'$, define $\varphi$ on $K_j$ by composing $\phi$ with the function sending $Y_1,X_2,Y_3$ to $I_{1,1},I_{1,2},I_{1,1}$, respectively. For every $K_j$ with $j\in J_Y'$, define $\varphi$ on $K_j$ by composing $\phi$ with the function sending $X_1,Y_2,X_3$ to $i_3,I_{3,1},I_{3,2}$, respectively; while for every $K_j$ with $j\in J_Y\setminus J_Y'$, define $\varphi$ on $K_j$ by composing $\phi$ with the function sending $X_1,Y_2,X_3$ to $i_1,I_{1,1},I_{1,2}$, respectively.

Since $|\phi^{-1}(X_0\cup Y_0\cup X_1\cup Y_1)|\leq\xi n$ and $\xi\ll1/k\ll\alpha$, we have $|\varphi^{-1}(I)|\leq\xi n$ for each $I\in\{i_1,i_2,i_3\}$, and $|\varphi^{-1}(I_{0,1})|+\alpha^4n/100\leq\sum_{i\in I_{0,1}}|V_i'|$. From definition, $\sum_{i\in I_{0,2}}|V_i'|, \sum_{i\in I_{3,2}}|V_i'|\geq2.4\alpha^2t_1$ and $\sum_{i\in I_{0,3}}|V_i'|, \sum_{i\in I_{3,1}}|V_i'|\geq2.4\alpha^2t_2$. Thus, by~\eqref{em1a:eq:1},~\eqref{em1a:eq:2}, and~\ref{lemma:reg:em1:2}, we have $|\varphi^{-1}(I)|+\alpha^3n/100\leq\sum_{i\in I}|V_i'|$ for each $I\in\{I_{1,1},I_{1,2}, I_{0,2}, I_{0,3}, I_{3,1}, I_{3,2}\}$. Therefore, we can apply Lemma~\ref{lem:maintechnicalembedding} to find a copy of $T$ in $G$. 
\end{proof}


\subsection{EM1b Embedding Method}\label{sec:EM1b}

\begin{figure}[h]
\begin{center}
\hspace{-0.7cm}\begin{minipage}{0.2\textwidth}
\begin{center}
\begin{tikzpicture}[scale=0.7,main1/.style = {circle,draw,fill=none,inner sep=2}, main2/.style = {circle,draw,fill=none,inner sep=3}, main3/.style = {circle,draw,fill=none,inner sep=1.5}]

\node[main1] (A0) at (0.5,0) {};

\foreach \n in {1,...,5}
{
\node[main1] (A\n) at ($(2,2.5-0.5*\n)$) {};
\node[main2] (B\n) at ($(3.5,2.5-0.5*\n)$) {};
}

\foreach \y in {1,...,8}
{
\node[main1] (D\y) at ($(4.5,1.9+0.8-\y*0.4)$) {};
}

\coordinate (Ax) at (2,-1) {};
\coordinate (Bx) at (3.7,-0.5) {};
\coordinate (By) at (3.7,-1.3) {};

\draw[red,thick] (Bx) -- (By);
\draw[red,thick] (A0) to[out=-20,in=180] (Bx);
\draw[red,thick] (A0) to[out=-45,in=180] (By);
\draw ($0.25*(A5)+0.25*(B5)+0.25*(Ax)+0.25*(Bx)-(-0.1,0.4)$) node {\textbf{H\L T$^-$}};

\foreach \x/\y in {1/1,1/3,1/5,2/2,2/8,3/5,3/2,3/4,4/1,4/3,5/6,5/7,4/6}
{
\draw[red,thick] (B\x) -- (D\y);
}

\foreach \x/\y in {1/2,3/4,4/5}
{
\draw[red,thick] (A\x) -- (A\y);
}

\draw[red,thick] (A1) to[out=-45,in=45] (A3);
\draw[red,thick] (A3) to[out=-45,in=45] (A5);
\draw[red,thick] (A2) to[out=-135,in=135] (A4);
\draw[red,thick] (A2) to[out=-135,in=135] (A5);

\foreach \n in {1,...,5}
{
\draw[red,thick] (A\n) -- (B\n);
\draw[red,thick] (A0) -- (A\n);
}


\draw ($(A0)+(0,0.5)$) node {$0$};

\draw ($(A1)+(0,0.55)$) node {$I_A'$};
\draw ($(B1)+(0,0.55)$) node {$I_B'$};

\draw ($(D1)+(0,0.5)$) node {$I_C$};

\draw ($(A1)+(0,-3.6)$) node {$I_A$};
\draw ($(B1)+(0,-3.6)$) node {$I_B$};

\end{tikzpicture}
\end{center}
\end{minipage}\hspace{-0.3cm}
\begin{minipage}{1.2cm}
\begin{center}
$\implies$
\end{center}
\end{minipage}\hspace{-0.3cm}
\begin{minipage}{0.2\textwidth}
\begin{center}
\begin{tikzpicture}[scale=0.7,main1/.style = {circle,draw,fill=none,inner sep=2}, main2/.style = {circle,draw,fill=none,inner sep=3}, main3/.style = {circle,draw,fill=none,inner sep=1.5}]

\node[main1] (A0) at (0.5,0) {};
\node[main1] (J0) at (0.5,1) {};
\draw [red,thick] (A0) -- (J0);
\draw ($(J0)+(0,0.5)$) node {$i_2$};
\foreach \n in {2,...,5}
{
\node[main1] (A\n) at ($(2,2.5-0.5*\n)$) {};
\node[main2] (B\n) at ($(3.25,2.5-0.5*\n)$) {};
\node[main1] (D\n) at ($(4.5,2.5-0.5*\n)$) {};
}

\coordinate (Ax) at (2,-1) {};
\coordinate (Bx) at (3.7,-0.5) {};
\coordinate (By) at (3.7,-1.3) {};

\draw[red,thick] (Bx) -- (By);
\draw[red,thick] (A0) to[out=-20,in=180] (Bx);
\draw[red,thick] (A0) to[out=-45,in=180] (By);
\draw ($0.25*(A5)+0.25*(B5)+0.25*(Ax)+0.25*(Bx)-(-0.1,0.4)$) node {\textbf{H\L T$^-$}};

\foreach \x/\y in {5/5,2/2,3/3,4/4}
{
\draw[red,thick] (B\x) -- (D\y);
}

\foreach \x/\y in {2/3,4/5}
{
\draw[red,thick] (A\x) -- (A\y);
}

\foreach \n in {2,...,5}
{
\draw[red,thick] (A\n) -- (B\n);
\draw[red,thick] (A0) -- (A\n);
}


\draw ($(A0)-(0,0.5)$) node {$i_1$};

\draw ($(A2)+(0,0.5)$) node {$Z_A$};
\draw ($(B2)+(0,0.5)$) node {$Z_B$};
\draw ($(D2)+(0,0.5)$) node {$Z_C$};

\end{tikzpicture}
\end{center}
\end{minipage}\hspace{-0.3cm}
\begin{minipage}{1.2cm}
\begin{center}
\textbf{and}
\end{center}
\end{minipage}\hspace{-0.5cm}
\begin{minipage}{0.2\textwidth}
\begin{center}
\begin{tikzpicture}[scale=0.7,main1/.style = {circle,draw,fill=none,inner sep=2}, main2/.style = {circle,draw,fill=none,inner sep=3}, main3/.style = {circle,draw,fill=none,inner sep=1.5}]

\node[main1] (A0) at (0.5,0) {};

\foreach \n in {3,...,5}
{
\node[main1] (A\n) at ($(2,2.5-0.5*\n)$) {};
\node[main2] (B\n) at ($(3.25,2.5-0.5*\n)$) {};
\node[main1] (D\n) at ($(4.5,2.5-0.5*\n)$) {};
}

\coordinate (Ax) at (2,-1) {};
\coordinate (Bx) at (3.7,-0.5) {};
\coordinate (By) at (3.7,-1.3) {};

\draw[red,thick] (Bx) -- (By);
\draw[red,thick] (A0) to[out=-20,in=180] (Bx);
\draw[red,thick] (A0) to[out=-45,in=180] (By);
\draw ($0.25*(A5)+0.25*(B5)+0.25*(Ax)+0.25*(Bx)-(-0.1,0.4)$) node {\textbf{H\L T$^-$}};

\foreach \n in {3,4,5}
{
\draw[red,thick] (A\n) -- (B\n);
\draw[red,thick] (A0) -- (A\n);
\draw[red,thick] (B\n) -- (D\n);
}

\draw ($(A0)-(0,0.5)$) node {$i_1$};

\draw ($(A3)+(0,0.5)$) node {$Z_A$};
\draw ($(B3)+(0,0.5)$) node {$Z_B$};
\draw ($(D3)+(0,0.5)$) node {$Z_C$};

\node[main1] (J0) at ($(A0)+(0,1)$) {};
\draw ($(J0)+(0,0.5)$) node {$i_2$};

\draw[red,thick] (A0) -- (J0);

\foreach \n in {3,4,5}
{
\draw[red,thick] (A\n) -- (J0);
}
\end{tikzpicture}
\end{center}
\end{minipage}\hspace{-0.3cm}
\begin{minipage}{1.2cm}
\begin{center}
\textbf{and}
\end{center}
\end{minipage}\hspace{-0.5cm}
\begin{minipage}{0.2\textwidth}
\begin{center}
\begin{tikzpicture}[scale=0.7,main1/.style = {circle,draw,fill=none,inner sep=2}, main2/.style = {circle,draw,fill=none,inner sep=3}, main3/.style = {circle,draw,fill=none,inner sep=1.5}]

\node[main1] (A0) at (0.5,0) {};

\foreach \n in {2,...,5}
{
\node[main1] (A\n) at ($(2,2.5-0.5*\n)$) {};
\node[main2] (B\n) at ($(3.25,2.5-0.5*\n)$) {};
}
\foreach \n in {2,4}
{
\node[main1] (D\n) at ($(4.5,2.5-0.5*\n)$) {};
}

\coordinate (Ax) at (2,-1) {};
\coordinate (Bx) at (3.7,-0.5) {};
\coordinate (By) at (3.7,-1.3) {};

\draw[red,thick] (Bx) -- (By);
\draw[red,thick] (A0) to[out=-20,in=180] (Bx);
\draw[red,thick] (A0) to[out=-45,in=180] (By);
\draw ($0.25*(A5)+0.25*(B5)+0.25*(Ax)+0.25*(Bx)-(-0.1,0.4)$) node {\textbf{H\L T$^-$}};

\foreach \x/\y in {2/2,4/4}
{
\draw[red,thick] (B\x) -- (D\y);
}

\foreach \x/\y in {2/3,4/5}
{
\draw[red,thick] (A\x) -- (A\y);
}

\foreach \n in {2,...,5}
{
\draw[red,thick] (A\n) -- (B\n);
\draw[red,thick] (A0) -- (A\n);
}


\node[main1] (J0) at ($(B2)+(0,1.25)$) {};
\node[main1] (I1) at ($0.5*(B3)+0.5*(B4)+(2.5,0)$) {};

\foreach \x in {2,4}
{
\draw[red,thick] (I1) -- (D\x);
}

\draw[red,thick] (A0) to[out=70,in=180] (J0);
\draw[red,thick] (J0) to[out=0,in=125] (I1);

\draw ($(A0)-(0,0.5)$) node {$i_1$};
\draw ($(J0)+(0,0.5)$) node {$i_2$};
\draw ($(I1)+(0,-0.5)$) node {$i_3$};

\draw ($(A2)+(0,0.5)$) node {$Z_A$};
\draw ($(B2)+(0,0.5)$) node {$Z_B$};
\draw ($(D2)+(0,0.5)$) node {$Z_C$};

\end{tikzpicture}
\end{center}
\end{minipage}
\end{center}

\vspace{-0.75cm}
\begin{center}
\hspace{-0.4cm}
\begin{tikzpicture}
\draw (0,0) node {$R_{\text{\bfseries EM1b}}$};
\draw (4.1,0) node {\textbf{I} \& \textbf{II}};
\draw (7.9,0) node {\textbf{III}};
\draw (12.15,0) node {\textbf{IV}};
\end{tikzpicture}
\end{center}
\vspace{-0.3cm}

\caption{The initial reduced graph $R_{\text{\bfseries EM1b}}$ in Lemma~\ref{lemma:em1b} on the left, and the three substructures within that we use to embed the tree in \textbf{Cases I} \& \textbf{II}, \textbf{Case III}, and \textbf{Case IV}, respectively.}\label{fig:EM1bproof}
\end{figure}
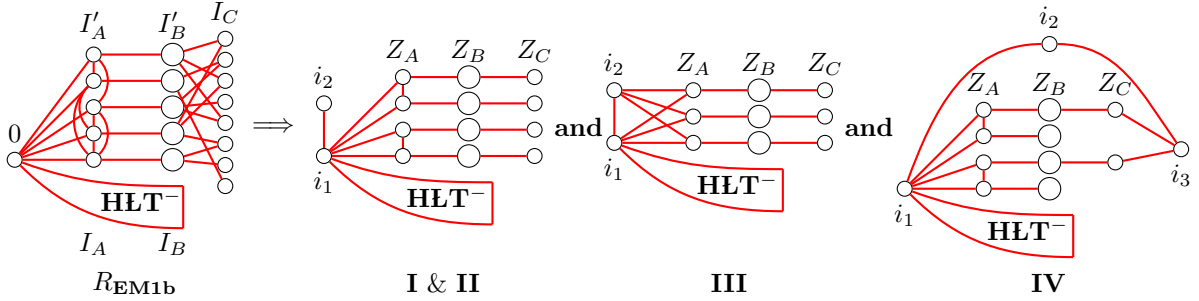


\begin{lemma}[\textbf{EM1b}]\label{lemma:em1b}
Let $1/n\ll 1/m\ll c\ll1/k\ll\eps\ll\gamma\ll\alpha\ll d\leq1$. Let $T$ be an $n$-vertex tree with $\Delta(T)\le cn$ and bipartition classes of sizes $t_1$ and $t_2$ satisfying $t_2\leq t_1\leq 2t_2$. Let $G$ be a graph with with a partition $V(G)=V_0\cup V_1\cup\cdots\cup V_k$. Let $R_{\text{\emph{\bfseries EM1b}}}$ be a graph with vertex set $[k]_0$, such that if $ij\in E(R_{\text{\emph{\bfseries EM1b}}})$ then $G[V_i,V_j]$ is $(\eps,d)$-regular. Suppose there is a partition $[k]=I_A\cup I_A'\cup I_B\cup I_B'\cup I_C$, such that the following properties hold (see Figure~\ref{fig:EM1bproof}).
 \stepcounter{propcounter}
 \begin{enumerate}[label = \emph{\textbf{\Alph{propcounter}\arabic{enumi}}}]
    \item\labelinthm{lemma:reg:em1b:1} $|V_i|=m$ for all $i\in\{0\}\cup I_A\cup I_A'\cup I_C$ and $|V_i|=t_1m/t_2$ for all $i\in I_B\cup I_B'$.
    \item\labelinthm{lemma:reg:em1b:2} $|\cup_{i\in I_A\cup I_A'}V_i|\geq(1-\gamma) t_2, |\cup_{i\in I_B\cup I_B'}V_i|\geq(1-\gamma)t_1$, $|\cup_{i\in I_A'}V_i|=10\alpha t_2$, $|\cup_{i\in I_B'}V_i|=10\alpha t_1$, and $|\cup_{i\in I_C}V_i|\geq\frac23t_2$.
    \item\labelinthm{lemma:reg:em1b:3} In $R_{\text{\emph{\bfseries EM1b}}}$, 0 is adjacent to every vertex in $I_A\cup I_A'$.
    \item\labelinthm{lemma:reg:em1b:4} In $R_{\text{\emph{\bfseries EM1b}}}$, there exists a perfect matching between $I_A$ and $I_B$, and a perfect matching between $I_A'$ and $I_B'$.
    \item\labelinthm{lemma:reg:em1b:5} In $R_{\text{\emph{\bfseries EM1b}}}$, every vertex in $I_B'$ is adjacent to at least $10\alpha|I_C|$ vertices in $I_C$.
    \item\labelinthm{lemma:reg:em1b:6} In $R_{\text{\emph{\bfseries EM1b}}}$, there exist at least $\alpha|I_A'|$ vertices with at least $\alpha|I_A'|$ neighbours in $I_A'$.
\end{enumerate}
Then, $G$ contains a copy of $T$.
\end{lemma}
\begin{proof}
Let $\xi$ satisfy $c\ll\xi\ll 1/k$. Let $S$ be the graph defined in Lemma~\ref{lem:maintreedecomp}. Using that lemma, take a homomorphism $\phi:T\to S$ such that each component of $T-\phi^{-1}(X_0\cup Y_0)$ has size at most $\xi n$ and $|\phi^{-1}(X_0\cup Y_0\cup X_1\cup Y_1)|\leq \xi n$. Without loss of generality, say $|\phi^{-1}(X_0\cup X_1\cup X_2\cup X_3)|=t_1$ and $|\phi^{-1}(Y_0\cup Y_1\cup Y_2\cup Y_3)|=t_2$. Let the components of $T-\phi^{-1}(X_0\cup Y_0)$ be $\{K_j:j\in J\}$. Moreover, partition $J$ as $J_X\cup J_Y$, so that $N_T(K_j)\subset\phi^{-1}(X_0)$ for each $j\in J_X$, and $N_T(K_j)\subset\phi^{-1}(Y_0)$ for each $j\in J_Y$.   

Let $\tau_{1,X}=|\phi^{-1}(X_2)|$, $\tau_{2,X}=|\phi^{-1}(Y_3)|$, $\tau_{1,Y}=|\phi^{-1}(X_3)|$, and $\tau_{2,Y}=|\phi^{-1}(Y_2)|$. Let $\gamma\ll\beta\ll\alpha$, and consider the following four cases. \textbf{I:} $\tau_{2,X}\geq3\beta t_2$. \textbf{II:} $\tau_{2,X}<3\beta t_2$, $\tau_{1,X}<100\beta t_1$, and $t_1<(1+200\beta)t_2$. \textbf{III:} $\tau_{2,X}<3\beta t_2$, $\tau_{1,X}<100\beta t_1$, and $t_1\geq(1+200\beta)t_2$. \textbf{IV:} $\tau_{2,X}<3\beta t_2$ and $\tau_{1,X}\geq100\beta t_1$.

\begin{figure}[h]
\input{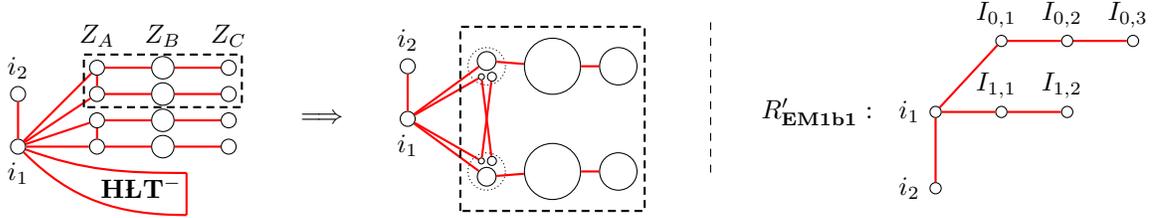}

\vspace{-0.3cm}

\caption{On the left, the transformation of the reduced graph structure used in {\bfseries Cases I} and {\bfseries II} to embed the tree. On the right, the auxiliary graph $R_{\text{\bfseries EM1b1}}'$ used when applying Lemma~\ref{lem:maintechnicalembedding}.}\label{fig:EM1bi}
\end{figure}

\medskip

\noindent \textbf{Cases I} \& \textbf{II.} By~\ref{lemma:reg:em1b:6}, we can greedily find a matching in $R_{\text{\bfseries EM1b}}[I_A']$ with size $\alpha|I_A'|/2=5\alpha^2t_2/m$. Let $Z_A\subset I_A'$ be the set of vertices covered by this matching, and let $Z_B\subset I_B'$ be the vertices matched with $Z_A$. By~\ref{lemma:reg:em1b:5}, we can greedily find a perfect matching in $R_{\text{\bfseries EM1b}}$ between $Z_B$ and some $Z_C\subset I_C$. 

Now, as depicted on the left of Figure~\ref{fig:EM1bi}, we further transform this structure into what we need to apply Lemma~\ref{lem:maintechnicalembedding}. Set $i_1=0$, pick $i_2$ arbitrarily from $I_A'\setminus Z_A$, and say it is matched with $i_2'\in I_B'\setminus Z_B$. For each $a\in Z_A$, let $a'\in Z_A$ be its neighbour in the matching above. Let $\eta\geq\beta$ be a constant to be chosen later depending on whether we are in \textbf{Case I} or \textbf{Case II}. Pick disjoint subsets $V_{a,1},V_{a,2}\subset V_a$ and $V_{a',1},V_{a',2}\subset V_{a'}$ such that $|V_{a,1}|=|V_{a',1}|=\eta t_1m/n$, $|V_{a,2}|=|V_{a',2}|=\eta t_2m/n$. By Lemma~\ref{lemma:regularity:1}, we can refine each of these new clusters, along with each cluster $V_i$ with index $i$ in $(I_A\cup I_A')\setminus(Z_A\cup \{i_2\})$ and $(I_B\cup I_B')\setminus(Z_B\cup\{i_2'\})$, into a maximum disjoint collection of smaller clusters with sizes $\gamma m$ or $\gamma t_1m/t_2$ accordingly, then pair them up so that they form $(\sqrt\eps,d-\eps)$-regular pairs. Note that at most $O(\gamma n)$ covered vertices are lost in this refinement process. Relabel these new refined clusters as $\{V_i':i\in I_{1,1}\}$ and $\{V_i':i\in I_{1,2}\}$ with $I_{1,1}$ indexing those with the smaller size. Then, we have
\begin{align*}
\sum_{i\in I_{1,1}}|V_i'|&\geq\sum_{i\in I_A\cup I_A'}|V_i|-m-|Z_A|\left(m-\frac{\eta t_2m}n\right)-O(\gamma n)\\
&\geq(1-O(\gamma))t_2-10\alpha^2t_2\left(1-\frac{\eta t_2}{n}\right)\geq(1-10\alpha^2+3\alpha^2\eta)t_2,
\end{align*}
and similarly $\sum_{i\in I_{1,2}}|V_i'|\geq(1-10\alpha^2+3\alpha^2\eta)t_1$.

Then, relabel the subsets $\{V_i\setminus(V_{i,1}\cup V_{i,2}):i\in Z_A\}$ as $\{V_i':i\in I_{0,1}\}$, and relabel the subsets $V_i$ with $i$ in $Z_B$ and $Z_C$ as $\{V_i':i\in I_{0,2}\}$ and $\{V_i':i\in I_{0,3}\}$, respectively. Note that $\sum_{i\in I_{0,1}}|V_i'|=|Z_A|(1-\eta)m=(10\alpha^2-10\alpha^2\eta)t_2$, $\sum_{i\in I_{0,2}}|V_i'|=10\alpha^2t_1$, and $\sum_{i\in I_{0,3}}|V_i'|=10\alpha^2t_2$.

Let $R_{\text{\bfseries EM1b1}}'$ be the graph on the right of Figure~\ref{fig:EM1bi}.  If we are in \textbf{Case I}, so $\tau_{2,X}\geq3\beta t_2$, set $\eta=\beta$ and define a homomorphism $\varphi:T\to R_{\text{\bfseries EM1b1}}'$ as follows. Let $\varphi(v)=i_1$ for every $v\in\phi^{-1}(X_0)$, and let $\varphi(v)=i_2$ for every $v\in\phi^{-1}(Y_0)$. For every $K_j$ with $j\in J_X$, independently with probability $10\alpha^2-\alpha^2\beta$, define $\varphi$ on $K_j$ by composing $\phi$ with the function sending $Y_1,X_2,Y_3$ to $I_{0,1},I_{0,2},I_{0,3}$, respectively; and with probability $1-10\alpha^2+\alpha^2\beta$, define $\varphi$ on $K_j$ by composing $\phi$ with the function sending $Y_1,X_2,Y_3$ to $I_{1,1},I_{1,2},I_{1,1}$, respectively. For every $K_j$ with $j\in J_Y$, independently with probability $10\alpha^2-\alpha^2\beta$, define $\varphi$ on $K_j$ by composing $\phi$ with the function sending $X_1,Y_2,X_3$ to $i_1,I_{0,1},I_{0,2}$, respectively; and with probability $1-10\alpha^2+\alpha^2\beta$, define $\varphi$ on $K_j$ by composing $\phi$ with the function sending $X_1,Y_2,X_3$ to $i_1,I_{1,1},I_{1,2}$, respectively.

By Lemma~\ref{lemma:mcdiarmid}, with positive probability we have $|\varphi^{-1}(I_{0,1}\cup I_{0,3})|=(10\alpha^2-\alpha^2\beta\pm\alpha^3\beta)t_2$, $|\varphi^{-1}(I_{0,3})|\geq20\alpha^2\beta t_2$ using that $\tau_{2,X}\geq3\beta t_2$, and $|\varphi^{-1}(I_{0,2})|=(10\alpha^2-\alpha^2\beta\pm\alpha^3\beta)t_1$. Thus, $|\varphi^{-1}(I_{0,1})|\leq(10\alpha^2-20\alpha^2\beta)t_2\leq\sum_{i\in I_{0,1}}|V_i'|-10\alpha^2\beta t_2$, and similarly for $I_{0,2}$ and $I_{0,3}$. It also follows that \[|\varphi^{-1}(I_{1,1})|\leq t_2-|\varphi^{-1}(I_{0,1}\cup I_{0,3})|\leq(1-10\alpha^2+2\alpha^2\beta)t_2\leq\sum_{i\in I_{1,1}}|V_i'|-\alpha^2\beta t_2,\] and similarly for $I_{1,2}$. Therefore, we can apply Lemma~\ref{lem:maintechnicalembedding} to find a copy of $T$ in $G$.

If we are in \textbf{Case II} instead, so $\tau_{2,X}<3\beta t_2$, $\tau_{1,X}<100\beta t_1$, and $t_1<(1+200\beta)t_2$, then we proceed similarly to above with the role of $X$ and $Y$ swapped, and with $\eta=1/2$. More specifically, we define a homomorphism $\varphi:T\to R_{\text{\bfseries EM1b1}}'$ as follows. Let $\varphi(v)=i_1$ for every $v\in\phi^{-1}(Y_0)$, and let $\varphi(v)=i_2$ for every $v\in\phi^{-1}(X_0)$.  For every $K_j$ with $j\in J_X$, independently with probability $9\alpha^2$, define $\varphi$ on $K_j$ by composing $\phi$ with the function sending $Y_1,X_2,Y_3$ to $i_1,I_{0,1},I_{0,2}$, respectively; and with probability $1-9\alpha^2$, define $\varphi$ on $K_j$ by composing $\phi$ with the function sending $Y_1,X_2,Y_3$ to $i_1,I_{1,1},I_{1,2}$, respectively. For every $K_j$ with $j\in J_Y$, independently with probability $9\alpha^2$, define $\varphi$ on $K_j$ by composing $\phi$ with the function sending $X_1,Y_2,X_3$ to $I_{0,1},I_{0,2},I_{0,3}$, respectively; and with probability $1-9\alpha^2$, define $\varphi$ on $K_j$ by composing $\phi$ with the function sending $X_1,Y_2,X_3$ to $I_{1,1},I_{1,2},I_{1,1}$, respectively.

By Lemma~\ref{lemma:mcdiarmid}, with positive probability we have $|\varphi^{-1}(I_{0,3})|\leq|\varphi^{-1}(I_{0,1}\cup I_{0,3})|=(9\pm0.1)\alpha^2t_1$, $|\varphi^{-1}(I_{0,1})|\leq\xi n+1000\alpha^2\beta t_1$ using $\tau_{1,X}<100\beta t_1$, and $|\varphi^{-1}(I_{0,2})|=(9\pm0.1)\alpha^2t_2$. It follows using $t_1<(1+200\beta)t_2$ that $|\varphi^{-1}(I)|\leq\sum_{i\in I}|V_i'|-\alpha^2n/100$ for each $I\in\{I_{0,1},I_{0,2},I_{0,3}\}$. Using $t_1<(1+200\beta)t_2$ again, we also get $|\varphi^{-1}(I_{1,1})|\leq(1-8.8\alpha^2)t_1\leq(1-8.7\alpha^2)t_2\leq\sum_{i\in I_{1,1}}|V_i'|-\alpha^2n/100$, and similarly for $I_{1,2}$. Therefore, we can apply Lemma~\ref{lem:maintechnicalembedding} to find a copy of $T$ in $G$.

\begin{figure}[h]
\input{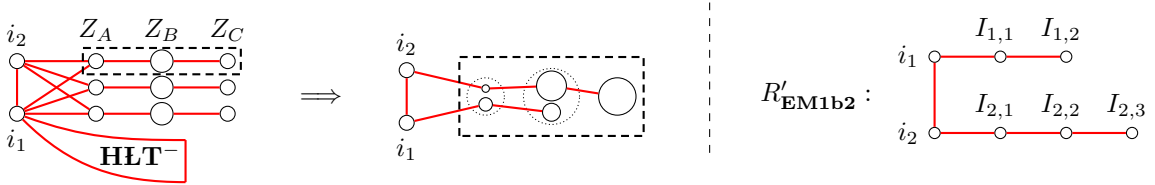}

\vspace{-0.3cm}

\caption{On the left, the transformation of the reduced graph structure used in {\bfseries Case III} to embed the tree. On the right, the auxiliary graph $R_{\text{\bfseries EM1b2}}'$ used when applying Lemma~\ref{lem:maintechnicalembedding}.}\label{fig:EM1bii}
\end{figure}

\noindent \textbf{Case III.}  In this case, we have $\tau_{2,X}<3\beta t_2$, $\tau_{1,X}<100\beta t_1$, and $t_1\geq(1+200\beta)t_2$. Thus, $\tau_{2,Y}\geq t_2-\tau_{2,X}-\xi n\geq(1-4\beta)t_2$, and similarly $\tau_{1,Y}\geq(1-101\beta)t_1$. Set $i_1=0$. By~\ref{lemma:reg:em1b:6}, we can find $i_2\in I_A'$ such that $i_2$ is adjacent to a set $Z_A$ of $\alpha|I_A'|=10\alpha^2t_2/m$ vertices in $I_A'$. Let $i_2'\in I_B'$ be the vertex matched with $i_2$, and let $Z_B\subset I_B'$ be the vertices matched with $Z_A$. By~\ref{lemma:reg:em1b:5}, we can greedily find a perfect matching between $Z_B$ and some subset $Z_C$ of $I_C$.

For convenience, denote $t_1/t_2$ by $\rho$, so $1+200\beta\leq\rho\leq2$ from assumptions. We now further transform this structure as depicted on the left of Figure~\ref{fig:EM1bii}. For each $a\in Z_A$, suppose it is matched with $b\in Z_B$, which is in turn matched with $c\in Z_C$. Take $U_a\subset V_a$ with size $2\sqrt\eps m$, take $U_b\subset V_b$ with size $m/\rho$, and let $U_c=V_c$. Then, $G[U_a,U_b],G[U_b,U_c]$ are both $(\sqrt\eps,d-\eps)$-regular by Lemma~\ref{lemma:regularity:1}. Relabel $\{U_a:a\in Z_A\},\{U_b:b\in Z_B\},\{U_c:c\in Z_C\}$ as $\{V_i':i\in I_{2,1}\},\{V_i':i\in I_{2,2}\},\{V_i':i\in I_{2,3}\}$, respectively. Note that $\sum_{i\in I_{2,1}}|V_i'|=20\alpha^2\sqrt\eps t_2\gg\xi n$, $\sum_{i\in I_{2,2}}|V_i'|=10\alpha^2t_2/\rho$, and $\sum_{i\in I_{2,3}}|V_i'|=10\alpha^2t_2$.

Next, again for each $a\in Z_A$, suppose it is matched with $b\in Z_B$. Let $W_b=V_b\setminus U_b$, so $|W_b|=(\rho-1/\rho)m$. Let $W_a\subset V_a\setminus U_a$ have size $(1-1/\rho^2)m$, possible as $\eps\ll1/\rho$. Observe that $|W_a|\geq(1-(1+200\beta)^{-2})m\geq200\beta m$. Refine the collections of clusters $\{W_a:a\in Z_A\}$ and $\{W_b:b\in Z_b\}$ above, and the clusters in $\{V_i:i\in(I_A\cup I_A')\setminus(Z_A\cup\{i_2\})\}$ and $\{V_i:i\in(I_B\cup I_B')\setminus(Z_B\cup\{i_2'\})\}$ down to clusters with sizes $\gamma m$ and $\gamma\rho m$ respectively. In the process, we lose $O(\gamma n)$ covered vertices, and the resulting refined clusters can be paired together again as $(\sqrt\eps,d-\eps)$-regular pairs by Lemma~\ref{lemma:regularity:1}. Relabel these refined clusters as $\{V_i':i\in I_{1,1}\}$ and $\{V_i':i\in I_{1,2}\}$, with $I_{1,1}$ indexing the smaller clusters. Note that
\begin{align*}
\sum_{i\in I_{1,1}}|V_i'|&\geq\sum_{i\in I_A\cup I_A'}|V_i|-\sum_{i\in Z_A\cup\{i_2\}}|V_i|+\sum_{a\in Z_A}|W_a|-O(\gamma n)\\
&\geq(1-\gamma)t_2-10\alpha^2t_2-m+10\alpha^2t_2(1-1/\rho^2)-O(\gamma n)\geq(1-(1+\beta)10\alpha^2/\rho^2)t_2,
\end{align*}
and
\begin{align*}
\sum_{i\in I_{1,2}}|V_i'|&\geq\sum_{i\in(I_B\cup I_B')\setminus\{i_2'\}}|V_i|-\sum_{b\in Z_B}|U_b|-O(\gamma n)\\
&\geq(1-\gamma)t_1-m-10\alpha^2t_2/\rho-O(\gamma n)\geq(1-(1+\beta)10\alpha^2/\rho^2)t_1.
\end{align*}

Since $\rho\geq1+200\beta$, we have $\rho(1-10\beta)(1-101\beta)\geq1+10\beta$, so we can find $p\in[0,1]$ such that
\[\frac{10\alpha^2(1+10\beta)}{(1-101\beta)\rho^2}\leq p\leq\frac{10\alpha^2(1-10\beta)}{\rho}.\]
Let $R_{\text{\bfseries EM1b2}}'$ be the graph on the right of Figure~\ref{fig:EM1bii}. Define a homomorphism $\varphi:T\to R_{\text{\bfseries EM1b2}}'$ as follows. Let $\varphi(v)=i_1$ for every $v\in\phi^{-1}(X_0)$, and let $\varphi(v)=i_2$ for every $v\in\phi^{-1}(Y_0)$. For every $K_j$ with $j\in J_X$, define $\varphi$ on $K_j$ by composing $\phi$ with the function sending $Y_1,X_2,Y_3$ to $I_{1,1},I_{1,2},I_{1,1}$, respectively. For every $K_j$ with $j\in J_Y$, independently with probability $p$, define $\varphi$ on $K_j$ by composing $\phi$ with the function sending $X_1,Y_2,X_3$ to $I_{2,1},I_{2,2},I_{2,3}$, respectively; and with probability $1-p$, define $\varphi$ on $K_j$ by composing $\phi$ with the function sending $X_1,Y_2,X_3$ to $i_1,I_{1,1},I_{1,2}$, respectively.

Using $\tau_{2,Y}\geq(1-4\beta)t_2$, $\tau_{1,Y}\geq(1-101\beta)t_1$, and Lemma~\ref{lemma:mcdiarmid}, with positive probability we have $|\varphi^{-1}(I_{2,1})|\leq\xi n$, \[(1+5\beta)10\alpha^2t_2/\rho^2\leq p(1-4\beta)t_2-5\alpha^2\beta t_2\leq|\varphi^{-1}(I_{2,2})|=p\tau_{2,Y}\pm5\alpha^2\beta t_2\leq(1-5\beta)10\alpha^2t_2/\rho,\] 
and similarly $(1+5\beta)10\alpha^2t_1/\rho^2\leq|\varphi^{-1}(I_{2,3})|\leq(1-5\beta)10\alpha^2t_1/\rho$. It follows that $|\varphi^{-1}(I)|$ is suitably smaller than $\sum_{i\in I}|V_i'|$ for each $I\in\{I_{2,1},I_{2,2},I_{2,3}\}$. Moreover, we have $|\varphi^{-1}(I_{1,1})|\leq(1-(1+5\beta)10\alpha^2/\rho^2)t_2\leq\sum_{i\in I_{1,1}}|V_i'|-40\alpha^2\beta t_2/\rho^2$, and similarly for $I_{1,2}$. Thus, we can apply Lemma~\ref{lem:maintechnicalembedding} to find a copy of $T$ in $G$.

\begin{figure}[h]
\input{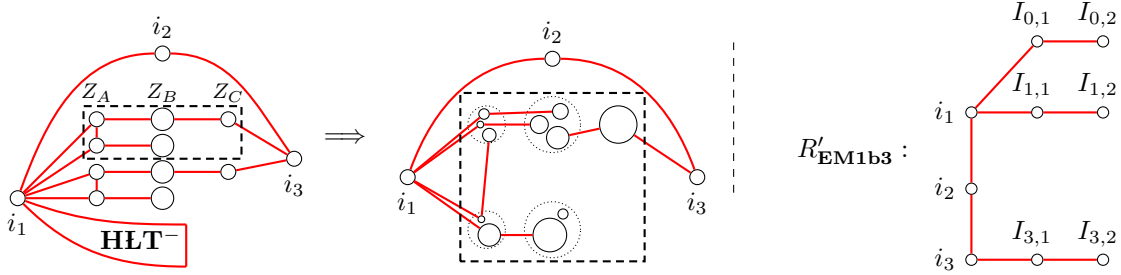}

\vspace{-0.3cm}

\caption{On the left, the transformation of the reduced graph structure used in {\bfseries Case IV} to embed the tree. On the right, the auxiliary graph $R_{\text{\bfseries EM1b3}}'$ used when applying Lemma~\ref{lem:maintechnicalembedding}.}\label{fig:EM1biii}
\end{figure}

\noindent \textbf{Case IV.} In this case, we assume that $\tau_{2,X}<3\beta t_2$ and $\tau_{1,X}\geq100\beta t_1$. Then, $\tau_{2,Y}\geq(1-4\beta)t_2$ and $\tau_{1,Y}<(1-100\beta)t_1$. Let $\rho_Y=\tau_{1,Y}/\tau_{2,Y}$, and observe that $\rho_Y<(1-40\beta)t_1/t_2$.

Let $i_1=0$. By~\ref{lemma:reg:em1b:6}, we can greedily find a matching $M$ in $R_{\text{\bfseries EM1b}}[I_A']$ with size $\alpha|I_A'|/2=5\alpha^2t_2/m$. Let $Z_A'\subset I_A'$ be the set of vertices covered by this matching, so $|Z_A'|=10\alpha^2t_2/m$, and let $Z_B'\subset I_B'$ be the vertices matched with $Z_A'$. By~\ref{lemma:reg:em1b:5}, and as in Claim~\ref{claim:backandforth}, we can find a matching $M'$ in $R_{\text{\bfseries EM1b}}[Z_B',I_C]$ with size $5\alpha|Z_B'|=50\alpha^3t_2/m$, and a vertex $i_3\in Z_B'\setminus V(M')$ that is adjacent to every vertex in $I_C\cap V(M')$. At the cost of halving the size, we can restrict $M'$ to a perfect matching between some $Z_B\subset Z_B'$ and $Z_C\subset I_C$ with $|Z_B|=|Z_C|=25\alpha^3t_2/m$, such that if $Z_A\subset Z_A'$ is the set matched with $Z_B$, then vertices in $Z_A$ all belong to the same side of the matching $M$. Let $i_2\in Z_A'\setminus Z_A$ be the vertex matched with $i_3$. 

We now further transform this structure as depicted on the left of Figure~\ref{fig:EM1biii}. For every $a\in Z_A$, suppose it is matched with $a'\in Z_A'\setminus Z_A$ by $M$ and with $b\in Z_B$. Say $b$ is matched with $c\in Z_C$ by $M'$, and $a'$ is matched with $b'\in Z_B'$. Let $U_c=V_c$, and pick $U_b\subset V_b$ with size $(\rho_Y+2\sqrt\eps)m$. Note that $G[U_b,U_c]$ is $(\sqrt\eps,d-\eps)$-regular by Lemma~\ref{lemma:regularity:1}. Relabel $\{U_c:c\in Z_C\}$ and $\{U_b:b\in Z_B\}$ as $\{V_i':i\in I_{3,1}\}$ and $\{V_i':i\in I_{3,2}\}$, respectively. Note that $\sum_{i\in I_{3,1}}|V_i'|=25\alpha^3t_2$ and $\sum_{i\in I_{3,2}}|V_i'|=25\alpha^3(\rho_Y+2\sqrt\eps)t_2\geq25\alpha^3\tau_{1,Y}+50\alpha^3\sqrt\eps t_2$.

Since $\rho_Y<(1-40\beta)t_1/t_2$, we have $|V_b\setminus U_b|\geq39\beta t_1m/t_2$, so $V_b\setminus U_b$ can be partitioned as $W_b\cup S_b$ with $|W_b|=20\beta t_1m/t_2$ and \[|S_b|=(1-20\beta)t_1m/t_2-(\rho_Y+2\sqrt\eps)m\geq19\beta t_1m/t_2.\] Partition $V_a$ into $W_a\cup S_a\cup L_a$ with $|W_a|=20\beta m$, $|S_a|=5\beta m$, and $|L_a|=(1-25\beta)m$. Recall that $a'$ is the vertex in $I_A'$ matched with $a$ by $M$, and it is matched with $b'\in I_B'$. Take $L_{a'}\subset V_{a'}$ and $L_{b'}\subset V_{b'}$ with sizes $5\beta m$ and $5\beta t_1m/t_2$, respectively. By shrinking exactly one of $S_a$ and $L_{a'}$ if necessary, we can ensure $|L_a|/|L_{a'}|=|S_b|/|S_a|$, $\max\{|L_{a'}|,|S_a|\}=5\beta m$, and $\min\{|L_{a'}|,|S_a|\}\geq\beta^2m$.

Refine all pairs of clusters of the forms $(W_a,W_b)$ and $(V_{a'}\setminus L_{a'},V_{b'}\setminus L_{b'})$, and all other matched clusters in $\{V_i:i\in(I_A\cup I_A')\setminus(Z_A\cup\{i_2\})\}$ and $\{V_i:i\in(I_B\cup I_B')\setminus(Z_B\cup\{i_3\})\}$ down to clusters with sizes $\gamma m$ and $\gamma t_1m/t_2$ accordingly. In the refinement process, we lose $O(\gamma n)$ covered vertices, and the resulting refined clusters can still be paired together as $(\sqrt\eps,d-\eps)$-regular pairs by Lemma~\ref{lemma:regularity:1}. Relabel these refined clusters as $\{V_i':i\in I_{1,1}\}$ and $\{V_i':i\in I_{1,2}\}$, with $I_{1,1}$ indexing the smaller clusters. Note that 
\begin{align*}
\sum_{i\in I_{1,1}}|V_i'|&\geq\sum_{i\in I_A\cup I_A'}|V_i|-\sum_{i\in Z_A\cup\{i_2\}}(1+5\beta)|V_i|+\sum_{a\in Z_A}|W_a|-O(\gamma n)\\
&\geq(1-O(\gamma))t_2-25\alpha^3(1+5\beta)t_2+500\alpha^3\beta t_2\geq(1-25\alpha^3+370\alpha^3\beta)t_2,
\end{align*}
and similarly $\sum_{i\in I_{1,2}}|V_i'|\geq(1-25\alpha^3+370\alpha^3\beta)t_1$. 

By Lemma~\ref{lemma:regularity:1}, we can also refine all pairs of clusters of the forms $(S_a,S_b)$ and $(L_{a'},L_a)$ down to clusters with sizes $\gamma m$ and $\gamma|L_a|m/|L_{a'}|$, then pair the new clusters into $(\sqrt\eps,d-\eps)$-regular pairs. Relabel the refined clusters as $\{V_i':i\in I_{0,1}\}$ and $\{V_i':i\in I_{0,2}\}$, with $I_{0,1}$ indexing the smaller clusters. Note that $\sum_{i\in I_{0,1}}|V_i'|\geq5\beta m\cdot25\alpha^3t_2/m-O(\gamma n)\geq100\alpha^3\beta t_2$. Moreover, using $\rho_Y=\tau_{1,Y}/\tau_{2,Y}\leq(t_1-\tau_{1,X})/\tau_{2,Y}$, we have
\begin{align*}
\sum_{i\in I_{0,2}}|V_i'|&\geq\sum_{a\in Z_A}|L_a|+\sum_{b\in Z_B}|S_b|-O(\gamma n)\\
&\geq25\alpha^3\left((1-25\beta)t_2+(1-20\beta)t_1-\frac{t_1-\tau_{1,X}}{\tau_{2,Y}}t_2-2\sqrt\eps t_2\right)-O(\gamma n)\\
&\geq25\alpha^3\left((1-26\beta)t_2+(1-20\beta)t_1-\frac{t_1}{1-4\beta}+\tau_{1,X}\right)\geq25\alpha^3\tau_{1,X}.
\end{align*}

Let $R_{\text{\bfseries EM1b3}}'$ be the graph on the right of Figure~\ref{fig:EM1biii}. Define a homomorphism $\varphi:T\to R_{\text{\bfseries EM1b3}}'$ as follows. Let $\varphi(v)=i_1$ for every $v\in\phi^{-1}(X_0)$, and let $\varphi(v)=i_2$ for every $v\in\phi^{-1}(Y_0)$. For every $K_j$ with $j\in J_X$, independently with probability $p=25\alpha^3-\alpha^3\beta$, define $\varphi$ on $K_j$ by composing $\phi$ with the function sending $Y_1,X_2,Y_3$ to $I_{0,1},I_{0,2},I_{0,1}$, respectively; and with probability $1-p$, define $\varphi$ on $K_j$ by composing $\phi$ with the function sending $Y_1,X_2,Y_3$ to $I_{1,1},I_{1,2},I_{1,1}$, respectively. For every $K_j$ with $j\in J_Y$, independently with probability $p=25\alpha^3-\alpha^3\beta$, define $\varphi$ on $K_j$ by composing $\phi$ with the function sending $X_1,Y_2,X_3$ to $i_3,I_{3,1},I_{3,2}$, respectively; and with probability $1-p$, define $\varphi$ on $K_j$ by composing $\phi$ with the function sending $X_1,Y_2,X_3$ to $i_1,I_{1,1},I_{1,2}$, respectively.

By Lemma~\ref{lemma:mcdiarmid}, with positive probability we have \[(25\alpha^3-150\alpha^3\beta)t_2\leq25\alpha^3(1-2\beta)(1-4\beta)t_2\leq|\varphi^{-1}(I_{3,1})|=p\tau_{2,Y}\pm\alpha^3\beta t_2/2\leq(25\alpha^3-\alpha^3\beta/2)t_2,\]
\[|\varphi^{-1}(I_{0,2}\cup I_{3,2})|=p(\tau_{1,X}+\tau_{1,Y})\pm\alpha^3\beta t_1\geq(25\alpha^3-\alpha^3\beta)(t_1-\xi n)-\alpha^3\beta t_1\geq(25\alpha^3-3\alpha^3\beta)t_1,\]
as well as $|\varphi^{-1}(I_{0,2})|\leq p\tau_{1,X}+\alpha^3\beta\tau_{1,X}/2\leq(25\alpha^3-\alpha^3\beta/2)\tau_{1,X}$, $|\varphi^{-1}(I_{3,2})|\leq p\tau_{1,Y}+\alpha^3\sqrt\eps t_2\leq(25\alpha^3-\alpha^3\beta)\tau_{1,Y}+\alpha^3\sqrt\eps t_2$, and $|\varphi^{-1}(I_{0,1})|\leq p\tau_{2,X}\pm\alpha^3\beta t_2\leq90\alpha^3\beta t_2$. 

It also follows that $|\varphi^{-1}(I_{1,1})|\leq t_2-|\varphi^{-1}(I_{3,1})|\leq (1-25\alpha^3+150\alpha^3\beta)t_2$ and $|\varphi^{-1}(I_{1,2})|\leq t_1-|\varphi^{-1}(I_{0,2}\cup I_{3,2})|\leq(1-25\alpha^3+3\alpha^3\beta)t_1$. Therefore, $|\varphi^{-1}(I)|$ is suitably smaller than $\sum_{i\in I}|V_i'|$ for each $I\in\{I_{1,1},I_{1,2},I_{0,1},I_{0,2},I_{3,1},I_{3,2}\}$. Thus, we can apply Lemma~\ref{lem:maintechnicalembedding} to find a copy of $T$ in $G$.
\end{proof}


\subsection{EM1c Embedding Method}\label{sec:EM1c}

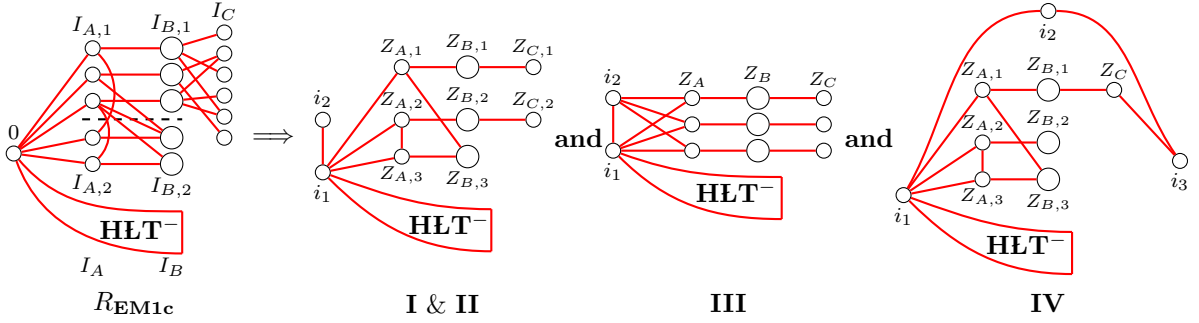
\begin{figure}[h]
\begin{center}
\hspace{-0.7cm}\begin{minipage}{0.2\textwidth}
\begin{center}
\begin{tikzpicture}[scale=0.7,main1/.style = {circle,draw,fill=none,inner sep=2}, main2/.style = {circle,draw,fill=none,inner sep=3}, main3/.style = {circle,draw,fill=none,inner sep=1.5}]

\node[main1] (A0) at (0.5,0) {};

\foreach \n in {1,...,3}
{
\node[main1] (A\n) at ($(2,2.5-0.5*\n)$) {};
\node[main2] (B\n) at ($(3.5,2.5-0.5*\n)$) {};
}
\def\droppp{0.2};
\foreach \n in {4,5}
{
\node[main1] (A\n) at ($(2,2.5-0.5*\n)-(0,\droppp)$) {};
\node[main2] (B\n) at ($(3.5,2.5-0.5*\n)-(0,\droppp)$) {};
}

\foreach \y in {1,...,6}
{
\node[main1] (D\y) at ($(4.5,1.9+0.8-\y*0.4)$) {};
}

\coordinate (Ax) at (2,-1) {};
\coordinate (Bx) at (3.7,-0.5) {};
\coordinate (By) at (3.7,-1.3) {};

\def\dropp{0.6};

\draw [thick,dashed,black] ($0.5*(A3)+0.5*(A4)-(0.2,0)$) -- ($0.5*(B3)+0.5*(B4)+(0.2,0)$);

\draw[red,thick] ($(Bx)-(0,\dropp)$) -- ($(By)-(0,\dropp)$);
\draw[red,thick] (A0) to[out=-50,in=180] ($(Bx)-(0,\dropp)$);
\draw[red,thick] (A0) to[out=-70,in=180] ($(By)-(0,\dropp)$);
\draw ($0.25*(A5)+0.25*(B5)+0.25*(Ax)+0.25*(Bx)-(-0.1,0.35)-(0,\dropp)$) node {\textbf{H\L T$^-$}};

\foreach \x/\y in {1/1,1/3,1/5,2/2,2/6,3/5,3/2,3/4}
{
\draw[red,thick] (B\x) -- (D\y);
}

\draw[red,thick] (A1) to[out=-45,in=45] (A4);
\draw[red,thick] (A3) to[out=-45,in=45] (A5);
\draw[red,thick] (A2) -- (B4);
\draw[red,thick] (A3) -- (B5);
\draw[red,thick] (A3) -- (B4);

\foreach \n in {1,...,5}
{
\draw[red,thick] (A\n) -- (B\n);
\draw[red,thick] (A0) -- (A\n);
}


\draw ($(A0)+(0,0.4)$) node {\footnotesize $0$};

\draw ($(A1)+(0,0.4)$) node {\footnotesize $I_{A,1}$};
\draw ($(B1)+(0,0.45)$) node {\footnotesize $I_{B,1}$};
\draw ($(A5)-(0,0.4)$) node {\footnotesize $I_{A,2}$};
\draw ($(B5)-(0,0.45)$) node {\footnotesize $I_{B,2}$};

\draw ($(A5)-(0,1.95)$) node {\footnotesize $I_A$};
\draw ($(B5)-(0,1.95)$) node {\footnotesize $I_B$};

\draw ($(D1)+(0,0.4)$) node {\footnotesize $I_C$};

\end{tikzpicture}
\end{center}
\end{minipage}\hspace{-0.3cm}
\begin{minipage}{1.2cm}
\begin{center}
$\implies$
\end{center}
\end{minipage}\hspace{-0.3cm}
\begin{minipage}{0.2\textwidth}
\begin{center}
\begin{tikzpicture}[scale=0.7,main1/.style = {circle,draw,fill=none,inner sep=2}, main2/.style = {circle,draw,fill=none,inner sep=3}, main3/.style = {circle,draw,fill=none,inner sep=1.5}]

\node[main1] (A0) at (0.5,0) {};
\node[main1] (J0) at (0.5,1) {};
\draw [red,thick] (A0) -- (J0);
\draw ($(J0)+(0,0.4)$) node {\footnotesize  $i_2$};

\foreach \n in {2,3}
{
\node[main1] (A\n) at ($(2,4-\n)$) {};
\node[main2] (B\n) at ($(3.25,4-\n)$) {};
\node[main1] (D\n) at ($(4.5,4-\n)$) {};
}

\def\droppp{0.2};
\foreach \n in {4}
{
\node[main1] (A\n) at ($(2,2.5-0.5*\n)-(0,\droppp)$) {};
\node[main2] (B\n) at ($(3.25,2.5-0.5*\n)-(0,\droppp)$) {};
}

\coordinate (Ax) at (2,-1) {};
\coordinate (Bx) at (3.7,-0.7) {};
\coordinate (By) at (3.7,-1.5) {};


\draw[red,thick] (Bx) -- (By);
\draw[red,thick] (A0) to[out=-20,in=180] (Bx);
\draw[red,thick] (A0) to[out=-45,in=180] (By);
\draw ($0.25*(A5)+0.25*(B5)+0.25*(Ax)+0.25*(Bx)-(-0.1,0.45)$) node {\textbf{H\L T$^-$}};

\foreach \x/\y in {2/2,3/3}
{
\draw[red,thick] (B\x) -- (D\y);
}

\draw[red,thick] (A2) -- (B4);

\foreach \x/\y in {3/4}
{
\draw[red,thick] (A\x) -- (A\y);
}

\foreach \n in {2,...,4}
{
\draw[red,thick] (A\n) -- (B\n);
\draw[red,thick] (A0) -- (A\n);
}


\draw ($(A0)-(0,0.4)$) node {\footnotesize $i_1$};

\draw ($(A2)+(0,0.35)$) node {\scriptsize $Z_{A,1}$};
\draw ($(B2)+(0,0.45)$) node {\scriptsize $Z_{B,1}$};
\draw ($(D2)+(0,0.35)$) node {\scriptsize $Z_{C,1}$};
\draw ($(A3)+(0,0.35)$) node {\scriptsize $Z_{A,2}$};
\draw ($(B3)+(0,0.45)$) node {\scriptsize $Z_{B,2}$};
\draw ($(D3)+(0,0.35)$) node {\scriptsize $Z_{C,2}$};
\draw ($(A4)-(0,0.35)$) node {\scriptsize $Z_{A,3}$};
\draw ($(B4)-(0,0.45)$) node {\scriptsize $Z_{B,3}$};


\end{tikzpicture}
\end{center}
\end{minipage}\hspace{-0.3cm}
\begin{minipage}{1.2cm}
\begin{center}
\textbf{and}
\end{center}
\end{minipage}\hspace{-0.5cm}
\begin{minipage}{0.2\textwidth}
\begin{center}
\begin{tikzpicture}[scale=0.7,main1/.style = {circle,draw,fill=none,inner sep=2}, main2/.style = {circle,draw,fill=none,inner sep=3}, main3/.style = {circle,draw,fill=none,inner sep=1.5}]

\node[main1] (A0) at (0.5,0) {};

\foreach \n in {3,...,5}
{
\node[main1] (A\n) at ($(2,2.5-0.5*\n)$) {};
\node[main2] (B\n) at ($(3.25,2.5-0.5*\n)$) {};
\node[main1] (D\n) at ($(4.5,2.5-0.5*\n)$) {};
}

\coordinate (Ax) at (2,-1) {};
\coordinate (Bx) at (3.7,-0.5) {};
\coordinate (By) at (3.7,-1.3) {};

\draw[red,thick] (Bx) -- (By);
\draw[red,thick] (A0) to[out=-20,in=180] (Bx);
\draw[red,thick] (A0) to[out=-45,in=180] (By);
\draw ($0.25*(A5)+0.25*(B5)+0.25*(Ax)+0.25*(Bx)-(-0.1,0.4)$) node {\textbf{H\L T$^-$}};


\foreach \n in {3,4,5}
{
\draw[red,thick] (A\n) -- (B\n);
\draw[red,thick] (A0) -- (A\n);
\draw[red,thick] (B\n) -- (D\n);
}

\draw ($(A0)-(0,0.4)$) node {\footnotesize $i_1$};

\draw ($(A3)+(0,0.35)$) node {\scriptsize $Z_A$};
\draw ($(B3)+(0,0.45)$) node {\scriptsize $Z_B$};
\draw ($(D3)+(0,0.35)$) node {\scriptsize $Z_C$};

\node[main1] (J0) at ($(A0)+(0,1)$) {};
\draw ($(J0)+(0,0.4)$) node {\footnotesize $i_2$};

\draw[red,thick] (A0) -- (J0);

\foreach \n in {3,4,5}
{
\draw[red,thick] (A\n) -- (J0);
}
\end{tikzpicture}
\end{center}
\end{minipage}\hspace{-0.3cm}
\begin{minipage}{1.2cm}
\begin{center}
\textbf{and}
\end{center}
\end{minipage}\hspace{-0.5cm}
\begin{minipage}{0.2\textwidth}
\begin{center}
\begin{tikzpicture}[scale=0.7,main1/.style = {circle,draw,fill=none,inner sep=2}, main2/.style = {circle,draw,fill=none,inner sep=3}, main3/.style = {circle,draw,fill=none,inner sep=1.5}]

\node[main1] (A0) at (0.5,0) {};
\foreach \n in {2}
{
\node[main1] (A\n) at ($(2,4-\n)$) {};
\node[main2] (B\n) at ($(3.25,4-\n)$) {};
\node[main1] (D\n) at ($(4.5,4-\n)$) {};
}

\foreach \n in {3}
{
\node[main1] (A\n) at ($(2,4-\n)$) {};
\node[main2] (B\n) at ($(3.25,4-\n)$) {};
}

\def\droppp{0.2};
\foreach \n in {4}
{
\node[main1] (A\n) at ($(2,2.5-0.5*\n)-(0,\droppp)$) {};
\node[main2] (B\n) at ($(3.25,2.5-0.5*\n)-(0,\droppp)$) {};
}

\coordinate (Ax) at (2,-1) {};
\coordinate (Bx) at (3.7,-0.7) {};
\coordinate (By) at (3.7,-1.5) {};


\draw[red,thick] (Bx) -- (By);
\draw[red,thick] (A0) to[out=-20,in=180] (Bx);
\draw[red,thick] (A0) to[out=-45,in=180] (By);
\draw ($0.25*(A5)+0.25*(B5)+0.25*(Ax)+0.25*(Bx)-(-0.1,0.5)$) node {\textbf{H\L T$^-$}};

\foreach \x/\y in {2/2}
{
\draw[red,thick] (B\x) -- (D\y);
}

\foreach \x/\y in {3/4}
{
\draw[red,thick] (A\x) -- (A\y);
}

\foreach \n in {2,...,4}
{
\draw[red,thick] (A\n) -- (B\n);
\draw[red,thick] (A0) -- (A\n);
}

\draw[red,thick] (A2) -- (B4);


\node[main1] (J0) at ($(B2)+(0,1.5)$) {};
\node[main1] (I1) at ($0.5*(B3)+0.5*(B4)+(2.5,0)$) {};

\foreach \x in {2}
{
\draw[red,thick] (I1) -- (D\x);
}

\draw[red,thick] (A0) to[out=60,in=180] (J0);
\draw[red,thick] (J0) to[out=0,in=125] (I1);

\draw ($(A0)-(0,0.4)$) node {\footnotesize $i_1$};
\draw ($(J0)-(0,0.4)$) node {\footnotesize $i_2$};
\draw ($(I1)-(0,0.4)$) node {\footnotesize $i_3$};

\draw ($(A2)+(0,0.35)$) node {\scriptsize $Z_{A,1}$};
\draw ($(B2)+(0,0.45)$) node {\scriptsize $Z_{B,1}$};
\draw ($(D2)+(0,0.35)$) node {\scriptsize $Z_{C}$};
\draw ($(A3)+(0,0.35)$) node {\scriptsize $Z_{A,2}$};
\draw ($(B3)+(0,0.45)$) node {\scriptsize $Z_{B,2}$};
\draw ($(A4)-(0,0.35)$) node {\scriptsize $Z_{A,3}$};
\draw ($(B4)-(0,0.45)$) node {\scriptsize $Z_{B,3}$};
\end{tikzpicture}
\end{center}
\end{minipage}
\end{center}

\vspace{-0.75cm}
\begin{center}
\hspace{-0.4cm}
\begin{tikzpicture}
\draw (0,0) node {$R_{\text{\bfseries EM1c}}$};
\draw (4.1,0) node {\textbf{I} \& \textbf{II}};
\draw (7.9,0) node {\textbf{III}};
\draw (12.15,0) node {\textbf{IV}};
\end{tikzpicture}
\end{center}
\vspace{-0.3cm}

\caption{The initial reduced graph $R_{\text{\bfseries EM1c}}$ in Lemma~\ref{lemma:em1c} on the left, and the three substructures within used to embed the tree in \textbf{Cases I} \& \textbf{II}, \textbf{Case III}, and \textbf{Case IV}, respectively.}\label{fig:EM1cproof}
\end{figure}
\begin{lemma}[\textbf{EM1c}]\label{lemma:em1c}
Let $1/n\ll 1/m\ll c\ll1/k\ll\eps\ll\gamma\ll\alpha\ll d\leq1$. Let $T$ be an $n$-vertex tree with $\Delta(T)\le cn$ and bipartition classes of sizes $t_1$ and $t_2$ with $t_2\leq t_1\leq 2t_2$. Let $G$ be a graph with with a partition $V(G)=V_0\cup V_1\cup\cdots\cup V_k$. Let $R_{\text{\emph{\bfseries EM1c}}}$ be a graph with vertex set $[k]_0$, such that if $ij\in E(R_{\text{\emph{\bfseries EM1c}}})$ then $G[V_i,V_j]$ is $(\eps,d)$-regular. Suppose there is a partition $[k]=I_A\cup I_{A,1}\cup I_{A,2}\cup I_B\cup I_{B,1}\cup I_{B,2}\cup I_C$, such that the following properties hold (see Figure~\ref{fig:EM1cproof}).
 \stepcounter{propcounter}
 \begin{enumerate}[label = \emph{\textbf{\Alph{propcounter}\arabic{enumi}}}]
    \item\labelinthm{lemma:reg:em1c:1} $|V_i|=m$ for all $i\in\{0\}\cup I_A\cup I_{A,1}\cup I_{A,2}\cup I_C$ and $|V_i|=t_1m/t_2$ for all $i\in I_B\cup I_{B,1}\cup I_{B,2}$.
    \item\labelinthm{lemma:reg:em1c:2} $|\cup_{i\in I_A\cup I_{A,1}\cup I_{A,2}}V_i|\geq(1-\gamma) t_2, |\cup_{i\in I_B\cup I_{B,1}\cup I_{B,2}}V_i|\geq(1-\gamma)t_1$, $|\cup_{i\in I_{A,1}}V_i|\geq|\cup_{i\in I_{A,2}}V_i|=10\alpha t_2$, $|\cup_{i\in I_{B,1}}V_i|\geq|\cup_{i\in I_{B,2}}V_i|=10\alpha t_1$, and $|\cup_{i\in I_C}V_i|\geq\frac23t_2$.
    \item\labelinthm{lemma:reg:em1c:3} In $R_{\text{\emph{\bfseries EM1c}}}$, 0 is adjacent to every vertex in $I_A\cup I_{A,1}\cup I_{A,2}$.
    \item\labelinthm{lemma:reg:em1c:4} In $R_{\text{\emph{\bfseries EM1c}}}$, there exist perfect matchings between each of the three pairs of sets $(I_A,I_B)$, $(I_{A,1},I_{B,1})$, and $(I_{A,2},I_{B,2})$.
    \item\labelinthm{lemma:reg:em1c:5} In $R_{\text{\emph{\bfseries EM1c}}}$, every vertex in $I_{B,1}$ is adjacent to at least $10\alpha|I_C|$ vertices in $I_C$.
    \item\labelinthm{lemma:reg:em1c:6} In $R_{\text{\emph{\bfseries EM1c}}}$, for every $a\in I_{A,2}$ and $b\in I_{B,2}$ with $ab\in E(R_{\text{\emph{\bfseries EM1c}}})$, both $a$ and $b$ have at least $\alpha|I_{A,1}|$ neighbours in $I_{A,1}$.
\end{enumerate}
Then, $G$ contains a copy of $T$.
\end{lemma}
\begin{proof}
We proceed similarly to the proof of Lemma~\ref{lemma:em1b}. Let $\xi$ satisfy $c\ll\xi\ll 1/k$. Let $S$ be the graph defined in Lemma~\ref{lem:maintreedecomp} and, using that lemma, take a homomorphism $\phi:T\to S$ such that each component of $T-\phi^{-1}(X_0\cup Y_0)$ has size at most $\xi n$ and $|\phi^{-1}(X_0\cup Y_0\cup X_1\cup Y_1)|\leq \xi n$. Without loss of generality, say $|\phi^{-1}(X_0\cup X_1\cup X_2\cup X_3)|=t_1$ and $|\phi^{-1}(Y_0\cup Y_1\cup Y_2\cup Y_3)|=t_2$. Let the components of $T-\phi^{-1}(X_0\cup Y_0)$ be $\{K_j:j\in J\}$. Moreover, partition $J$ as $J_X\cup J_Y$, so that $N_T(K_j)\subset\phi^{-1}(X_0)$ for each $j\in J_X$, and $N_T(K_j)\subset\phi^{-1}(Y_0)$ for each $j\in J_Y$.   

Let $\tau_{1,X}=|\phi^{-1}(X_2)|$, $\tau_{2,X}=|\phi^{-1}(Y_3)|$, $\tau_{1,Y}=|\phi^{-1}(X_3)|$, and $\tau_{2,Y}=|\phi^{-1}(Y_2)|$. Let $\gamma\ll\beta\ll\alpha$, and consider the following four cases. \textbf{I:} $\tau_{2,X}\geq3\beta t_2$. \textbf{II:} $\tau_{2,X}<3\beta t_2$, $\tau_{1,X}<100\beta t_1$, and $t_1<(1+200\beta)t_2$. \textbf{III:} $\tau_{2,X}<3\beta t_2$, $\tau_{1,X}<100\beta t_1$, and $t_1\geq(1+200\beta)t_2$. \textbf{IV:} $\tau_{2,X}<3\beta t_2$ and $\tau_{1,X}\geq100\beta t_1$.

\begin{figure}[h]
\input{figures/EM1c/pfEM1ci}

\vspace{-0.3cm}

\caption{On the left, the transformation of the reduced graph structure used in {\bfseries Cases I} and {\bfseries II} to embed the tree. On the right, the auxiliary graph $R_{\text{\bfseries EM1c1}}'$ used when applying Lemma~\ref{lem:maintechnicalembedding}.}\label{fig:EM1ci}
\end{figure}

\noindent \textbf{Cases I} \& \textbf{II.} 
Fix a submatching in $R_{\text{\bfseries EM1c}}[I_{A,2},I_{B,2}]$ with size $\alpha|I_{A,2}|/2=5\alpha^2t_2/m$, say between $Z_{A,3}$ and $Z_{B,3}$. By~\ref{lemma:reg:em1c:6}, we can greedily find two disjoint matchings $M_1$ and $M_2$ in $R_{\text{\bfseries EM1c}}[I_{A,1},Z_{B,3}]$ and $R_{\text{\bfseries EM1c}}[I_{A,1},Z_{A,3}]$ covering $Z_{B,3}$ and $Z_{A,3}$, respectively. Let $Z_{A,1}=V(M_1)\cap I_{A,1}$ and $Z_{A,2}=V(M_2)\cap I_{A,1}$. Let $Z_{B,1},Z_{B,2}\subset I_{B,1}$ be the vertices matched with $Z_{A,1}$ and $Z_{A,2}$, respectively. By~\ref{lemma:reg:em1c:5}, we can greedily find disjoint perfect matchings in $R_{\text{\bfseries EM1c}}[Z_{B,1}\cup Z_{B,2},I_C]$ between $Z_{B,1}$ and some $Z_{C,1}\subset I_C$, and between $Z_{B,2}$ and some $Z_{C,2}\subset I_C$. Let $Z_A=Z_{A,1}\cup Z_{A,2}\cup Z_{A,3}$ and $Z_B=Z_{B,1}\cup Z_{B,2}\cup Z_{B,3}$.

We now further transform this structure as depicted on the left of Figure~\ref{fig:EM1ci}. Set $i_1=0$, pick $i_2$ arbitrarily from $I_{A,2}\setminus Z_{A,3}$, and let $i_2'$ be the vertex in $I_{B,2}\setminus Z_{B,3}$ that $i_2$ is matched with. For each $a_3\in Z_{A,3}$ matched with $b_3\in Z_{B,3}$, say $a_3$ is matched with $a_2\in Z_{A,2}$ under $M_2$, and $b_3$ is matched with $a_1\in Z_{A,1}$ under $M_1$. Let $\eta\geq\beta$ be a constant to be chosen later depending on whether we are in \textbf{Case I} or \textbf{Case II}. Pick disjoint subsets $V_{a_3,1},V_{a_3,2}\subset V_{a_3}$ and $V_{a_2,1},V_{a_2,2}\subset V_{a_2}$ such that $|V_{a_3,1}|=|V_{a_2,1}|=\eta t_1m/n$ and $|V_{a_3,2}|=|V_{a_2,2}|=\eta t_2m/n$. Partition $V_{a_1}$ as $V_{a_1,1}\cup V_{a_1,2}$ with $|V_{a_1,1}|=\eta m$ and $|V_{a_1,2}|=(1-\eta)m$. Partition $V_{b_3}$ as $V_{b_3,1}\cup V_{b_3,2}$ with $|V_{b_3,1}|=\eta t_1m/t_2$ and $|V_{b_3,2}|=(1-\eta)t_1m/t_2$. By Lemma~\ref{lemma:regularity:1}, we can refine all pairs of clusters of the forms $(V_{a_3,2},V_{a_2,1}), (V_{a_2,2},V_{a_3,1}), (V_{a_1,1},V_{b_3,1}), (V_{a_3}\setminus(V_{a_3,1}\cup V_{a_3,2}),V_{b_3,2})$, along with all pairs of matched clusters indexed by $(I_A\cup I_{A,1}\cup I_{A,2})\setminus(Z_A\cup \{i_2\})$ and $(I_B\cup I_{B,1}\cup I_{B,2})\setminus(Z_B\cup\{i_2'\})$ into smaller clusters with sizes $\gamma m$ or $\gamma t_1m/t_2$ accordingly, then pair them up again into $(\sqrt\eps,d-\eps)$-regular pairs. Note that $O(\gamma n)$ covered vertices are lost in this refinement process. Relabel these new clusters as $\{V_i':i\in I_{1,1}\}$ and $\{V_i':i\in I_{1,2}\}$ with $I_{1,1}$ indexing those with the smaller size. Then, we have
\begin{align*}
\sum_{i\in I_{1,1}}|V_i'|&\geq\sum_{i\in I_A\cup I_{A,1}\cup I_{A,2}}|V_i|-m-2|Z_{A,2}|\left(m-\frac{\eta t_2m}n\right)-O(\gamma n)\\
&\geq(1-O(\gamma))t_2-10\alpha^2t_2\left(1-\frac{\eta t_2}{n}\right)\geq(1-10\alpha^2+3\alpha^2\eta)t_2,
\end{align*}
and similarly $\sum_{i\in I_{1,2}}|V_i'|\geq(1-10\alpha^2+3\alpha^2\eta)t_1$.

Then, let $U_i=V_{i,2}$ for every $i\in Z_{A,1}$ and let $U_i=V_i\setminus(V_{i,1}\cup V_{i,2})$ for every $i\in Z_{A,2}$. Relabel the subsets $\{U_i:i\in Z_{A,1}\cup Z_{A,2}\}$ as $\{V_i':i\in I_{0,1}\}$, and similarly use $I_{0,2}$ and $I_{0,3}$ to relabel the subsets $V_i$ with $i\in Z_{B,1}\cup Z_{B,2}$ and $i\in Z_{C,1}\cup Z_{C,2}$, respectively. Note that $\sum_{i\in I_{0,1}}|V_i'|=2|Z_{A,1}|(1-\eta)m=(10\alpha^2-10\alpha^2\eta)t_2$, $\sum_{i\in I_{0,2}}|V_i'|=10\alpha^2t_1$, and $\sum_{i\in I_{0,3}}|V_i'|=10\alpha^2t_2$. We are now in the same situation as \textbf{Cases I} \& \textbf{II} in the proof of Lemma~\ref{lemma:em1b}, so we can proceed in the same way to find a copy of $T$ in $G$.

\begin{figure}[h]
\input{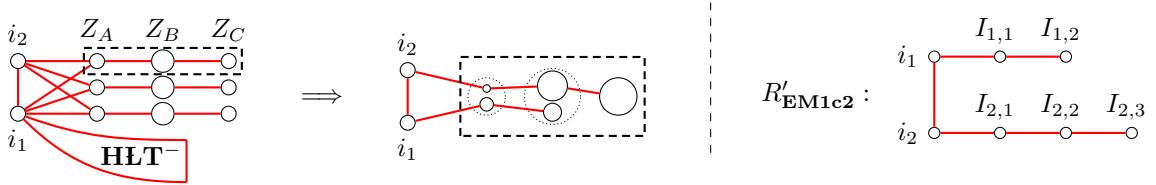}

\vspace{-0.3cm}

\caption{On the left, the transformation of the reduced graph structure used in {\bfseries Case III} to embed the tree. On the right, the auxiliary graph $R_{\text{\bfseries EM1c2}}'$ used when applying Lemma~\ref{lem:maintechnicalembedding}.}\label{fig:EM1cii}
\end{figure}

\noindent \textbf{Case III.} $\tau_{2,X}<3\beta t_2$, $\tau_{1,X}<100\beta t_1$, and $t_1\geq(1+200\beta)t_2$. Set $i_1=0$ and pick $i_2\in I_{A,2}$ arbitrarily. By~\ref{lemma:reg:em1c:6}, there is a set $Z_A\subset N_{R_{\text{\bfseries EM1c}}}(i_2,I_{A,1})$ with size $\alpha|I_{A,1}|=10\alpha^2t_2/m$. Let $Z_B\subset I_{B,1}$ be the indices matched with $Z_A$, and use~\ref{lemma:reg:em1c:5} to greedily find a perfect matching between $Z_B$ and some $Z_C\subset I_C$. This is the same structure used in \textbf{Case III} of the proof of Lemma~\ref{lemma:em1b}, so we can proceed in the same way to find an embedding of $T$ in $G$.

\begin{figure}[h]
\input{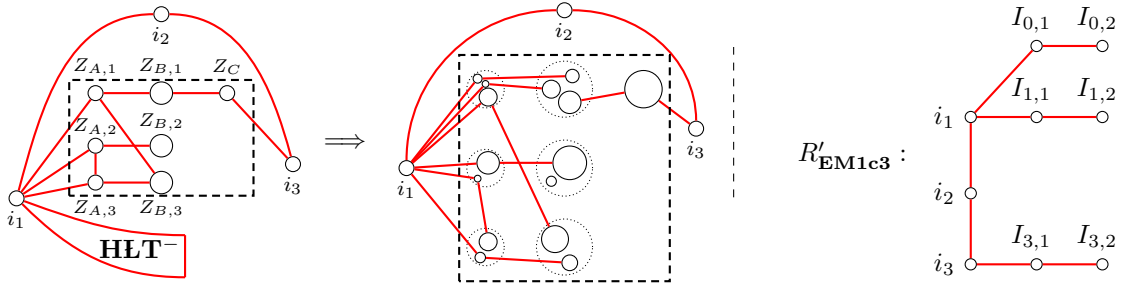}

\vspace{-0.3cm}

\caption{On the left, the transformation of the reduced graph structure used in {\bfseries Case IV} to embed the tree. On the right, the auxiliary graph $R_{\text{\bfseries EM1c3}}'$ used when applying Lemma~\ref{lem:maintechnicalembedding}.}\label{fig:EM1ciii}
\end{figure}
\noindent \textbf{Case IV.} $\tau_{2,X}<3\beta t_2$ and $\tau_{1,X}\geq100\beta t_1$. Let $\rho_Y=\tau_{1,Y}/\tau_{2,Y}$, and observe that $\rho_Y<(1-40\beta)t_1/t_2$. Fix a submatching in $R_{\text{\bfseries EM1c}}[I_{A,2},I_{B,2}]$ with size $\alpha|I_{A,1}|/2=5\alpha^2t_2/m$, say between $Z_{A,3}'$ and $Z_{B,3}'$. By~\ref{lemma:reg:em1c:6}, we can greedily find two disjoint matchings $M_1$ and $M_2$ in $R_{\text{\bfseries EM1c}}[I_{A,1},Z_{B,3}']$ and $R_{\text{\bfseries EM1c}}[I_{A,1},Z_{A,3}']$ covering $Z_{B,3}'$ and $Z_{A,3}'$, respectively. Let $Z_{A,1}'=V(M_1)\cap I_{A,1}'$ and $Z_{A,2}'=V(M_2)\cap I_{A,1}'$, and let $Z_{B,1}',Z_{B,2}'\subset I_{B,1}$ be the vertices matched with $Z_{A,1}'$ and $Z_{A,2}'$, respectively. Using~\ref{lemma:reg:em1c:5}, and as in Claim~\ref{claim:backandforth}, we can find a matching $M$ in $R_{\text{\bfseries EM1c}}[Z_{B,1}',I_C]$ with size $5\alpha|Z_{B,1}'|=25\alpha^3t_2/m$, and a vertex $i_3\in Z_{B,1}'\setminus V(M)$ that is adjacent to every vertex in $I_C\cap V(M)$. Say $M$ matches $Z_{B,1}\subset Z_{B,1}'$ with $Z_C\subset I_C$, let $Z_{A,1}\subset Z_{A,1}'$ be the set matched with $Z_{B,1}$, and let $i_2\in Z_{A,1}'\setminus Z_{A,1}$ be the vertex matched with $i_3$. Let $Z_{B,3}\subset Z_{B,3}'$ be matched with $Z_{A,1}$ by $M_1$, and suppose $Z_{B,3}$ is matched with $Z_{A,3}\subset Z_{A,3}'$, which is in turn matched with $Z_{A,2}\subset Z_{A,2}'$ by $M_2$. Finally, let $Z_{B,2}\subset Z_{B,2}'$ be matched with $Z_{A,2}$, set $Z_A=Z_{A,1}\cup Z_{A,2}\cup Z_{A,3}$, and set $Z_B=Z_{B,1}\cup Z_{B,2}\cup Z_{B,3}$. 

We now further transform this structure as depicted on the left of Figure~\ref{fig:EM1ciii}. Set $i_1=0$. 
For each $a_3\in Z_{A,3}$ matched with $b_3\in Z_{B,3}$, say $a_3$ is matched with $a_2\in Z_{A,2}$ by $M_2$, which is matched with $b_2\in Z_{B,2}$. Say $b_3$ is matched with $a_1\in Z_{A,1}$ by $M_1$, which is matched with $b_1\in Z_{B,1}$, which is in turn matched with $c\in Z_C$ by $M$. Let $U_c=V_c$, and pick $U_{b_1}\subset V_{b_1}$ with size $(\rho_Y+2\sqrt\eps)m$. Note that $G[U_{b_1},U_c]$ is $(\sqrt\eps,d-\eps)$-regular by Lemma~\ref{lemma:regularity:1}. Relabel $\{U_c:c\in Z_C\}$ and $\{U_{b_1}:b_1\in Z_{B,1}\}$ as $\{V_i':i\in I_{3,1}\}$ and $\{V_i':i\in I_{3,2}\}$, respectively. Note that $\sum_{i\in I_{3,1}}|V_i'|=25\alpha^3t_2$ and $\sum_{i\in I_{3,2}}|V_i'|=25\alpha^3(\rho_Y+2\sqrt\eps)t_2\geq25\alpha^3\tau_{1,Y}+50\alpha^3\sqrt\eps t_2$.

Since $\rho_Y<(1-40\beta)t_1/t_2$, we have $|V_{b_1}\setminus U_{b_1}|\geq39\beta t_1m/t_2$, so $V_{b_1}\setminus U_{b_1}$ can be partitioned as $W_{b_1}\cup L_{b_1}$ with $|W_{b_1}|=20\beta t_1m/t_2$ and \[|L_{b_1}|=(1-20\beta)t_1m/t_2-(\rho_Y+2\sqrt\eps)m\geq19\beta t_1m/t_2.\] Partition $V_{a_1}$ into $W_{a_1}\cup W_{a_1}'\cup L_{a_1}$ with $|W_{a_1}|=20\beta m$, $|W_{a_1}'|=(1-25\beta)m$, and $|L_{a_1}|=5\beta m$. Partition $V_{b_3}$ as $W_{b_3}\cup W_{b_3}'$ with $|W_{b_3}|=25\beta t_1m/t_2$ and $|W_{b_3}'|=(1-25\beta)t_1m/t_2$. Partition $V_{a_3}$ as $W_{a_3}\cup L_{a_3}$ with $|W_{a_3}|=25\beta m$ and $|L_{a_3}|=(1-25\beta)m$. Partition $V_{a_2}$ as $W_{a_2}\cup L_{a_2}$ with $|W_{a_2}|=(1-5\beta)m$ and $|L_{a_2}|=5\beta m$. Partition $V_{b_2}$ as $W_{b_2}\cup L_{b_2}$ with $|W_{b_2}|=(1-5\beta)t_1m/t_2$ and $|L_{b_2}|=5\beta t_1m/t_2$. By shrinking exactly one of $L_{a_1}$ or $L_{a_2}$ if necessary, we can ensure $|L_{b_1}|/|L_{a_1}|=|L_{a_3}|/|L_{a_2}|$, $\max\{|L_{a_1}|,|L_{a_2}|\}=5\beta m$, and $\min\{|L_{a_1}|,|L_{a_2}|\}\geq\beta^2m$.

Refine all pairs of clusters of the forms $(W_{a_1},W_{b_1}),(W_{a_1}',W_{b_3}'),(W_{a_3},W_{b_3}),(W_{a_2},W_{b_2})$, and all matched clusters in $\{V_i:i\in(I_A\cup I_{A,1}\cup I_{A,2})\setminus(Z_A\cup\{i_2\})\}$ and $\{V_i:i\in(I_B\cup I_{B,1}\cup I_{B,2})\setminus(Z_B\cup\{i_3\})\}$ down to clusters with sizes $\gamma m$ or $\gamma t_1m/t_2$ accordingly. In the refinement process, we lose $O(\gamma n)$ covered vertices, and the resulting refined clusters can be paired together again as $(\sqrt\eps,d-\eps)$-regular pairs by Lemma~\ref{lemma:regularity:1}. Relabel these refined clusters as $\{V_i':i\in I_{1,1}\}$ and $\{V_i':i\in I_{1,2}\}$, with $I_{1,1}$ indexing the smaller clusters. Note that 
\begin{align*}
\sum_{i\in I_{1,1}}|V_i'|&\geq\sum_{i\in I_A\cup I_{A,1}\cup I_{A,2}}|V_i|-m-|Z_{A,1}|(5\beta+5\beta+(1-25\beta))m-O(\gamma n)\\
&\geq(1-O(\gamma))t_2-25\alpha^3(1-15\beta)t_2\geq(1-25\alpha^3+370\alpha^3\beta)t_2,
\end{align*}
and similarly $\sum_{i\in I_{1,2}}|V_i'|\geq(1-25\alpha^3+370\alpha^3\beta)t_1$. 

Next, by Lemma~\ref{lemma:regularity:1}, we can refine all pairs of clusters of the forms $(L_{a_2},L_{a_3})$ and $(L_{a_1},L_{b_1})$ down to clusters with sizes $\gamma m$ and $\gamma|L_{b_1}|m/|L_{a_1}|$, and pair them up again into $(\sqrt\eps,d-\eps)$-regular pairs. Relabel the refined clusters as $\{V_i':i\in I_{0,1}\}$ and $\{V_i':i\in I_{0,2}\}$, with $I_{0,1}$ indexing the smaller clusters. Note that $\sum_{i\in I_{0,1}}|V_i'|\geq5\beta m\cdot25\alpha^3t_2/m-O(\gamma n)\geq100\alpha^3\beta t_2$. Moreover, using $\rho_Y=\tau_{1,Y}/\tau_{2,Y}\leq(t_1-\tau_{1,X})/\tau_{2,Y}$, we have
\begin{align*}
\sum_{i\in I_{0,2}}|V_i'|&\geq\sum_{a_3\in Z_{A,3}}|L_{a_3}|+\sum_{b_1\in Z_{B,1}}|L_{b_1}|-O(\gamma n)\\
&\geq25\alpha^3\left((1-25\beta)t_2+(1-20\beta)t_1-\frac{t_1-\tau_{1,X}}{\tau_{2,Y}}t_2-2\sqrt\eps t_2\right)-O(\gamma n)\\
&\geq25\alpha^3\left((1-26\beta)t_2+(1-20\beta)t_1-\frac{t_1}{1-4\beta}+\tau_{1,X}\right)\geq25\alpha^3\tau_{1,X}.
\end{align*}
We are now in the same situation as \textbf{Case IV} in the proof of Lemma~\ref{lemma:em1b}, so we can proceed in the same way to find an embedding of $T$ in $G$.
\end{proof}



\subsection{EM2a Embedding Method}\label{sec:EM2a}

\begin{figure}[h]
\input{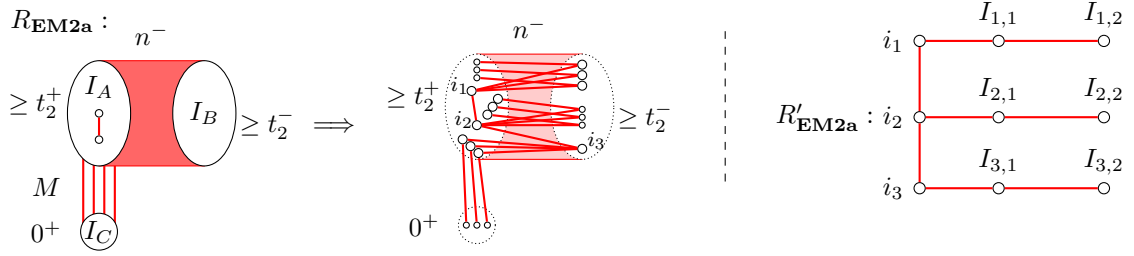}

\vspace{-0.3cm}

\caption{On the left, the initial reduced graph $R_{\text{\bfseries EM2a}}$ transformed into the substructure used to embed the tree in Lemma~\ref{lemma:em2a}. On the right, the auxiliary graph $R_{\text{\bfseries EM2a}}'$ used when applying Lemma~\ref{lem:maintechnicalembedding}.}\label{fig:EM2aproof}
\end{figure}

\begin{lemma}[\textbf{EM2a}]\label{lemma:em2a}
Let $1/n\ll c\ll 1/k\ll\eps\ll \eta \ll\alpha\ll d\leq1$. Let $T$ be an $n$-vertex tree with $\Delta(T)\le cn$ and bipartition classes of sizes $t_1$ and $t_2$ with $t_2\leq t_1\leq 2t_2$. Let $G$ be a graph with at most $2n$ vertices and a vertex partition $V_1\cup\cdots\cup V_k$ such that $|V_1|=|V_2|=\cdots=|V_k|=m$. Let $R_{\text{\emph{\bfseries EM2a}}}$ be a graph with vertex set $[k]$, such that if $ij\in E(R_{\text{\emph{\bfseries EM2a}}})$ then $G[V_i,V_j]$ is $(\eps,d)$-regular.
Suppose there is a partition $[k]=I_A\cup I_B\cup I_C$ such that the following properties hold (see Figure~\ref{fig:EM2aproof}).
\stepcounter{propcounter}
\begin{enumerate}[label = \emph{\textbf{\Alph{propcounter}\arabic{enumi}}}]
\item\labelinthm{lemma:reg:em2a:1} $|\cup_{i\in I_A}V_i|\geq t_2+200\alpha n, |\cup_{i\in I_B}V_i|\geq t_2-\alpha n, |\cup_{i\in I_A\cup I_B}V_i|=n-\alpha n$, and $|\cup_{i\in I_C}V_i|=100\alpha n$.
    \item\labelinthm{lemma:reg:em2a:2} $R_{\text{\emph{\bfseries EM2a}}}[I_A,I_B]$ is an $\eta$-almost complete bipartite graph.
    \item\labelinthm{lemma:reg:em2a:3} $R_{\text{\emph{\bfseries EM2a}}}[I_A,I_C]$ contains a matching $M$ covering $I_C$.
    \item\labelinthm{lemma:reg:em2a:4} $E(R_{\text{\emph{\bfseries EM2a}}}[I_A]-V(M))\neq \emptyset$.
\end{enumerate}
Then, $G$ contains a copy of $T$.
\end{lemma}
\begin{proof}
Let $\xi$ satisfy $c\ll\xi\ll 1/k$. Let $S$ be the graph defined in Lemma~\ref{lem:maintreedecomp} and, using that lemma, take a homomorphism $\phi:T\to S$ such that each component of $T-\phi^{-1}(X_0\cup Y_0)$ has size at most $\xi n$ and $|\phi^{-1}(X_0\cup Y_0\cup X_1\cup Y_1)|\leq \xi n$. Without loss of generality, say $|\phi^{-1}(X_0\cup X_1\cup X_2\cup X_3)|=t_1$ and $|\phi^{-1}(Y_0\cup Y_1\cup Y_2\cup Y_3)|=t_2$. Let the components of $T-\phi^{-1}(X_0\cup Y_0)$ be $\{K_j:j\in J\}$. Moreover, partition $J$ as $J_X\cup J_Y$, so that $N_T(K_j)\subset\phi^{-1}(X_0)$ for each $j\in J_X$, and $N_T(K_j)\subset\phi^{-1}(Y_0)$ for each $j\in J_Y$.

Let $\tau_{1,X}=|\phi^{-1}(X_2)|$, $\tau_{2,X}=|\phi^{-1}(Y_3)|$, $\tau_{1,Y}=|\phi^{-1}(X_3)|$ and $\tau_{2,Y}=|\phi^{-1}(Y_2)|$.  
Let $n_A=|\cup_{i\in I_A}V_i|$, $n_B=|\cup_{i\in I_B}V_i|$ and $n_C=|\cup_{i\in I_C}V_i|$. We now separate into the following two cases. \textbf{I:} $\tau_{1,X}+\tau_{2,Y}\geq n_B+20\alpha n$. \textbf{II:} $\tau_{1,X}+\tau_{2,Y}< n_B+20\alpha n$.

\medskip

\noindent \textbf{Case I.} $\tau_{1,X}+\tau_{2,Y}\geq n_B+20\alpha n$. Then, $(1-20\alpha)(\tau_{1,X}+\tau_{2,Y})\geq n_B-\alpha n$. Note that $\tau_{2,X}+\tau_{2,Y}\leq t_2\leq n_B+\alpha n$, so $(1-20\alpha)(\tau_{2,X}+\tau_{2,Y})\leq n_B-\alpha n$. Thus, we can find $p\in[0,1-20\alpha]$ such that $(1-20\alpha)\tau_{2,Y}+(1-20\alpha-p)\tau_{2,X}+p\tau_{1,X}=n_B-\alpha n$.

Let $R_{\text{\textbf{EM2a}}}'$ be the graph depicted on the right in Figure~\ref{fig:EM2aproof}, which is a subgraph of the graph $R'$ in Lemma~\ref{lem:maintechnicalembedding}.
Define a random homomorphism $\varphi:T\to R_{\text{\textbf{EM2a}}}'$ as follows. First, for vertices in $\phi^{-1}(X_0\cup Y_0)$, set $\varphi(v)=i_3$ if $\phi(v)=Y_0$, and set $\varphi(v)=i_2$ if $\phi(v)=X_0$. 

For each $j\in J_X$ independently at random, with probability $20\alpha$, define $\varphi$ on $K_j$ by composing $\phi$ with the map that sends $Y_1,X_2,Y_3$ to $i_3,I_{3,1},I_{3,2}$, respectively; with probability $p$ define $\varphi$ on $K_j$ by composing $\phi$ with the map that sends $Y_1,X_2,Y_3$ to $i_1,I_{1,1},I_{1,2}$, respectively; and with probability $1-20\alpha-p$ define $\varphi$ on $K_j$ by composing $\phi$ with the map that sends $Y_1,X_2,Y_3$ to $I_{2,1},I_{2,2},I_{2,1}$, respectively.

For each $j\in J_Y$ independently at random, with probability $20\alpha$, define $\varphi$ on $K_j$ by composing $\phi$ with the map that sends $X_1,Y_2,X_3$ to $I_{3,1},I_{3,2},I_{3,1}$, respectively; and with probability $1-20\alpha$ define $\varphi$ on $K_j$ by composing $\phi$ with the map that sends $X_1,Y_2,X_3$ to $i_2,I_{2,1},I_{2,2}$, respectively.

Since $|K_j|\leq\xi n$ for all $j\in J$, and $|\phi^{-1}(X_0\cup Y_0\cup X_1\cup Y_1)|\leq\xi n$, we can use Lemma~\ref{lemma:mcdiarmid} to conclude that with strictly positive probability, $|\varphi^{-1}(I_{3,1})|=20\alpha t_1\pm\alpha t_1$, $|\varphi^{-1}(I_{3,2})|=20\alpha t_2\pm\alpha t_2\leq n_C-\alpha n$, and $|\varphi^{-1}(I_{1,1}\cup I_{2,1})|=n_B-\alpha n\pm\alpha n/2$. It follows that $|\varphi^{-1}(\{I_{1,2},I_{2,2},I_{3,1}\})|\leq n-19\alpha t_2-(n_B-3\alpha n/2)\leq n_A+3\alpha n-19\alpha t_2\leq n_A-\alpha n$.

Using~\ref{lemma:reg:em2a:2}--\ref{lemma:reg:em2a:4}, we can find $i_1,i_2\in I_A\setminus V(M)$ and $i_3\in I_B$, such that $i_1i_2,i_2i_3\in E(R_\textbf{EM2a})$, and $i_3$ is adjacent to all but at most $\eta k$ vertices in $I_A\cap V(M)$. Thus, we can find $I_{3,1}\subset N_{R_\textbf{EM2a}}(i_3,I_A\cap V(M))$ covering $|\varphi^{-1}(I_{3,1})|+\alpha n/10$ vertices. Let $I_{3,2}\subset I_C$ denote the set of vertices matched with $I_{3,1}$ by $M$, then $|\varphi^{-1}(I_{3,b})|\leq\sum_{i\in I_{3,b}}|V_i|-\alpha n/10$ for each $b\in[2]$. 

Let $I_A'=I_A\setminus(I_{3,1}\cup\{i_1,i_2\})$ and $I_B'=(N_{R_\textbf{EM2a}}(i_1,I_B)\cap N_{R_\textbf{EM2a}}(i_2,I_B))\setminus\{i_3\}$. Note that from above, $|\varphi^{-1}(\{I_{1,2},I_{2,2}\})|\leq n_A-\alpha n-|\varphi^{-1}(I_{3,1})|\leq\sum_{i\in I_A'}|V_i|-\alpha n/5$, and $|\varphi^{-1}(I_{1,1}\cup I_{2,1})|\leq n_B-\alpha n/2$, so we can find partitions $I_A'=I_{1,2}'\cup I_{2,2}'$ and $I_B'=I_{1,1}'\cup I_{2,1}'$, such that $\sum_{i\in I_{a,b}'}|V_i|\geq|\varphi^{-1}(I_{a,b})|+\alpha n/20$ for each $a,b\in[2]$. Moreover, as $|I_{a,b}'|\geq\alpha n/20m$ by construction for each $a,b\in[2]$, and $R_\textbf{EM2a}[I_A,I_B]$ is $\eta$-almost complete, by choosing these two partitions randomly and applying Lemma~\ref{lemma:chernoff}, we can find a realisation such that both $R_\textbf{EM2a}[I_{1,1}',I_{1,2}']$ and $R_\textbf{EM2a}[I_{2,1}',I_{2,2}']$ are $10\eta$-almost complete. This allows us to apply Lemma~\ref{lemma:regularity:refine} to refine the clusters $V_i$ with $i\in I_{1,1}'\cup I_{1,2}'\cup I_{2,1}'\cup I_{2,2}'$ to obtain, for each $a,b\in[2]$, a set $\{V_i':i\in I_{a,b}\}$ of clusters of the same size, such that $\sum_{i\in I_{a,b}}|V_i'|\geq|\varphi^{-1}(I_{a,b})|+\alpha n/100$. Moreover, for each $a\in[2]$, the refined clusters indexed by $I_{a,1}$ and $I_{a,2}$ can be matched up, so that each pair is $(\sqrt\eps,d-\eps)$-regular (see Figure~\ref{fig:EM2aproof}). This allows us to use Lemma~\ref{lem:maintechnicalembedding} to find a copy of $T$ in $G$.

\medskip

\noindent \textbf{Case II.} $\tau_{1,X}+\tau_{2,Y}< n_B+20\alpha n$. Then, $(1-70\alpha)(\tau_{1,X}+\tau_{2,Y})\leq n_B-\alpha n$. Note that $\tau_{1,X}+\tau_{1,Y}\geq t_1-\xi n\geq n_B+100\alpha n$, so $(1-70\alpha)(\tau_{1,X}+\tau_{1,Y})\geq n_B-\alpha n$. Therefore, there exists $p\in[0,1-70\alpha]$ such that $(1-70\alpha)\tau_{1,X}+(1-70\alpha-p)\tau_{1,Y}+p\tau_{2,Y}=n_B-\alpha n$.

Similar to above, we define a random homomorphism $\varphi:T\to R_{\text{\textbf{EM2a}}}'$ as follows. First, for vertices in $\phi^{-1}(X_0\cup Y_0)$, set $\varphi(v)=i_2$ if $\phi(v)=Y_0$, and set $\varphi(v)=i_3$ if $\phi(v)=X_0$. 

For each $j\in J_X$ independently at random, with probability $70\alpha$, define $\varphi$ on $K_j$ by composing $\phi$ with the map that sends $Y_1,X_2,Y_3$ to $I_{3,1},I_{3,2},I_{3,1}$, respectively; and with probability $1-70\alpha$ define $\varphi$ on $K_j$ by composing $\phi$ with the map that sends $Y_1,X_2,Y_3$ to $i_2,I_{2,1},I_{2,2}$, respectively.

For each $j\in J_Y$ independently at random, with probability $70\alpha$, define $\varphi$ on $K_j$ by composing $\phi$ with the map that sends $X_1,Y_2,X_3$ to $i_3,I_{3,1},I_{3,2}$, respectively; with probability $p$ define $\varphi$ on $K_j$ by composing $\phi$ with the map that sends $X_1,Y_2,X_3$ to $i_1,I_{1,1},I_{1,2}$, respectively; and with probability $1-70\alpha-p$ define $\varphi$ on $K_j$ by composing $\phi$ with the map that sends $X_1,Y_2,X_3$ to $I_{2,1},I_{2,2},I_{2,1}$, respectively. 

Again, we can use Lemma~\ref{lemma:mcdiarmid} to conclude that with positive probability, $|\varphi^{-1}(I_{3,1})|=70\alpha t_2\pm\alpha t_2$, $|\varphi^{-1}(I_{3,2})|=70\alpha t_1\pm\alpha t_1\leq n_C-\alpha n$, and $|\varphi^{-1}(I_{1,1}\cup I_{2,1})|=n_B-\alpha n\pm\alpha n/2$. It follows that $|\varphi^{-1}(\{I_{1,2},I_{2,2}\})|+|\varphi^{-1}(I_{3,2})|\leq n-69\alpha t_2-(n_B-3\alpha n/2)\leq n_A+3\alpha n-69\alpha t_2\leq n_A-\alpha n$.

Similar to \textbf{Case I} above, by~\ref{lemma:reg:em2a:2}--\ref{lemma:reg:em2a:4} and after refining, we can find $i_1,i_2\in I_A$ and $i_3\in I_B$, along with three matchings of refined clusters of suitable sizes in $G$ attached to $i_1,i_2,i_3$ respectively, as depicted in Figure~\ref{fig:EM2aproof}, which allow us to apply Lemma~\ref{lem:maintechnicalembedding} to find a copy of $T$ in $G$.
\end{proof}


\subsection{EM2b Embedding Method}\label{sec:EM2b}

\begin{figure}[h]
\input{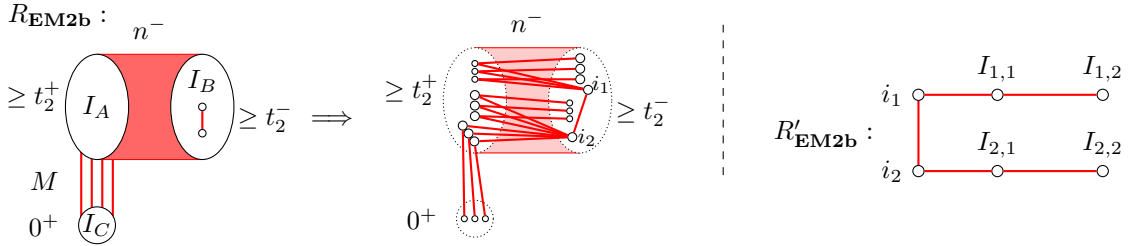}

\vspace{-0.3cm}

\caption{On the left, the initial reduced graph $R_{\text{\bfseries EM2b}}$ transformed into the substructure used to embed the tree in Lemma~\ref{lemma:em2b}. On the right, the auxiliary graph $R_{\text{\bfseries EM2b}}'$ used when applying Lemma~\ref{lem:maintechnicalembedding}.}\label{fig:EM2bproof}
\end{figure}

\begin{lemma}[\bfseries EM2b]\label{lemma:em2b}
Let $1/n\ll 1/m\ll c\ll1/k\ll\eps\ll\eta\ll\alpha\ll d\leq1$. Let $T$ be an $n$-vertex tree with $\Delta(T)\le cn$ and bipartition classes of sizes $t_1$ and $t_2$ with $t_2\le t_1\leq 2t_2$. Let $G$ be a graph with at most $2n$ vertices and a vertex partition $V_1\cup\cdots\cup V_k$ with $|V_1|=\cdots=|V_k|=m$. Let $R_{\text{\emph{\bfseries EM2b}}}$ be a graph with vertex set $[k]$, such that if $ij\in E(R_{\text{\emph{\bfseries EM2b}}})$ then $G[V_i,V_j]$ is $(\eps,d)$-regular.
Suppose there is a partition $[k]=I_A\cup I_B\cup I_C$ such that the following properties hold (see Figure~\ref{fig:EM2bproof}).
\stepcounter{propcounter}
\begin{enumerate}[label = \emph{\textbf{\Alph{propcounter}\arabic{enumi}}}]
    \item\labelinthm{lemma:reg:em2b:1} $|\cup_{i\in I_A}V_i|\geq t_2+200\alpha n, |\cup_{i\in I_B}V_i|\geq t_2-\alpha n, |\cup_{i\in I_A\cup I_B}V_i|=(1-\alpha)n$, and $|\cup_{i\in I_C}V_i|=100\alpha n$.
    \item\labelinthm{lemma:reg:em2b:2} $R_{\text{\emph{\bfseries EM2b}}}[I_A,I_B]$ is an $\eta$-almost complete bipartite graph.
    \item\labelinthm{lemma:reg:em2b:3} $R_{\text{\emph{\bfseries EM2b}}}[I_A,I_C]$ contains a matching $M$ covering $I_C$.
    \item\labelinthm{lemma:reg:em2b:4} $E(R_{\text{\emph{\bfseries EM2b}}}[I_B])\neq \emptyset$.
\end{enumerate}
Then, $G$ contains a copy of $T$.
\end{lemma}
\begin{proof}
Let $\xi$ satisfy $c\ll\xi\ll 1/k$. Let $S$ be the graph defined in Lemma~\ref{lem:maintreedecomp}. Using that lemma, take a homomorphism $\phi:T\to S$ such that each component of $T-\phi^{-1}(X_0\cup Y_0)$ has size at most $\xi n$ and $|\phi^{-1}(X_0\cup Y_0\cup X_1\cup Y_1)|\leq \xi n$. Without loss of generality, say $|\phi^{-1}(X_0\cup X_1\cup X_2\cup X_3)|=t_1$ and $|\phi^{-1}(Y_0\cup Y_1\cup Y_2\cup Y_3)|=t_2$. Let the components of $T-\phi^{-1}(X_0\cup Y_0)$ be $\{K_j:j\in J\}$, and partition $J$ as $J_X\cup J_Y$, so that $N_T(K_j)\subset\phi^{-1}(X_0)$ for each $j\in J_X$, and $N_T(K_j)\subset\phi^{-1}(Y_0)$ for each $j\in J_Y$.

Let $\tau_{1}=|\phi^{-1}(X_2\cup X_3)|$ and $\tau_{2}=|\phi^{-1}(Y_2\cup Y_3)|$. Let $n_A=|\cup_{i\in I_A}V_i|$, $n_B=|\cup_{i\in I_B}V_i|$, and $n_C=|\cup_{i\in I_C}V_i|$. Note that $\tau_{2}\leq t_2\leq n_B+\alpha n$, so $(1-20\alpha)\tau_{2}\leq n_B-\alpha n$. On the other hand, $(1-20\alpha)\tau_1\geq(1-20\alpha)(t_1-\xi n)\geq n_B-\alpha n$ as $n_B\leq t_1-201\alpha n$, so there exists $p\in[0,1-20\alpha]$ such that $p\tau_1+(1-20\alpha-p)\tau_2=n_B-\alpha n$.

Let $R_{\text{\textbf{EM2b}}}'$ be the graph depicted on the right in Figure~\ref{fig:EM2bproof}, which is a subgraph of the graph $R'$ in Lemma~\ref{lem:maintechnicalembedding}. Define a random homomorphism $\varphi:T\to R_{\text{\textbf{EM2a}}}'$ as follows. First, for vertices in $\phi^{-1}(X_0\cup Y_0)$, set $\varphi(v)=i_1$ if $\phi(v)=X_0$, and set $\varphi(v)=i_2$ if $\phi(v)=Y_0$. 

For each $j\in J_X$ independently at random, with probability $1-p$, define $\varphi$ on $K_j$ by composing $\phi$ with the map that sends $Y_1,X_2,Y_3$ to $i_2,I_{2,1},I_{2,2}$, respectively; and with probability $p$ define $\varphi$ on $K_j$ by composing $\phi$ with the map that sends $Y_1,X_2,Y_3$ to $I_{1,1},I_{1,2},I_{1,1}$, respectively.

For each $j\in J_Y$ independently at random, with probability $1-p$, define $\varphi$ on $K_j$ by composing $\phi$ with the map that sends $X_1,Y_2,X_3$ to $I_{2,1},I_{2,2},I_{2,1}$, respectively; and with probability $p$ define $\varphi$ on $K_j$ by composing $\phi$ with the map that sends $X_1,Y_2,X_3$ to $i_1,I_{1,1},I_{1,2}$, respectively.

Since $|K_j|\leq\xi n$ for all $j\in J$ and $|\phi^{-1}(X_0\cup Y_0\cup X_1\cup Y_1)|\leq\xi n$, we can use Lemma~\ref{lemma:mcdiarmid} to conclude that with strictly positive probability, $|\varphi^{-1}(I_{1,2}\cup I_{2,2})|=p\tau_1+(1-p)\tau_2\pm\alpha n\leq n_B+20\alpha\tau_2\leq n_B+n_C-10\alpha n$, and $|\varphi^{-1}(I_{1,1}\cup I_{2,1})|=(1-p)\tau_1+p\tau_2\pm\alpha n\leq n_A-\alpha n$. Therefore, like in the proof of Lemma~\ref{lemma:em2a}, by refining and using~\ref{lemma:reg:em2b:2}--\ref{lemma:reg:em2b:4}, we can find $i_1,i_2\in I_B$ such that $i_1i_2\in E(R_{\text{\bfseries EM2b}})$, along with two matchings of refined clusters of suitable sizes forming $(\sqrt\eps,d-\eps)$-regular pairs attached to $i_1,i_2$ as depicted in Figure~\ref{fig:EM2bproof}, which allow us to apply Lemma~\ref{lem:maintechnicalembedding} to find a copy of $T$ in $G$.
\end{proof}


\subsection{EM2c Embedding Method}\label{sec:EM2c}

\begin{figure}[h]
\input{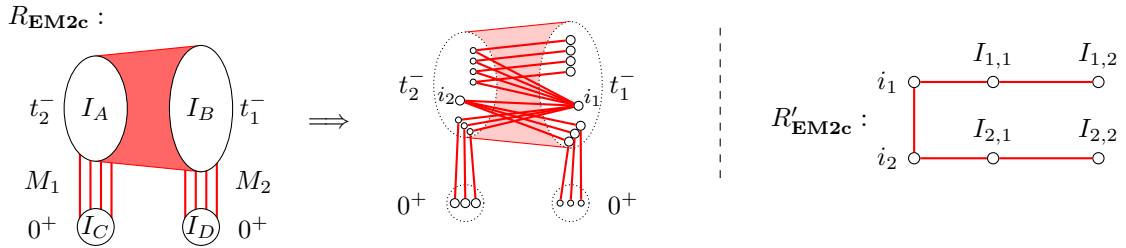}

\vspace{-0.3cm}

\caption{On the left, the initial reduced graph $R_{\text{\bfseries EM2c}}$ transformed into the substructure used to embed the tree in Lemma~\ref{lemma:em2c}. On the right, the auxiliary graph $R_{\text{\bfseries EM2c}}'$ used when applying Lemma~\ref{lem:maintechnicalembedding}.}\label{fig:EM2cproof}
\end{figure}

\begin{lemma}[\textbf{EM2c}]\label{lemma:em2c}
Let $1/n\ll 1/m\ll c\ll1/k\ll\eps\ll\eta\ll\alpha\ll d\leq1$. Let $T$ be an $n$-vertex tree with $\Delta(T)\le cn$ and bipartition classes of sizes $t_1$ and $t_2$ with $t_2\leq t_1\leq 2t_2$. Let $G$ be a graph with with a partition $V(G)=V_1\cup\cdots\cup V_k$ satisfying $|V_1|=\cdots=|V_k|=m$. Let  $R_{\text{\emph{\bfseries EM2c}}}$ be a graph with vertex set $[k]$, such that if $ij\in E(R_{\text{\emph{\bfseries EM2c}}})$ then $G[V_i,V_j]$ is $(\eps,d)$-regular. Suppose there is a partition $[k]=I_A\cup I_B\cup I_C\cup I_D$ such that the following properties hold (see Figure~\ref{fig:EM2cproof}).
\stepcounter{propcounter}
\begin{enumerate}[label = \emph{\textbf{\Alph{propcounter}\arabic{enumi}}}]
    \item\labelinthm{lemma:em2c:1} $|\cup_{i\in I_A}V_i|\geq(1-\alpha) t_2, |\cup_{i\in I_B}V_i|\geq(1-11\alpha)t_1$, $|\cup_{i\in I_C}V_i|=100\alpha t_2$, and $|\cup_{i\in I_D}V_i|=10\alpha t_2$.
    \item\labelinthm{lemma:em2c:2} $R_{\text{\emph{\bfseries EM2c}}}[I_A,I_B]$ is an $\eta$-almost complete bipartite graph.
    \item\labelinthm{lemma:em2c:3} $R_{\text{\emph{\bfseries EM2c}}}[I_A,I_C]$ contains a matching $M_1$ covering $I_C$.
    \item\labelinthm{lemma:em2c:4} $R_{\text{\emph{\bfseries EM2c}}}[I_B,I_D]$ contains a matching $M_2$ covering $I_D$.
\end{enumerate}
Then, $G$ contains a copy of $T$.
\end{lemma}
\begin{proof}
Let $\xi$ satisfy $c\ll\xi\ll 1/k$. Let $S$ be the graph defined in Lemma~\ref{lem:maintreedecomp}. Using that lemma, take a homomorphism $\phi:T\to S$ such that each component of $T-\phi^{-1}(X_0\cup Y_0)$ has size at most $\xi n$ and $|\phi^{-1}(X_0\cup Y_0\cup X_1\cup Y_1)|\leq \xi n$. Without loss of generality, say $|\phi^{-1}(X_0\cup X_1\cup X_2\cup X_3)|=t_1$ and $|\phi^{-1}(Y_0\cup Y_1\cup Y_2\cup Y_3)|=t_2$. Let the components of $T-\phi^{-1}(X_0\cup Y_0)$ be $\{K_j:j\in J\}$, and partition $J$ as $J_X\cup J_Y$, so that $N_T(K_j)\subset\phi^{-1}(X_0)$ for each $j\in J_X$, and $N_T(K_j)\subset\phi^{-1}(Y_0)$ for each $j\in J_Y$.   

Let $\tau_{1}=|\phi^{-1}(X_2\cup X_3)|$ and $\tau_{2}=|\phi^{-1}(Y_2\cup Y_3)|$. For each $\ast\in\{A,B,C,D\}$, let $n_\ast=|\cup_{i\in I_\ast}V_i|$. Let $R_{\text{\textbf{EM2c}}}'$ be the graph depicted on the right in Figure~\ref{fig:EM2cproof}, which is a subgraph of the graph $R'$ in Lemma~\ref{lem:maintechnicalembedding}. Define a random homomorphism $\varphi:T\to R_{\text{\textbf{EM2c}}}'$ as follows. First, for vertices in $\phi^{-1}(X_0\cup Y_0)$, set $\varphi(v)=i_1$ if $\phi(v)=X_0$, and set $\varphi(v)=i_2$ if $\phi(v)=Y_0$. 

For each $j\in J_X$ independently at random, with probability $4\alpha$, define $\varphi$ on $K_j$ by composing $\phi$ with the map that sends $Y_1,X_2,Y_3$ to $i_2,I_{2,1},I_{2,2}$, respectively; and with probability $1-4\alpha$ define $\varphi$ on $K_j$ by composing $\phi$ with the map that sends $Y_1,X_2,Y_3$ to $I_{1,1},I_{1,2},I_{1,1}$, respectively.

For each $j\in J_Y$ independently at random, with probability $4\alpha$, define $\varphi$ on $K_j$ by composing $\phi$ with the map that sends $X_1,Y_2,X_3$ to $I_{2,1},I_{2,2},I_{2,1}$, respectively; and with probability $1-4\alpha$ define $\varphi$ on $K_j$ by composing $\phi$ with the map that sends $X_1,Y_2,X_3$ to $i_1,I_{1,1},I_{1,2}$, respectively.

By Lemma~\ref{lemma:mcdiarmid}, with positive probability, $|\varphi^{-1}(I_{2,2})|=4\alpha\tau_2\pm\alpha\tau_2/2\leq n_D-\alpha n/10$, $|\varphi^{-1}(I_{2,1})|=4\alpha\tau_1\pm\alpha\tau_1/2\leq 10\alpha t_2-\alpha n/100$, $|\varphi^{-1}(I_{1,2}\cup I_{2,1})|\leq t_1\leq n_B+n_C-2\alpha n$, and $|\varphi^{-1}(I_{1,1})|=(1-4\alpha)\tau_2\pm\alpha\tau_2\leq n_A-\alpha n/10$. Therefore, like in Lemma~\ref{lemma:em2a}, by~\ref{lemma:em2c:2}--\ref{lemma:em2c:4} and after refining, we can find $i_1\in I_B$ and $i_2\in I_A$ such that $i_1i_2\in E(R_{\text{\bfseries EM2c}})$, along with two matchings of refined clusters of suitable sizes forming $(\sqrt\eps,d-\eps)$-regular pairs attached to $i_1,i_2$ as depicted in Figure~\ref{fig:EM2cproof}, which allow us to apply Lemma~\ref{lem:maintechnicalembedding} to find a copy of $T$ in $G$.
\end{proof}


\subsection{EM2d Embedding method}\label{sec:EM3}
\begin{figure}[h]
\input{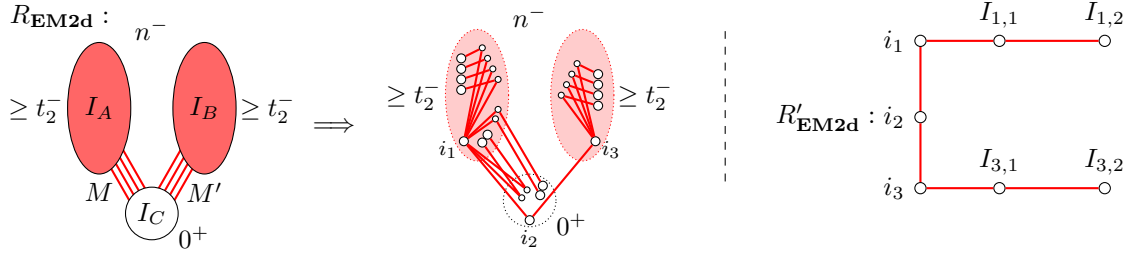}

\vspace{-0.3cm}

\caption{On the left, the initial reduced graph $R_{\text{\bfseries EM2d}}$ transformed into the substructure used to embed the tree in Lemma~\ref{lemma:em3}. On the right, the auxiliary graph $R_{\text{\bfseries EM2d}}'$ used when applying Lemma~\ref{lem:maintechnicalembedding}.}\label{fig:EM2dproof}
\end{figure}

\begin{lemma}[\textbf{EM2d}]\label{lemma:em3}
Let $1/n\ll 1/m\ll c\ll1/k\ll\eps\ll\eta\ll\alpha\ll d\leq1$. Let $T$ be an $n$-vertex tree with $\Delta(T)\le cn$ and bipartition classes of sizes $t_1$ and $t_2$ with $t_2\leq t_1\leq 2t_2$. Let $G$ be a graph with with a partition $V(G)=V_1\cup\cdots\cup V_k$ satisfying $|V_1|=\cdots=|V_k|=m$. Let  $R_{\text{\emph{\bfseries EM2d}}}$ be a graph with vertex set $[k]$, such that if $ij\in E(R_{\text{\emph{\bfseries EM2d}}})$ then $G[V_i,V_j]$ is $(\eps,d)$-regular. Suppose there is a partition $[k]=I_A\cup I_B\cup I_C$ such that the following properties hold (see Figure~\ref{fig:EM2dproof}).
\stepcounter{propcounter}
\begin{enumerate}[label = \emph{\textbf{\Alph{propcounter}\arabic{enumi}}}]
    \item\labelinthm{lemma:em2d:1} $|\cup_{i\in I_A}V_i|\geq|\cup_{i\in I_B}V_i|\geq(1-\alpha) t_2, |\cup_{i\in I_A\cup I_B}V_i|\geq(1-\alpha)n$, and $|\cup_{i\in I_A\cup I_B\cup I_C}V_i|=(1+100\alpha)n$.
    \item\labelinthm{lemma:em2d:2} $R_{\text{\emph{\bfseries EM2d}}}[I_A]$ and $R_{\text{\emph{\bfseries EM2d}}}[I_B]$ are both $\eta$-almost complete graphs.
    \item\labelinthm{lemma:em2d:3} $R_{\text{\emph{\bfseries EM2d}}}[I_A,I_C]$ contains a matching $M$ covering $I_C$, and a vertex in $I_A\cap V(M)$ that is adjacent to every vertex in $I_C\cap V(M)$.
    \item\labelinthm{lemma:em2d:4} $R_{\text{\emph{\bfseries EM2d}}}[I_B,I_C]$ contains a matching $M'$ covering $I_C$.
\end{enumerate}
Then, $G$ contains a copy of $T$.
\end{lemma}
\begin{proof}
Let $\xi$ satisfy $c\ll\xi\ll 1/k$. Let $S$ be the graph defined in Lemma~\ref{lem:maintreedecomp}. Using that lemma, take a homomorphism $\phi:T\to S$ such that each component of $T-\phi^{-1}(X_0\cup Y_0)$ has size at most $\xi n$ and $|\phi^{-1}(X_0\cup Y_0\cup X_1\cup Y_1)|\leq \xi n$. Without loss of generality, say $|\phi^{-1}(X_0\cup X_1\cup X_2\cup X_3)|=t_1$ and $|\phi^{-1}(Y_0\cup Y_1\cup Y_2\cup Y_3)|=t_2$. Let the components of $T-\phi^{-1}(X_0\cup Y_0)$ be $\{K_j:j\in J\}$, and partition $J$ as $J_X\cup J_Y$, so that $N_T(K_j)\subset\phi^{-1}(X_0)$ for each $j\in J_X$, and $N_T(K_j)\subset\phi^{-1}(Y_0)$ for each $j\in J_Y$.   

Let $\tau_X=|\phi^{-1}(X_2\cup Y_3)|$ and $\tau_Y=|\phi^{-1}(Y_2\cup X_3)|$. For each $\ast\in\{A,B,C\}$, let $n_\ast=|\cup_{i\in I_\ast}V_i|$. Let $R_{\text{\textbf{EM2d}}}'$ be the graph depicted on the right in Figure~\ref{fig:EM2dproof}, which is a subgraph of the graph $R'$ in Lemma~\ref{lem:maintechnicalembedding}. Since \[\tau_X+\tau_Y\leq n\leq n_A+n_B+2n_C-20\alpha n\leq2n_A+2n_C-20\alpha n,\] we can separate into the following two cases. \textbf{I:} $\tau_X\leq n_A+n_C-10\alpha n$. \textbf{II:} $\tau_Y\leq n_A+n_C-10\alpha n$. 

\medskip

\noindent \textbf{Case I.} As $\tau_X+\tau_Y\geq n-\xi n\geq n_A+n_C-10\alpha n$, there exists $p\in[0,1]$ such that $\tau_X+p\tau_Y=n_A+n_C-10\alpha n$. It follows that $(1-p)\tau_Y\leq n-\tau_X-p\tau_Y\leq n_B-50\alpha n$. Define a random homomorphism $\varphi:T\to R_{\text{\textbf{EM2d}}}'$ as follows. First, for vertices in $\phi^{-1}(X_0\cup Y_0)$, set $\varphi(v)=i_1$ if $\phi(v)=X_0$, and set $\varphi(v)=i_2$ if $\phi(v)=Y_0$. For each $j\in J_X$, define $\varphi$ on $K_j$ by composing $\phi$ with the map that sends $Y_1,X_2,Y_3$ to $I_{1,1},I_{1,2},I_{1,1}$, respectively. For each $j\in J_Y$ independently at random, with probability $p$, define $\varphi$ on $K_j$ by composing $\phi$ with the map that sends $X_1,Y_2,X_3$ to $i_1,I_{1,1},I_{1,2}$, respectively; and with probability $1-p$ define $\varphi$ on $K_j$ by composing $\phi$ with the map that sends $X_1,Y_2,X_3$ to $i_3,I_{3,1},I_{3,2}$, respectively.

By Lemma~\ref{lemma:mcdiarmid}, with positive probability, $|\varphi^{-1}(I_{1,1}\cup I_{1,2})|=\tau_X+p\tau_Y\pm\alpha n\leq n_A+n_C-5\alpha n$ and $|\varphi^{-1}(I_{3,1}\cup I_{3,2})|=(1-p)\tau_Y\pm\alpha n\leq n_B-10\alpha n$. 

We now further transform the structure as depicted on the left of Figure~\ref{fig:EM2dproof}. Let $i_1\in I_A\cap V(M)$ be the vertex given by~\ref{lemma:em2d:3} that is adjacent to every vertex in $I_C\cap V(M)$, suppose it is matched with $i_2\in I_C$ by $M$, and let $i_3\in I_B$ be the neighbour of $i_2$ in $M'$. By~\ref{lemma:em2d:2}, we can find a matching $M_1$ in $R_{\textbf{EM2d}}[I_A\cup I_C]$ containing the matching $M-\{i_1,i_2\}$, such that $i_1$ is adjacent to every $i\in V(M_1)$ and $\sum_{i\in V(M_1)}|V_i|=n_A+n_C-\alpha n$. Since $|\varphi^{-1}(I_{1,1}\cup I_{1,2})|+4\alpha n\leq n_A+n_C-\alpha n$ from above, we can find $m_1,m_2\geq\sqrt\eps m$ such that $m_1+m_2=m$ and $|\varphi^{-1}(I_{1,b})|+\alpha n\leq(n_A+n_C-\alpha n)m_b/m$ for each $b\in[2]$. For each $i\in V(M_1)$, partition $V_i$ as $V_{i,1}\cup V_{i,2}$, such that $|V_{i,1}|=m_1$ and $|V_{i,2}|=m_2$. Then, for every edge $ii'$ in $M_1$, both $G[V_{i,1},V_{i',2}]$ and $G[V_{i,2},V_{i',1}]$ are $(\sqrt\eps,d-\eps)$-regular by Lemma~\ref{lemma:regularity:1}. Relabel $\{V_{i,1}:i\in V(M_1)\}$ and $\{V_{i,2}:i\in V(M_1)\}$ as $\{V_i':i\in I_{1,1}\}$ and $\{V_i':i\in I_{1,2}\}$, respectively. Note that $\sum_{i\in I_{1,1}}|V_i'|\geq|\varphi^{-1}(I_{1,1})|+\alpha n$ and $\sum_{i\in I_{1,2}}|V_i'|\geq|\varphi^{-1}(I_{1,2})|+\alpha n$.

Similarly, we can use~\ref{lemma:em2d:2} to find a matching $M_2$ in $R_{\textbf{EM2d}}[I_B]$ with all vertices in $V(M_2)$ adjacent to $i_3$, then refine them accordingly to obtain clusters $\{V_i':i\in I_{3,1}\}$ and $\{V_i':i\in I_{3,2}\}$ that can be matched together to form $(\sqrt\eps,d-\eps)$-regular pairs, such that $\sum_{i\in I_{3,1}}|V_i'|\geq|\varphi^{-1}(I_{3,1})|+\alpha n$ and $\sum_{i\in I_{3,2}}|V_i'|\geq|\varphi^{-1}(I_{3,2})|+\alpha n$. Thus, we can apply Lemma~\ref{lem:maintechnicalembedding} to find a copy of $T$ in $G$.

\medskip

\noindent \textbf{Case II.} In this case, we use essentially the same argument except that the role of $X$ and $Y$ are flipped, so we omit the full details. That is, we set $\varphi(v)=i_1$ if $\phi(v)=Y_0$, and $\varphi(v)=i_2$ if $\phi(v)=X_0$. Then, every component $K_j$ with $j\in J_Y$ will be embedded into $I_A\cup I_C$, while every component $K_j$ with $j\in J_X$ will be embedded into $I_A\cup I_C$ with some suitable probability $p$, and into $I_B$ otherwise. 
\end{proof}




\section{Proof of Theorem~\ref{thm:stability}: Stability}\label{sec:stages}
In this section, we will prove Theorem~\ref{thm:stability} by following the 4-stage process outlined in Section~\ref{sec:outline:stability} and depicted in Figure~\ref{fig:stages}, using the embedding methods developed in Section~\ref{sec:embedregularity}. As will be justified when we put everything together in Section~\ref{sec:stabmain}, we may assume that $t_1\leq 2t_2$ throughout these 4 stages. Our starting point is the next result that easily follows from the proof of {\cite[Theorem~3]{haxell2002ramsey}} by Haxell, \L uczak, and Tingley applied with $\alpha=t_1/t_2$ and $n=(1-\eps)(t_1+2t_2)$. The reason it follows from their proof rather than directly from their theorem statement is that we need to `remember', and later use, the remaining regularity clusters in the graph that is not part of the structure found by Haxell, \L uczak, and Tingley, and thus not included in the statement of their theorem.
\begin{theorem}\label{thm:hltstart}
   Let $1/n\ll 1/m\ll1/k\ll\eps\ll1$. Let $t_1$ and $t_2$ satisfy $t_1+t_2=n$ and $t_2\leq t_1\leq 2t_2$.
   Let $G$ be any red/blue coloured complete graph on at least $(1-\eps)(t_1+2t_2)$ vertices. Then, there exist disjoint subsets $V_0,\ldots,V_k\subset V(G)$ that form an $\eps$-regular partition with corresponding red/blue coloured $(\eps,1/3)$-reduced graph $R$, a partition $[k]=I_A\cup I_B\cup I_C$, and a colour $\ast\in\{\text{red},\text{blue}\}$ such that the following hold (see \emph{\textbf{A-situation}} in Figure~\ref{fig:stages}).
\begin{itemize}
    \item $|V_i|=m$ for every $i\in\{0\}\cup I_A\cup I_C$, and $|V_i|=t_1m/t_2$ for every $i\in I_B$.
    \item $|I_A|=|I_B|=|I_C|=k$, with $km\geq(1-2\eps)t_2$.
    \item In $R_\ast$, 0 is adjacent to every $a\in I_A$. 
    \item $\red{R}[I_A,I_B]$ contains a perfect matching.
\end{itemize}
\end{theorem}

To reduce repetition, we will carry out the 4 stages in reverse order in Sections~\ref{sec:stage4}--\ref{sec:stage1}. That is, we start with Lemma~\ref{lem:stage4} in Section~\ref{sec:stage4} carrying out \textbf{Stage 4}, which states that given a \textbf{D-situation} obtained in an earlier stage, we can either find a monochromatic copy of $T$ or the reduced graph is extremal (\textbf{E-situation}). Then, we similarly proceed through the other stages in reverse order, ending with Lemma~\ref{lem:stage1} in Section~\ref{sec:stage1} carrying out \textbf{Stage 1}, which says that from the \textbf{A-situation} structure given by Theorem~\ref{thm:hltstart}, we can either find a monochromatic copy of $T$, or, combining with known results about the later stages, conclude that the reduced graph is extremal (\textbf{E-situation}).

To finish the proof, we still need two further results. In Section~\ref{sec:extpart}, we formalise what it means for a reduced graph to be extremal, and show in Lemma~\ref{lemma:regtoext} that this implies the original graph is extremal as well in the sense of Definition~\ref{def:extremality}. The second is a technical `cascading lemma' about maximum matchings, which is needed in \textbf{Stage 2} and proved separately in Section~\ref{sec:cascade}.

\subsection{Extremal regular partitions}\label{sec:extpart}

Recall from Definition~\ref{def:extremality} what it means for a red/blue coloured complete graph to be Type I $(\mu,t_1,t_2)$-extremal or Type II $(\mu,t_1,t_2)$-extremal, and that if it is extremal of either type then we say it is $(\mu,t_1,t_2)$-extremal. Now, we similarly define a notion of being extremal for the reduced graph of a regular partition. 


\begin{definition}\label{reduceextremality}
Let $1/n\ll1/k\ll\eps\ll1$, let $d,\mu\in [0,1]$, and let $t_1,t_2\in\mathbb{N}$ satisfy $t_1+t_2=n$. We say a red/blue-coloured graph $G$ has a \textit{Type I $(\mu,t_1,t_2)$-extremal $(\eps,d)$-regular partition} if there exists an $\eps$-regular partition $V_1\cup\cdots\cup V_k$ in $V(G)$ with corresponding $(\eps,d)$-reduced graph $R$, and a partition $[k]=I_A\cup I_B$ such that the following hold.
\begin{itemize}
    \item $|V_i|=m$ for every $i\in[k]$.
    \item $|I_A|\geq(1-\mu)n/m$ and $|I_B|\geq(1-\mu)t_2/m$.
    \item Both $\red{R}[I_A]$ and $\blue{R}[I_A,I_B]$ are $\mu$-almost empty, or both $\blue{R}[I_A]$ and $\red{R}[I_A,I_B]$ are $\mu$-almost empty.
\end{itemize}
We say $G$ has a \textit{Type II $(\mu,t_1,t_2)$-extremal $(\eps,d)$-regular partition} if there exists an $\eps$-regular partition $V_1\cup\cdots\cup V_k$ in $V(G)$ with corresponding $(\eps,d)$-reduced graph $R$, and a partition $[k]=I_A\cup I_B$ such that the following hold.
\begin{itemize}
    \item $|V_i|=m$ for every $i\in[k]$.
    \item $|I_A|,|I_B|\geq(1-\mu)t_1/m$.
    \item All of $\red{R}[I_A]$, $\red{R}[I_B]$, and $\blue{R}[I_A,I_B]$ are $\mu$-almost empty, or all of $\blue{R}[I_A]$, $\blue{R}[I_B]$, and $\red{R}[I_A,I_B]$ are $\mu$-almost empty.
\end{itemize}
\end{definition}

With suitable choices of constants, if a red/blue coloured complete graph $G$ has a Type I or Type II $(\mu',t_1,t_2)$-extremal $(\eps,d)$-regular partition, then it must be $(\mu,t_1,t_2)$-extremal, by the following lemma.
\begin{lemma}\label{lemma:regtoext}
Let $1/n\ll1/k\ll\eps\ll\mu'\ll d\ll\mu\ll1$, and let $t_1,t_2\in\mathbb{N}$ satisfy $t_1+t_2=n$. If $G$ is a red/blue coloured complete graph on at most $2n$ vertices that contains a Type I or Type II $(\mu',t_1,t_2)$-extremal $(\eps,d)$-regular partition, then $G$ is Type I or Type II $(\mu,t_1,t_2)$-extremal, respectively.
\end{lemma}
\begin{proof}
Suppose $G$ contains a Type I $(\mu',t_1,t_2)$-extremal $(\eps,d)$-regular partition $V_1\cup\cdots\cup V_k$ satisfying the conditions in Definition~\ref{reduceextremality}, where without loss of generality we assume that both $\red{R}[I_A]$ and $\blue{R}[I_A,I_B]$ are $\mu'$-almost empty. Let $V_A=\cup_{i\in I_A}V_i$ and $V_B=\cup_{i\in I_B}V_i$. From assumptions and using $km\leq|G|\leq 2n$, the number of red edges in $V_A$ is at most $\mu'|I_A|^2m^2+d|I_A|^2m^2+|I_A|m^2\leq\mu'k^2m^2+dk^2m^2+km^2\leq4(\mu'+d+1/k)n^2\leq5dn^2$. Therefore, by removing at most $\sqrt{5d} n$ vertices from $V_A$, we can obtain a subset $V_A'$ such that $\red{d}(u,V_A')\leq\sqrt{5d}n$ for every $u\in V_A'$.

Similarly, the number of blue edges between $V_A'$ and $V_B$ is at most $\mu'|I_A||I_B|m^2+d|I_A||I_B|m^2\leq4(\mu'+d)n^2\leq 5dn^2$. Thus, we can remove at most $\sqrt{5d} n$ vertices from each of $V_A'$ and $V_B$ to obtain subsets $U_1$ and $U_2$, respectively, such that for every $i\in[2]$ and every $u\in U_i$, $\blue{d}(u,U_{3-i})\leq\sqrt{5d} n$. Since $\mu\gg\sqrt{d}$, $U_1$ and $U_2$ show that $G$ is Type I $(\mu,t_1,t_2)$-extremal.

The case when $G$ contains a Type II $(\mu',t_1,t_2)$-extremal $(\eps,d)$-regular partition is similar and thus omitted. 
\end{proof}


\subsection{Stage 4}\label{sec:stage4}

\begin{lemma}[\textbf{Stage 4}]\label{lem:stage4}
Let $1/n\ll 1/m\ll c\ll1/k\ll\eps\ll\eta\ll\alpha\ll d\ll\mu\leq1$. Let $T$ be an $n$-vertex tree with $\Delta(T)\le cn$ and bipartition classes of sizes $t_1$ and $t_2$ with $t_2\leq t_1\leq2t_2$. Let $G$ be a red/blue coloured graph that contains a coloured $\eps$-regular partition $V_1\cup\cdots\cup V_k$ with $|V_1|=\cdots=|V_k|=m$ and corresponding red/blue coloured $(\eps,d)$-reduced graph $R$. Suppose there is a partition $[k]=I_A\cup I_B$ such that the following hold (see \textbf{\emph{D-situation}} in Figure~\ref{fig:stages}).
\stepcounter{propcounter}
\begin{enumerate}[label = \emph{\textbf{\Alph{propcounter}\arabic{enumi}}}]
    \item\labelinthm{lemma:stage4:1} $|I_A|=k_1$ and $|I_B|=k_2$, with $k_1m,k_2m\geq(1-\alpha)t_2$ and $(k_1+k_2)m\geq(1-\alpha)(n+t_2)$.
    \item\labelinthm{lemma:stage4:2} $\red{R}[I_A,I_B]$ is $\eta$-almost complete.
\end{enumerate}
Then, $G$ contains a monochromatic copy of $T$, or $G$ is Type I $(\mu,t_1,t_2)$-extremal, or $t_1\geq(2-\mu)t_2$ and $G$ is Type II $(\mu,t_1,t_2)$-extremal.
\end{lemma}
\begin{proof}
Without loss of generality, assume that $k_1\geq k_2$. Let $A=\cup _{i\in I_A}V_i$ and $B=\cup_{i\in I_B}V_i$. Note that from \ref{lemma:stage4:1}, $|A|+|B|\geq(1+200\alpha)n$ and $|A|=k_1m\geq(1-\alpha)(n+t_2)/2\geq t_2+200\alpha n$.

\medskip

\noindent\textbf{Case I.} $k_2m\geq t_2+100\alpha n$. If there exists an edge $ij$ in $\red{R}[I_A]$, then we can use~\ref{lemma:stage4:2} to greedily find a perfect matching between some $I_A'\subset I_A\setminus\{i,j\}$ of size $100\alpha n/m$ and some $I_B'\subset I_B$. This allows us to apply Lemma~\ref{lemma:em2a} (\textbf{EM2a}) to $I_A, I_B\setminus I_B'$, and $I_B'$ to embed $T$ in red. Similarly, we can use Lemma~\ref{lemma:em2b} (\textbf{EM2b}) to embed $T$ in red if there is an edge in $\red{R}[I_B]$. Thus, we may assume that both $\red{R}[I_A]$ and $\red{R}[I_B]$ are empty. This implies that both $\blue{R}[I_A]$ and $\blue{R}[I_B]$ contain at most $\eps k^2$ non-edges, so we can find $J_A\subset I_A$ and $J_B\subset I_B$, such that $|J_A|\geq(1-10\sqrt\eps)k_1, |J_B|\geq(1-10\sqrt\eps)k_2$, and both $\blue{R}[J_A], \blue{R}[J_B]$ are $10\sqrt\eps$-almost complete. 

If there is any edge $ib$ in $\blue{R}[J_A,J_B]$ with $i\in J_A$, then by moving $i$ out of $J_A$ and finding an arbitrary neighbour $a$ of $i$ in $J_A$, we get the structure required to apply Lemma~\ref{lemma:em3} (\textbf{EM2d}) to find a blue copy of $T$ in $G$. Therefore, we can assume that $\blue{R}[J_A,J_B]$ is empty.

If $k_1m\geq t_1+10\alpha n$, then after using Lemma~\ref{lemma:regularity:refine} to refine the clusters indexed by $J_A$ and $J_B$, we can embed $T$ in red using Lemma~\ref{lemma:hlt} (\textbf{H\L T}). If instead $k_1m<t_1+10\alpha n$, then it follows from $k_1m\geq(1-\alpha)(n+t_2)/2$ that $t_1\geq(2-100\alpha)t_2$. Also, we have $k_2m\geq(1-\alpha)(n+t_2)-k_1m\geq (2-50\alpha)t_2\geq(1-100\alpha)t_1$. Therefore, $G$ contains a Type II $(200\alpha,t_1,t_2)$-extremal $(\eps,d)$-regular partition, and so $G$ is Type II $(\mu,t_1,t_2)$-extremal by Lemma~\ref{lemma:regtoext}.

\medskip

\noindent\textbf{Case II.} $k_2m<t_2+100\alpha n$. Then $k_1m\geq(1-\alpha)(n+t_2)-k_2m\geq(1-102\alpha)n>t_2+500\alpha n$. 

If the maximum matching $M$ in $\red{R}[I_A]$ has size at least $101\alpha n/m$, then we can move one side of a submatching of $M$ with size $100\alpha n/m$ out of $I_A$ to obtain the structure needed to apply Lemma~\ref{lemma:em2a} (\textbf{EM2a}) to find a red copy of $T$. Otherwise, let $I_A'=I_A\setminus M$, $A'=\cup_{i\in I_A'}V_i$, and $k_1'=|I_A'|\geq k_1-202\alpha n/m$. Then, $\red{R}[I_A']$ is empty and so $\blue{R}[I_A']$ contains at most $\eps k^2$ non-edges. 

Note that $|A'|=k_1'm\geq(1-400\alpha)n$. If $|A'|=k_1'm\geq(1+5\alpha)n$, then we can easily find the structure to apply Lemma~\ref{lemma:em2c} (\textbf{EM2c}) to embed $T$ in $\blue{G}[A']$. Thus, we may assume that $|A'|\leq(1+5\alpha)n$, and so $|B|\geq(1-\alpha)(n+t_2)-|A'|-202\alpha n\geq(1-1000\alpha)t_2$. Let $\alpha\ll\beta\ll d$. If at least $2\beta n/m$ vertices in $I_A'$ have at least $2\beta n/m$ blue neighbours in $I_B$, then we can greedily find the structure to apply Lemma~\ref{lemma:em2c} (\textbf{EM2c}) to embed $T$ in $\blue{G}[A'\cup B]$. Otherwise, we can find $J_A\subset I_A'$ and $J_B\subset I_B$ such that $|J_A|\geq(1-3\beta)n/m$, $|J_B|\geq(1-10\sqrt\beta)t_2/m$, and $\blue{R}[J_A,J_B]$ is $10\sqrt\beta$-almost empty. This shows that $G$ contains a Type I $(10\sqrt\beta,t_1,t_2)$-extremal $(\eps,d)$-regular partition, so $G$ is Type I $(\mu,t_1,t_2)$-extremal by Lemma~\ref{lemma:regtoext}.
\end{proof}


\subsection{Stage 3}\label{sec:stage3}
\begin{lemma}[\textbf{Stage 3}]\label{lem:stage3}
Let $1/n\ll 1/m\ll c\ll1/k\ll\eps\ll\eta\ll\alpha\ll d\ll \mu\ll1$. Let $T$ be an $n$-vertex tree with $\Delta(T)\le cn$ and bipartition classes of sizes $t_1$ and $t_2$ with $t_2\leq t_1\leq 2t_2$.
Let $G$ be a red/blue coloured graph that contains a coloured $\eps$-regular partition $V_1\cup\cdots\cup V_k$ with $|V_1|=\cdots=|V_k|=m$ and corresponding blue/red coloured $(\eps,d)$-reduced graph $R$. Suppose there is a partition $[k]=I_A\cup I_B\cup I_C$ such that the following hold (see \textbf{\emph{C-situation}} in Figure~\ref{fig:stages}).
\stepcounter{propcounter}
\begin{enumerate}[label = \emph{\textbf{\Alph{propcounter}\arabic{enumi}}}]
    \item\labelinthm{lemma:stage3:1} $|I_A|=k_1$, $|I_B|=k_2$, and $|I_C|=k_3$, with $k_1m,k_2m,k_3m\geq(1-\alpha)t_2$ and $(k_1+k_2)m=(1-\alpha)n$.
    \item\labelinthm{lemma:stage3:2} $\red{R}[I_A,I_B]$ is $\eta$-almost complete.
\end{enumerate}
Then, $G$ contains a monochromatic copy of $T$, or $G$ is Type I $(\mu,t_1,t_2)$-extremal, or $t_1\geq(2-\mu)t_2$ and $G$ is Type II $(\mu,t_1,t_2)$-extremal.
\end{lemma}
\begin{proof}
Let $\alpha\leq\alpha'\ll\beta\ll d$. To avoid repetition, we first prove the following two claims dealing with two commonly occuring structures.

\begin{claim}\label{stage3claimA}
Suppose there exist disjoint $J_A,J_B,J_C\subset[k]$, such that the following hold (see Figure~\ref{fig:claimA}).
\stepcounter{propcounter}
\begin{enumerate}[label = \emph{\textbf{\Alph{propcounter}\arabic{enumi}}}]
    \item\labelinthm{lemma:stage3A:1} $|J_A|=(1-\alpha')t_1/m$ and $|J_B|=|J_C|=(1-\alpha')t_2/m$.
    \item\labelinthm{lemma:stage3A:2} $\red{R}[J_A,J_B]$ is $\eta$-almost complete.
    \item\labelinthm{lemma:stage3A:3} $\red{R}[J_B,J_C]$ contains a matching with size $100\alpha' t_2/m$.
\end{enumerate}
Then, $G$ contains a monochromatic copy of $T$, or $G$ is Type I $(\mu,t_1,t_2)$-extremal, or $t_1\geq(2-\mu)t_2$ and $G$ is Type II $(\mu,t_1,t_2)$-extremal.
\end{claim}
\begin{proof}[Proof of Claim~\ref{stage3claimA}]
If at least $100\alpha' t_2/m$ vertices in $J_A$ have at least $200\alpha' t_2/m$ red neighbours in $J_C$, then we can greedily find a matching of size $100\alpha' t_2/m$ in $\red{R}[J_A,J_C]$ disjoint from the matching given by~\ref{lemma:stage3A:3}. This allows us to apply Lemma~\ref{lemma:em2c} (\textbf{EM2c}) to find a red copy of $T$ in $G$. Otherwise, at most $100\alpha' t_2/m$ vertices in $J_A$ have at least $200\alpha' t_2/m$ red neighbours in $J_C$, so we can find $J_A'\subset J_A$ and $J_C'\subset J_C$ with $|J_A'|\geq(1-200\alpha')t_1/m$ and $|J_C'|\geq(1-20\sqrt{\alpha'})t_2/m$, such that $\blue{R}[J_A',J_C']$ is $20\sqrt{\alpha'}$-almost complete.

If at most $200\beta t_2/m$ vertices in $J_B$ have at least $200\beta t_2/m$ blue neighbours in $J_C'$, then we can find $J_B'\subset J_B$ and $J_C''\subset J_C'$, such that $|J_B'|\geq(1-300\beta)t_2/m$, $|J_C''|\geq(1-20\sqrt\beta)t_2/m$, and $\red{R}[J_A\cup J_C'',J_B']$ is $20\sqrt\beta$-almost complete. This gives the required structure (\textbf{D-situation}) in red to apply Lemma~\ref{lem:stage4} (\textbf{Stage 4}) with $\eta=20\sqrt\beta$ to finish the proof.

Thus, we may assume that at least $200\beta t_2/m$ vertices in $J_B$ have at least $200\beta t_2/m$ blue neighbours in $J_C'$, so we can find a blue matching of size $100\beta t_2/m$ in $\blue{R}[J_B,J_C']$ disjoint from the red matching given by~\ref{lemma:stage3A:3}. Arbitrarily pick a matching of size $10(\alpha'+\beta)t_2/m$ in $R[J_A]$. By pigeonhole, it either contains a red matching $\red{M}$ of size $10\alpha' t_2/m$ or a blue matching $\blue{M}$ of size $10\beta t_2/m$. In the former case, moving vertices on one side of $\red{M}$ out of $J_A$, and using~\ref{lemma:stage3A:2},~\ref{lemma:stage3A:3}, we have the structure to apply Lemma~\ref{lemma:em2c} (\textbf{EM2c}) to find a red copy of $T$ in $G$. In the latter case, moving vertices on one side of $\blue{M}$ out of $J_A$, and using that $\blue{R}[J_A,J_C']$ is $20\sqrt{\alpha'}$-almost complete, we can again apply Lemma~\ref{lemma:em2c} (\textbf{EM2c}) to find a blue copy of $T$ in $G$. 
\renewcommand{\qedsymbol}{$\boxdot$}
\end{proof}
\renewcommand{\qedsymbol}{$\square$}

\begin{claim}\label{stage3claimB}
Suppose there exist disjoint $J_A,J_B,J_C\subset[k]$, such that the following hold (see Figure~\ref{fig:claimB}).
\stepcounter{propcounter}
\begin{enumerate}[label = \emph{\textbf{\Alph{propcounter}\arabic{enumi}}}]
    \item\labelinthm{lemma:stage3B:1} $|J_A|\geq(t_2+200\alpha'n)/m$, $|J_B|\geq(t_2-\alpha'n)/m$, $|J_A|+|J_B|=(1-\alpha')n/m$, and $|J_C|=(1-\alpha')t_2/m$.
    \item\labelinthm{lemma:stage3B:2} $\red{R}[J_A,J_B]$ is $\eta$-almost complete.
    \item\labelinthm{lemma:stage3B:3} $\red{R}[J_A,J_C]$ contains a matching with size $100\alpha'n/m$.
\end{enumerate}
Then, $G$ contains a monochromatic copy of $T$, or $G$ is Type I $(\mu,t_1,t_2)$-extremal, or $t_1\geq(2-\mu)t_2$ and $G$ is Type II $(\mu,t_1,t_2)$-extremal.
\end{claim}
\begin{proof}[Proof of Claim~\ref{stage3claimB}]
If there is an edge in either $\red{R}[J_A]$ or $\red{R}[J_B]$, then we can apply Lemma~\ref{lemma:em2a} (\textbf{EM2a}) or Lemma~\ref{lemma:em2b} (\textbf{EM2b}), respectively, to find a red copy of $T$ in $G$. Thus, by removing at most $10\sqrt\eps$-fraction of vertices from $J_A$ and $J_B$, we may assume that $\blue{R}[J_A]$ and $\blue{R}[J_B]$ are both $10\sqrt\eps$-almost complete.

Suppose there exists a set of $10\beta t_2/m$ vertices in $J_C$, each of which has at least $10\beta t_2/m$ blue neighbours in both $J_A$ and $J_B$. Then, using a similar argument to Claim~\ref{claim:backandforth}, we can find a blue matching $M_1$ of size $\beta^2 t_2/m$ between $J_A$ and $J_C$ with a vertex in $V(M_1)\cap J_A$ adjacent to every vertex in $J_C':=V(M_1)\cap J_C$. Greedily, we can also find another blue matching between $J_C'$ and $J_B$ covering $J_C'$. This allows us to apply Lemma~\ref{lemma:em3} (\textbf{EM2d}) to find a blue copy of $T$ in $G$.

Thus, we may now assume that there exist two disjoint subsets $J_A^+,J_B^+\subset J_C$ containing all but at most $10\beta t_2/m$ vertices in $J_C$, such that every vertex in $J_A^+$ has at most $10\beta t_2/m$ blue neighbours in $J_B$, and every vertex in $J_B^+$ has at most $10\beta t_2/m$ blue neighbours in $J_A$. Using~\ref{lemma:stage3B:2}, and by removing at most $5\sqrt\beta t_2/m$ vertices from each of $J_A$ and $J_B$ to obtain $J_A^-$ and $J_B^-$, respectively, we can ensure that both $\red{R}[J_A^-,J_B^-\cup J_B^+]$ and $\red{R}[J_B^-,J_A^-\cup J_A^+]$ are $5\sqrt\beta$-almost complete. Note that if $|J_A^+|\leq10\beta t_2/m$, then $|J_A^-\cup J_B^-\cup J_B^+|\geq(1-20\beta)(n+t_2)/m$, so $J_A^-$ and $J_B^-\cup J_B^+$ form the required red structure (\textbf{D-situation}) to apply Lemma~\ref{lem:stage4} (\textbf{Stage 4}) with $\eta=5\sqrt\beta$ to finish the proof. Thus, we can assume that $|J_A^+|\geq10\beta t_2/m$, and similarly $|J_B^+|\geq10\beta t_2/m$.

As in the beginning of this proof, we may further assume that $\blue{R}[J_A^-\cup J_A^+]$ and $\blue{R}[J_B^-\cup J_B^+]$ are both $10\sqrt\eps$-almost complete, as the presence of any red edge within either of them allows us to apply Lemma~\ref{lemma:em2a} (\textbf{EM2a}) or Lemma~\ref{lemma:em2b} (\textbf{EM2b}) to find a red copy of $T$ in $G$. If there is any blue edge between $J_A^-\cup J_A^+$ and $J_B^-\cup J_B^+$, then we can move one end of this blue edge out and then find a blue copy of $T$ using Lemma~\ref{lemma:em3} (\textbf{EM2d}). Therefore, $\red{R}[J_A^-\cup J_A^+,J_B^-\cup J_B^+]$ contains the red structure (\textbf{D-situation}) needed to apply Lemma~\ref{lem:stage4} (\textbf{Stage 4}) to finish the proof.
\renewcommand{\qedsymbol}{$\boxdot$}
\end{proof}
\renewcommand{\qedsymbol}{$\square$}

Now we can carry out \textbf{Stage 3}. If at most $2\beta t_2/m$ vertices in $I_A\cup I_B$ have at least $2\beta t_2/m$ red neighbours in $I_C$, then we can find $I_{AB}\subset I_A\cup I_B$ and $I_C'\subset I_C$ with $|I_{AB}|\geq(1-2\beta)n/m$ and $|I_C'|\geq(1-5\sqrt\beta)t_2/m$, such that $\blue{G}[I_{AB}, I_C']$ is $5\sqrt\beta$-almost complete. Therefore, $I_{AB}$ and $I_C'$ form the blue structure (\textbf{D-situation}) required to apply Lemma~\ref{lem:stage4} (\textbf{Stage 4}) to finish the proof. Thus, we may assume that at least $2\beta t_2/m$ vertices in $I_A\cup I_B$ have at least $2\beta t_2/m$ red neighbours in $I_C$, from which it follows that there is a red matching of size $\beta t_2/m$ either between $I_A$ and $I_C$, or between $I_B$ and $I_C$.

\medskip

\noindent\textbf{Case I.} $|I_A|,|I_B|\geq(t_2+200\alpha n)/m$. Without loss of generality, assume that there is a red matching of size $\beta t_2/m$ between $I_A$ and $I_C$. Then, we can apply Claim~\ref{stage3claimB} to to finish the proof.

\medskip

\noindent\textbf{Case II.} Without loss of generality, assume that $|I_B|<(t_2+200\alpha n)/m$, so $|I_A|\geq(t_1-201\alpha n)/m\geq(1-500\alpha)t_1/m$ from~\ref{lemma:stage3:1}. If there is a red matching of size $\beta t_2/m$ between $I_B$ and $I_C$, then we can apply Claim~\ref{stage3claimA} with $\alpha'=500\alpha$ to finish the proof. 

Thus, we may assume that there is a red matching of size $\beta t_2/m$ between $I_A$ and $I_C$. If $t_1\leq(1+1000\alpha)t_2$, then $|I_B|\geq(1-\alpha)t_2/m\geq(1-2000\alpha)t_1/m$, so we can apply Claim~\ref{stage3claimA} with $\alpha'=2000\alpha$ to finish the proof. If instead $t_1>(1+1000\alpha)t_2$, then $|I_A|\geq(1-3\alpha)t_1/m\geq(t_2+200\alpha n)/m$, so we can apply Claim~\ref{stage3claimB} to finish the proof.
\end{proof}


\subsection{Stage 2}\label{sec:stage2}
\begin{lemma}[\textbf{Stage 2}]\label{lem:stage2}
Let $1/n\ll 1/m\ll c\ll1/k\ll\eps\ll\eta\ll\alpha\ll d\ll \mu\ll1$. Let $T$ be an $n$-vertex tree with $\Delta(T)\le cn$ and bipartition classes of sizes $t_1$ and $t_2$ with $t_2\leq t_1\leq 2t_2$.
Let $G$ be a red/blue coloured graph that contains a coloured $\eps$-regular partition $V_0\cup V_1\cup\cdots\cup V_k$ with $|V_0|=\cdots=|V_k|=m$ and corresponding red/blue coloured $(\eps,d)$-reduced graph $R$. Suppose there is a partition $[k]=I_A\cup I_B\cup I_C$ such that the following hold (see \textbf{\emph{B-situation}} in Figure~\ref{fig:stages}).
\stepcounter{propcounter}
\begin{enumerate}[label = \emph{\textbf{\Alph{propcounter}\arabic{enumi}}}]
    \item\labelinthm{lemma:stage2:1} $|I_A|=|I_B|=k_1$ with $k_1m=(1-\alpha)t_2$, and $|I_C|=k_2$ with $k_2m=(1-\alpha)t_1$.
    \item\labelinthm{lemma:stage2:2} In $\red{R}$, 0 is adjacent to every $i\in I_A\cup I_B$.
    \item\labelinthm{lemma:stage2:3} $\red{R}[I_A,I_B]$ is $\eta$-almost complete.
\end{enumerate}
Then, $G$ contains a monochromatic copy of $T$, or $G$ is Type I $(\mu,t_1,t_2)$-extremal, or $t_1\geq(2-\mu)t_2$ and $G$ is Type II $(\mu,t_1,t_2)$-extremal.
\end{lemma}
\begin{proof}
Let $M$ be a maximum red matching in $R$ between $I_A\cup I_B$ and $I_C$. For any $I\subset[k]$, for brevity we use $V(I)$ to denote $\cup_{i\in I}V_i$, the vertices covered by $I$. Let $\alpha\ll\beta\ll d$.

\medskip

\noindent\textbf{Case I.} $M$ contains at most $(t_1-t_2+2\beta n)/m$ edges. Let $X=(I_A\cup I_B)\cap V(M), X'=(I_A\cup I_B)\setminus V(M), Y=I_C\cap V(M)$, and $Y'=I_C\setminus V(M)$. By Lemma~\ref{lemma:improvemaximummatching}, there are partitions $X=X^+\cup X^-\cup \overline{X}$ and $Y=Y^+\cup Y^-\cup \overline{Y}$ such that $M$ matches $X^+$ with $Y^-$, $X^-$ with $Y^+$, and $\overline{X}$ with $\overline{Y}$, and $\red{R}[X'\cup X^-,Y'\cup Y^-\cup\overline{Y}]$ is empty. From assumption, $|X|=|Y|\leq(t_1-t_2+2\beta n)/m$, and the following hold.
\[|V(X'\cup X^-)|\geq|V(X')|\geq(1-\alpha)2t_2-(t_1-t_2+2\beta n)=(3-2\alpha-2\beta)t_2-(1+2\beta)t_1\geq(1-10\beta)t_2.\]
\[|V(Y'\cup Y^-\cup\overline{Y})|=|V(I_C)|-|V(Y^+)|\geq|V(I_C)|-|V(Y)|\geq(1-\alpha)t_1-(t_1-t_2+2\beta n)\geq(1-10\beta)t_2.\]
\[|V(X'\cup X^-\cup Y'\cup Y^-\cup\overline{Y})|=|V([k]\setminus X)|\geq (1-\alpha)(t_1+2t_2)-(t_1-t_2+2\beta n)\geq(1-10\beta)n.\]
Hence, $X'\cup X^-$ and $Y'\cup Y^-\cup\overline{Y}$ form the blue structure (\textbf{C-situation}) required to apply Lemma~\ref{lem:stage3} (\textbf{Stage 3}) to finish the proof.

\medskip

\noindent\textbf{Case II.} $M$ contains at least $(t_1-t_2+2\beta n)/m$ edges. Let $\eps\ll\gamma\ll\alpha$. Refine every cluster $V_i$ with $i\in I_A$ or $i\in I_B$ down to clusters of size $\gamma m$, and label these new clusters as $\{V_i':i\in J_A\}$ and $\{V_i':i\in J_B\}$, respectively. Refine every cluster $V_i$ with $i\in I_C$ down to clusters of size $\gamma t_1m/t_2$ and label the new clusters as $\{V_i':i\in J_C\}$. In this refinement process, at most $O(\gamma n)$ covered vertices are lost. Let $R'$ be the new reduced graph on $J_A\cup J_B\cup J_C$, where for each $\ast\in\{\text{red},\text{blue}\}$, $ij\in E(R_\ast')$ if and only if $G_\ast[V_i',V_j']$ is $(\sqrt\eps,d-\eps)$-regular.  Using~\ref{lemma:stage2:3}, Lemma~\ref{lemma:regularity:1}, and the matching $M$, we see that $\red{R'}[J_A,J_B]$ is also $\eta$-almost complete, and we can find a matching $M'$ in $\red{R'}[J_A\cup J_B,J_C]$ covering $(t_1-t_2+\beta n)t_2/t_1$ vertices in $\cup_{i\in J_A\cup J_B}V_i'$ and $t_1-t_2+\beta n$ vertices in $\cup_{i\in J_C}V_i'$.  

\medskip

\textbf{Case II.1.} $|V(M')\cap J_A|,|V(M')\cap J_B|\geq5\beta n/\gamma m$. Then, we can find disjoint subsets $J_{A,1},J_{A,2}\subset J_A\setminus V(M')$ and $J_{B,1},J_{B,2}\subset J_B\setminus V(M')$, such that $|J_{A,1}|/|J_{B,2}|=|J_{B,1}|/|J_{A,2}|\in[t_1/t_2,(1+\alpha)t_1/t_2]$, each of $J_{A,1},J_{A,2},J_{B,1},J_{B,2}$ covers at least $\alpha n$ vertices, and together with $V(M')\cap(J_A\cup J_B)$ they cover at least $(2-10\alpha)t_2$ vertices. Moreover, since $\red{R'}[J_A,J_B]$ is $\eta$-almost complete, by choosing the subsets above randomly and using Lemma~\ref{lemma:chernoff}, we can ensure that both $\red{R'}[J_{A,1},J_{B,2}]$ and $\red{R'}[J_{A,1},J_{B,2}]$ are $10\eta$-almost complete.
Therefore, for some $\eps\ll\gamma'\ll\gamma$, we can use Lemma~\ref{lemma:regularity:refine} to further appropriately refine these clusters along with those in $M'$ into two sets of smaller clusters of sizes $\gamma'm$ and $\gamma't_1m/t_2$, respectively, which can be paired up into $(\sqrt[4]\eps,d-2\sqrt\eps)$-regular pairs. Together, these new clusters cover at least \[(1-\alpha)(2-10\alpha)t_2+(1-\alpha)(t_1-t_2+\beta n)\geq(1+\beta/2)n\]
vertices. Finally, note that by~\ref{lemma:stage2:2}, each new cluster of size $\gamma'm$ forms an $(\sqrt\eps,d-\eps)$-regular pair with $V_0$, so we have the structure to apply Lemma~\ref{lemma:hlt} (\textbf{H\L T}) to find a red copy of $T$.

\medskip

\textbf{Case II.2.} One of $|V(M')\cap J_A|$ and $|V(M')\cap J_B|$ is at most $5\beta n/\gamma m$. Without loss of generality, assume it is $|V(M')\cap J_B|$, and assume that $M'$ is chosen so that $|V(M')\cap J_B|$ is maximised. Let $J_D$ be the vertices in $J_C$ matched by $M'$ to the vertices in $J_B$. Since $M'$ maximises $|V(M')\cap J_B|$, there is no red edge in $R'$ between $J_B\setminus V(M')$ and $J_C\setminus J_D$, as we can swap any such edge with an edge in $M'$ adjacent to $J_A$ to obtain a matching $M''$ with $|V(M'')\cap J_B|>|V(M')\cap J_B|$, a contradiction. Thus, $\blue{R'}[J_B\setminus V(M'),J_C\setminus J_D]$ contains at most $\eps(k+1)^2$ non-edges. Moreover, from $|V(M')\cap J_B|\leq5\beta n/\gamma m$, it follows that $J_B\setminus V(M')$ and $J_C\setminus J_D$ cover at least $(1-20\beta)t_2$ and $(1-20\beta)t_1$ vertices, respectively. Thus, after removing clusters with low degrees in $\blue{R'}[J_B\setminus V(M'),J_C\setminus J_D]$ and refining all remaining clusters down to a smaller common size, we get the required structure (\textbf{C-situation}) in blue to apply Lemma~\ref{lem:stage3} (\textbf{Stage 3}) to finish the proof.
\end{proof}


\subsection{Stage 1}\label{sec:stage1}

\begin{lemma}[\textbf{Stage 1}]\label{lem:stage1}
Let $1/n\ll 1/m\ll c\ll1/k\ll\eps\ll\alpha\ll d\ll\mu\ll1$. Let $T$ be an $n$-vertex tree with $\Delta(T)\le cn$ and bipartition classes of sizes $t_1$ and $t_2$ with $t_2\leq t_1\leq 2t_2$. Suppose $G$ is a red/blue coloured graph that contains an $\eps$-regular partition $V_0\cup V_1\cup\cdots\cup V_k$ with corresponding red/blue coloured $(\eps,d)$-reduced graph $R$, and a partition $[k]=I_A\cup I_B\cup I_C$ such that the following hold (see \textbf{\emph{A-situation}} in Figure~\ref{fig:stages}).
\stepcounter{propcounter}
\begin{enumerate}[label = \emph{\textbf{\Alph{propcounter}\arabic{enumi}}}]
    \item\labelinthm{lemma:stage1:1} $|V_i|=m$ for every $i\in\{0\}\cup I_A\cup I_C$, and $|V_i|=t_1m/t_2$ for every $i\in I_B$.
    \item\labelinthm{lemma:stage1:2} $|I_A|=|I_B|=|I_C|=k$, with $km=(1-\alpha)t_2$.
    \item\labelinthm{lemma:stage1:3} In $\red{R}$, 0 is adjacent to every $a\in I_A$.
    \item\labelinthm{lemma:stage1:4} $\red{R}[I_A,I_B]$ contains a perfect matching $M$.
\end{enumerate}
Then, $G$ contains a monochromatic copy of $T$, or $G$ is Type I $(\mu,t_1,t_2)$-extremal, or $t_1\geq(2-\mu)t_2$ and $G$ is Type II $(\mu,t_1,t_2)$-extremal.
\end{lemma}
\begin{proof}
Let $\alpha\ll\beta\ll d$. Let $X$ be the set of vertices in $I_A\cup I_B$ that have at least $20\beta|I_C|$ red neighbours in $I_C$.
Let $I_{B,1}=I_{B}\cap X$, and let $I_{A,1}$ be the vertices matched with $I_{B,1}$ by $M$. Let Let $I_{A,3}=(I_{A}\cap X)\setminus I_{A,1}$, and let $I_{B,3}$ be the vertices matched with $I_{A,3}$ by $M$. Let $I_{A,2}=I_{A}\setminus(I_{A,1}\cup I_{A,3}), I_{B,2}=I_{B}\setminus(I_{B,1}\cup I_{B,3})$, and note that $M$ gives a perfect matching between $I_{A,2}$ and $I_{B,2}$.

If $|I_{B,1}|\leq2\beta k$, then we can remove at most $5\sqrt{\beta}|I_C|$ vertices from $I_C$ to obtain $I_C'$, such that $\blue{R}[I_B\setminus I_{B,1}, I_C']$ is $5\sqrt{\beta}$-almost complete, both $I_B\setminus I_{B,1}$ and $I_C'$ cover at least $(1-6\sqrt\beta)t_2$ vertices, and they together cover at least $(1-6\sqrt\beta)n$ vertices. Thus, after refining all clusters down to a common smaller size, $I_B\setminus I_{B,1}$ and $I_C'$ provide the blue structure (\textbf{C-situation}) required to apply Lemma~\ref{lem:stage3} (\textbf{Stage 3}) with $\eta=5\sqrt\beta$ to finish the proof.

Now suppose $|I_{B,1}|\geq2\beta k$. In this case, assuming there is no red copy of $T$ in $G$, we make the following four deductions.

\medskip

\textbf{i)} $|I_{A,1}\cap X|\leq\beta|I_{A,1}|$. 

If not, let $I_{A,1}'\subset I_{A,1}\cap X$ have size $\beta^2k$, and let $I_{B,1}'\subset I_{B,1}$ be the vertices in $\red{R}$ matched with $I_{A,1}'$ by $M$. From the definition of $X$, we can find a perfect matching between $I_{A,1}'$ and some $I_{C,1}\subset I_C$, such that if $I_C'=I_C\setminus I_{C,1}$, then every vertex in $I_{B,1}'$ still has at least $10\beta^2|I_C'|$ neighbours in $I_C'$. After appropriate refinements, this structure allows us to apply Lemma~\ref{lemma:em1a} (\textbf{EM1a}) to find $T$ in red.

\medskip

\textbf{ii)} At most $\beta^2k$ vertices in $I_{B,3}$ have at least $\beta^2k$ red neighbours in $I_{A,1}$.

If not, we can find a perfect matching of size $\beta^2k$ between some $I_{B,3}'\subset I_{B,3}$ and $I_{A,1}'\subset I_{A,1}$. Let $I_{B,1}'\subset I_{B,1}$ be the vertices matched with $I_{A,1}'$ by $M$, and let $I_{A,3}'\subset I_{A,3}$ be the vertices matched with $I_{B,3}'$ by $M$. Since $I_{A,3}'\subset X$, we can find a perfect matching between $I_{A,3}'$ and some $I_{C,3}\subset I_C$, such that if $I_C'=I_C\setminus I_{C,3}$, then every vertex in $I_{B,1}'$ still has at least $10\beta^2|I_C'|$ neighbours in $I_C'$. After appropriate refinements, this structure allows us to apply Lemma~\ref{lemma:em1a} (\textbf{EM1a}) to find $T$ in red.

\medskip

\textbf{iii)} At most $\beta |I_{A,1}|$ vertices in $I_{A,1}$ have at least $\beta|I_{A,1}|$ red neighbours in $I_{A,1}$. 

If not, we can apply Lemma~\ref{lemma:em1b} (\textbf{EM1b}) to find $T$ in red.

\medskip

\textbf{iv)} At most $\beta k$ edges in $M[I_{A,2},I_{B,2}]$ satisfy that both of their endpoints have at least $\beta|I_{A,1}|$ red neighbours in $I_{A,1}$.

If not, we can apply Lemma~\ref{lemma:em1c} (\textbf{EM1c}) to find $T$ in red.

\medskip

From \textbf{iv)}, we can find a subset $J_{AB,2}\subset I_{A,2}\cup I_{B,2}$ containing a vertex in all but at most $\beta k$ edges in $M[I_{A,2},I_{B,2}]$, so that every vertex in $J_{AB,2}$ has at most $\beta|I_{A,1}|$ red neighbours in $I_{A,1}$. Similarly, using \textbf{i)}, \textbf{ii)}, and \textbf{iii)}, we can find $J_{B,3}\subset I_{B,3}$ containing all but at most $\beta^2k$ vertices in $I_{B,3}$, and $J_{A,1}\subset I_{A,1}\setminus X$ containing all but at most $2\beta|I_{A,1}|$ vertices in $I_{A,1}$, such that every vertex in $J:=J_{AB,2}\cup J_{B,3}\cup J_{A,1}$ has at most $2\beta|J_{A,1}|$ red neighbours in $J_{A,1}$, and at most $20\beta|I_C|$ red neighbours in $I_C$. In particular, by averaging and using that $J_{A,1}$ is non-empty, we can find some $j\in J_{A,1}$ with at most $2\beta|J|$ red neighbours in $J$ and at most $20\beta|I_C|$ red neighbours in $I_C$. Also, we have 
\[|J|\geq(1-2\beta)|I_{A,1}|+|I_{A,2}|-\beta k+|I_{A,3}|-\beta^2k\geq(1-4\beta)|I_A|\geq(1-5\beta)t_2/m.\] 
Therefore, we can find $J'\subset N_{\blue{R}}(j,J)$ and $I_C'\subset N_{\blue{R}}(j,I_C)$ with $|J'|\geq(1-2\beta)|J|\geq(1-10\beta)t_2/m$ and $|I_C'|\geq(1-20\beta-\sqrt{20\beta})|I_C|\geq(1-5\sqrt\beta)t_2/m$, such that $\blue{R}[J',I_C']$ is $5\sqrt{\beta}$-almost complete. Therefore, after refining all clusters down to a smaller common size, $J'$ and $I_C'$ provide the blue structure (\textbf{B-situation}) required to apply Lemma~\ref{lem:stage2} (\textbf{Stage 2}) with $\eta=5\sqrt\beta$ to finish the proof.
\end{proof}

\subsection{Cascading lemma}\label{sec:cascade}
\begin{lemma}[Cascading lemma]\label{lemma:improvemaximummatching}
Let $G$ be a bipartite graph with bipartition classes $A$ and $B$, and let $M$ be a maximum matching in $G$. Let $A_M=A\cap V(M)$, $B_M=B\cap V(M)$, $A'=A\setminus A_M$, and $B'=B\setminus B_M$. Then, $A_M$ and $B_M$ can be partitioned as $A_M=A^+\cup A^-\cup \overline{A}$ and $B_M=B^+\cup B^-\cup \overline{B}$ such that the following hold.
\begin{itemize}
    \item $M$ matches vertices in $A^+$ with vertices in $B^-$, vertices in $A^-$ with vertices in $B^+$, and vertices in $\overline{A}$ with vertices in $\overline{B}$.
    \item $G[A'\cup A^-,B'\cup B^-\cup\overline{B}]$ and $G[A'\cup A^-\cup\overline{A},B'\cup B^-]$ are both empty graphs.
\end{itemize}
\end{lemma}
\begin{proof}
First, by maximality of $M$, $G[A',B']$ must be an empty graph. Consider the following process. To initialise, set $A_0^+=A_0^-=B_0^+=B_0^-=\emptyset$, $\overline{A}_0=A_M$, and $\overline{B}_0=B_M$. Throughout this process, we will maintain the following conditions.
\stepcounter{propcounter}
\begin{enumerate}[label = \textbf{\Alph{propcounter}\arabic{enumi}}]
    \item\label{cascade:1} $A^+_i\cup A^-_i\cup\overline{A}_i$ is a partition of $A_M$ and $B^+_i\cup B^-_i\cup\overline{B}_i$ is a partition of $B_M$.
    \item\label{cascade:2} $M$ matches vertices in $A^+_i$ with vertices in $B^-_i$, vertices in $A^-_i$ with vertices in $B^+_i$, and vertices in $\overline{A}_i$ with vertices in $\overline{B}_i$.
    \item\label{cascade:3} $G[A'\cup A_i^-,B'\cup B^-_i]$ is the empty graph.
    \item\label{cascade:4} For every $a\in A'\cup A^-_i$, there exists an $M$-alternating path, possibly of length 0, that connects $a$ to some vertex $a'\in A'$, starts with an edge in $M$, and has all internal vertices in $A^-_i\cup B^+_i$. For every $b\in B'\cup B^-_i$, there exists an $M$-alternating path, possibly of length 0, that connects $b$ to some vertex $b'\in B'$, starts with an edge in $M$, and has all internal vertices in $B^-_i\cup A^+_i$.
\end{enumerate}
Note that~\ref{cascade:1}--\ref{cascade:4} are all satisfied when $i=0$. Suppose for some $i\geq 0$ we have found $A^+_i,A^-_i,\overline{A}_i$ and $B^+_i,B^-_i,\overline{B}_i$ satisfying \ref{cascade:1}--\ref{cascade:4}. If $G[\overline{A}_i,B'\cup B^-_i]$ and $G[A'\cup A^-_i,\overline{B}_i]$ are both empty graphs, then we are done by letting $A^+=A^+_i, A^-=A^-_i, \overline{A}=\overline{A}_i, B^+=B^+_i, B^-=B^-_i$, and $\overline{B}=\overline{B}_i$. Otherwise, we show that we can continue this process.

Indeed, suppose there exists $a\in \overline{A}_i$ adjacent to some $b_1\in B'\cup B^-_i$. Let $b\in\overline{B}_i$ be the vertex matched to $a$ by $M$. We claim that $b$ is not adjacent to any vertex in $A'\cup A^-_i$. Indeed, suppose $b$ is adjacent to $a_1\in A'\cup A^-_i$. Then by~\ref{cascade:4}, there exists an $M$-alternating path $P_1$ starting with an edge not in $M$ connecting some vertex $a_1'\in A'$ to $a_1$, with its internal vertices in $A^-_i\cup B^+_i$. Also by~\ref{cascade:4}, there exists some $M$-alternating path $P_2$ starting with an edge in $M$ connecting $b_1$ to some vertex $b_1'\in B'$, with its interval vertices in $B^-_i\cup A^+_i$. In particular, $P_1$ and $P_2$ are disjoint. Then, the path $P_1a_1bab_1P_2$ is an $M$-augmenting path beginning and ending with edges not in $M$ connecting $a'\in A'$ and $b'\in B'$, contradicting that $M$ is a maximum matching. Let $A^+_{i+1}=A^+_i\cup\{a\}, \overline{A}_{i+1}=\overline{A}_i\setminus\{a\}, B^-_{i+1}=B^-_i\cup\{b\}, \overline{B}_{i+1}=\overline{B}_i\setminus\{b\}$, and keep $A^-_{i+1},B^+_{i+1}$ unchanged. Then $G[A'\cup A^-_{i+1},B'\cup B^-_{i+1}]$ is still an empty graph, and $bab_1P_2$ is an $M$-alternating path starting with an edge in $M$ connecting $b$ to $b_1'\in B'$, with internal vertices in $B^-_{i+1}\cup A^+_{i+1}$. It follows that~\ref{cascade:1}--\ref{cascade:4} are all maintained.

If instead there exists $b\in\overline{B}_i$ adjacent to some $a_1\in A'\cup A^-_i$, then let $a\in\overline{A}_i$ be the vertex matched to $b$ in $M$, set $A^-_{i+1}=A^-_i\cup\{a\}, \overline{A}_{i+1}=\overline{A}_i\setminus\{a\}, B^+_{i+1}=B^+_i\cup\{b\}, \overline{B}_{i+1}=\overline{B}_i\setminus\{b\}$, and keep $A^+_{i+1}, B^-_{i+1}$ unchanged. Like above, \ref{cascade:1}--\ref{cascade:4} are all maintained, which finishes the proof.
\end{proof}

\subsection{Proof of Theorem~\ref{thm:stability}}\label{sec:stabmain}
\begin{proof}[Proof of Theorem~\ref{thm:stability}]
Let $1/n\ll c\ll\eps\ll\mu \ll 1$ and let $t_1,t_2\in \mathbb{N}$ satisfy $t_1+t_2=n$ and $t_1\geq t_2$. Let $G$ be a red/blue coloured complete graph with $\max\{t_1+2t_2,2t_1\}-1$ vertices. Suppose that $G$ is not Type I $(\mu,t_1,t_2)$-extremal, and either $t_1<(2-\mu)t_2$ or $G$ is not Type II $(\mu,t_1,t_2)$-extremal. Let $T$ be any $n$-vertex tree with $\Delta(T)\le cn$ and bipartition class sizes $t_1$ and $t_2$, we need to show that $G$ contains a monochromatic copy of $T$.

If $t_1\leq 2t_2$, let $T'=T$, $t_1'=t_1$ and $t_2'=t_2$. If $t_1\geq 2t_2+1$, then by Lemma~\ref{lemma:leaves:V1}, $T$ contains a set $L$ of $2/c$ leaves in $V_1$. Let $T'$ be a new tree obtained by attaching $\ceil{(t_1-2t_2)/2}$ new leaves to vertices in $L$, with no vertex in $L$ receiving more than $cn$ of these new leaves. Then, observe that $\Delta(T')\leq c|T'|$, and the bipartition classes of $T'$ have sizes $t_1'=t_1$ and $t_2'=\ceil{t_1/2}$. It follows that $t_1'\leq 2t_2'$ and $|G|=2t_1-1=2t_1'-1\geq t_1'+2t_2'-2\geq(1-\eps)(t_1'+2t_2')$. Therefore, Theorem~\ref{thm:hltstart} implies that we have the structure required to apply Lemma~\ref{lem:stage1}, which then implies that one of the following is true.  
\stepcounter{propcounter}
\begin{enumerate}[label = \textbf{\Alph{propcounter}\arabic{enumi}}]
    \item\label{2.2:a} $G$ contains a monochromatic copy of $T'$.
    \item\label{2.2:b} $G$ is Type I $(\mu/2,t_1',t_2')$-extremal.
    \item\label{2.2:c} $t_1'\geq(2-\mu/2)t_2'$ and $G$ is Type II $(\mu/2,t_1',t_2')$-extremal.
\end{enumerate}
Note that each of~\ref{2.2:b} and~\ref{2.2:c} implies that the same is true with $t_1,t_2,\mu$ in place of $t_1',t_2',\mu/2$, respectively, which contradicts our assumption. Therefore, it must be that~\ref{2.2:a} is true, which implies that $G$ contains a monochromatic copy of $T$ as well, finishing the proof.
\end{proof}

\section{Proof of Theorem~\ref{theorem:extremal:case1}: Type I extremal graphs}\label{sec:extremearg1}
In this section, we  prove Theorem~\ref{theorem:extremal:case1}. We will start by outlining the proof in Section~\ref{sec:outline:extcase1}, breaking it down into different cases that are then proved throughout the rest of the section. 


\subsection{Proof outline for Type I extremal graphs}\label{sec:outline:extcase1}
We start by recapping the situation in Theorem~\ref{theorem:extremal:case1}, where we have parameters $1/n\ll c\ll \mu\ll 1$. Let $T$ be an $n$-vertex tree with $\Delta(T)\le cn$ and bipartition classes of sizes $t_1$ and $t_2$ satisfying $t_1\geq t_2$.
Let $G$ be a red/blue coloured complete graph on $\max\{2t_1,t_1+2t_2\}-1$ vertices which is Type I $(\mu,t_1,t_2)$-extremal. This means that there are disjoint sets $U_1,U_2\subset V(G)$ such that $|U_1|\geq(1-\mu)n$, $|U_2|\geq(1-\mu)t_2$, $\red{d}(u,U_1)\leq \mu n$ for every $u\in U_1$, and $\blue{d}(u,U_{3-i})\leq \mu n$ for every $i\in[2]$ and $u\in U_i$. We wish to find a monochromatic copy of $T$ in $G$. As mentioned at the end of Section~\ref{sec:division}, we can and will assume that $t_1\leq 2t_2+1$, so $|G|\in\{n+t_2-1,n+t_2\}$.

Let us first partition the vertices outside of $U_1\cup U_2$. For some $\beta$ with $\mu \ll \beta\ll 1$, partition $V(G)=U_1^+\cup U_2^+$ so that vertices in $U_2^+$ have at least $\beta n$ red neighbours in $U_1$ and vertices in $U_1^+$ do not. In particular, $U_1\subset U_1^+$, $U_2\subset U_2^+$, and at least one of $|U_1^+|\geq n$ and $|U_2^+|\geq t_2$ must hold. Say $|U_1^+|=n+k$ for some $k$ satisfying $|k|\leq2\mu n$, from which it follows that $|U_2^+|\in\{t_2-k-1,t_2-k\}$.

A first approach to proving Theorem~\ref{theorem:extremal:case1} would be to embed $T$ into $G_{\mathrm{blue}}[U_1^+]$ if $|U_1^+|\geq n$, and into $G_{\mathrm{red}}[U_1^+,U_2^+]$ if $|U_2^+|\geq t_2$. We will be able to do this when the tree has many, say $n/100$, vertex-disjoint bare paths with length 5, which we call \textbf{Case}~\ref{CaseIA}. This is split into \textbf{Case}~\ref{CaseIA1} and \textbf{Case}~\ref{CaseIA2} depending on whether $k\geq 0$.
To embed such a tree $T$, we first remove many suitable bare paths with length 4, so that the remaining forest $T'$ can be embedded greedily. Then, we find a collection of vertex-disjoint paths with length 2 in $G$ that will form the middle parts of the missing paths, and make sure that they cover all the low degree vertices in $G$ if necessary. Finally, we attach these length 2 paths to the image of $T'$ to complete a copy of $T$ by verifying the appropriate Hall's matching conditions. These proofs are carried out in Section~\ref{sec:IA:paths}.

If $T$ does not have $n/100$ vertex-disjoint bare paths with length 5, we call this \textbf{Case}~\ref{CaseIB}. Note that by Lemma~\ref{lemma:paths-leaf}, $T$ must then have at least $n/20$ leaves. Unlike in \textbf{Case}~\ref{CaseIA} above where no spare vertex is required to embed $T$, it follows from a famous example of Koml\'os, S\"ark\'ozy and Szemer\'edi~\cite{komlos2001spanning} that we cannot necessarily embed $T$ into $G_{\mathrm{blue}}[U_1^+]$ even if $|U_1^+|\geq n$, or into $G_{\mathrm{red}}[U_1^+,U_2^+]$ even if $|U_2^+|\geq t_2$. For example, in our setting, it is easy to create an $n$-vertex tree $T$ with $\Delta(T)\leq cn$ in which there is a set $W$ of at most $2/c$ vertices such that every edge of the tree contains a vertex in $W$. Thus, $G_{\mathrm{blue}}[U_1^+]$ can only contain a blue copy of $T$ if it has a set $W'$ with size at most $2/c$ for which at most $k=|U_1^+|-n$ vertices in $U_1^+\setminus W'$ have no blue neighbour in $W'$.
An appropriate red sparse binomial random graph placed on $U_1^+$ can destroy this property, even when $|U_1^+|=n+\Theta_c(n)$, while having maximum degree at most $c n$ (see~\cite{komlos2001spanning} for a more detailed example). Thus, whether we embed the tree in red or blue will depend not only on $k$, but also on the number of edges in $\red{G}[U_1^+]$ and $\blue{G}[U_1^+,U_2^+]$.

We divide \textbf{Case}~\ref{CaseIB} (where $T$ has many leaves) into two further subcases, using a parameter $D$, which is defined to be the $30\mu n$-th biggest value of $\blue{d}(u,U_2^+)$ across all $u\in U_1^+$.
The hope is that this could allow us to embed $D$ leaves of the tree $T$ using the edges in $\blue{G}[U_1^+,U_2^+]$, thus creating more space to embed the rest of the tree into $G_{\mathrm{blue}}[U_1^+]$. With this reasoning, one might expect that if $k+D\geq 0$, then as $|U_1^+|+D=n+k+D\geq n$, we will be able to embed the tree in blue. However, an example similar to the one mentioned above from \cite{komlos2001spanning} shows that simiply requiring $k+D\geq 0$ is not enough, as comparatively few red edges in $G[U_1^+]$ can spoil this embedding attempt in blue. However, in Section~\ref{sec:IB1}, we will show that we can embed the tree in blue as long as $G[U_1^+]$ has at most $10^7(k+D+1)n$ red edges, which we call \textbf{Case}~\ref{CaseIB1}.

Our final case, \textbf{Case}~\ref{CaseIB2}, is when there are more than $10^7(k+D+1)n$ red edges in $G[U_1^+]$. These red edges will allow us to embed into $\red{G}[U_1^+]$ some small subtree of $T$, which contains at least $k+D+1$ vertices from $V_2$, the bipartition class of $T$ with size $t_2$. Note that to embed the remaining vertices of $V_2$, we now have at least $k+D+1-k-1\geq D$ spare vertices available in $U_2^+$. As most of the vertices in $U_1^+$ have at most $D$ blue neighbours in $U_2^+$, we will be able to embed the rest of the tree essentially greedily, with some slight complications such as vertices in $U_2^+\setminus U_2$ having fewer red neighbours in $U_1^+$.

In Section~\ref{sec:IA:paths}, we prove the embedding results used for both \textbf{Case}~\ref{CaseIA1} and \textbf{Case}~\ref{CaseIA2}. In Section~\ref{sec:IB1}, we prove the embedding result used for \textbf{Case}~\ref{CaseIB1}, stated so that it will also be useful in Section~\ref{sec:extremearg2}. In Section~\ref{sec:IB2}, we prove the embedding result used for \textbf{Case}~\ref{CaseIB2}. Finally, we put all of these together in Section~\ref{sec:IBfinal} to prove Theorem~\ref{theorem:extremal:case1}.
To finish this outline, we recap the different cases in the proof of Theorem~\ref{theorem:extremal:case1}, noting the main result that takes care of each of them.


\begin{enumerate}[label = \textbf{\Roman{enumi}}]
\item $G$ is Type I extremal.\label{CaseI}
\begin{enumerate}[label = \textbf{I.\Alph{enumii}}]
\item $T$ has at least $n/100$ vertex-disjoint bare paths with length 5.\label{CaseIA}
\begin{enumerate}[label = \textbf{I.\Alph{enumii}.\arabic{enumiii}}]
\item $k\geq 0$: $T$ embeds in blue.\hfill\emph{Lemma~\ref{proposition:bare-paths}}\label{CaseIA1}
\item $k<0$: $T$ embeds in red.\hfill\emph{Lemma~\ref{proposition:bipartite:bare-paths}}\label{CaseIA2}
\end{enumerate}
\item $T$ has at least $n/20$ leaves.\label{CaseIB}
\begin{enumerate}[label = \textbf{I.\Alph{enumii}.\arabic{enumiii}}]
\item $e(\red{G}[U_1^+])\leq 10^7(k+D+1)n$: $T$ embeds in blue.\hfill\emph{Lemma~\ref{lemma:caseIB1embedding}}\label{CaseIB1}
\item $e(\red{G}[U_1^+])>10^7(k+D+1)n$: $T$ embeds in red.\hfill\emph{Lemma~\ref{lemma:caseIB2embedding}}\label{CaseIB2}
\end{enumerate}
\end{enumerate}
\end{enumerate}


\subsection{Case~\ref{CaseIA}: trees with many bare paths in Type I extremal graphs}\label{sec:IA:paths}
In \textbf{Case}~\ref{CaseIA1}, we will use the following result that allows us to embed spanning trees with many short bare paths into an almost complete graph. A stronger version in which the degree condition is weakened like in Lemma~\ref{proposition:bipartite:bare-paths} also holds, though this is not needed in our proof.
\begin{lemma}\label{proposition:bare-paths} Let $1/n\ll \mu\ll 1$. Let $H$ be an $n$-vertex graph with $\delta(H)\ge (1-\mu)n$, and let $T$ be an $n$-vertex tree that contains $10\mu n$ vertex-disjoint bare paths with length 4. Then, for any $t\in V(T)$ not on any of these bare paths and any $s\in V(H)$, there is a copy of $T$ in $H$ with $t$ copied to $s$.
\end{lemma}
\begin{proof}
From assumption, there is a collection $\mathcal P=\{P_1,\ldots,P_\ell\}$ of $\ell=10\mu n$ vertex-disjoint bare paths in $T$ with length $4$, such that $t$ is not on any of these bare paths. Let $T'$ be the forest obtained by removing all internal vertices of the paths in $\mathcal P$ from $T$, so that $|T'|=n-3\ell$.

Since $\delta(H)\geq (1-\mu)n\geq |T'|$, by Lemma~\ref{lemma:greedy}, we can greedily find a copy $S'$ of $T'$ in $H$ with $t$ copied to $s$. Let $H'=H-V(S')$. Note that $|H'|=3\ell=30\mu n$, so $\delta(H')\geq|H'|-\mu n\geq|H'|/2$. Thus, by Dirac's theorem, $H'$ contains a Hamilton cycle. In particular, we can label the vertices in $H'$ as $w_1,\ldots,w_\ell,x_1,\ldots,x_\ell,y_1,\ldots,y_\ell$, so that for each $i\in [\ell]$, $x_iw_iy_i$ is a path in $H'$.

For each $i\in [\ell]$, let $u_i,v_i$ be the copies of the endpoints of $P_i$ in $S'$. Let $K$ be an auxiliary bipartite graph with bipartition classes $A=\{a_1,\ldots,a_\ell\}$ and $B=\{b_1,\ldots,b_\ell\}$, such that for any $i,j\in[\ell]$, there is an edge $a_ib_j$ in $K$ if and only if both $u_ix_j$ and $v_iy_j$ are edges in $H$. Since $\delta(H)\geq(1-\mu)n$, for every $i\in[\ell]$, $d_K(a_i),d_K(b_i)\geq\ell-2\mu n$. Then, for any $I\subset A$ with $0<|I|\leq\ell-2\mu n$, we have $|N_K(I,B)|\geq\ell-2\mu n\geq |I|$, while for any $I\subset A$ with $|I|>\ell-2\mu n$, we have $|N_K(I,B)|=|B|\geq|I|$, as any $b\in B\setminus N_K(I,B)$ would satisfy $d_K(b)<2\mu n<\ell-2\mu n$, a contradiction. Thus, by Lemma~\ref{lemma:Hall}, there is a perfect matching in $K$, say matching $a_i$ with $b_{\sigma(i)}$ for every $i\in[\ell]$. This then implies that $S'$ along with the paths $u_ix_{\sigma(i)}w_{\sigma(i)}y_{\sigma(i)}v_i$ for all $i\in[\ell]$ form a copy of $T$ in $H$, as required.
\end{proof}

In \textbf{Case}~\ref{CaseIA2}, we will use the following bipartite version of Lemma~\ref{proposition:bare-paths}. We prove it in a more flexible form so that we may also use it later in the proof of Lemma~\ref{lemma:caseIB2embedding}.
\begin{lemma} \label{proposition:bipartite:bare-paths}
Let $1/n\ll \mu\ll\beta\ll 1$. Let $H$ be a bipartite graph with bipartition classes $U_1$ and $U_2$ such that $d(u,U_2)\ge|U_2|-\mu n$ for every $u\in U_1$, $d(u,U_1)\geq\beta n$ for every $u\in U_2$, and $d(u,U_1)\geq|U_1|-\mu n$ for all but a set $W$ of at most $\mu n$ vertices in $U_2$.

Let $T$ be an $n$-vertex forest with bipartition classes $V_1$ and $V_2$ such that $|V_j|\le |U_j|$ for each $j\in [2]$. Let $R$ be a subforest of $T$ with $|R|\leq\beta n/2$, such that $T-R$ contains a collection $\mathcal{P}$ of $10\mu n$ vertex-disjoint bare paths of length $4$ whose endpoints are all in $V_2$. Suppose that $H-W$ contains a copy $S$ of $R$ with vertices in $V(R)\cap V_j$ copied into $U_j$ for each $j\in [2]$, then $S$ can be extended to a copy of $T$ in $H$, such that $W$ is covered by the central vertices of the bare paths in $\mathcal{P}$.
\end{lemma}
\begin{proof}
Let $\ell=10\mu n$, and label the paths in $\mathcal{P}$ as $P_1,\ldots, P_\ell$. Let $T'$ be the forest obtained by removing the internal vertices of $P_1,\ldots,P_\ell$ from $T$, noting that $|V(T')\cap V_1|=|V_1|-2\ell$ and $|V(T')\cap V_2|=|V_2|-\ell$.

Let $W=\{w_1,\ldots,w_r\}$, where $r\leq\mu n$. Since $d(w,U_1\setminus V(S))\geq\beta n-\beta n/2\geq2\mu n$ for every $w\in W$, we can greedily find distinct vertices $x_1,\ldots,x_r,y_1,\ldots,y_r\in U_1\setminus V(S)$, such that $x_i,y_i\in N(w_i)$ for every $i\in[r]$. Let $W^+=\{w_i,x_i,y_i:i\in[r]\}$, and $H'=H-W^+$.

Note that every vertex in $H'$ has at most $\mu n$ non-neighbours in the opposite side of the bipartition. Using $\ell=10\mu n$ and Lemma~\ref{lemma:bipartitegreedy}, we can greedily extend the copy $S$ of $R$ to a copy $S'$ of $T'$ in $H'$, in which $V(T')\cap V_{j}$ is copied into $U_j\setminus W^+$ for each $j\in [2]$. For every $i\in[\ell]$, let $u_i,v_i\in U_2$ be the copies of the two endpoints of $P_i$ in $S'$. To complete a copy of $T$, it suffices to find, for every $i\in[\ell]$, a $u_i,v_i$-path with length 4 using distinct new internal vertices.

Let $H''=H'-V(S')=H-W^+-V(S')$. Then $|V(H'')\cap U_1|\geq|V_1|-2r-(|V_1|-2\ell)=2(\ell-r)$, and similarly $|V(H'')\cap U_2|\geq|V_2|-r-(|V_2|-\ell)\geq\ell-r\geq5\mu n$. Moreover, every vertex in $H''$ has at most $\mu n$ non-neighbours in the opposite side. Arbitrarily pick distinct $w_{r+1},\ldots,w_{\ell}\in V(H'')\cap U_2$, we claim that there exist distinct $x_{r+1},\ldots,x_{\ell}, y_{r+1},\ldots,y_{\ell}\in V(H'')\cap U_1$ such that $x_i,y_i\in N(w_i)$ for every $i\in[\ell]\setminus[r]$. Indeed, by Lemma~\ref{lemma:hallmatching}, it suffices to show that for every $\emptyset\not=I\subset[\ell]\setminus[r]$, $|N(\{w_i:i\in I\},V(H'')\cap U_1)|\geq2|I|$. If $0<|I|\leq\ell-r-\mu n/2$, then $|N(\{w_i:i\in I\},V(H'')\cap U_1)|\geq|V(H'')\cap U_1|-\mu n\geq2(\ell-r)-\mu n\geq2|I|$. If $|I|>\ell-r-\mu n/2$, then $|N(\{w_i:i\in I\},V(H'')\cap U_1)|=|V(H'')\cap U_1|\geq2(\ell-r)\geq2|I|$, as any $u\in V(H'')\cap U_1$ that is not adjacent to any $w_i$ with $i\in I$ would have at least $|I|\geq2\mu n$ non-neighbours in $V(H'')\cap U_2$, a contradiction.

Therefore, together with $W^+$, we have found distinct $x_1,\ldots,x_\ell,y_1,\ldots,y_\ell,w_1,\ldots,w_\ell\in V(H)\setminus V(S')$, such that $x_i,y_i\in N(w_i)$ for every $i\in[\ell]$. Let $K$ be an auxiliary bipartite graph with bipartition classes $A=\{a_1,\ldots,a_\ell\}$ and $B=\{b_1,\ldots,b_\ell\}$, such that for any $i,j\in[\ell]$, $a_ib_j\in E(K)$ if and only if both $u_ix_j$ and $v_iy_j$ are in $E(H)$. From construction, for each $i\in[\ell]$, $u_i,v_i\in U_2\setminus W$, so $d_K(a_i)\geq\ell-2\mu n$. Similarly, $d_K(b_j)\geq\ell-2\mu n$ for every $j\in[\ell]$. Like in Lemma~\ref{proposition:bare-paths}, we can now use these degree conditions to verify that for every $I\subset A$, $|N_K(I,B)|\geq|I|$, so Lemma~\ref{lemma:Hall} gives a perfect matching in $K$, say matching $a_i$ with $b_{\sigma(i)}$ for every $i\in[\ell]$. Then, $S'$ together with the paths $u_ix_{\sigma(i)}w_{\sigma(i)}y_{\sigma(i)}v_i$ for all $i\in[\ell]$ form a copy of $T$ in $H$, with vertices in $W$ covered by the central vertices of the bare paths in $\mathcal{P}$, as required.
\end{proof}


\subsection{Case~\ref{CaseIB1}: embedding trees in almost complete graphs}\label{sec:IB1}
We now prove the main result to be used in \textbf{Case}~\ref{CaseIB1}. This is proved in a slightly stronger form so that we may also use it in Section~\ref{sec:extremearg2}.
\begin{lemma}\label{lemma:caseIB1embedding}
Let $1/n\ll c\ll \mu,\beta\ll1$, let $|k|\leq\mu n$ and $0\leq D\leq\mu n$ satisfy $k+D\geq0$. Let $G$ be a graph with a vertex partition $U_1\cup U_2$ such that $|U_1|=n+k$ and $\delta(G[U_1])\geq|U_1|-\beta n$. Let $X\subset U_1$ satisfy $|X|\leq \mu n$ and $d_G(u,U_2)\geq n/10$ for each $u\in X$. Suppose $e(G[U_1\setminus X])\leq10^7(k+D+1)n$, and there are at least $10\mu n$ vertices in $U_1$ with at least $D$ neighbours in $U_2$. 

Let $T$ be an $n$-vertex tree with $\Delta(T)\leq cn$ such that, if $D>0$, then $T$ has at least $n/20$ leaves. Then, $G$ contains a copy of $T$. Moreover, if $D=0$ and $X=\emptyset$, then for any $t\in T$ and $s\in U_1$, there is a copy of $T$ in $G$ with $t$ copied to $s$. 
\end{lemma}
\begin{proof}
If $D=0$ and $T$ has fewer than $n/20$ leaves, then $T$ has at least $n/100$ vertex-disjoint bare paths with length 4 by Lemma~\ref{lemma:paths-leaf}. Since $k\geq-D=0$, the result follows by applying Lemma~\ref{proposition:bare-paths} to $G[U_1]$.  

Therefore, we can assume that $D\geq0$ and $T$ has at least $n/20$ leaves. If $D\not=0$ or $X\not=\emptyset$, pick $t\in V(T)$ arbitrarily. Then, there exists a set $L'$ of $n/49$ leaves in $T$, which does not contain $t$ or any neighbours of $t$, and either all belong to $V_1$ or all belong to $V_2$. Let $P=N_T(L')$ be the set of parents of $L'$ in $T$, and note that $P$ is an independent set. 
Let $P_1\subset P$ be a set of size at most $D$ such that $t\notin P_1$ and $D\leq|N_T(P_1,L')|\leq n/150$, which is possible as $\Delta(T)\leq cn$. Similarly, let $P_2'\subset P\setminus P_1$ be a set of size at most $\lceil |X|/2\rceil$ such that $t\notin P_2'$ and $\lceil |X|/2\rceil \leq|N_T(P_2',L')|\leq n/150$. Let $P_3=P\setminus (P_1\cup P_2')$, so $|N_T(P_3,L')|\geq n/149$. In particular, we can add a set $L_3'$ of $\ceil{|X|/2}-|P_2'|$ leaves adjacent to $P_3$ to the set $P_2'$ to obtain a set $P_2$ of size $\lceil |X|/2\rceil$. Note that $P_1\cup P_2$ is still an independent set. Let $L=L'\setminus L_3'$ and let $m=|T-L|$. Observe that $n/50\leq |L|\leq n/49$ and $|N_T(P_3,L)|\geq n/150$.
Let $t_1=t$, and let $t_1,t_2,\ldots, t_{m}$ be an ordering of $V(T-L)$ so that each vertex apart from $t_1$ has exactly 1 neighbour in $T$ to its left in this ordering. Let $d_i=d_T(t_i,L)$ for every $i\in[m]$.

From assumption, we can take a set $Y$ of $2D$ vertices in $U_1\setminus X$, each having at least $D$ neighbours in $U_2$. Let $U^-_1=\{u\in U_1:d_G(u,Y)\geq D\text{ and }d_G(u,X)\geq |X|/2\}$. Since $\delta(G[U_1])\geq|U_1|-\beta n$, if $D>0$, by double counting there are at most $2D\beta n/D=2\beta n$ vertices $u\in U_1$ with $d_G(u,Y)<D$, while there is no such vertex if $D=0$. Similarly there are at most $2\beta n$ vertices $u\in U_1$ with $d_G(u,X)<|X|/2$. Thus, $|U_1\setminus U_1^-|\leq4\beta n$. If $D=0$ and $X=\emptyset$, note that $s\in U_1=U_1^-\setminus (X\cup Y)$ already. Otherwise, pick $s\in U_1^-\setminus (X\cup Y)$ arbitrarily.

Let $s_1=s$, so that $s_1\in U_1^-\setminus (X\cup Y)$. For each $1<i\leq m$ in turn, embed $t_i$ as follows, where $j_i<i$ is such that $t_{j_i}t_i\in E(T-L)$.
\stepcounter{propcounter}
\begin{enumerate}[label = \textbf{\Alph{propcounter}\arabic{enumi}}]
    \item\label{IB1embed1} If $t_i\in P_1$, select $s_{i}$ uniformly at random from $N_G(s_{j_i},Y\setminus \{s_1,\ldots,s_{i-1}\})$.
    \item\label{IB1embed1b} If $t_i\in P_2$, select $s_{i}$ uniformly at random from $N_G(s_{j_i},X\setminus \{s_1,\ldots,s_{i-1}\})$.
    \item\label{IB1embed2} If $t_i\not\in P_1\cup P_2$, select $s_{i}$ uniformly at random from $N_G(s_{j_i},U_1^-\setminus (X\cup Y\cup \{s_1,\ldots,s_{i-1}\}))$.
\end{enumerate}
Note that~\ref{IB1embed1} is always possible as $t_{j_i}\not\in P_1\cup P_2$, so $s_{j_i}\in U_1^-$ has at least $D$ neighbours in $Y$, and at most $|P_1|-1\leq D-1$ of them have been used. Similarly,~\ref{IB1embed1b} is always possible. 
\ref{IB1embed2} is always possible as 
\begin{align*}
|N_G(s_{j_i},U_1^-\setminus (X\cup Y\cup \{s_1,\ldots,s_{i-1}\}))|&\geq|U_1^-|-\beta n-|X|-|Y|-|T-L|\\
&\geq n+k-4\beta n-\beta n-\mu n-2D-n+|L|\geq n/100>0.   
\end{align*}Therefore, this random process always succeeds in producing a copy of $T-L$ in $G[U_1]$. Note that, in particular, we have $|X\setminus \{s_1,\ldots,s_m\}|=|X|-|P_2|=\lfloor |X|/2\rfloor$.

Now, let $j^*$ be the smallest integer so that $2^{j^*}>cn$. For each $j\in [j^*]$, let $I_j$ be the set of $\{i:t_i\in P_3\}$ with $2^{j-1}\leq d_i<2^j$, and say a vertex $v$ is \emph{$j$-bad} if
\[
|\{s_i:i\in I_j\}\cap N_G(v)|\leq \frac{2}{3}|I_j|-20(j^*-j+1).
\]
If $v$ is $j$-bad for some $j\in [j^*]$, then say $v$ is \emph{bad}. For each $v\in U_1\setminus X$, let $m_v$ be the number of non-neighbours of $v$ in $U_1\setminus X$, so that $m_v\leq \beta n$ from assumption.

\begin{claim}\label{clm:fewbad} For each $v\in U_1\setminus X$,
\[
\mathbb{P}(v\text{ is bad})\leq \frac{m_v}{10^8n}.
\]
\end{claim}
\begin{proof}[Proof of Claim~\ref{clm:fewbad}]
Let $j\in[j^*]$. Note that if $|I_j|<30(j^*-j+1)$, then $v$ cannot be $j$-bad, so $\mathbb{P}(v\text{ is $j$-bad})=0$. Assume now that $|I_j|\geq30(j^*-j+1)$. Note that for every $i\in I_j$, $t_i\not\in P_1\cup P_2$, so by \ref{IB1embed2} and conditioning on any choices of $s_1,\ldots,s_{i-1}$, we have
\[\mathbb{P}(s_i\not\in N_G(v)\mid s_1,\ldots,s_{i-1})\leq\frac{m_v}{|N_G(s_{j_i},U_1^-\setminus (X\cup Y\cup \{s_1,\ldots,s_{i-1}\}))|}\leq\frac{m_v}{n/100}.\]
Therefore, for each $I\subset I_j$, if $E_{v,j,I}$ is the event that $s_i\notin N_G(v)$ for each $i\in I$, then
\[
\mathbb{P}(E_{v,j,I})=\prod_{i\in I}\mathbb{P}(s_i\notin N_G(v)\mid s_{i'}\notin N_G(v)\text{ for each }i'\in I\text{ less than }i)\leq\left(\frac{100m_v}{n}\right)^{|I|}.
\]
Since $v$ is $j$-bad implies that $E_{v,j,I}$ holds for some $I\subset I_j$ with size $k_j:=\floor{|I_j|/3}\geq10(j^*-j+1)$, we can use a union bound to get
\[
\mathbb{P}(v\text{ is }j\text{-bad})\leq \binom{|I_j|}{k_j}\left(\frac{100m_v}{n}\right)^{k_j}\leq\left(\frac{10^4m_v}{n}\right)^{k_j}\leq\left(\frac{10^4m_v}{n}\right)^{10(j^*-j+1)}\leq\left(\frac{m_v}{10^{10}n}\right)^{(j^*-j+1)},
\]
where we used $m_v\leq\mu n$ and $1/n\ll\mu\ll1$. Thus,
\[
\mathbb{P}(v\text{ is bad})\leq \sum_{j=1}^{j^*}\mathbb{P}(v\text{ is $j$-bad})\leq\sum_{j=1}^{j^*}\left(\frac{m_v}{10^{10}n}\right)^{(j^*-j+1)}\leq\frac{m_v}{10^8n},
\]
as required.
\renewcommand{\qedsymbol}{$\boxdot$}
\end{proof}
\renewcommand{\qedsymbol}{$\square$}

By Claim~\ref{clm:fewbad}, we have
\[
\mathbb{E}|\{v\in U_1\setminus X:v\text{ is bad}\}|\leq \sum_{v\in U_1\setminus X}\frac{m_v}{10^8n}\leq \frac{2e(G^c[U_1\setminus X])}{10^8n}\leq \frac{k+D+1}{5},
\]
so by Markov's inequality, with probability at least $1/2$ there are at most $\floor{(k+D+1)/2}$ bad vertices in $U_1\setminus X$. Thus, we can take a realisation of this random embedding that has at most $\floor{(k+D+1)/2}$ bad vertices in $U_1\setminus X$.

Let $U_1'$ be the set of unused vertices in $U_1$, noting that $|U_1'|=n+k-|T-L|=k+|L|$ and $|U_1'\cap X|=\floor{|X|/2}$. To complete the embedding, we use Lemma~\ref{lemma:hallmatching} to embed the leaves in $L$ into $U_1'\cup U_2$. It suffices to show that $N_G(\{s_i:t_i\in J\},U_1'\cup U_2)\geq\sum_{i:t_i\in J}d_i$ for every $\emptyset\not=J\subset P$.

First suppose that $0<\sum_{i:t_i\in J}d_i\leq 999|L|/1000$, then, as required,
\[
|N_G(\{s_i:t_i\in J\},U_1')|\geq|U_1'|-\beta n=k+|L|-\beta n\geq999|L|/1000\geq\sum_{i:t_i\in J}d_i.
\]

Now suppose that $\sum_{i:t_i\in J}d_i>999|L|/1000$. Let $L_3=N_T(P_3,L)$, recall that $|L|\leq n/49$ and $|L_3|\geq n/150$. Thus, $\sum_{i:t_i\in J\cap P_3}d_i\geq|L_3|-|L|/1000\geq9|L_3|/10$. Then, for any $v\in U_1\setminus X$ which is not bad,
\begin{align*}
\sum_{\substack{i:t_i\in P_3,\\s_i\notin N_G(v)}}d_i&\leq\sum_{j=1}^{j^*}\sum_{i\in I_j:s_i\notin N_G(v)}2^j
\leq\sum_{j=1}^{j^*}\left(\frac{1}{3}|I_j|+20(j^*-j+1)\right)\cdot 2^j\\
&\leq\frac{1}{3}\sum_{j=1}^{j^*} \sum_{i\in I_j}2d_i+\sum_{j=1}^{j^*} 20(j^*-j+1)\cdot 2^j\\
&\leq\frac{2|L_3|}{3}+80\cdot2^{j^*}\leq\frac{2|L_3|}{3}+200cn<\frac{9|L_3|}{10}\leq\sum_{i:t_i\in J\cap P_3}d_i.
\end{align*}
Thus, $v$ is adjacent to some vertex in $\{s_i:t_i\in J\cap P_3\}$. Since at most $\floor{(k+D+1)/2}$ vertices in $U_1\setminus X$ are bad, this implies that
\begin{align*}
|N_G(\{s_i:t_i\in J\},U_1')|&\geq|U_1'|-|U_1'\cap X|-\floor{(k+D+1)/2}\\
&=|L|+k-\floor{|X|/2}-\floor{(k+D+1)/2}\geq|L|-D-\ceil{|X|/2},    
\end{align*}
so we are done if $\sum_{i:t_i\in J}d_i\leq|L|-D-\ceil{|X|/2}$.

Suppose now that $\sum_{i:t_i\in J}d_i>|L|-D-\ceil{|X|/2}$. If $J\cap P_2=\emptyset$, then $\sum_{i:t_i\in J}d_i\leq |L|-\ceil{|X|/2}$ as $\sum_{i:t_i\in P_2}d_i\geq\ceil{|X|/2}$. Hence, $J\cap P_1\neq\emptyset$ as $\sum_{i:t_i\in P_1}d_i\geq D$, and thus
\begin{align*}
|N_G(\{s_i:t_i\in J\},U_1'\cup U_2)|
&=|N_G(\{s_i:t_i\in J\},U_1')|+|N_G(\{s_i:t_i\in J\},U_2)|\\
&\geq|L|-D-\ceil{|X|/2}+D=|L|-\ceil{|X|/2}\geq\sum_{i:t_i\in J}d_i,
\end{align*}
as required. If $J\cap P_2\neq \emptyset$, then we have
\begin{align*}
|N_G(\{s_i:t_i\in J\},U_1'\cup U_2)|
&=|N_G(\{s_i:t_i\in J\},U_1')|+|N_G(\{s_i:t_i\in J\},U_2)|\\
&\geq|L|-D-\ceil{|X|/2}+n/10\geq |L|\geq\sum_{i:t_i\in J}d_i.
\end{align*}
Thus, by Lemma~\ref{lemma:hallmatching}, we can embed $L$ into $U_1'\cup U_2$ to finish a copy of $T$ in $G$.
\end{proof}


\subsection{Case~\ref{CaseIB2}: embedding trees in almost complete bipartite graphs}\label{sec:IB2}

In \textbf{Case}~\ref{CaseIB2}, using the notations in Section~\ref{sec:outline:extcase1}, we want to embed a small subtree into $\red{G}[U_1^+]$, before embedding the rest of $T$ into $\red{G}[U_1^+,U_2^+]$. To find this small subtree, we use the following result.
\begin{proposition}\label{prop:treeforIB2}
Let $n,m\in \mathbb{N}$ satisfy $1\leq m\leq n/18$. Let $T$ be an $n$-vertex tree with bipartition classes $V_1$ and $V_2$, where $|V_2|\geq |V_1|/3$. Then, $T$ contains a subtree $T'$ with $|T'|\leq 10^4m$, such that it contains at least $m$ vertices in $V_2$ and has at most 1 vertex in $V_2$ with a neighbour outside of $V(T')$ in $T$.
\end{proposition}
\begin{proof}
For any tree $R$ on at least $18m$ vertices, by repeated applications of Corollary~\ref{cor:splittree}, we can find subtrees $R_1$ and $R_2$ with a unique common vertex that decompose $R$, such that $6m\leq|R_1|\leq18m$. Iterating this in the tree $T$ to find a subtree with size between $6m$ and $18m$ at a time, we obtain a sequence $T_1,\ldots,T_\ell$ of subtrees satisfying the following properties.
\begin{itemize}
    \item $E(T_1),\ldots,E(T_\ell)$ partition $E(T)$.
    \item For every $j\in [\ell]$, $6m\le |T_j|\le 18m$.
    \item For every $j\in[\ell]$, $\cup_{i=1}^jT_i$ is a tree.
    \item\label{IB23} For every $2\leq j\leq\ell$, $T_j$ shares a unique vertex with $\cup_{i=1}^{j-1}T_i$.
\end{itemize}

Note that $\ell\leq n/(6m-1)$. We claim that $|V(T_i)\cap V_2|\geq m$ for at least $\ell/100$ indices $i\in[\ell]$. Indeed, if not, then there would be at most $\ell m+(\ell/100)\cdot 18m<n/4$ vertices in $V_2$, contradicting the assumption that $|V_2|\geq|V_1|/3$. Let $I=\{i\in[\ell]:|V(T_i)\cap V_2|\geq m\}$.

Consider the auxiliary graph $K$ with vertex set $[\ell]$, where for any $i<j$ in $[\ell]$, $ij\in E(K)$ if and only if $i$ is the smallest index such that $V(T_i)\cap V(T_j)\not=\emptyset$. Note that every $1<j\leq\ell$ has a unique neighbour in $[j-1]$ in $K$, so $K$ is a tree and $e(K)=\ell-1$. It follows that there exists $i\in I$ with $d_K(i)\leq 200$.

Let $t_i$ be the unique vertex shared by $T_i$ and $\cup_{j<i}T_j$. Then, $N_T(V(T_i)\setminus\{t_i\},\cup_{j>i}V(T_j))\leq200\cdot 18m$. Let $T'$ be the subtree of $T$ induced by $V(T_i)$ and $N_T(V(T_i)\setminus\{t_i\},\cup_{j>i}V(T_j))\cap V_1$. Then $|T'|\leq10^4m$, $|V(T')\cap V_2|\geq|V(T_i)\cap V_2|\geq m$, and every vertex in $V(T')\cap V_2$, except possibly $t_i$, has no neighbour in $T$ outside of $V(T')$.
\end{proof}

Using Proposition~\ref{prop:treeforIB2}, we can now prove the main result used for \textbf{Case}~\ref{CaseIB2}.
\begin{lemma}\label{lemma:caseIB2embedding}
Let $1/n\ll c\ll \mu \ll \beta \ll 1$, let $0\leq D\leq \mu n$, and let $|k|\leq\mu n$. Let $T$ be any $n$-vertex tree with $\Delta(T)\leq cn$ and bipartition classes $V_1$ and $V_2$ with sizes $t_1$ and $t_2$, respectively, that satisfy $t_2\leq t_1\leq 2t_2+1$. Let $G$ be a graph with a vertex partition $U_1\cup U_2$ such that $1.1t_1\leq |U_1|\leq 2n$, $|U_2|=t_2-k-1$, and the following hold. 
\begin{itemize}
    \item $d_G(u,U_{1})\geq \beta n$ for each $u\in U_2$, and $d_G(u,U_1)\geq |U_1|-\mu n$ for all but at most $\mu n$ vertices $u\in U_2$.
    \item $d_G(u,U_2)\geq |U_2|-D$ for all but at most $10\mu n$ vertices $u\in U_1$.
    \item There exists $X\subset U_1$ with $|X|\leq2\mu n$ such that $d_G(u,U_2)\geq |U_2|-\mu n$ for each $u\in U_1\setminus X$ and $e(G[U_1\setminus X])\geq 10^7(k+D+1)n$.
\end{itemize}
Then, $G$ contains a copy of $T$.
\end{lemma}
\begin{proof}
First observe that it suffices to prove this in the case when $k+D\geq-1$. Indeed, if $k+D<-1$, then let $k'=-D-1$ and remove $-k-D-1$ vertices from $U_2$ to obtain $U_2'$ with size $t_2-k'-1$. Note that trivially $e(G[U_1\setminus X])\geq 10^7(k'+D+1)n$, and all other assumptions still hold with $U_2'$ and $k'$ in place of $U_2$ and $k$, so $G$ contains a copy of $T$. Thus, we assume that $k+D\geq-1$ from now on.

Let $Y_1=\{u\in U_1:d_G(u,U_2)<|U_2|-D\}$ and $Y_2=\{u\in U_2:d_G(u,U_1)<|U_1|-\mu n\}$, so that $|Y_1|\leq 10\mu n$ and $|Y_2|\leq \mu n$. For each $i\in[2]$, let $U_i^-=U_i\setminus Y_i$. Let $Z\subset U_1^-$ be a random subset chosen by including each vertex independently at random with probability $\beta$. By Lemma~\ref{lemma:chernoff}, with high probability, we have $|Z|\leq 3\beta n$, $d_G(u,Z)\geq\beta^2n/2$ for each $u\in U_2$, and $e(G[U_1\setminus (X\cup Z)])\geq 10^7(k+D+1)n/2$, noting that the last condition is trivial when $k+D=-1$. Fix a choice of $Z$ with all of these properties.

As $e(G[U_1\setminus (X\cup Z)])\geq 10^7(k+D+1)n/2$, we can find a subgraph $H'\subset G[U_1\setminus (X\cup Z)]$ with minimum degree at least $10^7(k+D+1)/4$. Since each vertex $u\in U_1\setminus (X\cup Z)$ satisfies $d_G(u,U_2^-)\geq |U_2^-|-\mu n$, there are at least $|H'||U_2^-|/2$ edges in $G$ between $H'$ and $U_2^-$. Hence, there exists $v\in U_2^-$ with at least $|H'|/2\geq10^6(k+D+1)$ neighbours in $V(H')$.  

By Proposition~\ref{prop:treeforIB2}, there is a subtree $T'$ of $T$ with $|T'|\leq10^5(k+D+1)$ that contains at least $10(k+D+1)$ vertices in $V_2$, and has at most 1 vertex in $V(T')\cap V_2$ with a neighbour in $T-E(T')$. Call such a vertex $t$ if it exists. 
Let $F$ be the tree obtained from $T$ by contracting $T'$ to a single vertex $r$. Let $L_1$ be the set of leaves in $F$ that are in $V_1\setminus\{r\}$, observing that they are also leaves in $T$. Note that $|F-L_1|\geq|V_2|-10^5(k+D+1)\geq n/4$, so by Lemma~\ref{lemma:paths-leaf}, $F-L_1$ either has at least $n/100$ vertex-disjoint bare paths of length 5 or at least $n/100$ leaves excluding $r$. Note that any such leaf in $F-L_1$ must be in $V_2$, and is also a leaf in $T-L_1$. We now separate into several cases.

\medskip

\noindent\textbf{Case I.} $F-L_1$ contains at least $n/100$ vertex-disjoint bare paths with length 5, then $T-T'-L_1$ contains $n/200$ vertex-disjoint bare paths with length 4 with both endpoints in $V_2$. Since there are at $n/200$ central vertices on these paths, and $|L_1|\leq n$, by averaging, we can find a set $L_2$ of $10\mu n$ such central vertices that are adjacent to at most $200\cdot10\mu n\leq\beta^2n/10$ vertices in $L_1$. Embed $t$ to $v$ if $t$ exists, then in any case embed the rest of $T'$ greedily into $H'$. Note that this embeds at least $\max\{0,10k+10D+9\}$ vertices in $V_2$ into $U_1$, so we now have enough room to apply Lemma~\ref{proposition:bipartite:bare-paths} to extend this embedding of $T'$ to an embedding of $T-L_1$, with $T-T'-L_1$ embedded into the rest of $G[U_1^-\setminus Z,U_2]$ such that $Y_2$ is covered by vertices in $L_2$. To finish, greedily embed leaves in $L_1$ not adjacent to $L_2$ into $U_1^-\setminus Z$, possible as the parents of these leaves are embedded into $U_2^-$, and greedily embed leaves in $L_1$ adjacent to $L_2$ into $Z$, using that every vertex in $U_2$ has at least $\beta^2n/2\geq|N_T(L_2,L_1)|$ neighbours in $Z$.

\medskip

\noindent\textbf{Case II.} $F-L_1$ contains a set $L_2'$ of $n/200$ leaves not in $N_T(V(T'))$. Then, we can find a set $L_2\subset L_2'$ of size $10\mu n$ with $|N_T(L_2,L_1)|\leq 200\cdot10\mu n\leq\beta^2n/10$. Embed $t$ to $v$ if $t$ exists, then in any case embed the rest of $T'$ greedily into $H'$. Note that this embeds at least $\max\{0,10k+10D+9\}$ vertices in $V_2$ into $U_1$. Next, greedily extend this to embed the rest of $T-L_1-L_2$ into the rest of $G[U_1^-\setminus Z,U_2^-]$ with vertices in $V_i$ going into $U_i^-$ for each $i\in[2]$, which is possible as $|L_2|=10\mu n$ and $d(u,U_2^-)\geq|U_2^-|-D\geq t_2-|L_2|$ for every $u\in U_1^-$. We can then greedily embed $L_2$ into the rest of $U_2$, which is possible as there are at least $t_2-k-1-(t_2-\max\{0,10k+10D+9\}-|L_2|)\geq|L_2|+D$ vertices left in $U_2$, and every vertex in $U_1^-$ is adjacent to all but at most $D$ of them. Finally, greedily embed the leaves in $L_1$ not adjacent to $L_2$ into $U_1^-\setminus Z$, and greedily embed the leaves in $L_1$ adjacent to $L_2$ into $Z$, using that every vertex in $U_2$ has at least $\beta^2n/2\geq|N_T(L_2,L_1)|$ neighbours in $Z$.

\medskip

\noindent\textbf{Case III.} $F-L_1$ contains a set $L_2'$ of at least $n/200$ leaves in $N_T(V(T'))$. By Lemma~\ref{lemma:divide:vertices}, there exists subtrees $T_1,T_2$ decomposing $T'$ with a unique common vertex $t'$, such that $|V(T_1)\cap V_2|,|V(T_2)\cap V_2|\geq 3(k+D+1)$. Without loss of generality, suppose that $N_T(V(T_2),L_2')\geq n/500$, and pick a set $L_2\subset N_T(V(T_2),L_2')$ of size $10\mu n$, with none of them adjacent to $t'$, such that $N_T(L_2,L_1)\leq\beta^2n/10$. Note that at most two vertices in $V(T_1)\cap V_2$ can have a neighbour in $V(T)\setminus V(T_1)$, namely $t$ if it exists, and $t'$ if it is in $V_2$. 

If $t$ exists and $t'\in V_2$, then view $T_1$ as rooted at $t$, let the parent of $t'$ in $T_1$ be $p$, and let the parent of $p$ be $p'$. Embed $t$ to $v$, then greedily embed the rest of $T_1$ into $H'$ with the following exception. Let $W$ be the set of neighbours of the image of $p'$ in $H'$ that are still unused. If $p'$ is embedded into $H'$, then $|W|\geq10^7(k+D+1)/4-|T_1|\geq2\cdot 10^5(k+D+1)$, while if $p'$ coincides with $t$ then it is embedded to $v$, so $|W|\geq10^6(k+D+1)-|T_1|\geq2\cdot 10^5(k+D+1)$ as well. Like before, there exists $v'\in U_2^-$ with at least $|W|/2\geq10^5(k+D+1)$ neighbours in $W$. Embed $t'$ to $v'$, $p$ and $N_{T_1}(t')$ into $W$, then carry on greedily to finish the embedding of the rest of $T_1$ inside $H'$. If $t$ does not exist or $t'\not\in V_2$, we can greedily embed $T_1$ into $G[V(H')\cup\{v\}]$ such that $t$ is embedded to $v$ if it exists, and the same for $t'$ if it is in $V_2$. 

In any case, we have an embedding of $T_1$ into $G$, with all but at most two vertices embedded into $H'$, and every vertex in $V(T_1)\cap V_2$ that has any neighbour outside of $T_1$ is embedded into $U_2^-$. In particular, at least $3(k+D+1)-2\geq k+D+1$ vertices in $V_2$ are embedded into $U_1$ if $k+D\geq0$, while the same holds trivially if $k+D=-1$. Next, greedily embed the rest of $T-L_1-L_2$ into the rest of $G[U_1^-\setminus Z,U_2^-]$. We can then greedily embed $L_2$ into the rest of $U_2$, which is possible as there are at least $t_2-k-1-(t_2-k-D-1-|L_2|)=|L_2|+D$ vertices left in $U_2$ and every vertex in $U_1^-$ is adjacent to all but at most $D$ of them. Finally, greedily embed leaves in $L_1$ not adjacent to $L_2$ into $U_1^-\setminus Z$, and greedily embed leaves in $L_1$ adjacent to $L_2$ into $Z$, using that every vertex in $U_2$ has at least $\beta^2n/2\geq|N_T(L_2,L_1)|$ neighbours in $Z$.
\end{proof}


\subsection{Proof of Theorem~\ref{theorem:extremal:case1}}\label{sec:IBfinal}
Having proved all of the embedding results necessary for the Type I extremal case, we can now put them together to prove Theorem~\ref{theorem:extremal:case1}, following the outline in Section~\ref{sec:outline:extcase1}.

\begin{proof}[Proof of Theorem~\ref{theorem:extremal:case1}]
Let $1/n\ll c\ll\mu \ll 1$ and let $t_1,t_2\in \mathbb{N}$ satisfy $t_1+t_2=n$ and $t_1\geq t_2$. Let $G$ be a Type I $(\mu,t_1,t_2)$-extremal graph on $\max\{2t_1,t_1+2t_2\}-1$ vertices, and let $T$ be an $n$-vertex tree with $\Delta(T)\le cn$ and bipartition classes $V_1,V_2$ of sizes $t_1$ and $t_2$, respectively. From definition, there are disjoint subsets $U_1,U_2\subset V(G)$ such that $|U_1|\geq(1-\mu)n$, $|U_2|\geq(1-\mu)t_2$, $\red{d}(u,U_1)\leq \mu n$ for every $u\in U_1$, and $\blue{d}(u,U_{3-i})\leq \mu n$ for every $i\in[2]$ and every $u\in U_i$.

First assume that $t_1\leq 2t_2+1$, so $|G|\in\{n+t_2-1,n+t_2\}$. Let $\beta$ be such that $\mu \ll \beta \ll 1$. Let $U_2^+=\{v\in V(G):d_{\mathrm{red}}(u,U_1)\geq \beta n\}$, $U_1^+=V(G)\setminus U_2^+$, and note that $U_1\subset U_1^+$ and $U_2\subset U_2^+$. Let $k=|U_1^+|-n$, so $|k|\leq2\mu n$ and $|U_2^+|\in\{t_2-k-1,t_2-k\}$.
If $T$ has at least $n/100$ vertex-disjoint bare paths with length 5, then $T$ has $n/100$ vertex-disjoint bare paths with length 4 whose endpoints are all in $V_2$. Thus, there is a blue copy of $T$ in $G$ if $k\geq 0$ by Lemma~\ref{proposition:bare-paths}, and there is red copy of $T$ in $G$ if $k<0$ by Lemma~\ref{proposition:bipartite:bare-paths}.

Suppose, then, that $T$ does not have at least $n/100$ vertex-disjoint bare paths with length 5, then, by Lemma~\ref{lemma:paths-leaf}, $T$ has at least $n/20$ leaves. Let $D\geq0$ be the $30\mu n$-th biggest value of $\blue{d}(u,U_2^+)$ across all $u\in U_1^+$, and note that $D\leq3\mu n$. Let $X=\{u\in U_1^+:\blue{d}(u,U_2^+)\geq n/10\}$, and note that $X\subset U_1^+\setminus U_1$, so $|X|\leq2\mu n$. If $e(\red{G}[U_1^+\setminus X])<10^7(k+D+1)n$, then $k+D\geq0$, so there is a blue copy of $T$ in $G$ by Lemma~\ref{lemma:caseIB1embedding}, while if $e(\red{G}[U_1^+\setminus X])\geq10^7(k+D+1)n$, then there is a red copy of $T$ in $G$ by Lemma~\ref{lemma:caseIB2embedding}.

Finally, if $t_1\geq 2t_2+2$, then we can take $2/c$ leaves of $T$ in $V_1$, which are guaranteed to exist by Lemma~\ref{lemma:leaves:V1}, and attach $\lfloor (t_1-2t_2)/2\rfloor$ new leaves to them, with none of them receiving more than $cn$ new leaves. Let $T'$ be the new tree obtained in this way, and note that the bipartition classes of $T'$ have sizes $t_1'=t_1$ and $t_2'=\floor{t_1/2}$, with $t_1'\leq 2t_2'+1$ and $\max\{t'_1+2t'_2,2t'_1\}-1=2t_1'-1=2t_1-1=\max\{t_1+2t_2,2t_1\}-1$. Therefore, $G$ contains a monochromatic copy of $T'$ from above, and thus also contains a monochromatic copy of $T$.
\end{proof}

\section{Proof of Theorem~\ref{theorem:extremal:case2}: Type II extremal graphs}\label{sec:extremearg2}
In this section, we prove Theorem~\ref{theorem:extremal:case2}. We will start by outlining the proof in Section~\ref{sec:outline:extcase2}, breaking it down into different cases that are then proved throughout the rest of this section. 

\subsection{Proof outline for Type II extremal graphs}\label{sec:outline:extcase2}

We start by recapping the situation in Theorem~\ref{theorem:extremal:case2}, where we have parameters $1/n\ll c\ll\mu\ll 1$. Suppose $n=t_1+t_2$ with $t_1\geq(2-\mu)t_2$. Let $T$ be an $n$-vertex tree with $\Delta(T)\le cn$ and bipartition classes $V_1$ and $V_2$ of sizes $t_1$ and $t_2$, respectively.
Let $G$ be a red/blue coloured complete graph on $\max\{2t_1-1,t_1+2t_2-1\}$ vertices which is Type II $(\mu,t_1,t_2)$-extremal, which means that there are disjoint sets $U_1,U_2\subset V(G)$ such that $|U_1|,|U_2|\geq(1-\mu)t_1$, and for every $i\in [2]$ and $u\in U_i$, $\red{d}(u,U_i)\leq \mu n$ and $\blue{d}(u,U_{3-i})\leq \mu n$.

For some $\beta$ with $\mu \ll \beta\ll 1$, we will start by taking maximal disjoint sets $U_1^+,U_2^+\subset V(G)$ with $U_1\subset U_1^+$ and $U_2\subset U_2^+$, such that for every $i\in [2]$ and $u\in U_i^+$, $\red{d}(u,U_{3-i})\geq \beta n$. By relabelling if necessary, we can assume that $|U_1^+|\geq |U_2^+|$.

Our first two cases are reasonably easy. First, in \textbf{Case}~\ref{CaseIIY}, we assume that there exist two vertices $v_1,v_2$ with mostly blue neighbours in both $U_1^+$ and $U_2^+$. This will allow us to embed part of $T$ into $\blue{G}[U_1^+]$ and the rest of $T$, apart from at most 2 vertices, into $\blue{G}[U_2^+]$, then connect them together appropriately using $v_1$ and $v_2$. Using two vertices in this way is optimal, as $T$ may have a vertex whose removal creates exactly three subtrees of roughly equal sizes, so we could not easily fit two of them together into one of $G_{\mathrm{blue}}[U_1^+]$ or $G_{\mathrm{blue}}[U_2^+]$.

Next, in \textbf{Case}~\ref{CaseIIA}, we assume that there is a vertex $w$ that has at least $\beta n$ red neighbours in both $U_1^+$ and $U_2^+$. Then, we decompose $T$ into subtrees $T_1$ and $T_2$ with a common neighbour $t$, so that $T_1$ is a small subtree containing suitably more vertices in $V_1$ than in $V_2$ (see Proposition~\ref{prop:bipartitesplittree}). We embed $t$ to $w$, then embed the rest of $T_1,T_2$ greedily into $\red{G}[U_1^+,U_2^+]$ by embedding vertices in $V(T_1)\cap V_2$ and $V(T_2)\cap V_1$ into $U_1^+$, and vertices in $V(T_1)\cap V_1$ and $V(T_2)\cap V_2$ into $U_2^+$. Observe that $T_1$ and $T_2$ are embedded in opposite ways, which `rebalances' the vertices in $T$ across $U_1^+$ and $U_2^+$, so that there is enough room in each side for this embedding to be completed greedily.

Assuming neither of these cases hold, then $U_1^+$ and $U_2^+$ together cover all but at most one vertex of $G$, and $\delta(G_{\mathrm{blue}}[U_i^+])\geq |U_i^+|-\beta n$ for each $i\in [2]$. Assume without loss of generality that $|U_1^+|\geq |U_2^+|$. If there is a vertex $v$ not in $U_1^+\cup U_2^+$, then it has mostly blue neighbours in both $U_1^+$ and $U_2^+$, and we can use Corollary~\ref{cor:splittree} to decompose $T$ into subtrees $T_1$ and $T_2$ with a unique common vertex $t$ such that $|U_1^+\cup \{v\}|\geq |T_1|$ and $|U_2^+|\geq |T_2|+n/100$. 
In \textbf{Case}~\ref{CaseIIZ}, we assume that there are at most $10^6n$ red edges in $G[U_1^+]$. Using our work in Section~\ref{sec:IB1}, we can embed $T_1$ into $G_{\mathrm{blue}}[U_1^+\cup \{v\}]$ with $t$ embedded to $v$. Then, as $v$ has plenty of blue neighbours in $U_2^+$ and $|U_2^+|$ is comfortably larger than $|T_2|$, we can greedily embed $T_2$ into $G_{\mathrm{blue}}[U_2^+\cup \{v\}]$ to complete a blue copy of $T$.

Suppose now that we are not in \textbf{Cases}~\ref{CaseIIY}--\ref{CaseIIZ}. 
Let $k=t_1-|U_1^+|$, and note that $k\leq 1$ as $|U_1^+|\geq |U_2^+|$. Moreover, if $k=1$, then there exists $v\in V(G)\setminus(U_1^+\cup U_2^+)$, and so $G[U_1^+]$ contain at least $10^6n$ red edges since we are not in \textbf{Case}~\ref{CaseIIZ}. In theory, we have enough space while attempting to embed $T$ in red to embed all but at most $1$ vertex of $V_1$ into $U_1^+$, and $V_2$ into $U_2^+$. 
While this can always be done when $T$ has many vertex-disjoint bare paths, if $T$ has many leaves instead, then similar to
\textbf{Case}~\ref{CaseIB} discussed before, a small number of blue edges in $G[U_1^+,U_2^+]$ can prevent this embedding for some trees.

Therefore, we will first try to embed $T$ in blue again, using the following sparse cut structure.
\begin{definition}\label{def:sparsecut}
Let $T$ be an $n$-vertex tree. An \emph{$(\eps,d)$-sparse cut} in $T$ is a partition $V(T)=A\cup B$ such that the following hold. 
\begin{itemize}
    \item $T[A]$ is a tree and $|A|,|B|\leq (2/3-\eps)n$.
    \item For each $v\in A$, $d_T(v,B)\leq d$.
    \item $\{v\in A:d_T(v,B)>0\}$ is an independent set in $T$ with size at most $2\Delta(T)$.
\end{itemize}
\end{definition}
Such a sparse cut with $\eps\gg\mu$ and $d=\sqrt n$ will be found later using Proposition~\ref{lem:sparsecut}, and we will also need some additional properties guaranteed by Proposition~\ref{prop:splitwithsparsecut}. In \textbf{Case}~\ref{CaseIIB}, we assume that $T[A,B]$ can be embedded into $G_{\mathrm{blue}}[U_1^+,U_2^+]$, with, say $A$ embedded into $U_1^+$ and $B$ embedded into $U_2^+$. Then, we use the high minimum degree condition to find an embedding of $T[A]$ in $\blue{G}[U_1^+]$ that matches enough of the embedding of $T[A,B]$. Using the part that matches, we can greedily extend the embedding of $T[A]$ to embed most of the vertices in $B$ into $U_2^+$. The number of remaining vertices in $B$ will be small enough that they can be greedily embedded into the remaining part of $\blue{G}[U_1^+]$. 

Finally, in \textbf{Case}~\ref{CaseIIC}, we assume that $T[A,B]$ cannot be embedded into $G_{\mathrm{blue}}[U_1^+,U_2^+]$ as described above. The fact that we failed to do so will imply that there exist $U_A\subset U_1^+$ and $U_B\subset U_2^+$ of suitable sizes, such that every vertex in $U_1^+\setminus U_A$ has at most $\sqrt{n}$ blue neighbours in $U_2^+\setminus U_B$. Here, we focus on a specific case where $|U_1^+|\geq t_1$ and $T$ has a set $L$ of many leaves in $V_1$, the other case is handled similarly. First, we embed $T-L$ essentially randomly into $\red{G}[U_1^+,U_2^+]$, with vertices in $V_i$ embedded into $U_i^+$ for each $i\in[2]$, while ensuring that leaves in $L$ have their parents embedded into $U_2^+\setminus U_B$. Similar to \textbf{Case}~\ref{CaseIB1}, in $U_1^+$ we may have some `bad' vertices to which it is hard to embed the vertices in $L$. However, it will be likely that there is no bad vertex in $U_1^+\setminus U_A$ as all these vertices have very high red degrees into $U_2^+\setminus U_B$ from assumption. To make sure that we have no uncovered bad vertices in $U_A$, we use that the size of $U_A$ is related to the sparse cut $V(T)=A\cup B$. Roughly speaking, there will be enough components in $T[N_T(B,A)\cup B]$ that contains a vertex in $V_1\cap B$ for us to use these vertices to cover $U_A$. Finally, having ensured that there is no bad vertex, we can embed the leaves in $L$ to complete the embedding of $T$.

In Sections~\ref{subsec:caseIIY}--\ref{subsec:caseIIC}, we will prove the main embedding results used for \textbf{Cases}~\ref{CaseIIY}--\ref{CaseIIC}, respectively, before putting these all together in Section~\ref{sec:IIfinal} to prove  Theorem~\ref{theorem:extremal:case2}.
To finish this outline, we recap the different cases in the proof of Theorem~\ref{theorem:extremal:case2}, noting the main result that takes care of each of them. In what follows, by `otherwise' we mean that none of the previous cases hold.

\begin{enumerate}[label = \textbf{\Roman{enumi}}]\addtocounter{enumi}{1}
\item $G$ is Type II extremal.\label{CaseII}
\begin{enumerate}[label = \textbf{II.\Alph{enumii}}]
\item Two vertices have mostly blue neighbours in $U_1^+$ and $U_2^+$: $T$ embeds in blue.\hfill\emph{Lemma~\ref{lem:caseIIY}}\label{CaseIIY}
\item Some vertex has $\beta n$ red neighbours in both $U_1^+$ and $U_2^+$: $T$ embeds in red.\hfill\emph{Lemma~\ref{lem:caseIIA}}\label{CaseIIA}
\item Otherwise, but one vertex has mostly blue neighbours in $U_1^+$ and $U_2^+$, and $G[U_1^+]$ contains at most $10^6n$ red edges: $T$ embeds in blue.\hfill\emph{Lemma~\ref{lem:caseIIZ}}\label{CaseIIZ}
\item Otherwise, but $T[A,B]$ embeds into $G_{\mathrm{blue}}[U_1^+,U_2^+]$: $T$ embeds in blue.\hfill\emph{Lemma~\ref{lem:caseIIB}}\label{CaseIIB}
\item Otherwise, either $|U_1^+|\geq t_1$, or $|U_1^+|=t_1-1$ and there are at least $10^6n$ red edges in $G[U_1^+]$: $T$ embeds in red.\hfill\emph{Lemma~\ref{lem:caseIIC}}\label{CaseIIC}
\end{enumerate}
\end{enumerate}


\subsection{Case~\ref{CaseIIY}}\label{subsec:caseIIY}
For \textbf{Case}~\ref{CaseIIY}, we first use the following result to find a large subtree of $T$ that contains at most two vertices that have neighbours in the rest of the tree. Moreover, if there are two such vertices, then they are not adjacent.

\begin{proposition}\label{prop:splittreewith2vertices} Let $1/n\ll \eps\ll 1$. Let $T$ be an $n$-vertex tree.
Then, there is a partition $V(T)=A\cup B$ with $|A|,|B|\leq (2/3-\eps)n$ such that $T[A]$ is a tree and $\{v\in A:d_T(v,B)>0\}$ is an independent set in $T$ with size at most 2.
\end{proposition}
\begin{proof}
Using Corollary~\ref{cor:splittree}, let $T_1$ and $T_2$ be subtrees decomposing $T$ with a unique common vertex $t$, such that $\frac n3\leq|T_1|\leq|T_2|\leq 1+\frac{2n}3$. Furthermore, assume that $T_1$ and $T_2$ are chosen so that $|T_2|$ is minimised subject to these conditions. If $|T_2|\leq (2/3-\eps)n$, then set $A=V(T_2)$ and $B=V(T)\setminus A$, and note that the conditions in the lemma hold.

Suppose now that $|T_2|>(2/3-\eps)n$. If $t$ has only one neighbour, say $t'$, in $T_2$, then adding $tt'$ to $T_1$ and removing $t$ from $T_2$ gives two trees that contradict the minimality of $|T_2|$. If $T_2-t$ has a component $S$ with size at most $(1/3-2\eps)n$, then, letting $T_1'=T[\{t\}\cup V(S)]\cup T_1$ and $T_2'=T_2-V(S)$ gives a pair of trees $(T_1',T_2')$ that again contradicts the minimality of $|T_2|$, as $\max\{|T_1'|,|T_2'|\}<|T_2|$. Thus, $t$ must have exactly two neighbours in $T_2$, and $T_2-t$ is the disjoint union of two trees $S_1$ and $S_2$ with $(1/3-2\eps)n\leq|S_1|,|S_2|\leq(1/3+2\eps)n$. For each $i\in[2]$, let $t_i$ be the neighbour of $t$ in $S_i$.

Using Corollary~\ref{cor:splittree} again, let $S_1'$ and $S_2'$ be subtrees decomposing $S_1$ with a unique common vertex $t_1'$, such that $(1-6\eps)n/9\leq|S'_1|,|S'_2|\leq 1+(2+12\eps)n/9$. Relabelling if necessary, assume that $S_1'$ contains $t_1$.
Let $A=V(T_1)\cup V(S_1')\cup\{t_2\}$ and note that $T[A]$ is the tree made by connecting $T_1$ and $S_1'$ with the edge $tt_1$ and adding the edge $tt_2$. Let $B=V(T)\setminus A$, and note that the only vertices in $A$ with neighbours in $B$ in $T$ are $t_2$ and $t_1'$, and they are not adjacent in $T$ because they are in different components of $T_2-t$. Thus, $A$ and $B$ satisfy the required conditions.
\end{proof}

Using Proposition~\ref{prop:splittreewith2vertices}, it is now straightforward to prove the following result used in \textbf{Case}~\ref{CaseIIY}.

\begin{lemma}\label{lem:caseIIY}
Let $1/n\ll c\ll \mu \ll 1$. Let $T$ be an $n$-vertex tree with $\Delta(T)\leq cn$.
Let $G$ be a graph that contains two disjoint vertex sets $U_1$ and $U_2$ such that $|U_i|\geq (2/3-\mu)n$ and $\delta(G[U_i])\geq |U_i|-\mu n$ for each $i\in [2]$. Suppose there exist $v_1,v_2\in V(G)\setminus (U_1\cup U_2)$ such that $d_G(v_i,U_j)\geq |U_j|-\mu n$ for each $i,j\in [2]$. Then, $G$ contains a copy of $T$.
\end{lemma}
\begin{proof} Using Proposition~\ref{prop:splittreewith2vertices}, let $V(T)=A\cup B$ be a partition with $|A|,|B|\leq (2/3-10\mu)n$, such that $T[A]$ is a tree and $A':=\{v\in A:d_T(v,B)>0\}$ is an independent set in $T$ with $|A'|\leq 2$. For each $i\in[2]$, let $U_i^-=N_G(v_1,U_i)\cap N_G(v_2,U_i)$, then $|U_i^-|\geq (2/3-3\mu)n$ and $\delta(G[U_i^-])\geq |U_i^-|-\mu n$. Embed one vertex in $A'$ to $v_1$. Then, greedily extend this to an embedding of the tree $T[A]$ in $G[U_1^-\cup\{v_1,v_2\}]$, such that if there is another vertex in $A'$, then it is embedded to $v_2$. We can then extend this to copy of $T$ by embedding $T[A'\cup B]$ greedily in $G[U_2^-\cup\{v_1,v_2\}]$. 
\end{proof}


\subsection{Case~\ref{CaseIIA}}\label{subsec:caseIIA}
For \textbf{Case}~\ref{CaseIIA}, we need to find a subtree which has suitably more vertices in $V_1$ than in $V_2$, which we do with the following result.
\begin{proposition}\label{prop:bipartitesplittree}
Let $1/n\ll\mu\ll1$, and let $T$ be an $n$-vertex tree with bipartition classes $V_1$ and $V_2$ such that $|V_1|\geq 1.1|V_2|$. Then, there exists a decomposition of $T$ into subtrees $T_1$ and $T_2$ with a unique common vertex $v$, such that $10\mu n\leq|V(T_1)\cap V_1|-|V(T_1)\cap V_2|\leq 25\mu n$.
\end{proposition}
\begin{proof}
Among all $v\in T$ and all subtrees $T_1$ and $T_2$ decomposing $T$ with a unique common vertex $v$ that satisfy $|V(T_1)\cap V_1|-|V(T_1)\cap V_2|\geq12\mu n$, pick the combination that minimises $|T_1|$. Note that such $v,T_1,T_2 $ exist as $|V_1|\geq 1.1|V_2|$ implies that $|V_1|-|V_2|\geq 12\mu n$, so picking an arbitrary $v\in V(T)$, letting $T_1=T$ and $T_2=T[\{v\}]$ would satisfy these conditions.

First, consider the case when $\deg(v,T_1)\geq2$. If there is a component $S$ of $T_1-v$ that satisfies $|V(S)\cap V_2|-|V(S)\cap V_1|> 0$, then, transferring $S$ and the edge between $v$ and $S$ from $T_1$ to $T_2$ gives two subtrees $T_1',T_2'$ that still satisfy the required conditions but with $|T_1'|<|T_1|$, a contradiction. Thus, we can assume that every component of $T_1-v$ has at least as many vertices in $V_1$ as in $V_2$. Then, since there are at least 2 such components, at least one of them, say $S'$, satisfies $0\leq |V(S')\cap V_1|-|V(S')\cap V_2|\leq (|V(T_1)\cap V_1|-|V(T_1)\cap V_2|+1)/2$. In order for transferring $S'$ and the edge between $v$ and $S'$ from $T_1$ to $T_2$ to not contradict the minimality of $|T_1|$, we must have that $(|V(T_1)\cap V_1|-|V(T_1)\cap V_2|)-(|V(S')\cap V_1|-|V(S')\cap V_2|)<12\mu n$, and hence, $|V(T_1)\cap V_1|-|V(T_1)\cap V_2|\leq 25\mu n$, as required.

Suppose, then, that $\deg(v,T_1)=1$. Let $v'$ be the neighbour of $v$ in $T_1$. Let $T'_1=T_1-v$ and $T'_2=T_2+vv'$. Note that, in order to not get a contradiction, we must have $|V(T_1')\cap V_1|-|V(T_1')\cap V_2|< 12\mu n$, so $|V(T_1)\cap V_1|-|V(T_1)\cap V_2|<\mu n+1\leq 25\mu n$, as required.
\end{proof}
Using Proposition~\ref{prop:bipartitesplittree}, we can now prove the following lemma required in \textbf{Case}~\ref{CaseIIA}.
\begin{lemma}\label{lem:caseIIA}
Let $1/n\ll c\ll \mu\ll1$.
Let $T$ be an $n$-vertex tree with $\Delta(T)\leq cn$ and bipartition classes of sizes $t_1$ and $t_2$ satisfying $t_1\geq 1.1t_2$.
Let $G$ be a graph containing two disjoint subsets $U_1,U_2$ and an additional vertex $w$, such that for each $i\in [2]$, $|U_i|\geq t_1-\mu n$, $d_G(w,U_i)\geq\mu n$, and $d_G(u,U_{3-i})\geq|U_{3-i}|-\mu n$ for every $u\in U_i$.
Then, $G$ contains a copy of $T$.
\end{lemma}
\begin{proof}
By Proposition~\ref{prop:bipartitesplittree}, we can find subtrees $T_1$ and $T_2$ decomposing $T$ with a unique common vertex $v$, such that $10\mu n\leq|V(T_1)\cap V_1|-|V(T_1)\cap V_2|\leq25\mu n$. Embed $v$ to $w$, then we can greedily embed both $T_1$ and $T_2$ so that vertices in $V(T_2)\cap V_1$ and $V(T_1)\cap V_2$ go into $U_1$ and vertices in $V(T_1)\cap V_1$ and $V(T_2)\cap V_2$ go into $U_2$. This is possible because \[|V(T_1)\cap V_2|+|V(T_2)\cap V_1|\leq|V(T_1)\cap V_1|+|V(T_2)\cap V_1|-10\mu n\leq t_1+1-10\mu n\leq|U_1|-\mu n,\]
\[|V(T_1)\cap V_1|+|V(T_2)\cap V_2|\leq|V(T_1)\cap V_2|+|V(T_2)\cap V_2|+25\mu n\leq t_2+1+25\mu n\leq |U_2|-\mu n,\]
and $w$ has $\mu n\geq\Delta(T)$ neighbours in both $U_1$ and $U_2$.
\end{proof}


\subsection{Case~\ref{CaseIIZ}}\label{subsec:caseIIZ}
The embedding result for \textbf{Case}~\ref{CaseIIZ} follows easily from Lemma~\ref{lemma:caseIB1embedding} proved earlier for \textbf{Case}~\ref{CaseIB1}.

\begin{lemma}\label{lem:caseIIZ}
Let $1/n\ll c\ll \mu \ll 1$. Let $T$ be an $n$-vertex tree with $\Delta(T)\leq cn$.
Let $G$ be a graph that contains two disjoint vertex sets $U_1$, $U_2$ such that $|U_1|\geq\ceil{2n/3}-1$, $|U_2|\geq (2/3-\mu)n$, and $\delta(G[U_i])\geq |U_i|-\mu n$ for each $i\in [2]$. Suppose $G[U_1]$ contains at most $10^6n$ non-edges, and that there exists $v\in V(G)\setminus (U_1\cup U_2)$ with $d_G(v,U_i)\geq |U_i|-\mu n$ for each $i\in [2]$.
Then, $G$ contains a copy of $T$.
\end{lemma}
\begin{proof} Using Corollary~\ref{cor:splittree}, let $T_1$ and $T_2$ be a decomposition of $T$ into subtrees with a unique common vertex $t$, so that $\ceil{n/3}\leq|T_1|\leq|T_2|\leq\ceil{2n/3}$. Let $D=0$ and $k=|U_1\cup \{v\}|-|T_2|$, so that $k\geq 0$ and thus $G[U_1\cup\{v\}]$ has at most $10^6n+\mu n\leq 10^7(k+D+1)|T_2|$ non-edges. Then, by Lemma~\ref{lemma:caseIB1embedding}, $G[U_1\cup\{v\}]$ contains a copy of $T_2$, in which $t$ is copied to $v$. Since $|T_1|\leq 1+n/2\leq |U_2|-\mu n$, we can complete the embedding of $T$ by greedily finding a copy of $T_1$ in $G[U_2\cup \{v\}]$ with $t$ copied to $v$.
\end{proof}


\subsection{Case~\ref{CaseIIB}}\label{subsec:caseIIB}

We now give the embedding in \textbf{Case}~\ref{CaseIIB}, where there are enough blue edges between $U_1$ and $U_2$ to embed a large subtree of the tree in $U_2$ and connect this across to embed the rest into $U_1$.

\begin{lemma}\label{lem:caseIIB}
Let $1/n\ll c\ll \mu \ll \eps\ll 1$. Let $G$ be a graph that contains two disjoint vertex sets $U_1$, $U_2$ such that $|U_i|\geq(2/3-\eps/3)n$ and $\delta(G[U_i])\geq |U_i|-\mu n$ for each $i\in [2]$. Let $T$ be an $n$-vertex tree with $\Delta(T)\leq cn$. Suppose $T$ has an $(\eps,\sqrt n)$-sparse cut $V(T)=A\cup B$, such that $T[A,B]$ can be embedded into $G[U_1,U_2]$ with $A$ embedded into $U_1$ and $B$ embedded into $U_2$.
Then, $G$ contains a copy of $T$.
\end{lemma}
\begin{proof}
Let $A'=\{v\in A:d_T(v,B)>0\}$ and let $t_1\in A\setminus A'$, so $d_T(t_1,B)=0$. Let $m=|A|$ and extend $t_1$ to an ordering $t_1,\ldots,t_m$ of the vertices in $A$ so that each vertex $t_j$ except $t_1$ has exactly one neighbour to its left in $T[A]$ in this ordering.

Let $I=\{i\in[m]:t_i\in A'\}$, so $|I|\leq2cn$ from the definition of sparse cuts. From assumption, there is an embedding $\phi'$ of $T[A,B]$ into $G[U_1,U_2]$ with $A$ embedded into $U_1$ and $B$ embedded into $U_2$. Let $s'_i=\phi'(t_i)$ for each $i\in I$.

Pick $s_1\in U_1\setminus \{s'_i:i\in I\}$ uniformly at random. Then, for each $1<i\leq m$, let $j_i<i$ satisfy $t_{j_i}t_i\in E(T[A])$, and embed $t_i$ to some $s_i\in U_1$ randomly as follows.
\stepcounter{propcounter}
\begin{enumerate}[label = \textbf{\Alph{propcounter}\arabic{enumi}}]
    \item\label{IID1} If $i\in I$, then let $s_i=s'_i$ if $s_{j_i}s'_i\in E(G)$, otherwise pick $s_i$ uniformly at random from $N_G(s_{j_i},U_1)\setminus (\{s_j:j<i\}\cup \{s'_j:j\in I\})$.
    \item\label{IID2} If $i\notin I$, then pick $s_i$ uniformly at random from $N_G(s_{j_i},U_1)\setminus (\{s_j:j<i\}\cup \{s'_j:j\in I\})$.
\end{enumerate}

For each $i\in I$, let $T_i$ be the component containing $t_i$ in $T[A'\cup B]$, let $X_i=|T_i|-1$ if $s_i\neq s_i'$ and let $X_i=0$ otherwise. For each $i\in I$, the probability that $s_{j_i}$ is not in $N_G(s'_i)$ is at most $\mu n/(|U_1|-|A|-\mu n)\leq3\mu/\eps\leq\sqrt{\mu}$. Thus,
\[
\sum_{i\in I}\mathbb{E}(X_i)\leq \sqrt{\mu}\cdot \sum_{i\in I}|T_i|\leq \sqrt{\mu}n,
\]
so there is a realisation of $s_1,\ldots,s_m$ for which $\sum_{i\in I}X_i\leq \sqrt{\mu} n$. Take such a realisation, and let $I'\subset I$ be the set of $i\in I$ for which $s_i=s_i'$, so that $\sum_{i\in I\setminus I'}(|T_i|-1)=\sum_{i\in I}X_i\leq \sqrt{\mu}n$.

Let $\phi(t_i)=s_i$ for each $i\in [m]$ and note that this is an embedding of $T[A]$ in $G[U_1]$. Extend this to an embedding of $T[A\cup(\cup_{i\in I'}N_T(t_i))]$ using the embedding $\phi'$ of $T[A,B]$. Using $\delta(G[U_2])\geq |U_2|-\mu n\geq |B|$, we can greedily extend $\phi$ to an embedding of $T[A\cup(\cup_{i\in I'}T_i)]$ by embedding the vertices in $\cup_{i\in I'}(V(T_i)\cap B)$ into $U_2$. Then, using $\sum_{i\in I\setminus I'}(|T_i|-1)\leq \sqrt{\mu}n$,
\[\delta(G[U_1])\geq |U_1|-\mu n\geq (2/3-\eps/3)n-\mu n\geq |A|+\sum_{i\in I\setminus I'}(|T_i|-1),\]
so we can extend $\phi$ to an embedding of $T$ by greedily embedding the vertices in $\cup_{i\in I\setminus I'}(V(T_i)\cap B)$ into $U_1$. Thus, $G$ contains a copy of $T$.
\end{proof}


\subsection{Case~\ref{CaseIIC}}\label{subsec:caseIIC}
In the last case, \textbf{Case}~\ref{CaseIIC}, we aim to embed the tree $T$ into $\red{G}[U_1^+,U_2^+]$. Our embedding method here is the most involved, but it shares some similarities with \textbf{Case}~\ref{CaseIB1}. We will remove some leaves in $V_1$ from the tree $T$, and aim to embed the rest of the tree so that each remaining vertex in $G$ in the correct side has plenty of neighbours among the vertices that need leaves attached to them, which would allow us to complete the embedding using Lemma~\ref{lemma:hallmatching}. As mentioned before, the key difficulty is to ensure that the lower degree vertices in $G$ are covered, either in the initial stage by some carefully chosen vertices in $T$, or in the last stage by the leaves. An additional complication is that if $|U_1^+|\geq t_1$, then we will embed vertices in $V_i$ into $U_i^+$ for each $i\in[2]$, but if $|U_1^+|=t_1-1$, which implies that $|U_2^+|=t_1-1$ as well, we will instead embed vertices in $V_i$ into $U_{3-i}^+$ for each $i\in[2]$, except for one leaf in $V_1$ which needs to be embedded into $U_1^+$. The last part is possible as there will be some red edges in $G[U_1^+]$ in this case. To avoid repetition, we will prove the following embedding lemma that will later be applied to $(U_1^+,U_2^+)$ in the former case, and to $(U_2^+,U_1^+)$ in the latter case.

\begin{lemma}\label{lem:caseIIC}
Let $1/n\ll c\ll\mu\ll\alpha\ll\beta\ll\eps\ll1$. Let $G$ be a graph on at most $2n$ vertices that contains two disjoint vertex sets $U_1$, $U_2$ with $|U_1|\geq t_1-1$ and $|U_2|\geq(1-\mu)t_1$. Moreover, if $|U_1|=t_1-1$ then $e(G[U_2])\geq 10^6n$.
Suppose $\delta(G[U_1,U_2])\geq \beta n$ and, for each $i\in [2]$, all but at most $\mu n$ vertices $u\in U_i$ satisfy $d_G(u,U_{3-i})\geq |U_{3-i}|-\mu n$.

Let $0\leq \ell\leq 2cn$. Suppose there exist subsets $U_A\subset U_1$ and $U_B\subset U_2$ with $|U_A|<\ell$ and $|U_B|\leq(2/3-\eps)n$, such that $|N_G(u,U_2\setminus U_B)|\geq|U_2\setminus U_B|-\sqrt n$ for each $u\in U_1\setminus U_A$.

Let $T$ be an $n$-vertex tree with $\Delta(T)\leq cn$ and bipartition classes $V_1$ and $V_2$ of sizes $t_1$ and $t_2$, respectively, such that $t_1\geq(2-\mu)t_2$, and one of the following holds. 
\stepcounter{propcounter}
\begin{enumerate}[label =\emph{\textbf{\Alph{propcounter}\arabic{enumi}}}]
\item\labelinthm{subcase1} There is a set $L$ of $\alpha n$ leaves of $T$ in $V_1$ such that $d_T(u,L)\leq \sqrt{n}$ for each $u\in N_T(L)$.
\item\labelinthm{subcase2} $T$ contains a set $V_1'$ of $\ell$ vertices in $V_1$ and a disjoint set $L$ of $\alpha n$ leaves in $V_1$ such that vertices in $V_1'$ have no common neighbour, $|N_T(V_1')|\leq\eps n$, and $N_T(V_1')\cap N_T(L)=\emptyset$.
\end{enumerate}
Then, $G$ contains a copy of $T$.
\end{lemma}
\begin{proof}
Let $V_1'=\emptyset$ if $T$ does not satisfy~\ref{subcase2}. In both cases, let $s_1$ be a leaf of $T$ in $V_1\setminus V_1'$, which exists by Lemma~\ref{lemma:leaves:V1} using $t_1\geq(2-\mu)t_2$. Let $s_2$ be the neighbour of $s_1$ in $T$, and view $T$ as being rooted at $s_1$. 

Let $L'=L\setminus N_T(s_2)$, and let $P$ be the set of parents of $L'$. Note that $|L'|\geq\alpha n/2\gg\mu n$. Partition $P=P_1\cup P_2$, so that for each $j\in[2]$, $L_j:=N_T(P_j,L')$ has size at least $\alpha n/10$. If~\ref{subcase2} holds, then take an injection $\phi:U_A\to V_1'$ such that no vertex in $\phi(U_A)$ is adjacent to $s_2$, which is possible as $|U_A|<\ell=|V_1'|$ and vertices in $V_1'$ share no common neighbour. For every $u\in U_A$, let $\phi'(u)$ be the parent of $\phi(u)$ in the rooted tree $T$, and let $P'=\{\phi'(u):u\in U_A\}$. Note that from the assumption in~\ref{subcase2}, $P'\cap P=\emptyset$. If~\ref{subcase2} does not hold then let both $\phi$ and $\phi'$ be the empty function, and let $P'=\emptyset$.  

For each $j\in [2]$, let $U_j^-=\{u\in U_j:d_G(u,U_{3-j})\geq |U_{3-j}|-\mu n\}$, so that $|U_j\setminus U_j^-|\leq \mu n$. Select a random subset $Z\subset U_2^-$ by including each vertex independently at random with probability $\beta$. Using Lemma~\ref{lemma:chernoff}, with positive probability we have $|Z|\leq2\beta n$, $d_G(u,Z)\geq\beta^2n/2$ for each $u\in U_1$, and $e(G[U_2]-Z)>0$ if $|U_1|=t_1-1$. Fix a choice of $Z$ with these properties. Note that \[|U_2^-\setminus(U_B\cup Z)|\geq(1-\mu)t_1-\mu n-(2/3-\eps)n-2\beta n\geq2\alpha n.\] By adding vertices to $U_B$ if necessary, we may assume that $|U_2^-\setminus (U_B\cup Z)|=2\alpha n$.

Let $T'=T-L'$ and $m=|T'|$. Extend $s_1,s_2$ to an ordering $s_1,\ldots,s_m$ of the vertices in $T'$, such that for every $2\leq i\leq m$, $s_i$ has a unique neighbour in $T'$ to its left in this ordering. If $|U_1|=t_1-1$, pick $v_1v_2\in E(G[U_2]-Z)$ arbitrarily. Otherwise, arbitrarily pick $v_1v_2\in E(G[U_1,U_2]-Z)$ with $v_1\in U_1$ and $v_2\in U_2$. Embed $s_1$ to $v_1$ and $s_2$ to $v_2$. For each $3\leq i\leq m$, suppose $s_j$ has been embedded to $v_j$ for all $j<i$ with only vertices in $V_1'$ embedded into $U_A$, and let $j_i<i$ satisfy $s_{j_i}s_i\in E(T')$. Note that if $s_i\in V_2$ and $v_{j_i}\in U_A$, then $s_{j_i}\in V_1'$, so $s_i\not\in P\cup P'$ as vertices in $V_1'$ share no common neighbour and $N_T(V_1')\cap N_T(L')=\emptyset$. Embed $s_i$ randomly to some $v_i$ as follows.
\stepcounter{propcounter}
\begin{enumerate}[label =\textbf{\Alph{propcounter}\arabic{enumi}}]
\item\label{embed1} If $s_i\in P_1$, then randomly select $v_i$ from $N_G(v_{j_i},U_2^-\setminus(U_B\cup Z\cup\{v_1,\ldots,v_{i-1}\}))$.
\item\label{embed2} If $s_i\in P_2$, then randomly select $v_i$ from $N_G(v_{j_i},Z\setminus\{v_1,\ldots,v_{i-1}\})$.
\item\label{embed3} If $s_i\in P'$, say $s_i=\phi'(u)$, then randomly select $v_i$ from $N_G(v_{j_i},N_G(u,Z)\setminus\{v_1,\ldots,v_{i-1}\})$.
\item\label{embed4} If $s_i\in V_2\setminus(P\cup P')$ and $v_{j_i}\not\in U_A$, then randomly select $v_i$ from $N_G(v_{j_i},(U_2^-\cap U_B)\setminus(Z\cup\{v_1,\ldots,v_{i-1}\}))$.
\item\label{embed5} If $s_i\in V_2\setminus(P\cup P')$ and $v_{j_i}\in U_A$, then randomly select $v_i$ from $N_G(v_{j_i},Z\setminus\{v_1,\ldots,v_{i-1}\})$.
\item\label{embed6} If $s_i\in V_1$ and there is some $u\in U_A$ such that $\phi(u)=s_{i}$, then let $v_i=u$. 
\item\label{embed7} If $s_i\in V_1$ and there is no $u\in U_A$ such that $\phi(u)=s_{i}$, randomly select $v_{i}$ from $N_G(v_{j_i},U_1^-\setminus (U_A\cup\{v_1,\ldots,v_{i-1}\}))$.
\end{enumerate}
Note that~\ref{embed1}--\ref{embed7} can always be carried out to obtain an embedding of $T'$. Moreover, from~\ref{embed6}, $U_A$ is covered by this embedding if \ref{subcase2} holds. 

For each $i\in[m]$, let $d_{i,1}=d_T(s_i,L_1)$, $d_{i,2}=d_T(s_i,L_2)$, and $d_i=d_{i,1}+d_{i,2}$.
\begin{claim}\label{claim:canextend}
With high probability, for every vertex $v\in U_1$ not yet covered by the embedding of $T'$,
\[\sum_{i\in[m]:v_i\in N_G(v)}d_i\geq10\mu n.\]
\end{claim}
\begin{proof}[Proof of Claim~\ref{claim:canextend}]
Note that $d_1=d_2=0$. For each $v\in U_1\setminus U_A$ and $i\in [m]$, let $X_{v,i}=d_{i,1}$ if $v_i\in N_G(v)$ and $X_{v,i}=0$ otherwise. For each $v\in U_A$ and $i\in [m]$, let $X_{v,i}=d_{i,2}$ if $v_i\in N_G(v)$ and $X_{v,i}=0$ otherwise.

Regardless of whether~\ref{subcase1} or~\ref{subcase2} holds, for each $v\in U_1\setminus U_A$ and $i\in [m]$, let $s_{j_i}$ be the unique neighbour of $s_i$ to its left, then, using~\ref{embed1} and $|U_2^-\setminus (U_B\cup Z)|=2\alpha n$, 
\begin{align*}
\mathbb{P}(X_{v,i}\neq d_{i,1}\mid X_{v,1},\ldots,X_{v,i-1})&\leq\frac{|N_G(v_{j_i},U_2^-\setminus (N_G(v)\cup U_B\cup Z\cup \{v_1,\ldots,v_{i-1}\}))|}{|N_G(v_{j_i},U_2^-\setminus (U_B\cup Z\cup \{v_1,\ldots,v_{i-1}\}))|}\\
&\leq\frac{\sqrt n}{2\alpha n-\alpha n-\mu n}\leq \frac{1}{n^{1/3}}.
\end{align*}
Let $\ell_1=|L_1|=\sum_{i:s_i\in P_1}d_{i,1}\geq\alpha n/10$. Let $P_1'=\{s_i\in P_1:d_{i,1}\leq n/\log^2n\}$. By Lemma~\ref{lemma:azuma}, if $\sum_{i:s_i\in P_1'}d_{i,1}\geq\ell_1/2$, then for every $v\in U_1\setminus U_A$,
\begin{align}\label{eq:firstsmall}
\mathbb{P}\left(\sum_{i:s_i\in P_1'}X_{v,i}<\alpha n/40\right)\leq\exp\left(-\frac{\ell_1^2}{10^4\sum_{i:s_i\in P_1'}d_{i,1}^2}\right)\leq\exp\left(-\frac{\ell_1^2}{10^4(\frac{\ell_1\log^2n}{n})(\frac n{\log^2n})^2}\right)\leq \frac{1}{n^2}.
\end{align}
Otherwise, $\sum_{i:s_i\in P_1\setminus P_1'}d_{i,1}\geq\ell_1/2$. Since $|P_1\setminus P_1'|\leq \log^2n$, for each $v\in U_1\setminus U_A$, the probability that there is some $J\subset P_1\setminus P_1'$ with $|J|\geq 10$ and $X_{v,i}\neq d_{i,1}$ for each $i\in J$ is at most
\[
(\log^2n)^{10}\cdot \left(\frac{1}{n^{1/3}}\right)^{10}\leq \frac{1}{n^2},
\]
so with probability at least $1-1/n^2$, $\sum_{i:s_i\in P_1\setminus P_1'}X_{v,i}\geq\ell_1/2-10cn\geq\alpha n/40$. Combined with \eqref{eq:firstsmall} and using a union bound, we have with high probability that $\sum_{i\in[m]:v_i\in N_G(v)}d_i\geq\sum_{i:s_i\in P_1}X_{v,i}\geq\alpha n/40\geq10\mu n$ for all $v\in U_1\setminus U_A$.

Now let $v\in U_A$. If~\ref{subcase2} holds, then $v$ is covered by the embedding of $T'$ so there is nothing to prove. Suppose now that~\ref{subcase1} holds. For every $i\in[m]$, note that $d_{i,2}\leq d_i\leq\sqrt n$, and let $s_{j_i}$ be the unique neighbour of $s_i$ to its left. Then, using~\ref{embed2},
\begin{align*}
\mathbb{P}(X_{v,i}=d_{i,2}\mid X_{v,1},\ldots,X_{v,i-1})&\geq\frac{|N_G(v_{j_i},(Z\cap N_G(v))\setminus\{v_1,\ldots,v_{i-1}\})|}{|N_G(v_{j_i},Z\setminus\{v_1,\ldots,v_{i-1}\})|}\\
&\geq\frac{\beta^2n/2-\alpha n-\mu n}{2\beta n}\geq\beta/10.    
\end{align*}
Let $\ell_2=|L_2|=\sum_{i:s_i\in P_2}d_{i,2}\geq\alpha n/10$, then by Lemma~\ref{lemma:azuma},
\[
\mathbb{P}\left(\sum_{i:s_i\in P_2}X_{v,i}<\alpha\beta n/200\right)\leq\exp\left(-\frac{\beta^2\ell_2^2}{10^4\sum_{i:s_i\in P_2}d_{i,2}^2}\right)\leq\exp\left(-\frac{\beta^2\ell_2^2}{10^4(\ell_2/\sqrt n)\cdot (\sqrt n)^2}\right)\leq\frac1{n^2}.
\]
By a union bound, with high probability, all $v\in U_A$ satisfy $\sum_{i\in[m]:v_i\in N_G(v)}d_i\geq\sum_{i:s_i\in P_2}X_{v,i}\geq\alpha\beta n/200\geq10\mu n$.
\renewcommand{\qedsymbol}{$\boxdot$}
\end{proof}
\renewcommand{\qedsymbol}{$\square$}
Finally, it remains to embed $L'$. Let $U$ be the set of unused vertices in $U_1$, note that $|U|\geq|L'|$, and every $v\in U$ satisfies $\sum_{i\in[m]:v_i\in N_G(v)}d_i\geq10\mu n$ by Claim~\ref{claim:canextend}. We verify that Hall's condition holds between the set $\{v_i:s_i\in P\}$ of images of parents of $L'$ and $U$. Indeed, for any non-empty $J\subset P$, if $0<\sum_{i:s_i\in J}d_i\leq |L'|-\mu n$, then as $v_i\in U_2^-$ for each $i\in J$ by~\ref{embed1} and~\ref{embed2}, we have
\[
|N_G(\{v_i:s_i\in J\},U)|\geq |U|-\mu n\geq|L'|-\mu n\geq\sum_{i:s_i\in J}d_i.
\]
If instead $\sum_{i:s_i\in J}d_i>|L'|-\mu n$, then $|N_G(\{v_i:s_i\in J\},U)|=|U|\geq|L'|\geq\sum_{i:s_i\in J}d_i$, as any $v\in U\setminus N_G(\{v_i:s_i\in J\})$ would satisfy $\sum_{i\in[m]:v_i\in N_G(v)}d_i\leq\mu n$, a contradiction. Therefore, by Lemma~\ref{lemma:hallmatching}, we can embed $L'$ into $U$ to finish a copy of $T$ in $G$.
\end{proof}


\subsection{Proof of Theorem~\ref{theorem:extremal:case2}}\label{sec:IIfinal}
Our final step before we can prove Theorem~\ref{theorem:extremal:case2} is to prove the following results which give the desired sparse cut (see Definition~\ref{def:sparsecut}) used in \textbf{Case}~\ref{CaseIIB} and \textbf{Case}~\ref{CaseIIC}.

\begin{proposition}\label{lem:sparsecut}
Let $1/n\ll c\ll\mu\ll\eps\ll 1$. Let $T$ be an $n$-vertex tree with $\Delta(T)\leq cn$ and bipartition classes $V_1$ and $V_2$ satisfying $|V_1|\geq (2-\mu)|V_2|$.
Then, $T$ has an $(\eps,\sqrt n)$-sparse cut $V(T)=A\cup B$, such that at most two vertices in $\{v\in A:d_T(v,B)>0\}$ are adjacent to leaves of $T$ in $A$.
\end{proposition}
\begin{proof}
By Lemma~\ref{lemma:treehalfcomponent}, there is a vertex $v\in V(T)$ such that each component of $T-v$ has size at most $n/2$. Let $\ell\leq cn$ be the number of components in $T-v$ and let $T_1,\ldots,T_\ell$ be these components in order of decreasing size. Let $r$ be maximal subject to $\sum_{i=1}^r|T_i|\leq (2/3-2\eps)n$, and note that $r\geq 1$.

If $|T_{r}|\leq \sqrt{n}$, then as $|T_i|\leq\sqrt n$ for all $i\geq r$, there exists $r'>r$ such that $(1/3+3\eps/2)n\leq\sum_{i=r'}^{\ell}|T_i|\leq (1/3+3\eps/2)n+\sqrt n$, so $A=\{v\}\cup N_T(v)\cup (\cup_{i<r'}V(T_i))$ and $B=V(G)\setminus A$ form an $(\eps,\sqrt n)$-sparse cut. 

Assume now that $|T_r|>\sqrt{n}$, then as $|T_i|>\sqrt n$ for all $i\in[r]$, we have $r\leq\sqrt n$. If, moreover, $\sum_{i\in [r]}|T_i|\geq (1/3+\eps)n$, then $A=\{v\}\cup(\cup_{i=r+1}^{\ell}V(T_i))$ and $B=V(G)\setminus A$ form an $(\eps,\sqrt n)$-sparse cut. Thus, we can assume that $\sum_{i\in [r]}|T_i|<(1/3+\eps)n$.

Now, by the maximality of $r$, we have $|T_{r+1}|\geq (2/3-2\eps)n-(1/3+\eps)n=(1/3-3\eps)n$. As $|T_1|,|T_r|\geq |T_{r+1}|\geq (1/3-3\eps)n$, we must have $r=1$, and $(1/3-3\eps)n\leq |T_2|\leq |T_1|\leq (1/3+\eps)n$.
For each $j\in[2]$, let $v_j$ be the unique neighbour of $v$ in $T_j$, and view $T_j$ as a tree rooted at $v_j$. Let $v_j'$ be a vertex farthest away from $v_j$ in $T_j$ subject to the condition that the subtree $T_j'\subset T_j$ containing $v_j'$ and all of its descendents in $T_j$ has size at least $(1/3-10\eps)n$.

Let $\{S_i:i\in I_1\}$ be the components of $T_1'-v_1'$ and let $\{S_i:i\in I_2\}$ be the components of $T_2'-v_2'$. If $\sum_{i\in I_1:|S_i|<\sqrt{n}}|S_i|\geq5\eps n$, then we can find $I_1'\subset I_1$ such that $|S_i|\leq \sqrt{n}$ for each $i\in I_1'$ and $5\eps n\leq \sum_{i\in I_1'}|S_i|\leq6\eps n$. Then, $B=V(T_2)\cup (\cup_{i\in I_1'}V(S_i)\setminus N_T(v_1'))$ and $A=V(G)\setminus B$ form an $(\eps,\sqrt n)$-sparse cut. Similarly, an $(\eps,\sqrt n)$-sparse cut exists if $\sum_{i\in I_2:|S_i|<\sqrt{n}}|S_i|\geq5\eps n$. Thus, we can assume that $\sum_{i\in I_1:|S_i|\geq\sqrt{n}}|S_i|$ and $\sum_{i\in I_2:|S_i|\geq\sqrt{n}}|S_i|$ are both at least $(1/3-20\eps)n$. Let $I\subset\{i\in I_1\cup I_2:|S_i|\geq\sqrt n\}$ be minimal subject to $\sum_{i\in I}|S_i|\geq(1/3+\eps)n$. Then, minimality and the choices of $v_1',v_2'$ imply that $\sum_{i\in I}|S_i|<(1/3+\eps)n+(1/3-10\eps)n=(2/3-9\eps)n$, so $B=\cup_{i\in I}V(S_i)$ and $A=V(T)\setminus B$ form an $(\eps,\sqrt n)$-sparse cut. 

Therefore, $T$ always contains an $(\eps,\sqrt n)$-sparse cut $A\cup B$. Finally, it is easy to verify that in all cases above, there are at most two vertices in $\{v\in A:d_T(v,B)>0\}$ that are adjacent to leaves of $T$ in $A$.
\end{proof}

\begin{proposition}\label{prop:splitwithsparsecut}
Let $1/n\ll c\ll\mu\ll\alpha\ll \eps\ll 1$. Let $T$ be an $n$-vertex tree with $\Delta(T)\leq cn$ and bipartition classes $V_1$ and $V_2$ satisfying $|V_1|\geq (2-\mu)|V_2|$.
Then, $T$ has an $(\eps,\sqrt n)$-sparse cut $V(T)=A\cup B$ such that at least one of the following holds.
\stepcounter{propcounter}
\begin{enumerate}[label =\emph{\textbf{\Alph{propcounter}\arabic{enumi}}}]
    \item\labelinthm{sparsecut:1} There is a set $L$ of at least $\alpha n$ leaves of $T$ in $V_1$ such that $d_T(u,L)\leq \sqrt{n}$ for each $u\in N_T(L)$. 
    \item\labelinthm{sparsecut:2} $T$ contains a set $V_1'$ of $|\{v\in A:d_T(v,B)>0\}|$ vertices in $V_1$, and a disjoint set $L'$ of $\alpha n$ leaves in $V_1$, such that vertices in $V_1'$ have no common neighbour, $|N_T(V_1')|\leq\eps n$, and $N_T(V_1')\cap N_T(L')=\emptyset$.
\end{enumerate}
\end{proposition}
\begin{proof} 
Let $c\ll\gamma\ll\eps$. We begin with the following claim. 
\begin{claim}\label{sparseclaim}
Let $L$ be a set of leaves of $T$ in $V_1$. Suppose that $V(T)=A'\cup B'$ is a $(2\eps,\sqrt n)$-sparse cut of $T$, and let $A_1=\{a\in A':d_T(a,B')>0\}$. For each $a\in A_1$, let $R_a$ be the component of $T-(A'\setminus A_1)$ containing $a$, and suppose there exists $r_a\in V(R_a-a)\cap V_1$ with $d_T(r_a)\leq1/\gamma$. Moreover, assume that at least $1/3$ of the vertices in $\cup_{a\in A_1}R_a$ are in $L$, then $T$ has an $(\eps,\sqrt n)$-sparse cut $V(T)=A\cup B$ and some corresponding $L'$, $V_1'$ such that~\emph{\ref{sparsecut:2}} holds.
\end{claim}
\begin{proof}[Proof of Claim~\ref{sparseclaim}]
Arrange the components $R_a$, $a\in A_1$, in decreasing order of $|V(R_a)\cap L|/|R_a|$. Take a minimal collection $A_1'\subset A_1$, starting from the elements with the highest ratio, such that $\sum_{a\in A_1'}|V(R_a)\cap L|\geq2\alpha n$.  

Suppose first that $\sum_{a\in A_1'}|V(R_a)\cap L|\leq4\alpha n$. Noting that $(\sum_{a\in A_1}|V(R_a)\cap L|)/(\sum_{a\in A_1}|R_a|)\le \max_{a\in A_1}|V(R_a)\cap L|/|R_a|$, we have $\sum_{a\in A_1'}|R_a|\leq12\alpha n\ll\eps n$. Let $L'\subset\cup_{a\in A_1'}(V(R_a)\cap L)$ have size $\alpha n$, $B=\cup_{a\in A_1\setminus A_1'}V(R_a-a)$, $A=V(G)\setminus B$, $V_1'=\{r_a:a\in A_1\setminus A_1'\}$. Note that $|N_T(V_1')|\leq|A_1|/\gamma\leq2cn/\gamma\ll\eps n$ and $N_T(V_1')$ is disjoint from $N_T(L')$, so~\ref{sparsecut:2} holds with respect to the $(\eps,\sqrt n)$-sparse cut $V(T)=A\cup B$, $L'$, and $V_1'$. 

If $\sum_{a\in A_1'}|V(R_a)\cap L|>4\alpha n$, then the minimality of $A_1'$ implies that there exists $a^*\in A_1'\subset A_1$ satisfying $|V(R_{a^*})\cap L|\geq2\alpha n$. In particular, we can find a set $L'\subset V(R_{a^*})\cap L$ of size $\alpha n$, and some $r'_{a^*}\in V(R_{a^*})\cap(L\setminus L')$, such that $N_T(r'_{a^*})\cap N_T(L')=\emptyset$. Then, let $A=A'$ and $B=B'$, and note that~\ref{sparsecut:2} holds with respect to the $(\eps,\sqrt n)$-sparse cut $V(T)=A\cup B$, $L'$, and $V_1'=\{r_a:a\in A_1\setminus\{a^*\}\}\cup\{r'_{a^*}\}$.
\renewcommand{\qedsymbol}{$\boxdot$}
\end{proof}
\renewcommand{\qedsymbol}{$\square$}

Now, Proposition~\ref{lem:sparsecut} gives a $(100\eps,\sqrt n)$-sparse cut $V(T)=A'\cup B'$, such that at most two vertices in $\{v\in A':d_T(v,B')>0\}$ are adjacent to leaves of $T$ in $A'$. Let $A_1=\{v\in A':d_T(v,B')>0\}$. For each $a\in A_1$, let $R_a$ be the component of $T-(A'\setminus A_1)$ containing $a$. 
Let $A_2$ be the set of $a\in A_1$ such that $R_a-a$ contains at least $\gamma|R_a|$ vertices in $V_1$. Then, for each $a\in A_2$, there exists $r_a\in V(R_a-a)\cap V_1$ such that $d_T(r_a)\leq1/\gamma$.

\medskip

\noindent\textbf{Case I.} $\sum_{a\in A_2}|R_a-a|\geq (1/3+2\eps)n$. By Lemma~\ref{lemma:leaves:V1}, $T$ contains at least $t_1-t_2\geq(1/3-10\mu)n$ leaves in $V_1$. If $A'\setminus A_1$ contains at least $10\alpha n\geq\alpha n+2cn$ leaves of $T$ in $V_1$, then there is a set $L'$ of $\alpha n$ such leaves in $A'\setminus A_1$ with their parents not in $A_1$. Then, let $B=\cup_{a\in A_2}V(R_a-a)$, $A=V(G)\setminus B$, and note that $A$ and $B$ form an $(\eps,\sqrt n)$-sparse cut. Set $V_1'=\{r_a:a\in A_2\}$, then~\ref{sparsecut:2} holds with respect to $A$, $B$, $L'$, and $V_1'$.

If instead $A'\setminus A_1$ contains at most $10\alpha n$ leaves of $T$ in $V_1$, then at least $(1/3-\eps)n$ leaves of $T$ in $V_1$ are in $\cup_{a\in A_1}R_a$, so $\cup_{a\in A_2}R_a$ contains a set $L_1$ of at least $(1/3-10\eps)n$ of leaves of $T$ in $V_1$, as $\cup_{a\in A_1\setminus A_2}R_a$ contains at most $\gamma n$ vertices in $V_1$. In particular, at least $1/3$ of the vertices in $\cup_{a\in A_2}R_a$ are in $L_1$. Let $B''=\cup_{a\in A_2}V(R_a-a)$ and $A''=V(G)\setminus B''$, then we can apply Claim~\ref{sparseclaim} to the $(2\eps,\sqrt n)$-sparse cut $V(T)=A''\cup B''$ to finish the proof.

\medskip

\noindent\textbf{Case II.} $\sum_{a\in A_2}|R_a-a|<(1/3+2\eps)n$, so $\sum_{a\in A_1\setminus A_2}|R_a-a|\geq97\eps n$. Then, $\sum_{a\in A_1\setminus A_2}(|V(R_a-a)\cap V_2|-|V(R_a-a)\cap V_1|)\geq\sum_{a\in A_1\setminus A_2}((1-2\gamma)|R_a|-1)\geq95\eps n$. Consider the subtree $T'$ obtained by removing $R_a-a$ from $T$ for each $a\in A_1\setminus A_2$, and note that at most $|A_1\setminus A_2|\ll\eps n$ leaves in $T'$ are not leaves in $T$. By Lemma~\ref{lemma:leaves:V1}, $T'$ contains at least $|V(T')\cap V_1|-|V(T')\cap V_2|\geq t_1-t_2+95\eps n\geq(1/3+92\eps)n$ leaves in $V_1$. Thus, $T$ contains at least $(1/3+91\eps)n$ leaves in $V_1$.

If there is a set $L$ of at least $\alpha n$ leaves of $T$ in $V_1$ whose parents are all adjacent to at most $\sqrt n$ leaves of $T$, then~\ref{sparsecut:1} holds and we are done, so suppose otherwise. By Lemma~\ref{lemma:treehalfcomponent}, there is a vertex $v$ in $T$ such that every component of $T-v$ has size at most $n/2$. View $T$ as being rooted at $v$. Then, there is a set $L_2$ of at least $(1/3+90\eps)n$ leaves in $V_1$, each of whose parent in $T$ is adjacent to at least $\sqrt{n}$ leaves in $T$, and no leaf in $L_2$ or parent of leaf in $L_2$ is equal to $v$. Let $p_1,\ldots,p_k$ be the parents of $L_2$ in $T$, and note that $k\leq\sqrt n$. Then, for every $i\in[k]$, let $P_i$ be the subtree of $T$ induced by $p_i$ and all of its descendants in $T$, and note that $|P_i|\leq n/2$ as $p_i\not=v$. Let $I\subset[k]$ be the set of indices $i\in[k]$, such that $p_i$ is not a descendant of any $p_j$ with $j\not=i$. Let $I'\subset I$ be minimal subject to $\sum_{i\in I'}|P_i|\geq(1/3+2\eps)n$. 

\medskip

\noindent\textbf{Case II.1.} $\sum_{i\in I'}|P_i|\leq(2/3-2\eps)n$. Let $\overline{B}=\cup_{i\in I'}V(P_i)$ and $\overline{A}=V(T)\setminus \overline{B}$. Note that $T[\overline{A}]$ is a tree, and $V(T)=\overline{A}\cup\overline{B}$ is a $(2\eps,\sqrt n)$-sparse cut. Let $\overline{A}_1=\{a\in \overline{A}:d_T(a,\overline{B})>0\}$. For each $a\in \overline{A}_1$, let $\overline{R}_a$ be the component of $T-(\overline{A}\setminus \overline{A}_1)$ containing $a$, pick any $p_i\in N_T(a,\overline{B})$, and then pick $r_a$ to be any children of $p_i$ in $L_2$.

If $|\overline{B}\cap L_2|\leq(1/3+3\eps)n$, then at least $\eps n$ leaves in $L_2$ are in $\overline{A}$, so we can pick a set $L'$ of $\alpha n$ such leaves. Then, $\overline{A}$, $\overline{B}$, $L'$, and $V_1'=\{r_a:a\in\overline{A}_1\}$ satisfy~\ref{sparsecut:2}.

If $|\overline{B}\cap L_2|>(1/3+3\eps)n$, then at least $1/3$ of the vertices in $\cup_{a\in\overline{A}_1}R_a$ are leaves in $L_2$, so we are done by Claim~\ref{sparseclaim}.

\medskip

\noindent\textbf{Case II.2.} $\sum_{i\in I'}|P_i|>(2/3-2\eps)n$. Then, by minimality, $|P_i|\geq(1/3-4\eps)n$ for each $i\in I$, so $|I|=2$ as $|P_i|\leq n/2$. Without loss of generality, say $I=\{1,2\}$, and note that $|P_1|,|P_2|\leq(1/3+2\eps)n$ by minimality. Since at most $(1/3+2\eps)n$ vertices in $T$ are outside of $P_1\cup P_2$, $P_1\cup P_2$ contains at least $80\eps n$ leaves in $L_2$, so we may assume that, say, $P_1$ contains at least $40\eps n$ such leaves. 

Let $\{q_1,\ldots,q_m\}\subset\{p_3,\ldots,p_k\}$ be the set of descendants of $p_1$ in $T$, and assume they are ordered so that if $q_i$ is a descendant of $q_j$ in $T$, then $i\leq j$. Let $Q_i$ be the subtree of $T$ induced by $q_i$ and all of its descendants in $T$, and let $m'\in[m]$ be minimal subject to $|L_2\cap(\cup_{i=1}^{m'}V(Q_i))|\geq20\eps n$. Note that by minimality and $\Delta(T)\leq cn$, $|L_2\cap(\cup_{i=1}^{m'}V(Q_i))|\leq21\eps n$, and so $|(L_2\cap V(P_1))\setminus(\cup_{i=1}^{m'}V(Q_i))|\geq19\eps n$. Let $\overline{B}=V(P_2)\cup(\cup_{i=1}^{m'}V(Q_i))$, and observe that $(1/3-4\eps)n+20\eps n\leq|\overline{B}|\leq(2/3+4\eps)n-19\eps n$. Let $\overline{A}=V(G)\setminus\overline{B}$, and note that $V(T)=\overline{A}\cup\overline{B}$ form a $(2\eps,\sqrt n)$-sparse cut. Thus, we are now in the same situation as in \textbf{Case II.1}, and can finish the proof in the same way.
\end{proof}


Finally, we can put all the work of this section together to prove Theorem~\ref{theorem:extremal:case2}, following the outline in Section~\ref{sec:outline:extcase2}.

\begin{proof}[Proof of Theorem~\ref{theorem:extremal:case2}]
Let $1/n\ll c\ll\mu \ll 1$ and let $t_1,t_2\in\mathbb{N}$ satisfy $t_1+t_2=n$ and $t_1\geq(2-\mu)t_2$. Let $G$ be a Type II $(\mu,t_1,t_2)$-extremal graph on $\max\{2t_1,t_1+2t_2\}-1$ vertices, so from definition there are disjoint subsets $U_1,U_2\subset V(G)$ such that $|U_1|,|U_2|\geq (1-\mu)t_1$, and for each $i\in [2]$ and $u\in U_i$, $\red{d}(u,U_i)\leq \mu n$ and $\blue{d}(u,U_{3-i})\leq\mu n$. Let $T$ be an $n$-vertex tree with $\Delta(T)\leq cn$ and bipartition classes $V_1$ and $V_2$ with $|V_i|=t_i$ for each $i\in [2]$. We need to find a monochromatic copy of $T$ in $G$.

Let $\mu \ll \beta \ll 1$. Let $U_1^+,U_2^+\subset V(G)$ be maximal disjoint sets with $U_1\subset U_1^+$, $U_2\subset U_2^+$, and $\red{d}(u,U_{3-i})\geq \beta n$ for every $i\in [2]$ and $u\in U_i^+$. Note that $|U_i^+\setminus U_i|\leq2\mu n$ for each $i\in[2]$. By relabelling if necessary, we can assume that $|U_1^+|\geq |U_2^+|$.  

First (for \textbf{Case}~\ref{CaseIIY}), suppose that there are distinct vertices $v_1,v_2\in V(G)\setminus (U_1^+\cup U_2^+)$. By the maximality of $U_1^+$ and $U_2^+$, we have that $d_{\mathrm{blue}}(v_i,U_j)\geq |U_j|-\beta n$ for each $i,j\in [2]$. As $t_1\geq(2-\mu)t_2$, we have $|U_i|\geq t_1-\mu n\geq (2/3-10\mu)n$ for each $i\in [2]$.
Since $\delta(G_{\mathrm{blue}}[U_i])\geq |U_i|-\mu n$ for each $i\in [2]$, we can apply Lemma~\ref{lem:caseIIY} to find a copy of $T$ in $\blue{G}$.

Thus, we can assume that $|V(G)\setminus (U_1^+\cup U_2^+)|\leq 1$. It follows that $|U_1^+|\geq\ceil{(|G|-1)/2}\geq\ceil{2n/3}-1$, and $|U_1^+|\geq t_1-1$. Next (for \textbf{Case}~\ref{CaseIIA}), suppose there is some vertex $v\in V(G)$ with at least $\beta n$ red neighbours in both $U_1^+$ and $U_2^+$. Then, $v$ has at least $\beta n-2\mu n\geq\mu n$ red neighbour in both $U_1$ and $U_2$. Therefore, by Lemma~\ref{lem:caseIIA}, $G$ contains a red copy of $T$.

Hence, we can assume there is no such vertex $v$, from which we get $\delta(G_{\mathrm{blue}}[U_i^+])\geq|U_i^+|-\beta n$ for each $i\in [2]$. Next (for \textbf{Case}~\ref{CaseIIZ}), suppose $G[U_1^+]$ contains at most $10^6n$ red edges and there is exactly one vertex $w$ in $V(G)\setminus (U_1^+\cup U_2^+)$. Using the maximality like above, $\blue{d}(w,U_i^+)\geq|U_i^+|-2\beta n$ for each $i\in[2]$, so we can find a blue copy of $T$ in $G$ using Lemma~\ref{lem:caseIIZ}.

Thus, we can assume that either $G[U_1^+]$ has more than $10^6n$ red edges, or $V(G)\setminus (U_1^+\cup U_2^+)=\emptyset$. Let $\beta\ll\eps\ll1$. Let $V(T)=A\cup B$ be an $(\eps,\sqrt n)$-sparse cut given by Proposition~\ref{prop:splitwithsparsecut}.
Suppose (for \textbf{Case}~\ref{CaseIIB}) that $T[A,B]$ can be embedded into $G_{\mathrm{blue}}[U_1^+,U_2^+]$, either with $A$ embedded into $U_1^+$ and $B$ embedded into $U_2^+$, or the other way around, then $G$ contains a blue copy of $T$ by Lemma~\ref{lem:caseIIB}.

Finally, suppose (for \textbf{Case}~\ref{CaseIIC}) that $T[A,B]$ cannot be embedded into $G_{\mathrm{blue}}[U_1^+,U_2^+]$. Let $\ell=|\{v\in A:d_T(v,B)>0\}|$ and list the elements in $\{v\in A:d_T(v,B)>0\}$ as $a_1,\ldots,a_\ell$. For each $i\in [\ell]$, let $d_i=d_T(a_i,B)\leq\sqrt n$. 

If $|U_1^+|\geq t_1$, let $I\subset [\ell]$ be a maximal set for which there are distinct vertices $\{w_i:i\in I\}\subset U_1^+$ and disjoint subsets $\{W_i\subset U_2^+:i\in I\}$, such that $W_i\subset N_{\mathrm{blue}}(w_i)$ and $|W_i|=d_i$ for each $i\in I$. As we are not in \textbf{Case}~\ref{CaseIIB}, $|I|<\ell$. Let $U_A=\{w_i:i\in I\}$ and $U_B=\cup_{i\in I}W_i$. Then, the maximality of $I$ implies that every vertex in $U_1^+\setminus U_A$ has at most $\sqrt{n}$ blue neighbours in $U_2^+\setminus U_B$. Therefore, by Lemma~\ref{lem:caseIIC}, $G$ contains a copy of $T$ in red. 

If $|U_1^+|=t_1-1$, then we must have $|V(G)\setminus(U_1^+\cup U_2^+)|=1$ and $|U_2^+|=t_1-1$. Since we are not in \textbf{Case}~\ref{CaseIIZ}, $G[U_1^+]$ has at least $10^6n$ red edges. We now proceed as above but swapping the role of $U_1^+$ and $U_2^+$. Let $I\subset [\ell]$ be a maximal set for which there are distinct vertices $\{w_i:i\in I\}\subset U_2^+$ and disjoint subsets $\{W_i\subset U_1^+:i\in I\}$, such that $W_i\subset N_{\mathrm{blue}}(w_i)$ and $|W_i|=d_i$ for each $i\in I$. Let $U_A=\{w_i:i\in I\}$ and $U_B=\cup_{i\in I}W_i$. Then, $|I|<\ell$ and the maximality of $I$ implies that every vertex in $U_2^+\setminus U_A$ has at most $\sqrt{n}$ blue neighbours in $U_1^+\setminus U_B$. Therefore, by Lemma~\ref{lem:caseIIC}, $G$ contains a copy of $T$ in red. This completes the proof of Theorem~\ref{theorem:extremal:case2}.
\end{proof}

\bibliographystyle{abbrv}
\bibliography{Ramseytrees}

\begin{thebibliography}{10}

\bibitem{balister2024upper}
P.~Balister, B.~Bollob{\'a}s, M.~Campos, S.~Griffiths, E.~Hurley, R.~Morris,
  J.~Sahasrabudhe, and M.~Tiba.
\newblock Upper bounds for multicolour {R}amsey numbers.
\newblock {\em arXiv:2410.17197}, 2024.

\bibitem{besomi2019degree}
G.~Besomi, M.~Pavez-Sign{\'e}, and M.~Stein.
\newblock Degree conditions for embedding trees.
\newblock {\em SIAM Journal on Discrete Mathematics}, 33(3):1521--1555, 2019.

\bibitem{bollobás1998modern}
B.~Bollob{\'a}s.
\newblock {\em Modern Graph Theory}.
\newblock Springer, 1998.

\bibitem{BONDY197346}
J.~Bondy and P.~Erd\H{o}s.
\newblock Ramsey numbers for cycles in graphs.
\newblock {\em Journal of Combinatorial Theory, Series B}, 14(1):46--54, 1973.

\bibitem{Burr1974}
S.~A. Burr.
\newblock Generalized {R}amsey theory for graphs - a survey.
\newblock In {\em Graphs and Combinatorics}, pages 52--75. Springer, 1974.

\bibitem{Burr1973ONTM}
S.~A. Burr and P.~Erd{\H o}s.
\newblock On the magnitude of generalized {R}amsey numbers for graphs.
\newblock In {\em Infinite and finite sets}, pages 215--240. J{\'a}nos Bolyai
  Mathematical Society, 1975.

\bibitem{burr1976extremal}
S.~A. Burr and P.~Erd{\H o}s.
\newblock Extremal {R}amsey theory for graphs.
\newblock {\em Utilitas Mathematica}, 9:247--258, 1976.

\bibitem{campos2023}
M.~Campos, S.~Griffiths, R.~Morris, and J.~Sahasrabudhe.
\newblock An exponential improvement for diagonal {R}amsey.
\newblock {\em arXiv:2303.09521}, 2023.

\bibitem{chvatal1983ramsey}
V.~Chv{\'a}tal, V.~R{\"o}dl, E.~Szemer{\'e}di, and W.~T. Trotter~Jr.
\newblock The {R}amsey number of a graph with bounded maximum degree.
\newblock {\em Journal of Combinatorial Theory, Series B}, 34(3):239--243,
  1983.

\bibitem{dubo2025ramsey}
F.~F. Dub{\'o} and M.~Stein.
\newblock On the {R}amsey number of the double star.
\newblock {\em Discrete Mathematics}, 348(1):114227, 2025.

\bibitem{erdos1947some}
P.~Erd{\H{o}}s.
\newblock Some remarks on the theory of graphs.
\newblock {\em Bulletin of the {A}merican {M}athematical {S}ociety},
  53(4):292--294, 1947.

\bibitem{erdHos1982ramsey}
P.~Erd{\H{o}}s, R.~J. Faudree, C.~C. Rousseau, and R.~H. Schelp.
\newblock {R}amsey numbers for brooms.
\newblock {\em Congressus Numerantium}, 35:283--293, 1982.

\bibitem{erdHos1995discrepancy}
P.~Erd{\H{o}}s, Z.~F{\"u}redi, M.~Loebl, and V.~T.~S{\'o}s.
\newblock Discrepancy of trees.
\newblock {\em Studia Scientiarum Mathematicarum Hungarica}, 30(1-2):47--57,
  1995.

\bibitem{erdos1935combinatorial}
P.~Erd{\H{o}}s and G.~Szekeres.
\newblock A combinatorial problem in geometry.
\newblock {\em Compositio {M}athematica}, 2:463--470, 1935.

\bibitem{faudree1974all}
R.~J. Faudree and R.~H. Schelp.
\newblock All {R}amsey numbers for cycles in graphs.
\newblock {\em Discrete Mathematics}, 8(4):313--329, 1974.

\bibitem{gerencser1967ramsey}
L.~Gerencs\'er and A.~Gy\'arf\'as.
\newblock On {R}amsey-type problems.
\newblock {\em Annales Universitatis Scientiarium Budapestinensis de Rolando
  Eötvös Nominatae, Sectio Mathematica}, 10:167--170, 1967.

\bibitem{GROSSMAN1979247}
J.~W. Grossman, F.~Harary, and M.~Klawe.
\newblock Generalized {R}amsey theory for graphs, x: double stars.
\newblock {\em Discrete Mathematics}, 28(3):247--254, 1979.

\bibitem{gupta2024optimizing}
P.~Gupta, N.~Ndiaye, S.~Norin, and L.~Wei.
\newblock Optimizing the {CGMS} upper bound on {R}amsey numbers.
\newblock {\em arXiv:2407.19026}, 2024.

\bibitem{hall1935representatives}
P.~Hall.
\newblock On representatives of subsets.
\newblock {\em Journal of the London Mathematical Society}, s1-10(1):26--30,
  1935.

\bibitem{Harary1972}
F.~Harary.
\newblock Recent results on generalized {R}amsey theory for graphs.
\newblock In {\em Graph Theory and Applications}, pages 125--138. Springer,
  1972.

\bibitem{haxell2002ramsey}
P.~E. Haxell, T.~{\L}uczak, and P.~W. Tingley.
\newblock {R}amsey numbers for trees of small maximum degree.
\newblock {\em Combinatorica}, 22(2):287--320, 2002.

\bibitem{JLR2000}
S.~Janson, T.~{\L}uczak, and A.~Ruci{\'n}ski.
\newblock {\em Random graphs}.
\newblock Wiley-Interscience, 2000.

\bibitem{komlos2001spanning}
J.~Koml{\'o}s, G.~N. S{\'a}rk{\"o}zy, and E.~Szemer{\'e}di.
\newblock Spanning trees in dense graphs.
\newblock {\em Combinatorics, Probability and Computing}, 10(5):397--416, 2001.

\bibitem{komlos1995szemeredi}
J.~Koml{\'o}s and M.~Simonovits.
\newblock Szemer{\'e}di's {R}egularity {L}emma and its applications in graph
  theory.
\newblock In {\em Combinatorics, {P}aul {E}rd{\H o}s is eighty, {V}olume 2},
  pages 295--352. J{\'a}nos Bolyai Mathematical Society, 1996.

\bibitem{Krivelevichtrees}
M.~Krivelevich.
\newblock Embedding spanning trees in random graphs.
\newblock {\em SIAM Journal on Discrete Mathematics}, 24(4):1495--1500, 2010.

\bibitem{lee2017ramsey}
C.~Lee.
\newblock {R}amsey numbers of degenerate graphs.
\newblock {\em Annals of Mathematics}, 185(3):791--829, 2017.

\bibitem{McDiarmid_1989}
C.~McDiarmid.
\newblock On the method of bounded differences.
\newblock In {\em Surveys in Combinatorics, 1989}, pages 148--188. Cambridge
  University Press, 1989.

\bibitem{montgomery2019spanning}
R.~Montgomery.
\newblock Spanning trees in random graphs.
\newblock {\em Advances in Mathematics}, 356:106793, 2019.

\bibitem{norin2016asymptotics}
S.~Norin, Y.~R. Sun, and Y.~Zhao.
\newblock Asymptotics of {R}amsey numbers of double stars.
\newblock {\em arXiv:1605.03612}, 2016.

\bibitem{pokrovskiy2024hyperstability}
A.~Pokrovskiy.
\newblock Hyperstability in the {E}rd{\H o}s-{S}{\'o}s conjecture.
\newblock {\em arXiv:2409.15191}, 2024.

\bibitem{radziszowski2012small}
S.~Radziszowski.
\newblock Small {R}amsey numbers.
\newblock {\em The Electronic Journal of Combinatorics}, DS1, 2024.

\bibitem{RamseyOnAP}
F.~P. Ramsey.
\newblock On a problem of formal logic.
\newblock {\em Proceedings of The London Mathematical Society},
  s2-30(1):264--286, 1930.

\bibitem{rosta1973ramsey}
V.~Rosta.
\newblock On a {R}amsey-type problem of {J. A.} {B}ondy and {P}. {E}rd{\H{o}}s.
  {II}.
\newblock {\em Journal of Combinatorial Theory, Series B}, 15(1):105--120,
  1973.

\bibitem{stein2020tree}
M.~Stein.
\newblock Tree containment and degree conditions.
\newblock In {\em Discrete Mathematics and Applications}, pages 459--486.
  Springer, 2020.

\bibitem{wormald1999differential}
N.~C. Wormald.
\newblock The differential equation method for random graph processes and
  greedy algorithms.
\newblock In {\em Lectures on Approximation and Randomized Algorithms}, pages
  73--155. Polish Scientific Publishers, 1999.

\bibitem{zhao2011proof}
Y.~Zhao.
\newblock Proof of the $(n/2-n/2-n/2)$ conjecture for large $ n$.
\newblock {\em The Electronic Journal of Combinatorics}, 18(1):P27, 2011.

\end{thebibliography}

\end{document}